\documentclass[12pt]{amsart}

\usepackage{fullpage}
\usepackage{paralist}
\usepackage{graphicx}
\usepackage{hyperref}

 \usepackage{amsmath,amssymb,amscd,amstext,amsthm,amsfonts}
 \usepackage{latexsym,float,graphics,color,epsfig,euscript,subfigure,wrapfig,ifthen}
 \usepackage{extarrows}
 \usepackage{stmaryrd}

\usepackage{paralist,graphicx}
\usepackage{cite,url,hyperref}
\usepackage{tikz,pgfplots}
\usetikzlibrary{calc,matrix,arrows,angles,patterns, quotes,intersections, 3d}

\usepackage{relsize}

\theoremstyle{plain}
\newtheorem{theorem}{Theorem}[section]
\newtheorem{lemma}[theorem]{Lemma}
\newtheorem{proposition}[theorem]{Proposition}
\newtheorem{corollary}[theorem]{Corollary}

\theoremstyle{definition}
\newtheorem{definition}[theorem]{Definition}
\newtheorem{example}[theorem]{Example}

\theoremstyle{remark}
\newtheorem{remark}{Remark}

\DeclareMathOperator{\Div}{\mathrm{Div}}
\DeclareMathOperator{\divf}{\mathrm{div}}
\DeclareMathOperator{\RDiv}{\mathbb{R}\mathrm{Div}}
\DeclareMathOperator{\DivPlus}{\mathrm{Div}_+}
\DeclareMathOperator{\DivPlusD}{\mathrm{Div}_+^d}
\DeclareMathOperator{\DivPlusThree}{\mathrm{Div}_+^3}
\DeclareMathOperator{\DivD}{\mathrm{Div}^d}
\DeclareMathOperator{\DivPlusR}{\mathrm{Div}_+^r}

\DeclareMathOperator{\RDivPlus}{\mathbb{R}\mathrm{Div}_+}
\DeclareMathOperator{\RDivPlusD}{\mathbb{R}\mathrm{Div}_+^d}
\DeclareMathOperator{\RDivPlusM}{\mathbb{R}\mathrm{Div}_+^m}
\DeclareMathOperator{\RDivPlusDM}{\mathbb{R}\mathrm{Div}_+^{cd+m}}
\DeclareMathOperator{\RDivD}{\mathbb{R}\mathrm{Div}^d}
\DeclareMathOperator{\mR}{\mathbb{R}}
\DeclareMathOperator{\mZ}{\mathbb{Z}}

\DeclareMathOperator{\CPA}{\mathrm{CPA}(\Gamma)}

\DeclareMathOperator{\measure}{\mathrm{Meas}^0}

\DeclareMathOperator{\Red}{\mathrm{Red}}

\DeclareMathOperator{\dist}{\mathrm{dist}_\Gamma}

\DeclareMathOperator{\Mod}{\mathrm{mod}}

\DeclareMathOperator{\Gmin}{\Gamma_{\mathrm{min}}}
\DeclareMathOperator{\Gmax}{\Gamma_{\mathrm{max}}}
\DeclareMathOperator{\supp}{\mathrm{supp}}

\DeclareMathOperator{\Tan}{\operatorname{Tan}}

\newcommand{\mbbN}{{\mathbb N}}

\newcommand{\mbbP}{{\mathbb P}}

\newcommand{\mbbR}{{\mathbb R}}

\newcommand{\mbbZ}{{\mathbb Z}}
\newcommand{\mbbTP}{\mathbb{TP}}
\newcommand{\mbbTPP}{\mathbb{TP}^0}

\newcommand{\mcalB}{{\mathcal B}}
\newcommand{\mcalC}{{\mathcal C}}

\newcommand{\mcalL}{{\mathcal L}}

\newcommand{\mcalN}{{\mathcal N}}

\newcommand{\mcalU}{{\mathcal U}}

\newcommand{\mfrakC}{{\mathfrak C}}

\newcommand{\ve}{{\mathbf e}}
\newcommand{\vt}{{\mathbf t}}

\newcommand{\outdeg}{\mathrm{outdeg}}

\newcommand{\Xmin}{X_{\min}}
\newcommand{\Xmax}{X_{\max}}

\newcommand{\minplus}{\oplus}
\newcommand{\maxplus}{\boxplus}
\newcommand{\hminplus}{\widehat{\oplus}}
\newcommand{\hmaxplus}{\widehat{\boxplus}}

\newcommand{\lowerpi}{\underline{\pi}}
\newcommand{\upperpi}{\overline{\pi}}

\newcommand{\lowerTheta}{\underline{\Theta}}
\newcommand{\upperTheta}{\overline{\Theta}}

\newcommand{\lowereta}{\underline{\eta}}
\newcommand{\uppereta}{\overline{\eta}}

\newcommand{\tconv}{\mathrm{tconv}}
\newcommand{\lowertconv}{\underline{\mathrm{tconv}}}
\newcommand{\uppertconv}{\overline{\mathrm{tconv}}}
\newcommand{\loweruppertconv}{\underline{\overline{\mathrm{tconv}}}}
\newcommand{\lowerTCONV}{\underline{\mathrm{TCONV}}}
\newcommand{\upperTCONV}{\overline{\mathrm{TCONV}}}
\newcommand{\lowerupperTCONV}{\underline{\overline{\mathrm{TCONV}}}}

\newcommand{\cl}{cl}

\title{Idempotent Analysis, Tropical Convexity and Reduced Divisors}
\author{Ye Luo}
\address{School of Information Science and Engineering, Xiamen University, Xiamen, Fujian 361005, China}
\email{luoye80@gmail.com, luoye@xmu.edu.cn}


\date{}
\begin{document}

\begin{abstract}
We investigate a canonical extension of  a conventional combinatorial notion of reduced divisors  to  a notion of tropical projections, which  can be defined as the unique minimizers of the so-called $B$-pseudonorms with respect to compact tropical convex sets.  In this paper, we build the foundation of a theory of idempotent analysis using tropical projections and obtain a series of subsequent results, e.g. tropical retracts, construction of compact tropical convex sets and a set-theoretical characterization of tropical weak independence. In particular, we prove a tropical version of Mazur's Theorem on closed tropical convex hulls and discover a fixed point theorem for tropical projections. As the main application of our machinery  of  tropical convexity analysis, we investigate the divisor theory on metric graphs based on tropical projections. We extend the notion of linear systems and redefine the notion of reduced divisors to all linear systems instead of only to complete linear systems. Moreover, we explore the correspondence between reduced divisor maps to dominant tropical trees and harmonic morphisms to metric trees. Furthermore, we propose a notion called the geometric rank for linear systems on metric graphs which resolves the discrepancy between the interpretations of gonality of metric graphs using the conventional Baker-Norine rank function and using harmonic morphisms to  metric trees.
\end{abstract}

\maketitle

\section{Introduction}
In this paper, we present  a new framework for idempotent  analysis over tropical semirings based on a notion called  tropical projection. The backgrounds of this paper come from a few areas with intrinsic connections:

\begin{enumerate}
\item \emph{Idempotent analysis.}  Idempotent analysis is an analysis theory over idempotent semirings, developed by Maslov and his collaborators in 1980s originally as a framework to describe the semicalssical limit of quantum mechanics via Maslov dequantization \cite{KM97,LM05}. Equivalent under negation, the min-plus algebra $(\mbbR_{\min},\odot,\minplus)$ and the max-plus algebra $(\mbbR_{\max},\odot,\maxplus)$ where $\mbbR_{\min}=\mbbR\bigcup\{+\infty\}$, $\mbbR_{\max}=\mbbR\bigcup\{-\infty\}$, $a\odot b:=a+b$, $a\minplus b:=\min(a,b)$ and $a\maxplus b:=\max(a,b)$ are the most intensively studied idempotent semirings. ($(\mbbR_{\min},\odot,\minplus)$ and $(\mbbR_{\max},\odot,\maxplus)$ are also commonly called tropical semirings nowadays and sometimes the addition identities $+\infty$ and $-\infty$ are ignored.)  Under a correspondence principle of idempotent analysis \cite{LM98}, many important results over the field of real or complex numbers have counterparts over idempotent semirings, while the most striking observation is probably that the Legendre transform is just the counterpart of  the Fourier transform in idempotent analysis. 

\item \emph{Tropical geometry and tropical convexity.} 
As a rapidly developing subject in mathematics, tropical geometry is a theory of geometry over tropical semirings which can be described as a piece-wise linear version of algebraic geometry \cite{MS15, Viro08}. A framework for tropical geometry has been developed by Kontsevich-Soibelman \cite{KS01} and Mikhalkin \cite{Mikhalkin05} with applications in enumerative algebraic geometry and homological mirror symmetry, and another framework has been developed by Baker-Payne-Rabinoff \cite{BPR16,BPR13} based on Berkovich analytification \cite{Berkovich12}. The ``linearity'' features of tropical geometry can be captured by the notion of tropical convexity discussed by Develin and Sturmfels  \cite{DS04}, which is intrinsically related to max-linear (or min-linear) systems \cite{Butkovivc10}.

\item \emph{Chip-firing games, the divisor theory on finite graphs/metric graphs  and reduced divisors.}
The chip-firing game on a graph (or the abelian sandpile model) has become a huge topic in combintorics \cite{CP18}, partly because it is related to many areas of mathematics and theoretical physics. Baker and Norine \cite{BN07} have discovered that the classical Riemann--Roch theorem and Abel-Jacobi theory for algebraic curves have analogues for finite graphs with a motivation of degenerating divisors on an algebraic curve $C$ to divisors on the dual graph of the special fiber of a semistable model of $C$.  Their result was immediately extended to the context of abstract tropical curves which are essentially metric graphs  \cite{GK08}. People soon realize that this remarkable work should be unified into the world of tropical geometry under the framework in \cite{BPR16}, and tropical proofs of some theorems in conventional algebraic geometry have been obtained. For example, Cools, Draisma, Payne and Robeva \cite{CDPR12} have derived a new proof of the Brill--Noether Theorem based on Baker's specialization lemma  \cite{Baker08} and an explicit computation of the Baker-Norine combinatorial rank function for a special type of metric graphs. Note that because the degeneration of divisors on algebraic curves to finite graphs or metric graphs has some subtlety, even though the graph-theoretical Riemann--Roch  in \cite{BN07} is formulated after the algebro-geometric Riemann--Roch, one can hardly borrow techniques from algebraic geometry and there are only purely combinatorial proofs of the graph-theoretical Riemann--Roch so far. Actually the divisor theory on finite graphs or metric graphs is closely related to the chip-firing games, and the most important tool in Baker-Norine's original proof of the graph-theoretical Riemann--Roch and many follow-up works is based on a notion called reduced divisors (also named G-parking functions in combinatorics \cite{PS04}). 
\end{enumerate}

Let $X$ be a topological space and $BC(X)$ be the the space of all bounded and continuous real functions on $X$. Instead of equipping $\mbbR$  (when the addition identities $\pm\infty$ are ignored)  with only one of min-plus algebra and max-plus algebra like most other works, both $\minplus$ and $\maxplus$ (called lower tropical addition and upper tropical addition respectively) are put into our scenario. Then operations on $BC(X)$ can be induced from operations on $\mbbR$. For $f,g\in BC(X)$ and $c\in\mbbR$, the lower tropical addition $f\minplus g:=\min(f,g)$, upper tropical addition $f\maxplus g := \max(f,g)$ and tropical scalar multiplication $c\odot f  :=c+f$ are simply defined by pointwise operations (see details in Subsection~\ref{SS:TOper}). The tropical projective space $\mbbTP(X)$ is defined as $BC(X)$ modulo tropical scalar multiplication. Also, a subspace of $BC(X)$ closed under tropical scalar multiplication and lower tropical addition (respectively upper tropical addition) modulo tropical scalar multiplication is a subspace of $\mbbTP(X)$ which is said to be lower tropically convex (respectively upper tropically convex). Note that conventionally in most works on tropical convexity, only the finite dimensional case is explored, i.e., when $X$ is a finite set equipped with discrete topology in our setting. Here we don't have such a restriction. Moreover, $\mbbTP(X)$ is a normed space whose norm is naturally defined as $\Vert [f]\Vert=\max(f)-\min(f)$ where $[f]\in\mbbTP(X)$ is the equivalence class of $f\in BC(X)$ (Definition~\ref{D:TropNorm}). Actually $(\mbbTP(X),\Vert\cdot \Vert)$ is also a Banach space (Proposition~\ref{P:TPBanach}).

For an element $\gamma\in \mbbTP(X)$ and a compact subset $T$ of $\mbbTP(X)$ which is lower (or upper) tropically convex, unlike the case in conventional convex geometry, it is not generally true that the minimizer of the distance function between elements in $T$ and $\gamma$ is a singleton.
Actually because of  the degeneration nature of ``tropicalization'', similar phenomena are quite common in the ``tropical world''. In other words, one may only be able to develop a coarse analysis theory from the notion of tropical norm. We resolve this non-uniqueness issue by introducing a notion of $B$-pseudonorms, which enables use to build a more refined analysis theory on a tropical projective space as desired. 

More precisely, the $B$-pseudonorms on $\mbbTP(X)$ with respect to some Borel measure $\mu$ on $X$ (Definition~\ref{D:pseudonorm}) are a series of pseudonorms $\llfloor \cdot \rrfloor_p$ and $\llceil \cdot \rrceil_p$ for $p\in[1,\infty]$. Here ``pseudo'' means not necessarily symmetric, i.e., $\llfloor \alpha \rrfloor_p\neq \llfloor -\alpha \rrfloor_p$ and $\llceil \alpha \rrceil_p\neq \llceil -\alpha \rrceil_p$ for $p\in[1,\infty)$ in general. But we have $\llfloor \alpha \rrfloor_p = \llceil -\alpha \rrceil_p$ and $\llfloor \cdot \rrfloor_\infty=\llceil \cdot \rrceil_\infty=\Vert \cdot \Vert$. One of the main results of this paper is the following theorem. 

(\textbf{A Restatement of Theorem~\ref{T:main}})  
For a compact lower (respectively upper) tropically convex subset of  $T$ of $\mbbTP(X)$ and an arbitrary element $\gamma$ in $\mbbTP(X)$,  there exists a unique element $\lowerpi_T(\gamma)$ (respectively $\upperpi_T(\gamma)$) in $T$ which minimizes all $B$-pseudonorm functions $\alpha\mapsto\llfloor \alpha-\gamma\rrfloor_p$ (respectively $\alpha\mapsto\llceil \alpha-\gamma\rrceil_p$) with $\alpha\in T$ for $p\in[1,\infty)$. 

Note that in the above theorem, $\lowerpi_T(\gamma)$ and $\upperpi_T(\gamma)$ do not depend on $p$ as long as $p\in[1,\infty)$, and are called lower and upper tropical projections of $\gamma$ to $T$ respectively. Actually such tropical projections are even more intrinsic. Using the criteria of tropical projections in Corollary~\ref{C:CritTropProj}, we see that $\lowerpi_T(\gamma)$ and $\upperpi_T(\gamma)$ can be characterized purely set-theoretically which means that $\lowerpi_T(\gamma)$ and $\upperpi_T(\gamma)$ are even independent of the underlying measure $\mu$ of $X$ (Remark~\ref{R:TropProjMeas}). 

Recall that the notion of $b$-functions is introduced in \cite{BS13} where reduced divisors can be characterized as the minimizers of  $b$-functions. In our scenario, $b$-functions are special cases of $B$-pseudonorms and thus reduced divisors are special cases of tropical projections (see detailed discussions in Subsection~\ref{SS:bFuncBPseudoNorm}). Just as techniques based on reduced divisors  have shown  tremendous power in studying the divisor theory on finite graphs and metric graphs, the notion of tropical projections and its features are also extremely useful for exploring the theory of tropical convexity analysis. Here we summarize some of the results that we have derived:

\begin{enumerate}
\item We show that for a lower (respectively upper) tropical convex set $W$ and a compact lower (respectively upper) tropical convex subset $T$ of $W$, there is always a strong deformation retraction from $W$ to $T$ such that the intermediate sets are also lower (respectively upper) tropical convex (Theorem~\ref{T:TropRetract}). A direct corollary is that $W$ itself must be contractible (Corollary~\ref{C:TropContr}).  
\item We propose a systematic approach to construct compact tropical convex sets (Theorem~\ref{T:construction}) which has a consequence that finitely generated tropical convex hulls (also called tropical polytopes) are always compact (Corollary~\ref{C:Polytope}). Moreover, we prove the following tropical version of Mazur's Theorem on conventional closed convex hulls (Theorem~\ref{T:TropMazur}): the closed tropical convex hull of a compact subset of $\mbbTP(X)$ is also compact. 
\item We study several notions of independence (Definition~\ref{D:TropIndep} and Remark~\ref{R:TropIndep}) in the context of tropical projective spaces including tropical weak independence, Gondran-Minoux independence and tropical independence (the last one is essential in Jensen-Payne's tropical proofs of the Gieseker-Petri Theorem \cite{JP14}  and the maximal rank conjecture for quadrics \cite{JP16}). In particular, we provide a purely set-theoretical criterion for tropical weak independence (Theorem~\ref{T:CriterionTropIndep}) and discuss the extremals of tropical convex sets (Theorem~\ref{T:extremal}).
\item We prove the following fixed-point theorem about tropical projections (Theorem~\ref{T:FixedPoint}): the tropical projections bouncing back and forth between a compact lower tropically convex set and a compact upper tropically convex set stabilize after at most two steps. Note that unlike most of the other results in this paper which have statements for lower and upper tropical convexity separately, this fixed-point theorem involves both lower and upper tropical convex sets. 
\end{enumerate}

We apply this machinery of tropical convexity analysis to the divisor theory on metric graphs (or abstract tropical curves). 

\begin{enumerate}
\item Using the potential theory on metric graphs (Appendix~\ref{S:potential}) and the language of chip-firing moves, we convert several definitions and statements in the previously developed theory of tropical convexity analysis to those in the context of divisor theory on metric graphs, e.g., tropical convexity, tropical projection and its criterion (Theorem~\ref{T:TropProjDiv}), and tropical independence and its criterion (Theorem~\ref {T:FiniteCriterion}). 
\item In the divisor theory of algebraic curves, for each divisor $D$ on an algebraic curve,  the complete linear system $|D|$ associated to $D$ is the projective space consisting of all effective divisors linearly equivalent to $D$, and a linear subspace of $|D|$ is called a linear system whose rank is simply its dimension. However, different complete linear systems on different algebraic curves can degenerate as subsets of the same complete linear system on a single finite graph or metric graph. As a result, even through we may still define the complete linear system $|D|$ associate to a divisor $D$ on a metric graph as the set of all effective divisors linearly equivalent to $D$ (as in \cite{BN07} and the whole subsequent works), $|D|$ is  not purely dimensional in general. This is an obstacle in defining linear systems as the ``linear'' subspaces of $|D|$  as in the case of algebraic curves.  Here we note that $|D|$ is a tropical polytope and propose to define the  linear systems on a metric graph as the tropical polytopes contained in complete linear systems (Definition~\ref{D:LinSys}). 
\item For each point $q$ on a metric graph $\Gamma$ and each nonempty complete linear system $|D|$ on $\Gamma$,  there exists a unique divisor called the $q$-reduced in $|D|$ (Definition~\ref{D:RedDiv}) which is characterized in \cite{BS13} as the minimizer of the so-called $b$-functions restricted to $|D|$ with respect to $q$. Based on an observation that  $b$-functions are just special $B$-pseudonorms, we give the following definition of reduced divisors for all linear systems (Definition~\ref{D:GeneralRed}) instead of only for complete ones as in convention: the $q$-reduced divisor in a linear system $T$ is the tropical projection of $d\cdot(q)$ to $T$ where $d$ is the degree of $T$. In this sense, we naturally extend the notion of reduced divisor maps introduced in \cite{Amini13} which is originally defined as a map from $\Gamma$ to a complete linear system $|D|$ to a more general scope where the target can be any linear system (Definition~\ref{D:RedDivMap}). 
\item We pay special attention to those one dimensional linear systems which we call tropical trees (Definition~\ref{D:TropTree}). Using the techniques based on tropical projections and reduced divisors, we provide criteria for (1) tropical trees (Proposition~\ref{P:TropTreeCrit}), (2) dominant tropical trees which are tropical trees with support being the whole metric graph (Proposition~\ref{P:DomTreeCrit}), and (3) reduced divisor maps to dominant tropical trees (Proposition~\ref{P:CritRedDivMap}). We prove that a harmonic morphism \cite{BN09} from a modification of a metric graph $\Gamma$ to a metric tree essentially corresponds to the reduced divisor map from $\Gamma$ to a dominant tropical tree (Theorem~\ref{T:RedHarm}).
\item   Recall that the gonality of an algebraic curve is defined as  the minimum degree of divisors of rank one or equivalently the minimum degree of finite maps from the algebraic curve to a projective line. However, because of some subtlety in the process of divisor degeneration from curves to metric graphs, these two equivalent interpretations of curve gonality diverge into two non-equivalent notions of gonality: the divisorial gonality \cite{Baker08} and the stable gonality \cite{CKK15} for metric graphs. In particular, the divisorial gonality of $\Gamma$ is the minimum degree of divisors with Baker-Norine  rank being one,  and the stable gonality of $\Gamma$  is the minimum degree of harmonic morphisms from modifications of $\Gamma$ to a metric tree  (Definition~\ref{D:StaGonality}). Therefore, to compute the stable gonality of $\Gamma$ is equivalent to find the minimum degree of dominant tropical trees on $\Gamma$ (Proposition~\ref{P:SGon}). We want to mention that recently there is another independent work also trying to compute stable gonality by looking into linear systems in \cite{Kageyama18}. Other than the  Baker-Norine combinatorial rank and the Caporaso algebraic rank \cite{Caporaso15}, we propose a new rank function called the geometric rank for divisors on metric graphs such that the stable gonality of $\Gamma$ can also be accounted as the the minimum degree of divisors of geometric rank one. 
\end{enumerate}

Since idempotent mathematics has wide applications in optimization problems \cite{K94}, we expect that the notions of $B$-pseudonorms and tropical projections introduced in this paper open up new trends in idempotent optimization. Some of the results in this paper can be directly extended from over tropical semirings to over idempotent semirings in general. In our follow-up work, we will investigate idempotent functional analysis and  operator theory based on the notion of tropical projection. 

The rest of the paper can be divided into two parts. The first part consists of Sections~\ref{S:PreAna}-\ref{S:FixedPoint},  laying out the foundation of the theory of  tropical convexity analysis built around tropical projections. The second part consists of Section~\ref{S:AppMetGra}, in which we discuss in detail an application of our machinery to the divisor theory on metric graphs.  More specifically,  in Section~\ref{S:PreAna}, we give some preliminary results on the  analysis of tropical projective spaces; in Section~\ref{S:Tconvexity}, we discuss the notion of tropical convexity and several types of independence in the context of tropical projective spaces; in Section~\ref{S:Bpseudonorm}, we define $B$-pseudonorms on tropical projective spaces and discuss the main theorem about tropical projections; in Section~\ref{S:TropRetract}, we investigate deformation retracts to tropical convex sets; in Section~\ref{S:Compact}, we provide a general approach of constructing compact tropical convex sets and prove the tropical version of Mazur's theorem on closed convex hulls; in Section~\ref{S:TropIndep}, we provide a set-theoretical criterion of tropical independence; in Section~\ref{S:FixedPoint}, we give a fixed-point theorem for tropical projections; Section~\ref{S:AppMetGra} is about the application of the whole machinery to the divisor theory on metric graphs, in which we discuss the notions of tropical convexity, tropical projection and tropical independence on a divisor space of a metric graph, define a general notion of  linear systems of divisors, define a notion of tropical trees as $1$-dimensional linear systems, study the relation between $b$-functions and $B$-pseudonorms, show that reduced divisors are essentially special tropical projections, give a definition of reduced divisors to linear systems in general, prove that harmonic morphisms from a metric graph to a metric tree are no more than reduced divisor maps to tropical trees, and finally introduce a new rank function called geometric rank function on the space of divisors.

\section{Preliminary Analysis on Tropical Projective Spaces} \label{S:PreAna}
Let $X$ be a topological space.
Let $BC(X)$ be the real linear space  of all bounded and continuous real functions on $X$. Let $BC^0(X)$ be linear subspace of $BC(X)$ whose elements are bounded and continuous functions whose infimum and supremum on $X$ are both attainable. 

Throughout this paper, we denote the infimum, supremum, minimum (when existing) and maximum (when existing) of a real-valued function $f$  on $X$ by $\inf(f)$, $\sup(f)$, $\min(f)$ and $\max(f)$ respectively. And by abuse of notation, we let $\inf_{i\in I}(f_i)$ (respectively $\sup_{i\in I}(f_i)$) be a function on $X$ whose value at $x\in X$ is the infimum  (respectively supremum) value of $\{f_i(x)\}_{i\in I}$, and let $\min(f_1,\cdots,f_n)$ (respectively $\max(f_1,\cdots, f_n)$) be a function on $X$ whose value at $x\in X$ is the minimum  (respectively maximum) value of $f_1(x),\cdots,f_n(x)$. Moreover, we also write the constant function of value $c$ on $X$  simply as $c$ sometimes. 

Let $\Vert \cdot \Vert_\infty$ be the  uniform norm on $BC(X)$.
Recall that  $(BC(X),\Vert \cdot \Vert_\infty)$ is a Banach space. 
\begin{lemma} \label{L:BCBC0}
$(BC(X),\Vert \cdot \Vert_\infty)$  is the completion of $(BC^0(X),\Vert \cdot \Vert_\infty)$. 
\end{lemma}
\begin{proof}
To show that $BC^0(X)$ is dense in $BC(X)$, we need to show that any $f\in BC(X)$ is approachable by a sequence $f_1,f_2,\cdots$ in $BC^0(X)$. 
Consider $f\in BC(X)\setminus BC^0(X)$ whose infimum and supremum are $a$ and $b$ respectively.  Then we can choose a decreasing sequence $a_1,a_2,\cdots$ converging to $a$ and an increasing sequence  $b_1,b_2,\cdots$ converging to $b$ such that $a_1\leq b_1$. Let $f_n=\max(a_n,\min(b_n,f))$. Then  $f_n$ has a minimum value $a_n$ and a maximum value $b_n$ and is clearly continuous which means $f_n\in BC^0(X)$. Moreover, the sequence $f_1,f_2,\cdots$ converges to $f$ uniformly. Therefore $BC^0(X)$ is dense in $BC(X)$. 
\end{proof}

We let   $\mbbTPP(X):=BC^0(X)/\sim$ and $\mbbTP(X):=BC(X)/\sim$ where $f_1\sim f_2$ if $f_1-f_2$ is a constant function. We call $\mbbTPP(X)$ and $\mbbTP (X)$ the \emph{inner tropical projective space} and the \emph{(outer) tropical projective space} on $X$ respectively. 

Note that if $X$ is compact, then $BC(X)=BC^0(X)=C(X)$  and $\mbbTP(X)=\mbbTPP(X)$ where $C(X)$ is the linear space of all continuous functions on $X$. 

For $f\in BC(X)$ and $V\subseteq BC(X)$, we have the following notations:

\begin{enumerate}
\item $[f]\in \mbbTP(X)$ is the equivalence class of $f$ and we call $[f]$ the projectivization of $f$;
\item $[V]:=\{[f]\mid f\in V\}$;
\item $\underline{f}:=f-\inf(f)$ and $\overline{f}:=f-\sup(f)$; 
\item  $\Xmin(f):=f^{-1}(\inf(f))$ and $\Xmax(f):=f^{-1}(\sup(f))$;
\item For $\epsilon>0$, $\Xmin^\epsilon(f):=\{x\in X\mid f(x)<\inf(f)+\epsilon\}$ and  $\Xmax^\epsilon(f):=\{x\in X\mid f(x)>\sup(f)-\epsilon\}$.
\end{enumerate} 

We call $\Xmin(f)$ and $\Xmax(f)$ the minimizer and maximizer of $f$ respectively. Note that $\Xmin(f)=\bigcap_{\epsilon>0} \Xmin^\epsilon(f)$ and $\Xmax(f)=\bigcap_{\epsilon>0} \Xmax^\epsilon(f)$. Clearly if $f\in BC^0(X)$, then $\Xmax(f)$ and $\Xmin(f)$ are both nonempty. But if $f\in BC(X)\setminus BC^0(X)$, then at least one of $\Xmax(f)$ and $\Xmin(f)$ must be empty. Note that  for each $g\in[f]$, $\Xmin(f)=\Xmin(g)$, $\Xmax(f)=\Xmax(g)$, $\Xmin^\epsilon(f)=\Xmin^\epsilon(g)$ and $\Xmax^\epsilon(f)=\Xmax^\epsilon(g)$. Hence by abuse of notation, we also write $\Xmin([f]):=\Xmin(f)$, $\Xmax([f]):=\Xmax(f)$, $\Xmin^\epsilon([f]):=\Xmin^\epsilon(f)$ and $\Xmax^\epsilon([f]):=\Xmax^\epsilon(f)$. 

The following are several easily verifiable facts:

\begin{enumerate}
\item $\mbbTP(X)$ and $\mbbTPP(X)$  are both  real linear spaces with the zero element  being the equivalence class $[0]$ of all constant functions on $X$ and $\lambda_1 [f_1]+\lambda_2 [f_2]:=[\lambda_1 f_1+\lambda_2 f_2]$ for all $\lambda_1,\lambda_2\in \mbbR$; 
\item $-\overline{f}=\underline{-f}$;
\item $\Xmin(f)=\Xmin(cf)$ and $\Xmax(f)=\Xmax(cf)$ for all $c>0$; 
\item $\Xmin(f)=\Xmax(-f)$;
\item $\Xmin^\epsilon(f)=\Xmin^{c\epsilon}(cf)$ and $\Xmax^\epsilon(f)=\Xmax^{c\epsilon}(cf)$ for all  $c>0$ and $\epsilon>0$;
\item $\Xmin^\epsilon(f)=\Xmax^\epsilon(-f)$ for all $\epsilon>0$;
\end{enumerate}

Here we introduce a new notation ``$\Cap$''. For each $\alpha,\beta\in \mbbTP(X)$,  we say $\Xmin(\alpha)\Cap\Xmin(\beta)\neq\emptyset$  if for each $\epsilon>0$, $\Xmin^\epsilon(\alpha)\bigcap\Xmin^\epsilon(\beta)\neq\emptyset$, and say $\Xmin(\alpha)\Cap\Xmin(\beta)=\emptyset$  if there exists $\epsilon>0$ such that  $\Xmin^\epsilon(\alpha)\bigcap\Xmin^\epsilon(\beta)=\emptyset$. Accordingly,  we say $\Xmax(\alpha)\Cap\Xmax(\beta)\neq\emptyset$  if for each $\epsilon>0$, $\Xmax^\epsilon(\alpha)\bigcap\Xmax^\epsilon(\beta)\neq\emptyset$, and say $\Xmax(\alpha)\Cap\Xmax(\beta)=\emptyset$  if there exists $\epsilon>0$ such that  $\Xmax^\epsilon(\alpha)\bigcap\Xmax^\epsilon(\beta)=\emptyset$.

\begin{lemma} \label{L:XminXmax}
For $\alpha,\beta\in \mbbTP(X)$, if $\Xmin(\alpha)\Cap\Xmin(\beta)\neq\emptyset$, then $\Xmin(\alpha)\bigcap\Xmin(\beta)=\Xmin(\alpha+\beta)$, and if $\Xmax(\alpha)\Cap\Xmax(\beta)\neq\emptyset$, then $\Xmax(\alpha)\bigcap\Xmax(\beta)=\Xmax(\alpha+\beta)$. 
\end{lemma}

\begin{lemma} \label{L:XminXmaxTP0}
For $\alpha,\beta\in \mbbTP^0(X)$, 
\begin{enumerate}
\item $\Xmin(\alpha)\bigcap\Xmin(\beta)=\Xmin(\alpha+\beta)\neq\emptyset$ if and only if $\Xmin(\alpha)\Cap\Xmin(\beta)\neq\emptyset$;
\item $\Xmin(\alpha)\bigcap\Xmin(\beta)=\emptyset$ if and only if $\Xmin(\alpha)\Cap\Xmin(\beta)=\emptyset$;
\item $\Xmax(\alpha)\bigcap\Xmax(\beta)=\Xmax(\alpha+\beta)\neq\emptyset$ if and only if $\Xmax(\alpha)\Cap\Xmax(\beta)\neq\emptyset$;
\item $\Xmax(\alpha)\bigcap\Xmax(\beta)=\emptyset$ if and only if $\Xmax(\alpha)\Cap\Xmax(\beta)=\emptyset$.
\end{enumerate}
\end{lemma}
\begin{proof}
This can be easily verified by definition. 
\end{proof}

\begin{lemma} \label{L:SpeIneq}
Let $\alpha=[f],\beta=[g]\in \mbbTP(X)$. Then
\begin{enumerate}
\item $\underline{f+g}\leq \underline{f}+\underline{g}$ with equality holds if and only if $\Xmin(\alpha)\Cap\Xmin(\beta)\neq\emptyset$;
\item $\overline{f+g}\geq \overline{f}+\overline{g}$ with equality holds if and only if $\Xmax(\alpha)\Cap\Xmax(\beta)\neq\emptyset$. 
\end{enumerate}
\end{lemma}
\begin{proof}
For (1), $\underline{f+g}=f+g-\inf(f+g)\leq f+g-\inf(f)-\inf(g)=\underline{f}+\underline{g}$ with equality holds if and only if $\inf(f+g)=\inf(f)+\inf(g)$ if and only if $\Xmin(\alpha)\Cap\Xmin(\beta)\neq\emptyset$. 

For (2), $\overline{f+g}=f+g-\sup(f+g)\geq f+g-\sup(f)-\sup(g)=\overline{f}+\overline{g}$ with equality holds if and only if $\sup(f+g)=\sup(f)+\sup(g)$ if and only if $\Xmax(\alpha)\Cap\Xmax(\beta)\neq\emptyset$. 
\end{proof}

\begin{definition} \label{D:TropNorm}
The \emph{tropical norm}  on $\mbbTP(X)$ is a function $\Vert \cdot \Vert:\mbbTP(X)\to [0,\infty)$ defined by
$\Vert [f] \Vert:=\sup(f)-\inf(f)=\sup(\underline{f})$ for all $[f]\in\mbbTP(X)$. 
\end{definition}

\begin{lemma} \label{L:TropNorm}
The tropical norm is a norm, i.e., 
\begin{enumerate}
\item $\Vert c\alpha \Vert = |c| \Vert \alpha \Vert$ for all $c\in \mbbR$ and $\alpha\in \mbbTP(X)$. 
\item $\Vert \alpha \Vert = 0$ if and only if $\alpha=[0]$. 
\item $\Vert \alpha+\beta \Vert\leq \Vert \alpha\Vert +\Vert \beta \Vert$ for all $\alpha,\beta\in\mbbTP(X)$ where the equality holds if and only if $\Xmin(\alpha)\Cap\Xmin(\beta)\neq\emptyset$ and $\Xmax(\alpha)\Cap\Xmax(\beta)\neq\emptyset$.
\end{enumerate}
 
\end{lemma}
\begin{proof}
(1) and (2) are straightforward. For (3), let $\alpha=[f]$ and $\beta=[g]$. Then as in Lemma~\ref{L:SpeIneq}, $\inf(f+g)\geq\inf(f)+\inf(g)$ where equality holds if and only if $\Xmin(\alpha)\Cap\Xmin(\beta)\neq\emptyset$, and $\sup(f+g)\leq\sup(f)+\sup(g)$ where the equality holds if and only if  $\Xmax(\alpha)\Cap\Xmax(\beta)\neq\emptyset$. Therefore, 
$\Vert \alpha+\beta \Vert=\sup(f+g)-\inf(f+g)\leq \sup(f)+\sup(g)-\inf(f)-\inf(g)$ where the equality holds if and only if $\Xmin(\alpha)\Cap\Xmin(\beta)\neq\emptyset$ and $\Xmax(\alpha)\Cap\Xmax(\beta)\neq\emptyset$.
\end{proof}

As in convention, the tropical norm induces a metric $\rho$ on $\mbbTP(X)$ where $\rho(\alpha,\beta)=\Vert \alpha- \beta \Vert$. And in future discussions, we always assume that $\mbbTP(X)$ is equipped with this metric topology.

\begin{lemma} \label{L:NormComparison}
For $f\in BC(X)$, we have the following relation between $\Vert f\Vert_\infty$ and $\Vert[f]\Vert$:
\begin{enumerate}
\item $\Vert f \Vert_\infty =\Vert [f]\Vert$ if $\inf(f)=0$ or $\sup(f)=0$;
\item $\Vert f \Vert_\infty \geq \Vert [f]\Vert$ if $\inf(f)\geq0$ or $\sup(f)\leq 0$;
\item $\Vert f \Vert_\infty \leq \Vert [f]\Vert\leq 2 \Vert f \Vert_\infty$ if $\inf(f)\leq 0$ and $\sup(f)\geq 0$.
\end{enumerate}
\end{lemma}

\begin{proof}
This lemma can be easily verified. 
\end{proof}

\begin{proposition} \label{P:TPBanach}
The normed space $(\mbbTP(X),\Vert \cdot \Vert)$ is a real Banach space which is the completion of $(\mbbTP^0(X),\Vert \cdot \Vert)$.
\end{proposition}
\begin{proof}
We will use the inequalities in Lemma~\ref{L:NormComparison}. 

Let $\{[f_n]\}_n$ be a Cauchy sequence in $\mbbTP(X)$. Then given any $\epsilon>0$, there exists $N>0$ such that $\Vert [f_n]-[f_m]\Vert<\epsilon$ for all $m,n\geq N$. Note that $\Vert \underline{f_n}-\underline{f_m} \Vert_\infty\leq \Vert [f_n]-[f_m]\Vert$ since $\inf(\underline{f_n}-\underline{f_m})\leq 0$ and $\sup(\underline{f_n}-\underline{f_m})\geq 0$.  Therefore $\{\underline{f_n}\}_n$ is a Cauchy sequence in $BC(X)$ and converges to a function $f\in BC(X)$ since $(BC(X),\Vert \cdot \Vert_\infty)$ is complete. 

To show that $(\mbbTP(X),\Vert \cdot \Vert)$ is a Banach space, it remains to show that $[f]$ is the limit of $\{[f_n]\}_n$, i.e., $\Vert [f_n]-[f]\Vert\to 0$ as $n\to \infty$. Note that it can be easily shown that $\inf:BC(X)\to \mbbR$ is a continuous function and hence $\inf(f)=\lim\limits_{n\to \infty}\inf (\underline{f_n})=0$. It follows that $\Vert [f_n]-[f]\Vert \leq 2 \Vert \underline{f_n}-f \Vert_\infty\to 0$ as $n\to \infty$. 

Now to show that $(\mbbTP^0(X),\Vert \cdot \Vert)$ is dense in  $(\mbbTP(X),\Vert \cdot \Vert)$, we note that $BC^0(X)$ is dense in $BC(X)$ by Lemma~\ref{L:BCBC0}. For each $[f]\in\mbbTP(X)$, we choose a sequence $f_1,f_2,\cdots$ in $BC^0(X)$ converging to $f$ and claim that the sequence $[f_1],[f_2],\cdots$ in $\mbbTP^0(X)$  converges to $[f]$. This follows from the fact that $\Vert [f_n]-[f]\Vert$ is always bounded by $2\Vert f_n-f\Vert_\infty$. 

\end{proof}

\section{Tropical Convexity} \label{S:Tconvexity}

\subsection{Tropical Operations} \label{SS:TOper}

\begin{definition}
For $c\in \mbbR$ and elements $f,g\in \mbbR^X$, we define the following tropical operations on $\mbbR^X$:
\begin{enumerate}
\item \emph{lower tropical  addition} $f\minplus g :=\min(f,g)$,
\item \emph{upper tropical addition} $f\maxplus g :=\max(f,g)$,
\item \emph{tropical scalar multiplication} $c\odot f  :=c+f$,
\item \emph{tropical multiplication} $f\otimes g:=f+g$, (if $c$ is also considered as a constant function, then $c\otimes f = c\odot f$)
\item \emph{tropical division} $f\oslash g:=f-g$;
\item The negation $-f=0\oslash f$ of $f$ is also called the \emph{tropical inverse} of $f$. 
\end{enumerate}
\end{definition}

\begin{lemma}\label{L:TropOper}
The following are some properties of the above tropical operations.

\begin{enumerate}
\item All these tropical operations are closed in $BC(X)$.
\item $\minplus$, $\maxplus$ and $\otimes$ are all commutative.
\item $\minplus$ and  $\maxplus$ are idempotent, i.e., $f\minplus f=f$ and $f\maxplus f = f$.
\item If $f\leq g$ (i.e.,$ f(x)\leq g(x)$ for all $x\in X$), then $f\minplus g=f$ and $f\maxplus g=g$.
\item $f\minplus(f\maxplus g) = f\maxplus(f\minplus g)=f$.
\item $f\otimes(g\minplus h)=(f\otimes g)\minplus (f\otimes h)$.
\item $f\otimes(g\maxplus h)=(f\otimes g)\maxplus (f\otimes h)$.
\item $f\maxplus(g\minplus h)=(f\maxplus g)\minplus (f\maxplus h)$.
\item $f\minplus(g\maxplus h)=(f\minplus g)\maxplus (f\minplus h)$.
\item $(f\minplus g)\otimes (f\maxplus g)=f\otimes g$.
\item $(-f)\minplus(-g)=-(f\maxplus g)$.
\end{enumerate}

\end{lemma}

\begin{proof}
(1)-(7) are straightforward to verify. 

For (8) and (9), it can be verified that for each $x\in X$, 
$$f\maxplus(g\minplus h)(x)=(f\maxplus g)\minplus (f\maxplus h)(x)=
\begin{cases}
f & \mbox{if } g(x)\leq f(x)\ \mbox{or } h(x)\leq f(x) \\
g &\mbox{if } f(x)\leq g(x)\leq h(x) \\ 
h & \mbox{if } f(x)\leq h(x)\leq g(x) 
\end{cases}$$
and 
$$f\minplus(g\maxplus h)(x)=(f\minplus g)\maxplus (f\minplus h)(x)=
\begin{cases}
f & \mbox{if } g(x)\geq f(x)\ \mbox{or } h(x)\geq f(x) \\
g &\mbox{if } f(x)\geq g(x)\geq h(x) \\ 
h & \mbox{if } f(x)\geq h(x)\geq g(x).
\end{cases}$$

For (10), we have
$(f\minplus g)\otimes (f\maxplus g)=\min(f,g)+\max(f,g)=f+g=f\otimes g$. 

For (11), we have 
$(-f)\minplus(-g)=\min(-f,-g)=-\max(f,g)=-(f\maxplus g)$. 
\end{proof}

\subsection{Tropical Paths and Segments}

\begin{definition} \label{D:tpath}
For each $\alpha=[f]$ and $\beta=[g]$ in $\mbbTP(X)$, we define two types of \emph{tropical paths} from $\alpha$ to $\beta$. Let $d=\rho(\alpha,\beta)$ be the distance between $\alpha$ and $\beta$. 
\begin{enumerate}
\item The \emph{lower tropical path} from $\alpha$ to $\beta$ is an injective map $P_{(\alpha,\beta)}:[0,d]\to \mbbTP(X)$ given by $t\mapsto [\min(t,\underline{g-f})+f]$;
\item The \emph{upper tropical path} from $\alpha$ to $\beta$ is an injective map $P^{(\alpha,\beta)}:[0,d]\to \mbbTP(X)$ given by $t\mapsto [\max(-t,\overline{g-f})+f]$.
\item Both lower and upper tropical paths are \emph{tropical paths}.
\end{enumerate}
Respectively, we define two types of  \emph{tropical segments} (also called \emph{t-segments}) connecting $\alpha$ and $\beta$ as follows.
\begin{enumerate}
\item The \emph{lower tropical segment} $\lowertconv(\alpha,\beta)$ connecting $\alpha$ and $\beta$ is the image of $P_{(\alpha,\beta)}$ in $\mbbTP(X)$;
\item The \emph{upper tropical segment} $\uppertconv(\alpha,\beta)$ connecting $\alpha$ and $\beta$ is the image of $P^{(\alpha,\beta)}$ in $\mbbTP(X)$.
\item Both lower and upper tropical segments are called \emph{tropical segments}. 
\end{enumerate}
\end{definition}

\begin{remark} \label{R:TropPath}
Clearly, $P_{(\alpha,\beta)}(0)=P^{(\alpha,\beta)}(0)=\alpha$ and $P_{(\alpha,\beta)}(d)=P^{(\alpha,\beta)}(d)=\beta$ by definition. Tropical paths can be translated by any $\gamma\in \mbbTP(X)$ as follows: for each $t\in[0,d]$, $P_{(\alpha+\gamma,\beta+\gamma)}(t)=P_{(\alpha,\beta)}(t)+\gamma$ and $P^{(\alpha+\gamma,\beta+\gamma)}(t)=P^{(\alpha,\beta)}(t)+\gamma$. Furthermore, we can scale tropical paths as follows: for each $t\in[0,d]$ and $c\in[0,\infty)$, $P_{(c\alpha,c\beta)}(ct)=cP_{(\alpha,\beta)}(t)$ and $P^{(c\alpha,c\beta)}(ct)=cP^{(\alpha,\beta)}(t)$.
\end{remark}

\begin{remark}
If $\alpha,\beta\in \mbbTP^0(X)$, then the tropical segments  $\lowertconv(\alpha,\beta)$ and $\uppertconv(\alpha,\beta)$ are both contained in $\mbbTP^0(X)$. 
\end{remark}

\begin{lemma} \label{L:TropLinear}
For each $\alpha=[f]$ and $\beta=[g]$ in $\mbbTP(X)$, $\lowertconv(\alpha,\beta)=\{[a\odot f \minplus b\odot g]\mid a, b\in \mbbR\}$ and $\uppertconv(\alpha,\beta)=\{[a\odot f \maxplus b\odot g]\mid a, b\in \mbbR\}$. Moreover, $\lowertconv(\alpha,\beta)=\lowertconv(\beta,\alpha)$ and $\uppertconv(\alpha,\beta)=\uppertconv(\beta,\alpha)$.
\end{lemma}
\begin{proof}
Let $u=\inf(g-f)$ and $v=\sup(g-f)$. Note that $c\odot f\minplus g = c\odot f$ if and only if $c\odot f\maxplus g = g$ if and only if $c\leq u$, and $c\odot f\minplus g = g$ if and only if $c\odot f\maxplus g = c\odot f$ if and only if $c\geq v$. Thus $\{[a\odot f \minplus b\odot g]\mid a, b\in \mbbR\}=\{[c\odot f \minplus g]\mid c\in [u,v]\}$ and $\{[a\odot f \maxplus b\odot g]\mid a, b\in \mbbR\}=\{[c\odot f \maxplus g]\mid c\in [u,v]\}$. 

On the other hand, $\lowertconv(\alpha,\beta)=\{[\min(t,\underline{g-f})+f]\mid t\in[0,v-u]\}$ and $\uppertconv(\alpha,\beta)=\{[\max(-t,\overline{g-f})+f]\mid t\in[0,v-u] \}$. 
Now $\min(t,\underline{g-f})+f=\min(t,g-f-u)+f=\min(t+f,g-u)=t\odot f \minplus (-u)\odot g \sim (t+u)\odot f\minplus g$ and $\max(-t,\overline{g-f})+f=\max(-t,g-f-v)+f=\max(f-t,g-v)=(-t)\odot f \maxplus (-v)\odot g\sim (v-t)\odot f \maxplus g$. Since $t\in[0,v-u]$, we have $t+u,v-t\in[u,v]$. Thus we get  $\lowertconv(\alpha,\beta)=\{[c\odot f \minplus g]\mid c\in [u,v]\}=\{[a\odot f \minplus b\odot g]\mid a, b\in \mbbR\}$ and $\uppertconv(\alpha,\beta)=\{[c\odot f \maxplus g]\mid c\in [u,v]\}=\{[a\odot f \maxplus b\odot g]\mid a, b\in \mbbR\}$. The commutativity of $\lowertconv$ and  $\uppertconv$ also follows easily. 
\end{proof}

\begin{remark}
Lemma~\ref{L:TropLinear} actually says that the lower (respectively upper) tropical segment between $\alpha=[g]$ and $\beta=[g]$ is the projectivization of the lower (respectively upper) tropical linear space spanned by $f$ and $g$.  
\end{remark}

\begin{lemma} \label{L:neg_tseg}
For $\alpha,\beta$ in $\mbbTP(X)$ with $d=\rho(\alpha,\beta)$, we have $-P_{(\alpha,\beta)}(t)=P^{(-\alpha,-\beta)}(t)$ for all $t\in[0,d]$. Moreover, $-\lowertconv(\alpha,\beta)=\uppertconv(-\alpha,-\beta)$. In other words, the tropical inverse of a lower tropical segment is an upper tropical segment and vice versa. 
\end{lemma}

\begin{proof}
Suppose  $\alpha=[f]$ and $\beta=[g]$. Then $-P_{(\alpha,\beta)}(t)=-[\min(t,\underline{g-f})+f]=[-\min(t,\underline{g-f})-f]=[\max(-t,\overline{(-g-(-f))})+(-f)]=P^{(-\alpha,-\beta)}(t)$. And it follows $-\lowertconv(\alpha,\beta)=\uppertconv(-\alpha,-\beta)$. 
\end{proof}

We will also use the following notations to represent tropical segments which are respectively called closed, open and half-closed half-open (upper or lower) tropical segments as in the convention what we call the  intervals on the real line:
 \begin{align*}
 \underline{[\alpha,\beta]}&=\lowertconv(\alpha,\beta),\  \overline{[\alpha,\beta]}=\uppertconv(\alpha,\beta),\\
  \underline{(\alpha,\beta)}&=\lowertconv(\alpha,\beta)\setminus\{\alpha,\beta\},\  \overline{(\alpha,\beta)}=\uppertconv(\alpha,\beta)\setminus\{\alpha,\beta\}, \\
 \underline{(\alpha,\beta]}&=\lowertconv(\alpha,\beta)\setminus\{\alpha\},\  \overline{(\alpha,\beta]}=\uppertconv(\alpha,\beta)\setminus\{\alpha\}, \\ 
 \underline{[\alpha,\beta)}&=\lowertconv(\alpha,\beta)\setminus\{\beta\},\  \overline{[\alpha,\beta)}=\uppertconv(\alpha,\beta)\setminus\{\beta\}.\\
\end{align*}
Note that unless otherwise specified, when saying tropical segments we mean the closed ones by default.

\begin{proposition} \label{P:TropSeg}
Let $\alpha,\beta$ be elements in $\mbbTP(X)$ with $d=\rho(\alpha,\beta)$. We summarize some properties of tropical paths and segments as follows. 
\begin{enumerate}
\item For each $\alpha', \beta'$ in $\underline{[\alpha,\beta]}$ (respectively in $\overline{[\alpha,\beta]}$), we have  $\underline{[\alpha',\beta']} \subseteq \underline{[\alpha,\beta]}$ (respectively $\overline{[\alpha',\beta']} \subseteq \overline{[\alpha,\beta]}$). 
\item $P_{(\alpha,\beta)}(t)=P_{(\beta,\alpha)}(d-t)$ (respectively $P^{(\alpha,\beta)}(t)=P^{(\beta,\alpha)}(d-t)$) for all $t\in[0,d]$.
\item   $P_{(\alpha,\beta)}$ (respectively $P^{(\alpha,\beta)}$) is an isometry from $[0,d]$ to $\underline{[\alpha,\beta]}$ (respectively $\overline{[\alpha,\beta]}$). 
\item $\underline{[\alpha,\beta]}$ and $\overline{[\alpha,\beta]}$ are compact (and thus closed and bounded) subsets of $\mbbTP(X)$. 
\item The following are equivalent:
\begin{enumerate}
\item $\gamma\in \underline{[\alpha,\beta]}$ (respectively $\gamma\in \overline{[\alpha,\beta]}$);
\item $\underline{[\alpha,\beta]}=\underline{[\alpha,\gamma]}\bigcup \underline{[\gamma,\beta]}$ 
\\(respectively $\overline{[\alpha,\beta]}=\overline{[\alpha,\gamma]}\bigcup \overline{[\gamma,\beta]}$);
\item $\Xmin(\alpha-\gamma)\bigcup\Xmin(\beta-\gamma)=X$
\\(respectively $\Xmax(\alpha-\gamma)\bigcup\Xmax(\beta-\gamma)=X$).
\end{enumerate}
\item The intersection of any two lower (respectively upper) tropical segments is a lower (respectively upper) tropical segment. 
\item For two lower tropical segments $\underline{[\alpha_1,\beta_1]}$ and $\underline{[\alpha_2,\beta_2]}$, if $\beta_1\in \underline{(\alpha_2,\beta_2]}$ and $\alpha_2\in \underline{[\alpha_1,\beta_1)}$, then $\underline{[\alpha_1,\beta_1]}\bigcup\underline{[\alpha_2,\beta_2]}=\underline{[\alpha_1,\beta_2]}$.  Respectively, for two upper tropical segments $\overline{[\alpha_1,\beta_1]}$ and $\overline{[\alpha_2,\beta_2]}$, if $\beta_1\in \overline{(\alpha_2,\beta_2]}$ and $\alpha_2\in \overline{[\alpha_1,\beta_1)}$, then $\overline{[\alpha_1,\beta_1]}\bigcup\overline{[\alpha_2,\beta_2]}=\overline{[\alpha_1,\beta_2]}$.

\end{enumerate}

\end{proposition}

\begin{proof}
\begin{enumerate}
\item Let $\alpha=[f]$ and $\beta=[g]$. For $\alpha',\beta'\in \underline{[\alpha,\beta]}$, we may assume that  $\alpha'=P_{(\alpha,\beta)}(t_1)=[f']$ and $\beta'=P_{(\alpha,\beta)}(t_2)=[g']$ where $f'=\min(t_1,\underline{g-f})+f$ and  $g'=\min(t_2,\underline{g-f})+f$ for $0\leq t_1\leq t_2\leq d$. Note that $g'-f'=\underline{g'-f'}=\min(t_2,\underline{g-f})-\min(t_1,\underline{g-f})$ and $\rho(\alpha',\beta')=\Vert g'-f'\Vert = \max(g'-f')=t_2-t_1$. Therefore, for $0\leq t\leq t_2-t_1$,
\begin{align*}
P_{(\alpha',\beta')}(t)&=[\min(t,g'-f')+f]=[\min(t+\min(t_1,\underline{g-f}),\min(t_2,\underline{g-f}))+f]\\
&=[\min(t+t_1,\underline{g-f})+f]=P_{(\alpha,\beta)}(t+t_1)
\end{align*}
 which implies $\underline{[\alpha',\beta']} \subseteq \underline{[\alpha,\beta]}$. An analogous argument holds for $\overline{[\alpha',\beta']} \subseteq \overline{[\alpha,\beta]}$.
 \item For $t\in[0,d]$, 
 \begin{align*}
 & P_{(\alpha,\beta)}(t)=[\min(t,\underline{g-f})+f]=[\min(t+f,\underline{g-f}+f]=[\min(t+f,g-\min(g-f)]  \\
 &=[\min(t+f,d+g+\min(f-g)]=[\min(d-t,f-g-\min(f-g))+g]=P_{(\beta,\alpha)}(d-t).
 \end{align*}
 An analogous argument holds for $P^{(\alpha,\beta)}(t)=P^{(\beta,\alpha)}(d-t)$. 
\item Following the argument in (1), $\rho(P_{(\alpha,\beta)}(t_1),P_{(\alpha,\beta)}(t_2))=\rho(P^{(\alpha,\beta)}(t_1),P^{(\alpha,\beta)}(t_2))=|t_2-t_1|$ for all $0\leq t_1,t_2\leq d$ and thus $\underline{[\alpha,\beta]}$ and $\overline{[\alpha,\beta]}$ are both isometric to the interval $[0,d]$. 
\item It follows from (3) straightforwardly. 

\item Here we just show it for $\underline{[\alpha,\beta]}$ and the case for $\overline{[\alpha,\beta]}$ follows from a similar argument. 
 The equivalence of (a) and (b) is clear by (1). Let $\alpha=[f]$, $\beta=[g]$ and $\gamma = [h]$. 
 Now suppose $\gamma\in \underline{[\alpha,\beta]}$. We may let $h=\min(t,\underline{g-f})+f$ for some $t\in[0,d]$. Then $\Xmin(\alpha-\gamma)=\Xmin(f-h)=\Xmin(-\min(t,\underline{g-f}))=\{x\in X\mid \underline{g-f}(x)\geq t\}$ and $\Xmin(\beta-\gamma)=\Xmin(g-h)=\Xmin(g-f-\min(t,\underline{g-f}))=\Xmin(\max(t,\underline{g-f}))=\{x\in X\mid \underline{g-f}(x)\leq t\}$. Therefore  $\Xmin(\alpha-\gamma)\bigcup \Xmin(\beta-\gamma) =X$. 
 
 Conversely, suppose $\Xmin(g-h)\bigcup \Xmin(f-h)=\Xmin(g-h)\bigcup \Xmax(h-f) =X$. Let $t=\Vert h-f\Vert$. Then $\underline{h-f}=\min(t,\underline{g-f})$ and thus $\gamma=[h]=[\min(t,\underline{g-f}))+f]=P_{(\alpha,\beta)}(t)$. 
 
 \item Let $T_1$ and $T_2$ be two lower (respectively upper)  tropical segments  with $T$ being their intersection. Then by (1), whenever $\alpha,\beta\in T$, we must have $\underline{[\alpha,\beta]}\subseteq T$ (respectively $\overline{[\alpha,\beta]}\subseteq T$). Since $T$ is compact, $T$ must be a  lower (respectively upper)  tropical segment. 
 
 \item Again here we only show the case for lower tropical segments while the case for $\overline{[\alpha,\beta]}$ follows analogously. Actually we just need to show $\beta_1,\alpha_2\in \underline{[\alpha_1,\beta_2]}$ and the statement will follow from (1). Suppose $\rho(\alpha_1,\beta_1)=d_1$, $\rho(\alpha_2,\beta_2)=d_2$. Note that $\Xmin(\beta_1-\alpha_2)=\Xmin(\beta_2-\alpha_2)$ since $\beta_1\in \underline{(\alpha_2,\beta_2]}$. Then $\Xmin(\alpha_1-\alpha_2)\bigcup\Xmin(\beta_2-\alpha_2)=\Xmin(\alpha_1-\alpha_2)\bigcup\Xmin(\beta_1-\alpha_2)=X$ since $\alpha_2\in\underline{[\alpha_1,\beta_1]}$. Thus $\alpha_2\in \underline{[\alpha_1,\beta_2]}$. Analogously we can show that $\beta_1\in \underline{[\alpha_1,\beta_2]}$.
\end{enumerate}
 
\end{proof}

\subsection{Tropical Convexity and Several Types of  Independence on Tropical Projective Spaces}

\begin{definition} \label{D:TropConv}
A subset $T$ of $\mbbTP(X)$ is said to be \emph{lower tropically convex}  (respectively \emph{upper tropically convex} ) if for every $\alpha,\beta\in T$, the whole tropical segment $\underline{[\alpha,\beta]}$ (respectively $\overline{[\alpha,\beta]}$) connecting $\alpha$ and $\beta$ is contained in $T$. 
\end{definition}

\begin{remark}
By Proposition~\ref{P:TropSeg}(1), all (closed, open and half-closed half-open) lower and upper tropical segments are lower and upper tropically convex respectively. 
\end{remark}

The following lemmas follow from Definition~\ref{D:TropConv} directly. 

\begin{lemma}
If $T$ is lower or upper tropically convex, then $T$ is connected. 
\end{lemma}

\begin{lemma}
$\mbbTP^0(X)$ is both lower and upper tropically convex. 
\end{lemma}

\begin{lemma}
The intersection of an arbitrary collection of lower or upper tropical convex sets is lower or upper tropically convex respectively. 
\end{lemma}

\begin{definition} \label{D:TropIndep}
Let $S$ be a subset of $\mbbTP(X)$. 
\begin{enumerate}
\item The \emph{lower tropical convex hull $\lowertconv(S)$} (respectively \emph{upper tropical convex hull $\uppertconv(S)$}) generated by $S$ is the intersection of all lower (respectively upper) tropically convex subsets of $\mbbTP(X)$ containing $S$, and we say $S$ is a \emph{generating set} of $\lowertconv(S)$ (respectively  $\uppertconv(S)$). Clearly, $S$ is lower tropically convex if and only if $S=\lowertconv(S)$, and  $S$ is upper tropically convex if and only if $S=\uppertconv(S)$. 
 
\item We say an (upper or lower) tropical convex hull is \emph{finitely generated} if it can be generated by a finite set. We also call a finitely generated (upper or lower) tropical convex hull an (upper or lower respectively) \emph{tropical polytope}. 
\item  
\begin{enumerate}
\item If $x\notin \lowertconv(S\setminus\{x\})$ (respectively $x\notin \uppertconv(S\setminus\{x\})$) for every $x\in S$, then we say $S$ is \emph{lower (tropically) weakly independent} (respectively \emph{upper (tropically) weakly independent}). 
\item If $\lowertconv(S_1)\bigcap \lowertconv(S_2)= \emptyset$ (respectively $\uppertconv(S_1)\bigcap \uppertconv(S_2)= \emptyset$) for each  partition $\{S_1,S_2\}$ of $S$ with $S_1,S_2\neq \emptyset$,  then we say $S$ is \emph{lower Gondran-Minoux independent} (respectively \emph{upper Gondran-Minoux independent}). 
\end{enumerate}

\end{enumerate}
\end{definition}

\begin{lemma}  \label{L:OpTropHull}
We may do translation, dilation and tropical inversion of  tropical convex hulls as follows:
\begin{enumerate}
\item $\alpha+\lowertconv(S)=\lowertconv(\alpha+S)$ and $\alpha+\uppertconv(S)=\lowertconv(\alpha+S)$ for any $\alpha\in\mbbTP(X)$;
\item $c\cdot \lowertconv(S)=\lowertconv(c\cdot S)$ and $c\cdot \uppertconv(S)=\lowertconv(c\cdot S)$ for any $c\geq 0$; 
\item $-\lowertconv(S)=\uppertconv(-S)$ and  $-\uppertconv(S)=\lowertconv(-S)$. 
\end{enumerate}

\end{lemma}

\begin{proof}
This follows from Remark~\ref{R:TropPath} and Lemma~\ref{L:neg_tseg}. 
\end{proof}

For $V\subseteq BC(X)$, we say $\hminplus(V):=\{(c_1\odot f_1)\minplus\cdots\minplus (c_m\odot f_m)\mid m\in\mbbN, c_i\in \mbbR,f_i\in V\}$ is the \emph{lower tropical torus} spanned by $V$ and $\hmaxplus(V):=\{(c_1\odot f_1)\maxplus\cdots\maxplus (c_m\odot f_m)\mid m\in\mbbN, c_i\in \mbbR,f_i\in V\}$ is the \emph{upper tropical torus} spanned by $V$. By abuse of notation, for $S\subseteq\mbbTP(X)$, we also write $\hminplus(S):=\hminplus(V_S)$ and $\hmaxplus(S):=\hmaxplus(V_S)$ where $V_S=\{f\in BC(X)\mid [f]\in S\}$. Note that Lemma~\ref{L:TropLinear} essentially says $\lowertconv(\alpha,\beta)=[\hminplus(\{\alpha,\beta\})]$ and $\uppertconv(\alpha,\beta)=[\hmaxplus(\{\alpha,\beta\})]$. In the following theorem, we show that this is generally true for all $S\subseteq\mbbTP(X)$.

\begin{theorem} \label{T:TconvTlinear}
For any $S\subseteq\mbbTP(X)$, $\lowertconv(S)=[\hminplus(S)]$ and $\uppertconv(S)=[\hmaxplus(S)]$. Moreover, $\uppertconv(\lowertconv(S))=\lowertconv(\uppertconv(S))$ which is both lower and upper tropically convex. 
\end{theorem}

\begin{proof}
To show that $\lowertconv(S)=[\hminplus(S)]$, first we note that $[\hminplus(S)]\subseteq \lowertconv(S)$, i.e.,  $[f_1\minplus \cdots \minplus f_m]\in \lowertconv(S)$ for all $m\in\mbbN$ and $[f_1],\cdots, [f_m]\in S$. But this can be derived by applying Lemma~\ref{L:TropLinear} inductively on $m$. 
More precisely, suppose $[f_1\minplus \cdots \minplus f_m]\in \lowertconv(S)$ is true for all $[f_1],\cdots, [f_m]\in S$. Then since $\lowertconv(S)$ is lower tropically convex, $[f_1\minplus \cdots \minplus f_m\minplus f_{m+1}]\in \lowertconv(S)$ must also be true for all $[f_1],\cdots, [f_m],[f_{m+1}]\in S$ by Lemma~\ref{L:TropLinear}. It remains to show that $[\hminplus(S)]$ itself is lower tropically convex. Consider $\alpha=[f_1\minplus \cdots \minplus f_m]$ and $\beta=[g_1\minplus \cdots \minplus g_n]\in [\hminplus(S)]$ where $[f_1],\cdots, [f_m],[g_1],\cdots [g_n]\in S$. We have  $\lowertconv(\alpha,\beta)=[\hminplus(\{\alpha,\beta\})]=\{(a\odot f_1\minplus \cdots \minplus a\odot f_m) \minplus (b\odot g_1\minplus \cdots \minplus b\odot g_n)\mid a,b\in \mbbR\}\subseteq [\hminplus(S)]$. 

Using an analogous argument, we can also show that $\uppertconv(S)=[\hmaxplus(S)]$. 

Now let us show that $\uppertconv(\lowertconv(S))=\lowertconv(\uppertconv(S))$. By the above results, an element of $\uppertconv(\lowertconv(S))$ can be written as $\alpha=[(f_{11}\minplus \cdots \minplus f_{1n_1}])\maxplus\cdots \maxplus (f_{m1}\minplus \cdots \minplus f_{mn_m}])$ for some $m,n_1,\cdots, n_m\in\mbbN$ and $[f_{ij}]\in S$. Using Lemma~\ref{L:TropOper}(8), $\alpha$ can also be written as $\alpha=\minplus_{1\leq i_j\leq n_j,1\leq j\leq m}(f_{1i_1}\maxplus\cdots\maxplus f_{mi_m})$ which lies in $\lowertconv(\uppertconv(S))$. Therefore, $\uppertconv(\lowertconv(S))\subseteq\lowertconv(\uppertconv(S))$. Analogously, using Lemma~\ref{L:TropOper}(9), we can show that $[(f_{11}\maxplus \cdots \maxplus f_{1n_1}])\minplus\cdots \minplus (f_{m1}\maxplus \cdots \maxplus f_{mn_m}])=\maxplus_{1\leq i_j\leq n_j,1\leq j\leq m}(f_{1i_1}\minplus\cdots\minplus f_{mi_m})$  which implies $\lowertconv(\uppertconv(S))\subseteq\uppertconv(\lowertconv(S))$.

\end{proof}

\begin{remark}
We will write $\loweruppertconv(S):=\lowertconv(\uppertconv(S))=\uppertconv(\lowertconv(S))$. 
\end{remark}

\begin{corollary} \label{C:LocalFinite}
Let $T$ be a lower (respectively upper) tropical convex hull generated by $S$. For each $x\in T$, there exists a finite subset $S_x$ of $S$ such that $x$ is in the lower (respectively upper) tropical convex hull generated by $S_x$. 
\end{corollary}
\begin{proof}
This is essentially a restatement of Theorem~\ref{T:TconvTlinear}. 
\end{proof}

\begin{remark} \label{R:TropIndep}
There are several different notions of ``linear'' independence defined over tropical semirings \cite{AGG09}. 

As studied in \cite{CR79}, in max-plus linear algebra, a family $V$ of vectors in $\mbbR_{\max}^n$ is said to be \emph{weakly independent} if no vector in $V$ is a max-plus linear combination of the others. Then by Theorem~\ref{T:TconvTlinear}, this corresponds exactly to our definition of upper tropical weak independence in the context of tropical projective spaces in Definition~\ref{D:TropIndep}(3a). 

Another notion of  linear independence in max-plus linear algebra is the Gondran-Minous independence \cite{GM84} which says that a family $V$ of vectors in $\mbbR_{\max}^n$ is Gondran-Minoux independent if for any partition $\{V_1,V_2\}$ of $V$, the intersection of the max-plus linear spaces generated by $V_1$ and $V_2$ is trivial. Again, by Theorem~\ref{T:TconvTlinear}, this corresponds to Definition~\ref{D:TropIndep}(3b) in the context of tropical projective spaces. 

In tropical geometry, there is another notion of linear independence, which is usually called tropical independence \cite{RST05}. This notion of tropical independence was applied to linear systems on metric graphs  by Jensen and Payne   in their tropical proofs of the Gieseker-Petri Theorem \cite{JP14} and the maximal rank conjecture for quadrics \cite{JP16}. More precisely, $\{f_1,\cdots,f_n\}\subseteq BC(X)$ (or $\{[f_1],\cdots,[f_n]\}\subseteq \mbbTP(X)$) is said to be \emph{lower tropically dependent} (respectively \emph{upper tropically dependent}) if there are real numbers $c_1,\cdots,c_n$ such that the minimum $\min\{f_1(x)+c_1,\cdots,f_n(x)+c_n\}$ (respectively the maximum $\max\{f_1(x)+c_1,\cdots,f_n(x)+c_n\}$) occurs at least twice at every point $x\in X$.  Otherwise, $\{f_1,\cdots,f_n\}\subseteq BC(X)$ (or $\{[f_1],\cdots,[f_n]\}\subseteq \mbbTP(X)$) is said to be \emph{lower tropically independent} (respectively \emph{upper tropically independent}). 

Gondran-Minoux independence is clearly a stronger notion than weak independence, while tropical independence is even stronger.  Actually suppose $\{[f_1],\cdots,[f_n]\}\subseteq \mbbTP(X)$  is lower Gondran-Minoux dependent  (the case of upper Gondran-Minoux independence follows analogously). Then without loss of generality, we may assume that there exists $$[f]\in \lowertconv(\{[f_1],\cdots,[f_m]\})\bigcap \lowertconv(\{[f_{m+1}],\cdots,[f_n]\})$$ for some $1\leq m \leq n-1$. Then by Theorem~\ref{T:TconvTlinear}, this means that 
\begin{align*}
f&=(c_1\odot f_1)\minplus\cdots\minplus (c_m\odot f_m)=\min\{f_1+c_1,\cdots,f_m+c_m\} \\
 &=(c_{m+1}\odot f_{m+1})\minplus\cdots\minplus (c_n\odot f_n)=\min\{f_{m+1}+c_{m+1},\cdots,f_n+c_n\} \\
 &=(c_1\odot f_1)\minplus\cdots\minplus (c_n\odot f_n) = \min\{f_1+c_1,\cdots,f_n+c_n\}
\end{align*}
for some $c_1,\cdots, c_n\in\mbbR$. Therefore, for all $x\in X$, the minimum of $\min\{f_1(x)+c_1,\cdots,f_n(x)+c_n\}$ is taken at both $f_i(x)+c_i$ and $f_j(x)+c_j$ for some $1\leq i\leq m$ and $m+1\leq j \leq n$. As a result, this means that $\{[f_1],\cdots,[f_n]\}\subseteq \mbbTP(X)$  is lower tropically dependent. \qed
\end{remark}

\begin{example} \label{E:tconv}
Let $X$ be the finite set $\{1,2,3\}$ with discrete topology. Then $BC(X)=BC^0(X) = \mbbR^3$ and $\mbbTP(X) = \mbbTP^0(X)=\mbbR^2$. In particular,  an element $f$ in $BC(X)$ can be written as $f=(x_1,x_2,x_3)=x_1\ve_1+x_2\ve_2+x_3\ve_3$ with $x_1,x_2,x_3\in\mbbR$, and correspondingly the element $[f]$ in $\mbbTP(X)$ can be written as $[f]=(x_1:x_2:x_3)$ under the tropical projective coordinates. Note that in tropical projective coordinates, $(x_1:x_2:x_3)=(y_1:y_2:y_3)$ if and only if $x_1-y_1=x_2-y_2=x_3-y_3$. Therefore, $[f]=(x_1:x_2:x_3)=(x_1-x_3:x_2-x_3:0)$ and we also write $[f]=(x_1-x_3, x_2-x_3)=(x_1-x_3)\ve_1+(x_1-x_3)\ve_2$. In this way, elements in the $x_1x_2$-plane is in one-to-one correspondence to elements in $\mbbTP(X)$, i.e., the point $(x_1,x_2)$ in the $x_1x_2$-plane represents $(x_1:x_2:0)$ in $\mbbTP(X)$. 

Let $\alpha=(x^\alpha_1:x^\alpha_2:0)=(x^\alpha_1,x^\alpha_2)$ and $\beta=(x^\beta_1:x^\beta_2:0)=(x^\beta_1,x^\beta_2)$ be two elements in $\mbbTP(X)$ respectively. Then by definition, the distance $\rho(\alpha,\beta)= \max(x^\beta_1-x^\alpha_1, x^\beta_2-x^\alpha_2,0)-\min(x^\beta_1-x^\alpha_1, x^\beta_2-x^\alpha_2,0)$. Depending on the relative positions of $\alpha$ and $\beta$, $\rho(\alpha,\beta)$ and the  tropical paths from $\alpha$ to $\beta$ can be written as follows:
\begin{enumerate}
\item $x^\beta_1 - x^\alpha_1\geq x^\beta_2 - x^\alpha_2\geq 0$: 
\begin{align*}
\rho(\alpha,\beta) &=x^\beta_1 - x^\alpha_1 \\
P_{(\alpha,\beta)}(t) &= 
\begin{cases}
(x^\alpha_1+t,x^\alpha_2+t) & \mbox{if } 0\leq t\leq x^\beta_2 - x^\alpha_2 \\
(x^\alpha_1+t,x^\beta_2) &\mbox{if } x^\beta_2 - x^\alpha_2\leq t \leq x^\beta_1 - x^\alpha_1
\end{cases} \\
P^{(\alpha,\beta)}(t) &= 
\begin{cases}
(x^\alpha_1+t,x^\alpha_2) & \mbox{if } 0\leq t\leq (x^\beta_1 - x^\alpha_1)-(x^\beta_2 - x^\alpha_2) \\
(x^\alpha_1+t,x^\beta_2-(x^\beta_1-x^\alpha_1)+t) &\mbox{if } (x^\beta_1 - x^\alpha_1)-(x^\beta_2 - x^\alpha_2) \leq t \leq x^\beta_1 - x^\alpha_1
\end{cases}
\end{align*}

\item $x^\beta_2 - x^\alpha_2\geq x^\beta_1 - x^\alpha_1\geq 0$:
\begin{align*}
\rho(\alpha,\beta) &=x^\beta_2 - x^\alpha_2 \\
P_{(\alpha,\beta)}(t) &= 
\begin{cases}
(x^\alpha_1+t,x^\alpha_2+t) & \mbox{if } 0\leq t\leq x^\beta_1 - x^\alpha_1 \\
(x^\beta_1,x^\alpha_2+t) &\mbox{if } x^\beta_1 - x^\alpha_1\leq t \leq x^\beta_2 - x^\alpha_2
\end{cases} \\
P^{(\alpha,\beta)}(t) &= 
\begin{cases}
(x^\alpha_1,x^\alpha_2+t) & \mbox{if } 0\leq t\leq (x^\beta_2 - x^\alpha_2)-(x^\beta_1 - x^\alpha_1) \\
(x^\beta_1-(x^\beta_2 - x^\alpha_2)+t,x^\alpha_2+t) &\mbox{if } (x^\beta_2 - x^\alpha_2)-(x^\beta_1 - x^\alpha_1) \leq t \leq x^\beta_2 - x^\alpha_2
\end{cases}
\end{align*}
\item $x^\beta_1 - x^\alpha_1\leq 0$ and $x^\beta_2 - x^\alpha_2\geq 0$:
\begin{align*}
\rho(\alpha,\beta) &=(x^\beta_2 - x^\alpha_2)-(x^\beta_1 - x^\alpha_1) \\
P_{(\alpha,\beta)}(t) &= 
\begin{cases}
(x^\alpha_1-t,x^\alpha_2) & \mbox{if } 0\leq t\leq -(x^\beta_1 - x^\alpha_1) \\
(x^\beta_1,x^\alpha_2+(x^\beta_1 - x^\alpha_1)+t) &\mbox{if } -(x^\beta_1 - x^\alpha_1) \leq t \leq (x^\beta_2 - x^\alpha_2)-(x^\beta_1 - x^\alpha_1)
\end{cases} \\
P^{(\alpha,\beta)}(t) &= 
\begin{cases}
(x^\alpha_1,x^\alpha_2+t) & \mbox{if } 0\leq t\leq x^\beta_2 - x^\alpha_2 \\
(x^\alpha_1+(x^\beta_2 - x^\alpha_2)-t,x^\beta_2) &\mbox{if } x^\beta_2 - x^\alpha_2 \leq t \leq (x^\beta_2 - x^\alpha_2)-(x^\beta_1 - x^\alpha_1)
\end{cases}
\end{align*}

\item $x^\beta_1 - x^\alpha_1\leq x^\beta_2 - x^\alpha_2\leq 0$:
\begin{align*}
\rho(\alpha,\beta) &=-(x^\beta_1 - x^\alpha_1) \\
P_{(\alpha,\beta)}(t) &= 
\begin{cases}
(x^\alpha_1-t,x^\alpha_2) & \mbox{if } 0\leq t\leq (x^\beta_2 - x^\alpha_2) -(x^\beta_1 - x^\alpha_1)\\
(x^\alpha_1-t,x^\beta_2-(x^\beta_1-x^\alpha_1)-t) &\mbox{if } (x^\beta_2 - x^\alpha_2) -(x^\beta_1 - x^\alpha_1) \leq t \leq -(x^\beta_1 - x^\alpha_1)
\end{cases} \\
P^{(\alpha,\beta)}(t) &= 
\begin{cases}
(x^\alpha_1-t,x^\alpha_2-t) & \mbox{if } 0\leq t\leq -(x^\beta_2 - x^\alpha_2) \\
(x^\alpha_1-t,x^\beta_2) &\mbox{if } -(x^\beta_2 - x^\alpha_2)\leq t \leq -(x^\beta_1 - x^\alpha_1)
\end{cases} 
\end{align*}

\item $x^\beta_2- x^\alpha_2\leq x^\beta_1 - x^\alpha_1\leq 0$:
\begin{align*}
\rho(\alpha,\beta) &=-(x^\beta_2 - x^\alpha_2) \\
P_{(\alpha,\beta)}(t) &= 
\begin{cases}
(x^\alpha_1,x^\alpha_2-t) & \mbox{if } 0\leq t\leq (x^\beta_1 - x^\alpha_1) -(x^\beta_2 - x^\alpha_2)\\
(x^\beta_1-(x^\beta_2 - x^\alpha_2)-t,x^\alpha_2-t) &\mbox{if } (x^\beta_1 - x^\alpha_1) -(x^\beta_2 - x^\alpha_2) \leq t \leq -(x^\beta_2 - x^\alpha_2)
\end{cases}\\
P^{(\alpha,\beta)}(t) &= 
\begin{cases}
(x^\alpha_1-t,x^\alpha_2-t) & \mbox{if } 0\leq t\leq -(x^\beta_1 - x^\alpha_1) \\
(x^\beta_1,x^\alpha_2-t) &\mbox{if } -(x^\beta_1 - x^\alpha_1)\leq t \leq -(x^\beta_2 - x^\alpha_2)
\end{cases} 
\end{align*}

\item $x^\beta_1 - x^\alpha_1\geq 0$ and $x^\beta_2 - x^\alpha_2\leq 0$:
\begin{align*}
\rho(\alpha,\beta) &=(x^\beta_1 - x^\alpha_1) -(x^\beta_2 - x^\alpha_2)\\
P_{(\alpha,\beta)}(t) &= 
\begin{cases}
(x^\alpha_1,x^\alpha_2-t) & \mbox{if } 0\leq t\leq -(x^\beta_2 - x^\alpha_2) \\
(x^\alpha_1+(x^\beta_2 - x^\alpha_2)+t,x^\beta_2) &\mbox{if } -(x^\beta_2 - x^\alpha_2) \leq t \leq (x^\beta_1 - x^\alpha_1)-(x^\beta_2 - x^\alpha_2)
\end{cases}\\
P^{(\alpha,\beta)}(t) &= 
\begin{cases}
(x^\alpha_1+t,x^\alpha_2) & \mbox{if } 0\leq t\leq x^\beta_1 - x^\alpha_1 \\
(x^\beta_1,x^\alpha_2+(x^\beta_1 - x^\alpha_1)-t) &\mbox{if } x^\beta_1 - x^\alpha_1 \leq t \leq (x^\beta_1 - x^\alpha_1)-(x^\beta_2 - x^\alpha_2)
\end{cases} 
\end{align*}
\end{enumerate}

\begin{figure}
\centering
\begin{tikzpicture}[scale=0.9]
\begin{scope}[shift={(0,14.5)}] 
\draw (0,4) node[anchor=east] {(a)};

\begin{scope}[shift={(1,0)}] 

\draw (3,3.5) node[draw, anchor=south] {Lower Tropical Segments};

\draw[->,line width=0.8pt] (0,0) -- (5,0) node[right] {$x_1$};
\draw[->,line width=0.8pt] (0,0) -- (0,4) node[left] {$x_2$};

\coordinate (origin) at (2.5,2);

\def\len{1}
\def\shift{0.5}

\coordinate (A) at ($(origin)+(0,\len)$);
\coordinate (B) at ($(origin)+(\len,0)$);
\coordinate (C) at ($(origin)+(-\len,-\len)$);

\coordinate (A1) at ($(A)+(\shift,0)$);
\coordinate (B1) at ($(B)+(\shift,0)$);
\coordinate (O1) at ($(origin)+(\shift,0)$);

\coordinate (B2) at ($(B)+(\shift,-\shift)$);
\coordinate (C2) at ($(C)+(\shift,-\shift)$);
\coordinate (O2) at ($(origin)+(\shift,-\shift)$);

\draw [line width=1.2pt] (A) -- (origin) -- (C);
\draw [line width=1.2pt] (A1) -- (O1) --   (B1);
\draw [line width=1.2pt] (B2) -- (O2) --  (C2);

\fill [blue] (A) circle (3pt);
\fill [blue] (B1) circle (3pt);
\fill [blue] (C) circle (3pt);
\fill [blue] (A1) circle (3pt);
\fill [blue] (B2) circle (3pt);
\fill [blue] (C2) circle (3pt);

\draw [blue] (A) node[anchor=east] { $\alpha_1$};
\draw [blue] (A1) node[anchor=west] { $\alpha_2$};
\draw [blue] (B1) node[anchor=west] { $\beta_1$};
\draw [blue] (B2) node[anchor=west] { $\beta_2$};
\draw [blue] (C2) node[anchor=east] { $\gamma_1$};
\draw [blue] (C) node[anchor=east] { $\gamma_2$};

\end{scope}

\begin{scope}[shift={(9,0)}] 

\draw (3,3.5) node[draw, anchor=south] {Upper Tropical Segments};

\draw[->,line width=0.8pt] (0,0) -- (5,0) node[right] {$x_1$};
\draw[->,line width=0.8pt] (0,0) -- (0,4) node[left] {$x_2$};

\coordinate (origin) at (2.5,1.5);

\def\len{1}
\def\shift{0.5}

\coordinate (A) at ($(origin)+(-\len,0)$);
\coordinate (B) at ($(origin)+(\len,\len)$);
\coordinate (C) at ($(origin)+(0,-\len)$);

\coordinate (A1) at ($(A)+(0,\shift)$);
\coordinate (B1) at ($(B)+(0,\shift)$);
\coordinate (O1) at ($(origin)+(0,\shift)$);

\coordinate (B2) at ($(B)+(\shift,0)$);
\coordinate (C2) at ($(C)+(\shift,0)$);
\coordinate (O2) at ($(origin)+(\shift,0)$);

\draw [line width=1.2pt] (A) -- (origin) -- (C);
\draw [line width=1.2pt] (A1) -- (O1) --   (B1);
\draw [line width=1.2pt] (B2) -- (O2) --  (C2);

\fill [blue] (A) circle (3pt);
\fill [blue] (B1) circle (3pt);
\fill [blue] (C) circle (3pt);
\fill [blue] (A1) circle (3pt);
\fill [blue] (B2) circle (3pt);
\fill [blue] (C2) circle (3pt);

\draw [blue] (A) node[anchor=east] { $\alpha_3$};
\draw [blue] (A1) node[anchor=east] { $\alpha_4$};
\draw [blue] (B1) node[anchor=west] { $\beta_3$};
\draw [blue] (B2) node[anchor=west] { $\beta_4$};
\draw [blue] (C2) node[anchor=west] { $\gamma_3$};
\draw [blue] (C) node[anchor=east] { $\gamma_4$};

\end{scope}

\end{scope}


\begin{scope}[shift={(0,5)}]  
\draw (0,9) node[anchor=east] {(b)};

\draw (9,8) node[anchor=south] {\large $\lowertconv(\{A_i,B_i,C_i\})$};

\draw (9,3.3) node[anchor=south] {\large $\uppertconv(\{A_i,B_i,C_i\})$};

\draw [line width=0.5pt] (0,-0.1) --   (16.8,-0.1) -- (16.8,9) -- (0,9) -- (0,-0.1);
\draw [line width=0.5pt] (0,4.3) --   (16.8,4.3);


\begin{scope}[shift={(0,4.5)}] 
\coordinate (origin) at (2,2);
\def\len{1.5}

\coordinate (A1) at ($(origin)+(0,\len)$);
\coordinate (B1) at ($(origin)+(\len,0)$);
\coordinate (C1) at ($(origin)+(-\len,-\len)$);

\draw [line width=1.2pt] (origin) --   (A1);
\draw [line width=1.2pt] (origin) --   (B1);
\draw [line width=1.2pt] (origin) --   (C1);

\fill [blue] (A1) circle (3pt);
\fill [blue] (B1) circle (3pt);
\fill [blue] (C1) circle (3pt);

\draw [blue] (A1) node[anchor=east] { $A_1$};
\draw [blue] (B1) node[anchor=north] {$B_1$};
\draw [blue] (C1) node[anchor=north] { $C_1$};
\end{scope}

\begin{scope}[shift={(0,0)}] 
\coordinate (origin) at (2,2);
\def\len{1.5}

\coordinate (A1) at ($(origin)+(0,\len)$);
\coordinate (B1) at ($(origin)+(\len,0)$);
\coordinate (C1) at ($(origin)+(-\len,-\len)$);
\coordinate (a1) at ($(origin)+(\len,\len)$);
\coordinate (b1) at ($(origin)+(0,-\len)$);
\coordinate (c1) at ($(origin)+(-\len,0)$);

\fill [black!30,opacity=1]  (A1)--(a1)--(B1)--(b1)--(C1)--(c1);
\draw [line width=1.2pt]  (A1)--(a1)--(B1)--(b1)--(C1)--(c1)--(A1);

\fill [blue] (A1) circle (3pt);
\fill [blue] (B1) circle (3pt);
\fill [blue] (C1) circle (3pt);

\draw [blue] (A1) node[anchor=south] { $A_1$};
\draw [blue] (B1) node[anchor=west] {$B_1$};
\draw [blue] (C1) node[anchor=north] { $C_1$};

\end{scope}


\begin{scope}[shift={(3,4.5)}] 
\coordinate (origin) at (2,2.5);

\def\len{1.5}
\def\del{1}

\coordinate (A2) at ($(origin)+(0,\len-\del)$);
\coordinate (B2) at ($(origin)+(\len,0)$);
\coordinate (C2) at ($(origin)+(\del-\len,-\len)$);
\coordinate (b2) at ($(origin)+(\del,0)$);
\coordinate (c2) at ($(origin)+(0,-\del)$);

\fill [black!30,opacity=1] (origin) --   (b2) -- (c2) -- (origin);

\draw [line width=1.2pt] (c2) --   (A2);
\draw [line width=1.2pt] (origin) --   (B2);
\draw [line width=1.2pt] (C2) --   (b2);

\fill [blue] (A2) circle (3pt);
\fill [blue] (B2) circle (3pt);
\fill [blue] (C2) circle (3pt);

\draw [blue] (A2) node[anchor=south] { $A_2$};
\draw [blue] (B2) node[anchor=south] {$B_2$};
\draw [blue] (C2) node[anchor=north] { $C_2$};

\end{scope}

\begin{scope}[shift={(3,0)}] 
\coordinate (origin) at (2,2.5);

\def\len{1.5}
\def\del{1}

\coordinate (A2) at ($(origin)+(0,\len-\del)$);
\coordinate (B2) at ($(origin)+(\len,0)$);
\coordinate (C2) at ($(origin)+(\del-\len,-\len)$);
\coordinate (a2) at ($(origin)+(\len,\len-\del)$);
\coordinate (b2) at ($(origin)+(0,-\len)$);
\coordinate (c2) at ($(origin)+(\del-\len,0)$);

\fill [black!30,opacity=1]  (A2)--(a2)--(B2)--(b2)--(C2)--(c2);
\draw [line width=1.2pt]  (A2)--(a2)--(B2)--(b2)--(C2)--(c2)--(A2);

\fill [blue] (A2) circle (3pt);
\fill [blue] (B2) circle (3pt);
\fill [blue] (C2) circle (3pt);

\draw [blue] (A2) node[anchor=south] { $A_2$};
\draw [blue]  (B2) node[anchor=west] {$B_2$};
\draw [blue] (C2) node[anchor=north] { $C_2$};
\end{scope}


\begin{scope}[shift={(6,4.5)}] 
\coordinate (origin) at (1.7,2.5);

\def\len{1.5}

\coordinate (A3) at (origin);
\coordinate (B3) at ($(origin)+(\len,0)$);
\coordinate (C3) at ($(origin)+(0,-\len)$);

\fill [black!30,opacity=1] (A3) --   (B3) -- (C3) -- (A3);

\draw [line width=1.2pt] (A3) --   (B3) -- (C3) -- (A3);

\fill [blue] (A3) circle (3pt);
\fill [blue] (B3) circle (3pt);
\fill [blue] (C3) circle (3pt);

\draw [blue] (A3) node[anchor=south] { $A_3$};
\draw [blue] (B3) node[anchor=south] {$B_3$};
\draw [blue] (C3) node[anchor=north] { $C_3$};

\end{scope}

\begin{scope}[shift={(6,0)}] 
\coordinate (origin) at (1.7,2.5);

\def\len{1.5}

\coordinate (A3) at (origin);
\coordinate (B3) at ($(origin)+(\len,0)$);
\coordinate (C3) at ($(origin)+(0,-\len)$);

\fill [black!30,opacity=1] (A3) --   (B3) -- (C3) -- (A3);

\draw [line width=1.2pt] (A3) --   (B3) -- (C3) -- (A3);

\fill [blue] (A3) circle (3pt);
\fill [blue] (B3) circle (3pt);
\fill [blue] (C3) circle (3pt);

\draw [blue] (A3) node[anchor=south] { $A_3$};
\draw [blue] (B3) node[anchor=south] {$B_3$};
\draw [blue] (C3) node[anchor=north] { $C_3$};

\end{scope}

\begin{scope}[shift={(9,4.5)}] 
\coordinate (origin) at (1.9,2.5);

\def\len{1.5}
\def\del{1}

\coordinate (A4) at ($(origin)+(\del-\len,0)$);
\coordinate (B4) at ($(origin)+(\len,\len-\del)$);
\coordinate (C4) at ($(origin)+(0,-\len)$);
\coordinate (a4) at ($(origin)+(0,\len-\del)$);
\coordinate (b4) at ($(origin)+(\len,0)$);
\coordinate (c4) at ($(origin)+(\del-\len,-\len)$);

\fill [black!30,opacity=1]  (A4)--(a4)--(B4)--(b4)--(C4)--(c4);
\draw [line width=1.2pt]  (A4)--(a4)--(B4)--(b4)--(C4)--(c4)--(A4);

\fill [blue] (A4) circle (3pt);
\fill [blue] (B4) circle (3pt);
\fill [blue] (C4) circle (3pt);

\draw [blue] (A4) node[anchor=south east] { $A_4$};
\draw [blue] (B4) node[anchor=south] {$B_4$};
\draw [blue] (C4) node[anchor=north] { $C_4$};

\end{scope}

\begin{scope}[shift={(9,0)}] 
\coordinate (origin) at (1.9,2.5);

\def\len{1.5}
\def\del{1}

\coordinate (A4) at ($(origin)+(\del-\len,0)$);
\coordinate (B4) at ($(origin)+(\len,\len-\del)$);
\coordinate (C4) at ($(origin)+(0,-\len)$);
\coordinate (b4) at ($(origin)+(\del,0)$);
\coordinate (c4) at ($(origin)+(0,-\del)$);

\fill [black!30,opacity=1] (origin) --   (b4) -- (c4) -- (origin);

\draw [line width=1.2pt] (A4) --   (b4);
\draw [line width=1.2pt] (origin) --   (C4);
\draw [line width=1.2pt] (c4) --   (B4);

\fill [blue] (A4) circle (3pt);
\fill [blue] (B4) circle (3pt);
\fill [blue] (C4) circle (3pt);

\draw [blue] (A4) node[anchor=south] { $A_4$};
\draw [blue] (B4) node[anchor=south] {$B_4$};
\draw [blue] (C4) node[anchor=north] { $C_4$};

\end{scope}

\begin{scope}[shift={(12,4.5)}] 
\coordinate (origin) at (2.5,2);

\def\len{1.5}

\coordinate (A5) at ($(origin)+(-\len,0)$);
\coordinate (B5) at ($(origin)+(\len,\len)$);
\coordinate (C5) at ($(origin)+(0,-\len)$);
\coordinate (a5) at ($(origin)+(0,\len)$);
\coordinate (b5) at ($(origin)+(\len,0)$);
\coordinate (c5) at ($(origin)+(-\len,-\len)$);

\fill [black!30,opacity=1]  (A5)--(a5)--(B5)--(b5)--(C5)--(c5);
\draw [line width=1.2pt]  (A5)--(a5)--(B5)--(b5)--(C5)--(c5)--(A5);

\fill [blue] (A5) circle (3pt);
\fill [blue] (B5) circle (3pt);
\fill [blue] (C5) circle (3pt);

\draw [blue] (A5) node[anchor=east] { $A_5$};
\draw [blue] (B5) node[anchor=west] {$B_5$};
\draw [blue] (C5) node[anchor=north] { $C_5$};

\end{scope}

\begin{scope}[shift={(12,0)}] 
\coordinate (origin) at (2.5,2);

\def\len{1.5}

\coordinate (A5) at ($(origin)+(-\len,0)$);
\coordinate (B5) at ($(origin)+(\len,\len)$);
\coordinate (C5) at ($(origin)+(0,-\len)$);

\draw [line width=1.2pt] (origin) --   (A5);
\draw [line width=1.2pt] (origin) --   (B5);
\draw [line width=1.2pt] (origin) --   (C5);

\fill [blue] (A5) circle (3pt);
\fill [blue] (B5) circle (3pt);
\fill [blue] (C5) circle (3pt);

\draw [blue] (A5) node[anchor=south] {$A_5$};
\draw [blue] (B5) node[anchor=north west] {$B_5$};
\draw [blue] (C5) node[anchor=west] { $C_5$};

\end{scope}

\end{scope}


\begin{scope}[shift={(0,0)}]  

\draw (0,4) node[anchor=east] {(c)};
\begin{scope}[shift={(0,0)}]   	
    	
\coordinate (origin) at (2,2);

\def\radius{1.5}

\draw [blue, line width=2pt]  (origin) circle (\radius);

\draw ($(origin)+(0,1.8)$) node[anchor=south] {\large $S$};
\end{scope}

\begin{scope}[shift={(4,0)}] 

\coordinate (origin) at (2,2);

\def\radius{1.5}

\coordinate (N) at ($(origin)+(0,\radius)$);
\coordinate (S) at ($(origin)+(0,-\radius)$);
\coordinate (E) at ($(origin)+(\radius,0)$);
\coordinate (W) at ($(origin)+(-\radius,0)$);

\coordinate (NE) at ($(origin)+(\radius,\radius)$);
\coordinate (SW) at ($(origin)+(-\radius,-\radius)$);
\coordinate (Nw) at ($(origin)+({(1-sqrt(2))*\radius},\radius)$);
\coordinate (nW) at ($(origin)+(-\radius,{(sqrt(2)-1)*\radius})$);
\coordinate (nw) at ($(origin)+({-(sqrt(2)/2)*\radius},{(sqrt(2)/2)*\radius})$);
\coordinate (Se) at ($(origin)+({(sqrt(2)-1)*\radius},-\radius)$);
\coordinate (sE) at ($(origin)+(\radius,{(1-sqrt(2))*\radius})$);
\coordinate (se) at ($(origin)+({(sqrt(2)/2)*\radius},{-(sqrt(2)/2)*\radius})$);

\fill [black!30,opacity=1] (S) --   (SW) -- (W) -- (S);
\fill [black!30,opacity=1] (E) --   (sE) -- (se) -- (E);
\fill [black!30,opacity=1] (N) --   (Nw) -- (nw) -- (N);
\fill [black!30,opacity=1] (origin) circle (\radius);

\draw [line width=1.2pt] (S) --   (SW) -- (W);
\draw [line width=1.2pt] (E) --   (sE) -- (se);
\draw [line width=1.2pt] (N) --   (Nw) -- (nw);
\draw [blue, line width=1.2pt]  (origin) circle (\radius);

\draw ($(origin)+(0,1.8)$) node[anchor=south] {\large $\lowertconv(S)$};
\end{scope}


\begin{scope}[shift={(8,0)}] 

\coordinate (origin) at (2,2);

\def\radius{1.5}

\coordinate (N) at ($(origin)+(0,\radius)$);
\coordinate (S) at ($(origin)+(0,-\radius)$);
\coordinate (E) at ($(origin)+(\radius,0)$);
\coordinate (W) at ($(origin)+(-\radius,0)$);

\coordinate (NE) at ($(origin)+(\radius,\radius)$);
\coordinate (SW) at ($(origin)+(-\radius,-\radius)$);
\coordinate (Nw) at ($(origin)+({(1-sqrt(2))*\radius},\radius)$);
\coordinate (nW) at ($(origin)+(-\radius,{(sqrt(2)-1)*\radius})$);
\coordinate (nw) at ($(origin)+({-(sqrt(2)/2)*\radius},{(sqrt(2)/2)*\radius})$);
\coordinate (Se) at ($(origin)+({(sqrt(2)-1)*\radius},-\radius)$);
\coordinate (sE) at ($(origin)+(\radius,{(1-sqrt(2))*\radius})$);
\coordinate (se) at ($(origin)+({(sqrt(2)/2)*\radius},{-(sqrt(2)/2)*\radius})$);

\fill [black!30,opacity=1] (N) --   (NE) -- (E) -- (N);
\fill [black!30,opacity=1] (W) --   (nW) -- (nw) -- (W);
\fill [black!30,opacity=1] (S) --   (Se) -- (se) -- (S);
\fill [black!30,opacity=1] (origin) circle (\radius);

\draw [line width=1.2pt] (N) --   (NE) -- (E);
\draw [line width=1.2pt] (W) --   (nW) -- (nw);
\draw [line width=1.2pt] (S) --   (Se) -- (se);
\draw [blue, line width=1.2pt]  (origin) circle (\radius);

\draw ($(origin)+(0,1.8)$) node[anchor=south] {\large $\uppertconv(S)$};
\end{scope}

\begin{scope}[shift={(12,0)}] 

\coordinate (origin) at (2,2);

\def\radius{1.5}

\coordinate (N) at ($(origin)+(0,\radius)$);
\coordinate (S) at ($(origin)+(0,-\radius)$);
\coordinate (E) at ($(origin)+(\radius,0)$);
\coordinate (W) at ($(origin)+(-\radius,0)$);

\coordinate (NE) at ($(origin)+(\radius,\radius)$);
\coordinate (SW) at ($(origin)+(-\radius,-\radius)$);
\coordinate (Nw) at ($(origin)+({(1-sqrt(2))*\radius},\radius)$);
\coordinate (nW) at ($(origin)+(-\radius,{(sqrt(2)-1)*\radius})$);
\coordinate (nw) at ($(origin)+({-(sqrt(2)/2)*\radius},{(sqrt(2)/2)*\radius})$);
\coordinate (Se) at ($(origin)+({(sqrt(2)-1)*\radius},-\radius)$);
\coordinate (sE) at ($(origin)+(\radius,{(1-sqrt(2))*\radius})$);
\coordinate (se) at ($(origin)+({(sqrt(2)/2)*\radius},{-(sqrt(2)/2)*\radius})$);

\fill [black!30,opacity=1] (NE) --   (Nw) -- (nW) -- (SW) -- (Se) -- (sE) -- (NE);
\fill [black!30,opacity=1] (origin) circle (\radius);

\draw [line width=1.2pt] (NE) --   (Nw) -- (nW) -- (SW) -- (Se) -- (sE) -- (NE);
\draw [blue,line width=1.2pt]  (origin) circle (\radius);

\draw ($(origin)+(0,1.8)$) node[anchor=south] {\large $\loweruppertconv(S)$};
\end{scope}
\end{scope}
\end{tikzpicture}
\caption{(a) Some examples of lower and upper tropical segments. (b) Tropical convex hulls generated by three points $A_i$, $B_i$ and $C_i$ with different relative positions for $i$ being $1$ through $5$. (c) Tropical convex hulls generated by a circle $S$.} \label{F:tconv}
\end{figure}
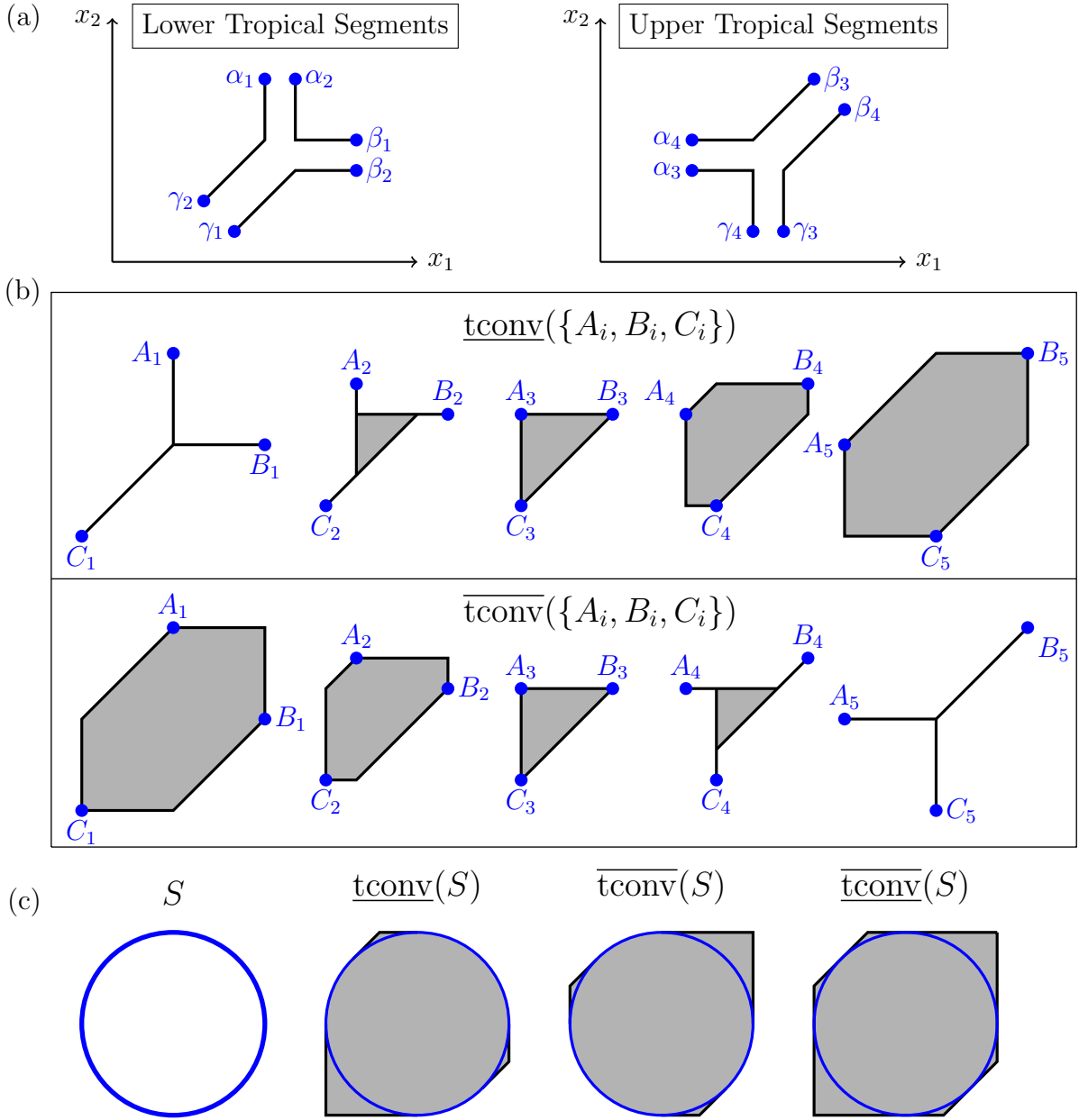

Figure~\ref{F:tconv}(a) shows some lower and upper tropical segments in $\mbbTP(X)$ represented in the $x_1x_2$-plane. In particular, for the lower tropical paths on the left panel, the lower tropical paths from $\gamma_1$ to $\beta_2$, from $\gamma_2$ to $\alpha_1$, from $\beta_1$ to $\alpha_2$, from $\beta_2$ to $\gamma_1$, from $\alpha_1$ to $\gamma_2$ and from $\alpha_2$ to $\beta_1$ correspond to Case (1)-(6) above respectively; for the upper tropical paths on the left panel, the lower tropical paths from $\alpha_4$ to $\beta_3$, from $\gamma_3$ to $\beta_4$, from $\gamma_4$ to $\alpha_3$, from $\beta_3$ to $\alpha_4$, from $\beta_4$ to $\gamma_3$ and from $\alpha_3$ to $\gamma_4$ correspond to Case (1)-(6) above respectively.

Figure~\ref{F:tconv}(b) shows the upper and lower tropical polytopes generated by  the triples $\{A_i,B_i,C_i\}$  for $i=1,\cdots,5$. Note that the fact that these sets are (lower or upper) tropically convex can be easily verified by showing that the (lower or upper) tropical segments connecting any two points in one of the sets is fully contained in that set. More specifically, we suppose 
\begin{align*}
& A_1 = (0,1),\ B_1=(1,0),\ C_1=(-1,-1);\quad & A_2 = (0,1),\ B_2=(1,\frac{2}{3}),\ C_2=(-\frac{1}{3},-\frac{1}{3}); \\
& A_3 = (0,1),\ B_3=(1,1),\ C_3=(0,0);\quad & A_4 = (-\frac{1}{3},\frac{2}{3}),\ B_4=(1,1),\ C_4=(0,-\frac{1}{3}); \\
& A_5 = (-1,0),\ B_5=(1,1),\ C_5=(0,-1).
\end{align*}
There are a few observations about these sets worth mentioning here:
\begin{enumerate}
\item All these sets are compact subsets of $\mbbR^2$ under the tropical metric topology which can be verified straightforwardly in this case. Actually in Section~\ref{S:Compact}, we show that in general all tropical polytopes are compact (Corollary~\ref{C:Polytope}). 
\item Depending on the relative positions of $A_i$, $B_i$ and $C_i$, the tropical polytopes can be purely $1$-dimensional ($\lowertconv(\{A_1,B_1,C_1\})$ and $\uppertconv(\{A_5,B_5,C_5\})$), purely $2$-dimensional ($\lowertconv(\{A_3,B_3,C_3\})$, $\lowertconv(\{A_4,B_4,C_4\})$, $\lowertconv(\{A_5,B_5,C_5\})$, $\uppertconv(\{A_1,B_1,C_1\})$, $\uppertconv(\{A_2,B_2,C_2\})$ and $\uppertconv(\{A_3,B_3,C_3\})$) and not of pure dimension ($\lowertconv(\{A_2,B_2,C_2\})$ and $\uppertconv(\{A_4,B_4,C_4\})$). 
\item Note that $\lowertconv(\{A_3,B_3,C_3\})=\uppertconv(\{A_3,B_3,C_3\})$, $\lowertconv(\{A_4,B_4,C_4\})=\uppertconv(\{A_2,B_2,C_2\})$, and 
 $\lowertconv(\{A_5,B_5,C_5\})=\uppertconv(\{A_1,B_1,C_1\})$ which are all both lower and upper tropically convex. Therefore, we conclude that 
 \begin{align*}
 \loweruppertconv(\{A_3,B_3,C_3\})=\lowertconv(\{A_3,B_3,C_3\})=\uppertconv(\{A_3,B_3,C_3\}) \\
\loweruppertconv(\{A_4,B_4,C_4\})=\loweruppertconv(\{A_2,B_2,C_2\})=\lowertconv(\{A_4,B_4,C_4\})=\uppertconv(\{A_2,B_2,C_2\}) \\
\loweruppertconv(\{A_5,B_5,C_5\})=\loweruppertconv(\{A_1,B_1,C_1\})=\lowertconv(\{A_5,B_5,C_5\})=\uppertconv(\{A_1,B_1,C_1\}).
 \end{align*}
\end{enumerate}

Figure~\ref{F:tconv}(c) shows a circle $S$ and the tropical convex hulls $\lowertconv(S)$, $\uppertconv(S)$ and $\loweruppertconv(S)$ generated by $S$. Note $\loweruppertconv(S)$ is a hexagon. In this case, $S$ is an infinite compact set and $\lowertconv(S)$, $\uppertconv(S)$ and $\loweruppertconv(S)$ are all compact. It is also generally true that the closed tropical convex hulls generated by a compact set is also compact (Theorem~\ref{T:TropMazur}), which is the tropical version of Mazur's theorem proved in Section~\ref{S:Compact}. 

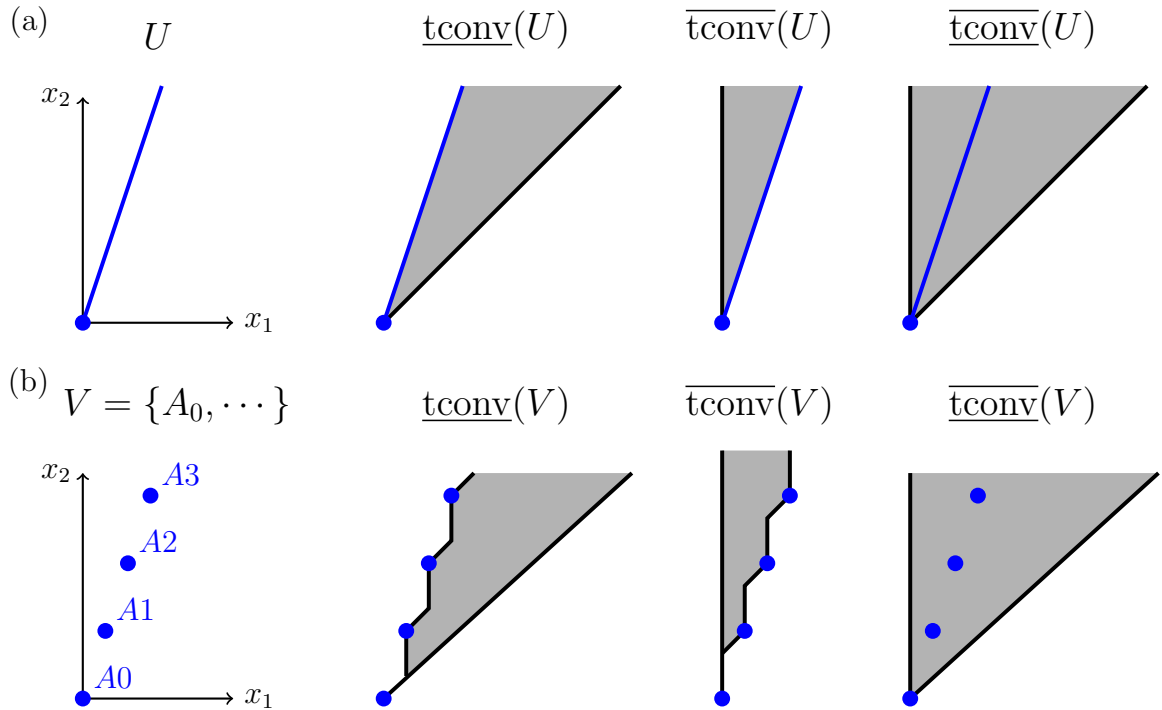
\begin{figure}
\centering
\begin{tikzpicture}

\begin{scope} 
\draw (-0.3,4) node[anchor=east] {(a)};
\begin{scope}

\draw[->,line width=0.8pt] (0,0) -- (2,0) node[right] {$x_1$};
\draw[->,line width=0.8pt] (0,0) -- (0,3) node[left] {$x_2$};

\def\x{3.5*0.3};
\def\y{3.5*0.9};

\fill [blue] (0,0) circle (3pt);

\draw  [blue,line width=1.5pt]  (0,0) --  (\x,\y);

\draw ($(1,3.5)$) node[anchor=south] {\large $U$};
	
\end{scope}

\begin{scope}[shift={(4,0)}] 

\def\x{3.5*0.3};
\def\y{3.5*0.9};

\fill [black!30,opacity=1] (0,0) -- (\x,\y) -- (\y,\y) -- cycle;

\draw  [line width=1.5pt] (0,0) --  (\y,\y);

\fill [blue] (0,0) circle (3pt);

\draw  [blue,line width=1.5pt]  (0,0) --  (\x,\y);

\draw ($(1.5,3.5)$) node[anchor=south] {\large $\lowertconv(U)$};
	
\end{scope}

\begin{scope}[shift={(8.5,0)}]

\def\x{3.5*0.3};
\def\y{3.5*0.9};

\fill [black!30,opacity=1] (0,0) -- (0,\y) -- (\x,\y) -- cycle;

\draw  [line width=1.5pt] (0,0) --  (0,\y);

\fill [blue] (0,0) circle (3pt);

\draw  [blue,line width=1.5pt]  (0,0) --  (\x,\y);

\draw ($(0.5,3.5)$) node[anchor=south] {\large $\uppertconv(U)$};
	
\end{scope}

\begin{scope}[shift={(11,0)}] 

\def\x{3.5*0.3};
\def\y{3.5*0.9};

\fill [black!30,opacity=1] (0,0) -- (\y,\y) -- (0,\y) -- cycle;

\draw  [line width=1.5pt] (0,\y) -- (0,0) --  (\y,\y);

\fill [blue] (0,0) circle (3pt);

\draw  [blue,line width=1.5pt]  (0,0) --  (\x,\y);

\draw ($(1.5,3.5)$) node[anchor=south] {\large $\loweruppertconv(U)$};
\end{scope}
\end{scope}

\begin{scope}[shift={(0,-5)}] 
\draw (-0.3,4.2) node[anchor=east] {(b)};
\begin{scope}

\draw[->,line width=0.8pt] (0,0) -- (2,0) node[right] {$x_1$};
\draw[->,line width=0.8pt] (0,0) -- (0,3) node[left] {$x_2$};

\def\genx{0.3};
\def\geny{0.9};
\def\n{3};
\def\N{6}

\coordinate (gen) at (\genx,\geny);
\foreach \i in {0,1,...,\n} 
	\coordinate (gen\i) at ($(\i*\genx,\i*\geny)$);

\foreach \i in {0,1,...,\n} 
	\fill [blue] (gen\i) circle (3pt);
	
\foreach \i in {0,1,...,\n} 
	\draw [blue] (gen\i) node[anchor=south west] {$A\i$};
	
\draw ($(1.25,3.5)$) node[anchor=south] {\large $V=\{A_0,\cdots\}$};
	
\end{scope}

\begin{scope}[shift={(4,0)}] 

\def\genx{0.3};
\def\geny{0.9};
\def\n{3};
\def\N{6}

\coordinate (gen) at (\genx,\geny);
\foreach \i in {0,1,...,\n} 
	\coordinate (gen\i) at ($(\i*\genx,\i*\geny)$);

\foreach \i in {0,...,\n} 
	\coordinate (mid\i) at ($(gen\i)+(\genx,\genx)$);
	
	\coordinate (V) at ($(\n*\geny+\geny-\genx,\n*\geny+\genx)$);
	
	\fill [black!30,opacity=1] (mid0) node{} \foreach \i in {1,...,\n} {--(gen\i)--(mid\i) node{}} --(V) -- cycle;

\draw  [line width=1.5pt] (mid0) node{} \foreach \i in {1,...,\N}{ -- \ifodd\i ++(0,\geny-\genx) \else ++(\genx,\genx) \fi node{}};
\draw  [line width=1.5pt] (0,0) --  (V);

\foreach \i in {0,1,...,\n} 
	\fill [blue] (gen\i) circle (3pt);
	
\draw ($(1.5,3.5)$) node[anchor=south] {\large $\lowertconv(V)$};

\end{scope}

\begin{scope}[shift={(8.5,0)}] 
\def\genx{0.3};
\def\geny{0.9};
\def\n{3};
\def\N{6}

\coordinate (gen) at (\genx,\geny);
\foreach \i in {0,1,...,\n} 
	\coordinate (gen\i) at ($(\i*\genx,\i*\geny)$);

\foreach \i in {0,...,\n} 
	\coordinate (mid\i) at ($(gen\i)+(0,\geny-\genx)$);
	
	\coordinate (V) at ($(0,\n*\geny+\geny-\genx)$);
	
	\fill [black!30,opacity=1] (mid0) node{} \foreach \i in {1,...,\n} {--(gen\i)--(mid\i) node{}} --(V) -- cycle;

\draw  [line width=1.5pt] (mid0) node{} \foreach \i in {1,...,\N}{ -- \ifodd\i ++(\genx,\genx) \else ++(0,\geny-\genx) \fi node{}};
\draw  [line width=1.5pt] (0,0) --  (V);

\foreach \i in {0,1,...,\n} 
	\fill [blue] (gen\i) circle (3pt);
	
\draw ($(0.5,3.5)$) node[anchor=south] {\large $\uppertconv(V)$};
\end{scope}

\begin{scope}[shift={(11,0)}] 

\def\genx{0.3};
\def\geny{0.9};
\def\n{3};
\def\N{6}

\coordinate (gen) at (\genx,\geny);
\foreach \i in {0,1,...,\n} 
	\coordinate (gen\i) at ($(\i*\genx,\i*\geny)$);
	
\coordinate (V1) at ($(\n*\geny+\geny-\genx,\n*\geny+\genx)$);
\coordinate (V2) at ($(0,\n*\geny+\genx)$);
	
	\fill [black!30,opacity=1] (0,0) -- (V1) --(V2) -- cycle;

\draw  [line width=1.5pt] (V2) -- (0,0) --  (V1);

\foreach \i in {0,1,...,\n} 
	\fill [blue] (gen\i) circle (3pt);

\draw ($(1.5,3.5)$) node[anchor=south] {\large $\loweruppertconv(V)$};
\end{scope}

\end{scope}
\end{tikzpicture}

\caption{Examples of non-compact tropical convex sets: (a)  Tropical convex hulls generated by the ray $x_2=3x_1$ in the first quadrant.  (b) Tropical convex hulls generated by points $(x_1,3x_1)$ for $x_1=0,1, \cdots$.} \label{F:tconv2}
\end{figure}

Now let us consider some tropical convex hulls generated by noncompact sets. 
Figure~\ref{F:tconv2}(a) shows the tropical convex hulls generated by the ray $U$ defined by  $x_2=3x_1$ with $x_1\geq 0$.  Note that for any  two points $\alpha$ and $\beta$ in $U$, the lower tropical segment $\underline{[\alpha,\beta]}$ is of type $\underline{[\alpha_1,\gamma_2]}$ in Figure~\ref{F:tconv}(a) and the upper tropical segment $\overline{[\alpha,\beta]}$ is of type $\overline{[\beta_4,\gamma_3]}$ in Figure~\ref{F:tconv}(a). For comparison, Figure~\ref{F:tconv2}(b) shows the tropical convex hulls generated by the countable set $V$ defined by  $x_2=3x_1$ with $x_1=0,1,\cdots$.

\qed
\end{example}

\section{$B^p$-Pseudonorms and Tropical Projections}    \label{S:Bpseudonorm}

\subsection{The Definition of $B^p$-Pseudonorms}
Now suppose the underlying space $X$ is a locally compact Hausdorff space equipped with a finite nontrivial Borel measure $\mu$. To specify the measure $\mu$, we also write  $\mbbTP(X)$ as  $\mbbTP(X,\mu)$. Note that all functions in $BC(X)$ are $\mu$-measurable and $\mu$-integrable.  Then we can define some pseudonorms and pseudometrics on $\mbbTP(X,\mu)$ (here ``pseudo'' means not necessarily symmetric). 

\begin{definition} \label{D:pseudonorm}
For a given $1\leq p \leq\infty$, the \emph{$\underline{p}$-pseudonorm} or \emph{$\underline{B}^p$-pseudonorm} on $\mbbTP(X,\mu)$ is a function: $\llfloor\cdot\rrfloor_p:\mbbTP(X,\mu)\to[0,\infty)$ defined by $\llfloor[f]\rrfloor_p:=\Vert \underline{f}\Vert_p$ where $\Vert \cdot\Vert_p$ is the $p$-norm on $BC(X)$. More precisely, for all $[f]\in \mbbTP(X,\mu)$, when $1\leq p <\infty$, $\llfloor[f]\rrfloor_p=\left(\int_X (\underline{f})^pd\mu\right)^{1/p}=\left(\int_X (f-\inf(f))^pd\mu\right)^{1/p}$, and the $\infty$-pseudonorm $\llfloor\cdot\rrfloor_\infty$ on $\mbbTP(X,\mu)$ is defined as $\llfloor[f]\rrfloor_\infty=\lim\limits_{p\to\infty}\llfloor[f]\rrfloor_p$. We also define the \emph{$\overline{p}$-pseudonorm} or  \emph{$\overline{B}^p$-pseudonorm} on $\mbbTP(X,\mu)$ as $\llceil[f]\rrceil_p=\Vert \overline{f}\Vert_p=\Vert -\overline{f}\Vert_p=\left(\int_X (-\overline{f})^pd\mu\right)^{1/p}=\left(\int_X (\sup(f)-f)^pd\mu\right)^{1/p}$ with $p\in[1,\infty)$ and $\llceil[f]\rrceil_\infty=\lim\limits_{p\to\infty}\llceil[f]\rrceil_p$. Both $\underline{p}$-pseudonorm and $\overline{p}$-pseudonorm are called \emph{$p$-pseudonorms} or \emph{$B^p$-pseudonorms}. All $B^p$-pseudonorms are also simply called \emph{$B$-pseudonorms}.
\end{definition}

\subsection{Null Sets}
\begin{lemma} \label{L:null}
For $p\in[1,\infty)$, $\llfloor[f]\rrfloor_p=0$ if and only if $\underline{f}=0$ almost everywhere, and $\llceil[f]\rrceil_p=0$ if and only if $\overline{f}=0$ almost everywhere. 
\end{lemma}
\begin{proof}
Recall that for a non-negative $\mu$-measurable function $F$, $\int_X F d\mu=0$ if and only if $F=0$ almost everywhere (i.e., $\mu(\{x\in X\mid F(x)\neq 0\})=0$). By letting $F=(\underline{f})^p$ and $F=(-\overline{f})^p$ respectively, the statement follows. 
\end{proof}

We have the following notations of null sets:
\begin{enumerate}
\item $\underline{\mcalN}(X,\mu):=\{[f]\in \mbbTP(X,\mu)\mid \underline{f}=0\ \text{almost everywhere}\}$;
\item $\overline{\mcalN}(X,\mu)=\{[f]\in \mbbTP(X,\mu)\mid \overline{f}=0\ \text{almost everywhere}\}$;
\item $\mcalN(X,\mu)=\{[f]\in \mbbTP(X,\mu)\mid f\ \text{is equal to a constant almost everywhere}\}$.
\end{enumerate}
The sets will also be simply denoted as $\underline{\mcalN}$, $\overline{\mcalN}$ and $\mcalN$ respectively when $X$ and $\mu$ are presumed. 

\begin{lemma} \label{L:null2}
The following are some basic properties of $\underline{\mcalN}$, $\overline{\mcalN}$ and $\mcalN$.
\begin{enumerate}
\item $[f]\in\underline{\mcalN}$ if and only if $\llfloor[f]\rrfloor_p=0$ for some $p\in[1,\infty)$ if and only if $\llfloor[f]\rrfloor_p=0$ for all $p\in[1,\infty)$.
\item $[f]\in\overline{\mcalN}$ if and only if $\llceil[f]\rrceil_p=0$ for some $p\in[1,\infty)$ if and only if $\llceil[f]\rrceil_p=0$ for all $p\in[1,\infty)$.
\item $\underline{\mcalN}=-\overline{\mcalN}$.
\item $\underline{\mcalN}\bigcap\overline{\mcalN}=\{[0]\}$. 
\item $\underline{\mcalN}$ (respectively $\overline{\mcalN}$) is a positive cone, i.e., $\underline{\mcalN}\bigcap(-\underline{\mcalN}) =\{[0]\}$ (respectively  $\overline{\mcalN}\bigcap(-\overline{\mcalN}) =\{[0]\}$) and if $\alpha,\beta\in\underline{\mcalN}$ and $a,b\geq 0$, $a\alpha+b\beta\in\underline{\mcalN}$ (respectively $a\alpha+b\beta\in\overline{\mcalN}$).
\item $\mcalN$ is a linear subspace of $\mbbTP(X,\mu)$ spanned by $\underline{\mcalN}(X,\mu)\bigcup\overline{\mcalN}(X,\mu)$. 
\item $\mcalN=\{[0]\}$ if and only if $\underline{\mcalN}=\{[0]\}$ if and only if $\overline{\mcalN}=\{[0]\}$.
\end{enumerate}
\end{lemma}
\begin{proof}
(1) and (2) follows from Lemma~\ref{L:null} directly. 

For (3), we observe that  $[f]\in\underline{\mcalN}$ if and only if $\llfloor[f]\rrfloor_1=\Vert \underline{f}\Vert_1=0$ if and only if $\llceil-[f]\rrceil_1=\Vert \overline{-f}\Vert_1=0$ if and only if $-[f]\in\overline{\mcalN}$.

(4) follows from the fact that $[f]=[0]$ if and only if $\underline{f}=0$ almost everywhere and $\overline{f}=0$.  

Using (3)  and (4), (5) can be easily verified by  definition. 

For (6), it is straightforward to verify that  $\mcalN$ is a linear subspace containing both $\underline{\mcalN}(X,\mu)$ and $\overline{\mcalN}(X,\mu)$. It remains to show that each $[f]\in \mcalN$ can be written as $[f]=[g]+[h]$ where $[g]\in \underline{\mcalN}(X,\mu)$ and $[h]\in \overline{\mcalN}(X,\mu)$. By definition, we know that $f=c$ for some constant $c$ almost everywhere. We let $g=\max(f,c)$ and $h=\min(f,c)$. Then it is clear that $[f]=[g]+[h]$, $[g]\in \underline{\mcalN}(X,\mu)$ is and $[h]\in \overline{\mcalN}(X,\mu)$. 

(7) follows from (3), (4) and (6).

\end{proof}

We say that  $\mcalN$ is trivial if  $\mcalN=\{[0]\}$. 

\begin{lemma}
If the measure of every nonempty open subset of $X$ is nonzero, then $\mcalN$ is trivial. 
\end{lemma}
\begin{proof}
If $[f]\in \underline{\mcalN}$, then  $\mu(\{x\in X\mid \underline{f}(x)\neq 0\})=0$ by definition. Therefore $\{x\in X\mid \underline{f}(x)\neq 0\}=\emptyset$ and thus $[f]=[0]$. By Lemma~\ref{L:null2}(7), $\mcalN$ is trivial. 
\end{proof}

\subsection{Basic Properties of $B^p$-Pseudonorms}

\begin{proposition} \label{P:BnormProperty}
We summarize some properties of the $B^p$-pseudonorms as follows.
\begin{enumerate}

\item $\llceil\alpha\rrceil_p=\llfloor-\alpha\rrfloor_p$.
\item For $1\leq p <\infty$, $\llfloor\alpha\rrfloor_p \leq \mu(X)^{1/p} \Vert \alpha\Vert$ and $\llceil\alpha\rrceil_p \leq \mu(X)^{1/p} \Vert \alpha\Vert$. 
\item The $\underline{p}$-pseudonorms and $\overline{p}$-pseudonorms are continuous functions.
\item Fixing $\alpha\in\mbbTP(X)$, the functions $\mu(X)^{-1/p}\llfloor\alpha\rrfloor_p$ and $\mu(X)^{-1/p}\llceil\alpha\rrceil_p$ are nondecreasing with respect to $p\in [1,\infty]$.
\item $\llfloor\alpha\rrfloor_1+\llceil\alpha\rrceil_1=\llfloor\alpha\rrfloor_1+\llfloor -\alpha\rrfloor_1=\llceil\alpha\rrceil_1+\llceil-\alpha\rrceil_1=\mu(X)\cdot\Vert \alpha \Vert$.
\item $\llfloor\alpha\rrfloor_\infty=\llceil\alpha\rrceil_\infty=\Vert \alpha \Vert$.
\item $\llfloor c\alpha\rrfloor_p=c\llfloor \alpha\rrfloor_p$ and $\llceil c\alpha\rrceil_p=c\llceil \alpha\rrceil_p$ when $c>0$. 
\item For $p\in[1,\infty)$,  $\alpha=[f]$ and $\beta=[g]$, if $\underline{f}\leq\underline{g}$ almost everywhere, then $ \llfloor \alpha\rrfloor_p \leq \llfloor \beta\rrfloor_p$, and  if $\overline{f}\geq\overline{g}$ almost everywhere, then $ \llceil \alpha\rrfloor_p \leq \llceil \beta\rrceil_p$.
\item (The triangle inequalities) For $1\leq p \leq\infty$, $\llfloor \alpha+\beta\rrfloor_p \leq \llfloor \alpha\rrfloor_p +\llfloor \beta\rrfloor_p$ and $\llceil \alpha+\beta\rrceil_p \leq \llceil \alpha\rrceil_p +\llceil \beta\rrceil_p$.

\item $\llfloor \alpha+\beta\rrfloor_1 = \llfloor \alpha\rrfloor_1 +\llfloor \beta\rrfloor_1$ if and only if 
$ \Xmin(\alpha)\Cap \Xmin(\beta) \neq\emptyset$.
\item $\llceil \alpha+\beta\rrceil_1 = \llceil \alpha\rrceil_1 +\llceil \beta\rrceil_1$ if and only if 
$ \Xmax(\alpha)\Cap \Xmax(\beta) \neq\emptyset$.

\item $\Vert \alpha+\beta \Vert= \Vert \alpha\Vert +\Vert \beta \Vert$ if and only if  $\llfloor \alpha+\beta\rrfloor_1 = \llfloor \alpha\rrfloor_1 +\llfloor \beta\rrfloor_1$ and $\llceil \alpha+\beta\rrceil_1 = \llceil \alpha\rrceil_1 +\llceil \beta\rrceil_1$. 

\item For any $p\in[1,\infty)$, $\alpha\in\mbbTP(X)$, $\beta_1\in\underline{\mcalN}$ and $\beta_2\in\overline{\mcalN}$, if $\Xmin(\alpha)\Cap  \Xmin(\beta_1)\neq \emptyset$ and $ \Xmax(\alpha)\Cap \Xmax(\beta_2) \neq\emptyset$, then $\llfloor \alpha+\beta_1\rrfloor_p = \llfloor \alpha\rrfloor_p$ and $\llceil \alpha+\beta_2\rrceil_p = \llceil \alpha\rrceil_p$.

\item[For (14) and (15), we suppose that $\mcalN(X,\mu)$ is trivial.] 
\item For $p\in(1,\infty)$, $\llfloor \alpha+\beta\rrfloor_p = \llfloor \alpha\rrfloor_p +\llfloor \beta\rrfloor_p$ if and only if  $\llceil \alpha+\beta\rrceil_p = \llceil \alpha\rrceil_p +\llceil \beta\rrceil_p$ if and only if either $\alpha=[0]$ or $\beta=c\alpha$ for some $c\geq0$.

\item For any $p\in[1,\infty)$, $\alpha=[f]$ and $\beta=[g]$, if $\underline{f}\lneqq\underline{g}$, then $ \llfloor \alpha\rrfloor_p < \llfloor \beta\rrfloor_p$, and  if $\overline{f}\gneqq\overline{g}$, then $ \llceil \alpha\rrfloor_p < \llfloor \beta\rrceil_p$.

\end{enumerate}
\end{proposition}

\begin{proof}
Let $\alpha=[f]$. 

For (1), we have $\llceil\alpha\rrceil_p=\llceil[f]\rrceil_p=\Vert \overline{f}\Vert_p=\Vert -\overline{f}\Vert_p=\Vert \underline{-f}\Vert_p=\llfloor[-f]\rrfloor_p=\llfloor-\alpha\rrfloor_p$.

For (2), $\llfloor\alpha\rrfloor_p=\left(\int_X (\underline{f})^pd\mu\right)^{1/p}\leq \left(\int_X (\max(f)-\min(f))^pd\mu\right)^{1/p}=\left(\mu(X)\Vert \alpha \Vert^p\right)^{1/p}=\mu(X)^{1/p} \Vert \alpha\Vert$. Moreover, $\llceil \alpha\rrceil =\llfloor -\alpha \rrfloor\leq\mu(X)^{1/p} \Vert -\alpha\Vert = \mu(X)^{1/p} \Vert \alpha\Vert$. 

(3) follows from (2) directly. 

For (4), we know from Jensen's inequality that $\mu(X)^{-1/p}\Vert F\Vert_p$ is non-decreasing for any function $F$. Then (4) follows by letting $F$ be $\underline{f}$ and $\overline{f}$ respectively. 

For (5), note that $\underline{f}-\overline{f}$ is a constant function of value $\max(f)-\min(f)=\Vert f \Vert$. Therefore $\llfloor\alpha\rrfloor_1+\llceil\alpha\rrceil_1=\int_X (f-\min(f))d\mu + \int_X (\max(f)-f)d\mu=\int_X(\max(f)-\min(f))d\mu=\mu(X)\cdot\Vert \alpha \Vert$.

For (6), $\llfloor[f]\rrfloor_\infty=\lim\limits_{p\to\infty}\llfloor[f]\rrfloor_p=\lim\limits_{p\to\infty}\Vert \underline{f}\Vert_p=\Vert\underline{f}\Vert_\infty=\Vert [f]\Vert$. Analogously, $\llceil [f] \rrceil_\infty = \llfloor[-f]\rrfloor_\infty=\Vert [-f] \Vert=\Vert [f] \Vert$. 

(7) and (8) follow from the definition of pseudonorms directly.

Now let $\alpha=[f]$ and $\beta=[g]$. 

For (9), we have $\llfloor \alpha+\beta\rrfloor_p=\Vert \underline{f+g}\Vert_p\leq\Vert \underline{f}+\underline{g}\Vert_p\leq \Vert \underline{f}\Vert_p+\Vert \underline{g}\Vert_p=\llfloor \alpha\rrfloor_p+\llfloor\beta\rrfloor_p$ and $\llceil \alpha+\beta\rrceil_p = \llfloor -\alpha-\beta\rrfloor_p \leq  \llfloor -\alpha\rrfloor_p +\llfloor -\beta\rrfloor_p=\llceil \alpha\rrceil_p +\llceil \beta\rrceil_p$.

For (10), note that  $\underline{f+g}+(\inf(f+g)-\inf(f)-\inf(g))=\underline{f}+\underline{g}$ where $\inf(f+g)-\inf(f)-\inf(g)\geq0$. Therefore, 
$\llfloor \alpha+\beta\rrfloor_1 + (\inf(f+g)-\inf(f)-\inf(g))\mu(X)= \llfloor \alpha\rrfloor_1 +\llfloor\beta\rrfloor_1$.  As in Lemma~\ref{L:SpeIneq}, the statement follows from the fact that $\inf(f+g)-\inf(f)-\inf(g)=0$ if and only if  $ \Xmin(\alpha)\Cap \Xmin(\beta) \neq\emptyset$. Moreover, (11) can be derived by replacing $\alpha$ and $\beta$ with $-\alpha$ and $-\beta$ respectively in the above argument. 

(12) follows from (5), (10) and (11) directly. 

For (13), suppose $\alpha=[f]$, $\beta_1=[g_1]$ and $\beta_2=[g_2]$. Let $Y_1=\Xmin(g_1)$ and $Y_2=\Xmax(g_2)$. Then $\mu(Y_1)=\mu(Y_2)=\mu(X)$ and thus for any non-negative $\mu$-measurable function $F$, $\int_X F d\mu=\int_{Y_1} F d\mu=\int_{Y_2} F d\mu$. Since $\Xmin(\alpha)\Cap  \Xmin(\beta_1)\neq \emptyset$ and $ \Xmax(\alpha)\Cap \Xmax(\beta_2) \neq\emptyset$, we have $\underline{f+g_1}=\underline{f}+\underline{g_1}$ and $\overline{f+g_2}=\overline{f}+\overline{g_2}$. Therefore, 
\begin{align*}
\llfloor \alpha+\beta_1\rrfloor_p &= \left(\int_X (\underline{f+g_1})^pd\mu\right)^{1/p}= \left(\int_X (\underline{f}+\underline{g_1})^pd\mu\right)^{1/p} = \\
&=  \left(\int_{Y_1} (\underline{f}+\underline{g_1})^pd\mu\right)^{1/p}=\left(\int_{Y_1} (\underline{f})^pd\mu\right)^{1/p} =\left(\int_X(\underline{f})^pd\mu\right)^{1/p}=\llfloor \alpha\rrfloor_p
\end{align*} and 
\begin{align*}
\llceil \alpha+\beta_1\rrceil_p &= \left(\int_X (-\overline{f+g_1})^pd\mu\right)^{1/p}= \left(\int_X (-\overline{f}-\overline{g_1})^pd\mu\right)^{1/p} = \\
&=  \left(\int_{Y_2} (-\overline{f}-\overline{g_1})^pd\mu\right)^{1/p}=\left(\int_{Y_2} (-\overline{f})^pd\mu\right)^{1/p} =\left(\int_X(-\overline{f})^pd\mu\right)^{1/p}=\llceil \alpha\rrceil_p.
\end{align*} 

For (14), clearly if either $\alpha=[0]$ or $\beta=c\alpha$ for some $c\geq0$, then $\llfloor \alpha+\beta\rrfloor_p = \llfloor \alpha\rrfloor_p +\llfloor \beta\rrfloor_p$ and $\llceil \alpha+\beta\rrceil_p = \llceil \alpha\rrceil_p +\llceil \beta\rrceil_p$ by definition of pseudonorms. Now suppose $\llfloor \alpha+\beta\rrfloor_p = \llfloor \alpha\rrfloor_p +\llfloor \beta\rrfloor_p$. Then $\Vert \underline{f+g}\Vert_p=\Vert\underline{f}\Vert_p+\Vert\underline{g}\Vert_p$ and we must have  $\Vert \underline{f}+\underline{g}\Vert_p=\Vert\underline{f}\Vert_p+\Vert\underline{g}\Vert_p$ since $\Vert \underline{f+g}\Vert_p\leq \Vert \underline{f}+\underline{g}\Vert_p\leq \Vert\underline{f}\Vert_p+\Vert\underline{g}\Vert_p$. Recall that by Minkowski inequality, $\Vert \underline{f}+\underline{g}\Vert_p\leq\Vert\underline{f}\Vert_p+\Vert\underline{g}\Vert_p$ with equality for $1< p <\infty$ if and only if either $\underline{f}=0$ or $\underline{g}=c\underline{f}$ for some $c\geq0$. It follows that  either $\alpha=[0]$ or $\beta=c\alpha$ for some $c\geq0$. For the case $\llceil \alpha+\beta\rrceil_p = \llceil \alpha\rrceil_p +\llceil \beta\rrceil_p$, a similar argument applies. 

(15) is a special case of (8) where the inequalities are strict. Suppose $\underline{f}\lneqq\underline{g}$. Then it is easy to see that in this case $\underline{g-f}=\underline{g}-\underline{f}\gneqq 0$.  Since $\mcalN(X,\mu)$ is trivial and $\beta-\alpha\neq[0]$, we know that  $\llfloor \beta-\alpha\rrfloor_p> 0$ for all $p\in[1,\infty)$. Now 
$ \llfloor \alpha\rrfloor_p^p+ \llfloor \beta-\alpha\rrfloor_p^p = \int_X( (\underline{f})^p+ (\underline{g-f})^p)d\mu\leq  \int_X( (\underline{f}+\underline{g-f})^p)d\mu= \int_X(\underline{g})^pd\mu= \llfloor \beta\rrfloor_p^p$ and thus  $\llfloor \alpha\rrfloor_p< \llfloor \beta\rrfloor_p$. 

Analogously, suppose $\overline{f}\gneqq\overline{g}$  which implies $\overline{g-f}=\overline{g}-\overline{f}\lneqq 0$. Again, $\llceil \beta-\alpha\rrceil_p> 0$ for all $p\in[1,\infty)$ and we get 
$ \llceil \alpha\rrceil_p^p+ \llceil \beta-\alpha\rrceil_p^p = \int_X( (-\overline{f})^p+ (-\overline{g-f})^p)d\mu \leq \int_X( (-\overline{f}-\overline{g-f})^p)d\mu= \int_X(-\overline{g})^pd\mu= \llceil \beta\rrceil_p^p$ which means $\llceil \alpha\rrceil_p < \llceil \beta\rrceil_p$.

\end{proof}

\subsection{The Main Theorem of Tropical Projections}
For the rest of the paper, we assume  that $X$ is a locally compact Hausdorff space, $\mu$ is a Borel measure on $X$ such that $\mu(X)\in(0,\infty)$ and $\mcalN(X,\mu)$ is trivial. 

\begin{theorem} \label{T:main}
 For a compact subset $T$ of $\mbbTP(X)$ and an arbitrary element $\gamma$ in $\mbbTP(X)$, consider  the following real-valued  functions on $T$ defined by $\gamma$ using the $p$-pseudonorms with respect to $\mu$:

\begin{enumerate}
\item $\Theta^{(T,\gamma)}_\infty:T \to [0,\infty)$ given by $\alpha\mapsto\Vert\alpha-\gamma\Vert$,
\item $\lowerTheta^{(T,\gamma)}_p:T \to [0,\infty)$ with $p\in[1,\infty)$ given by  $\alpha\mapsto\llfloor \alpha-\gamma\rrfloor_p$ and 
\item $\upperTheta^{(T,\gamma)}_p:T \to [0,\infty)$ with $p\in[1,\infty)$ given by  $\alpha\mapsto\llceil \alpha-\gamma\rrceil_p$. 
\end{enumerate}
We have the following conclusions:

\begin{enumerate}
\item Suppose $T$ is lower tropically convex. For each element $\gamma\in \mbbTP(X)$, there is a unique element $\lowerpi_T(\gamma)$ called the lower tropical projection of $\gamma$ to $T$ which minimizes $\lowerTheta^{(T,\gamma)}_p$  for all $p\in[1,\infty)$. Moreover, the minimizer of $\Theta^{(T,\gamma)}_\infty$ is compact and lower tropically convex which contains $\lowerpi_T(\gamma)$. 
\item Suppose $T$ is upper tropically convex. For each element $\gamma\in \mbbTP(X)$, there is a unique element $\upperpi_T(\gamma)$ called the upper tropical projection of $\gamma$ to $T$ which minimizes $\upperTheta^{(T,\gamma)}_p$  for all $p\in[1,\infty)$. Moreover, the minimizer of $\Theta^{(T,\gamma)}_\infty$ is compact and upper tropically convex which contains $\upperpi_T(\gamma)$. 
\item Suppose $T$ is both lower and upper tropically convex. Then for each element $\gamma\in \mbbTP(X)$, the minimizer of $\Theta^{(T,\gamma)}_\infty$  is also both lower and upper tropically convex. In addition, $\lowerpi_T(\gamma)=\upperpi_T(\gamma)$ if and only if the minimizer of $\Theta^{(T,\gamma)}_\infty$ is identical to the singleton $\{\lowerpi_T(\gamma)\}=\{\upperpi_T(\gamma)\}$.
\end{enumerate}
\end{theorem}

\begin{proof}
We will use Proposition~\ref{P:working}. Recall that a closed subset of a complete metric space is complete. Moreover, since $T$ is compact and the $p$-pseudonorms are continuous, the minimizers of $\Theta^{(T,\gamma)}_\infty$, $\lowerTheta^{(T,\gamma)}_p$ and $\upperTheta^{(T,\gamma)}_p$ are nonempty. 
\begin{enumerate}
\item Choose an element $\alpha$ from the nonempty minimizer of $\lowerTheta^{(T,\gamma)}_p$ for some $p\in[1,\infty)$. We first need to show that the minimizer of $\lowerTheta^{(T,\gamma)}_p$  is actually the singleton $\{\alpha\}$. For each element $\beta$ in $T$, consider the lower tropical path $P_{(\alpha,\beta)}$ from  $\alpha$ to $\beta$. Note that since $T$ is lower tropically convex, the whole segment $\underline{[\alpha,\beta]}$ is contained in $T$. By Case (1) and (2) of Proposition~\ref{P:working}, $\lowereta_p(t)=\llfloor P_{(\alpha,\beta)}(t)-\gamma\rrfloor_p=\lowerTheta^{(T,\gamma)}_p(P_{(\alpha,\beta)}(t))$ is either strictly increasing or strictly decreasing at $t=0$. But since $\alpha=P_{(\alpha,\beta)}(0)$ minimizes $\lowerTheta^{(T,\gamma)}_p$, only Case (1) can happen and $\lowereta_p(t)$ must be strictly increasing. This means that the minimizer of $\lowerTheta^{(T,\gamma)}_p$  is exactly the singleton $\{\alpha\}$. Moreover, for all $p\in[1,\infty)$, the minimizers of $\lowerTheta^{(T,\gamma)}_p$  are all identical to $\{\alpha\}$ and we can just let $\lowerpi_T(\gamma)=\alpha$. This is because the condition $\Xmin(\beta-\alpha)\Cap \Xmin(\alpha-\gamma) \neq\emptyset$ for Case (1) is independent of $p$. For the same reason, $\lowerpi_T(\gamma)$ minimizes $\Theta^{(T,\gamma)}_\infty$. 

Let $T_{\min}$ be the minimizer $\Theta^{(T,\gamma)}_\infty$.  Then $T_{\min}$ is compact since $T_{\min}$ is a closed subset of the compact set $T$. To show that $T_{\min}$  is lower tropically convex, we choose two elements $\alpha$ and $\beta$ from $T_{\min}$. Then also by case (1) and (2) of Proposition~\ref{P:working}, the function $\lowereta_\infty(t)=\llfloor P_{(\alpha,\beta)}(t)-\gamma\rrfloor_\infty=\Vert P_{(\alpha,\beta)}(t)-\gamma \Vert=\Theta^{(T,\gamma)}_\infty(P_{(\alpha,\beta)}(t))$ must be a constant function with value being the minimum of $\Theta^{(T,\gamma)}_\infty$. This means that the whole segment $\underline{[\alpha,\beta]}$ is contained in $T_{\min}$. Hence  $T_{\min}$ is lower tropically convex. 

\item It follows from a proof analogous to the above proof of (1) while instead Case (3) and (4) of Proposition~\ref{P:working} are employed and functions  $\uppereta_p(t)=\upperTheta^{(T,\gamma)}_p(P^{(\alpha,\beta)}(t))$ and $\uppereta_\infty(t)=\Theta^{(T,\gamma)}_\infty(P^{(\alpha,\beta)}(t))$ are considered. Again,  we let $\upperpi_T(\gamma)=\alpha$ and only Case (3) can happen for all $p\in[1,\infty]$. 

\item It is clear from (1) and (2) that  the minimizer of $\Theta^{(T,\gamma)}_\infty$  must also be both lower and upper tropically convex when $T$ is both lower and upper tropically convex. It remains to show that if $\lowerpi_T(\gamma)=\upperpi_T(\gamma)=\alpha$, then the  minimizer of $\Theta^{(T,\gamma)}_\infty$  must also be the singleton $\{\alpha\}$. Actually as in the above arguments for (1) and (2), we know that  Case (1) and (3) of Proposition~\ref{P:working} will happen simultaneously for all $\beta\in T$. Then by Lemma~\ref{L:TropNorm}, 
$$\Vert\beta-\gamma\Vert=\Vert\beta-\alpha\Vert + \Vert\alpha-\gamma\Vert$$ since $\Xmin(\beta-\alpha)\Cap \Xmin(\alpha-\gamma) \neq\emptyset$ and $\Xmax(\beta-\alpha)\Cap \Xmax(\alpha-\gamma) \neq\emptyset$. If $\beta\neq\alpha$, then $\Theta^{(T,\gamma)}_\infty(\beta)=\Vert\beta-\gamma\Vert>\Vert\alpha-\gamma\Vert=\Theta^{(T,\gamma)}_\infty(\alpha)$. This  means that the minimum of $\Theta^{(T,\gamma)}_\infty$ is $\Vert\alpha-\gamma\Vert$  and the minimizer of $\Theta^{(T,\gamma)}_\infty$ is  the singleton $\{\alpha\}$. 
\end{enumerate} 
\end{proof}

\begin{remark}
Accordingly,  $\lowerpi_T$ and $\upperpi_T$ can be considered as  maps from $\mbbTP(X)$ to $T$ which are called \emph{lower and upper tropical projections} respectively.
\end{remark}

\begin{remark} \label{R:NoCompact}
The existence of $\lowerpi_T(\gamma)$ (or $\upperpi_T(\gamma)$)  in Theorem~\ref{T:main} is guaranteed by the compactness of $T$. If this compactness condition is withdrawn, as long as we know  the minimizer of $\lowerTheta^{(T,\gamma)}_p$  when $T$ is upper tropically convex (or $\upperTheta^{(T,\gamma)}_p$ when $T$ is lower tropically convex) is nonempty for some $p\in[1,\infty)$,  the existence and uniqueness of  $\lowerpi_T(\gamma)$ (or $\upperpi_T(\gamma)$ respectively) are still guaranteed. A conjecture is that we may only assume $T$ to be a closed instead of compact subset of $\mbbTP(X)$ to make the theorem hold. 
\end{remark}

\begin{proposition} \label{P:working}
Let $\alpha,\beta$ be distinct elements in  $\mbbTP(X)$ such that $\rho(\alpha,\beta)=d$.  For $p\in[1,\infty]$ and $\gamma\in\mbbTP(X)$, consider the functions $\lowereta_p(t)=\llfloor P_{(\alpha,\beta)}(t)-\gamma\rrfloor_p$ and $\uppereta_p(t)=\llceil P^{(\alpha,\beta)}(t)-\gamma\rrceil_p$ for $t\in[0,d]$. Then we have the following cases:
\begin{enumerate}
\item $\Xmin(\beta-\alpha)\Cap \Xmin(\alpha-\gamma) \neq\emptyset$: In this case, $\lowereta_\infty(t)$ is non-decreasing and $\lowereta_p(t)$ with $p\in[1,\infty)$ is strictly increasing for $t\in[0,d]$. Moreover, for $t\in[0,d]$, $\lowereta_1(t)=\llfloor P_{(\alpha,\beta)}(t)-\alpha\rrfloor_1+\llfloor \alpha-\gamma\rrfloor_1$. 

\item $\Xmin(\beta-\alpha)\Cap\Xmin(\alpha-\gamma)=\emptyset$: In this case, $\lowereta_\infty(t)$ is non-increasing at $t=0$ and $\lowereta_p(t)$ with $p\in[1,\infty)$ is strictly decreasing at $t=0$. 

\item $\Xmax(\beta-\alpha)\Cap \Xmax(\alpha-\gamma) \neq\emptyset$: In this case, $\uppereta_\infty(t)$ is non-decreasing and $\uppereta_p(t)$ with $p\in[1,\infty)$ is strictly increasing for $t\in[0,d]$. Moreover, for $t\in[0,d]$, $\uppereta_1(t)=\llceil P^{(\alpha,\beta)}(t)-\alpha\rrceil_1+\llceil \alpha-\gamma\rrceil_1$.

\item $ \Xmax(\beta-\alpha)\Cap\Xmax(\alpha-\gamma)=\emptyset$: In this case, $\uppereta_\infty(t)$ is non-increasing at $t=0$ and $\uppereta_p(t)$ with $p\in[1,\infty)$ is strictly decreasing at $t=0$. 
\end{enumerate}
\end{proposition}
\begin{remark}
We say a function $f(t)$ is non-decreasing (resp. non-increasing, strictly increasing, strictly decreasing, or locally constant) at $t_0$ if there exists $\delta>0$ such that $f(t)$ is non-decreasing (resp. non-increasing, strictly increasing, strictly decreasing, or constant) on $[t_0,t_0+\delta]$. 
\end{remark}
\begin{proof}
Let $\alpha=[f]$, $\beta=[g]$ and $\gamma=[h]$. Recall that by definition, $P_{(\alpha,\beta)}(t)= [\min(t,\underline{g-f})+f]$ and $P^{(\alpha,\beta)}(t)= [\max(-t,\overline{g-f})+f]$ for $t\in[0,d]$. Then $P_{(\alpha,\beta)}(t)-\alpha=[\min(t,\underline{g-f})]$, $P^{(\alpha,\beta)}(t)-\alpha=[\max(-t,\overline{g-f})]$, $P_{(\alpha,\beta)}(t)-\gamma=[\min(t,\underline{g-f})+f-h]$ and $P^{(\alpha,\beta)}(t)-\gamma=[\max(-t,\overline{g-f})+f-h]$. Moreover, for $t\in(0,d]$, $\Xmin(P_{(\alpha,\beta)}(t)-\alpha)=\Xmin(\min(t,\underline{g-f}))=\Xmin(\beta-\alpha)=\Xmin(g-f)$ and  $\Xmax(P^{(\alpha,\beta)}(t)-\alpha)=\Xmax(\max(-t,\overline{g-f}))=\Xmax(\beta-\alpha)=\Xmax(g-f)$. 

For (1), suppose $\Xmin(\beta-\alpha)\Cap \Xmin(\alpha-\gamma) \neq\emptyset$. Then for all $t\in(0,d]$, 
   $\Xmin(P_{(\alpha,\beta)}(t)-\alpha)\Cap\Xmin(\alpha-\gamma)=\Xmin(\beta-\alpha)\Cap\Xmin(\alpha-\gamma)\neq \emptyset$. 

Moreover,  $P_{(\alpha,\beta)}(t)-\gamma=[\min(t,\underline{g-f})+f-h]$ and  in this case, $\underline{\min(t,\underline{g-f})+f-h}=\min(t,\underline{g-f})+\underline{f-h}$.

Now $\lowereta_\infty(t)=\Vert [\min(t,\underline{g-f})+\underline{f-h}]\Vert=\sup(\min(t,\underline{g-f})+\underline{f-h})$ is clearly non-decreasing for $t\in[0,d]$. In addition, as in Proposition~\ref{P:BnormProperty}(10), this implies that for $t\in[0,d]$, $\lowereta_1(t)=\llfloor P_{(\alpha,\beta)}(t)-\gamma\rrfloor_1=\llfloor P_{(\alpha,\beta)}(t)-\alpha\rrfloor_1+\llfloor \alpha-\gamma\rrfloor_1$. 

To show that  $\lowereta_p(t)$ with $p\in[1,\infty)$ is strictly increasing for $t\in[0,d]$,  consider  $t_1,t_2\in[0,d]$ such that $t_1<t_2$ and we claim that $\lowereta_p(t_1)=\llfloor [\min(t_1,\underline{g-f})+\underline{f-h}]\rrfloor_p<\lowereta_p(t_2)=\llfloor [\min(t_2,\underline{g-f})+\underline{f-h}]\rrfloor_p$. Note that $\min(t_1,\underline{g-f})+\underline{f-h}\lneqq \min(t_2,\underline{g-f})+\underline{f-h}$ and the claim follows from Proposition~\ref {P:BnormProperty}(15). 

For (2), suppose $ \Xmin(\beta-\alpha)\Cap\Xmin(\alpha-\gamma)=\emptyset$. In this case, we note that $\beta-\gamma=[g-h]=[\underline{g-f}+\underline{f-h}]$ and $\underline{g-f}+\underline{f-h}=\underline{g-h}+\delta$ where $\delta$ must be strictly larger than $0$. Again recall that  $P_{(\alpha,\beta)}(t)-\gamma=[\min(t,\underline{g-f})+f-h]$ and in this case we claim that  $\underline{\min(t,\underline{g-f})+f-h}=\min(0,\underline{g-f}-t)+\underline{f-h}=\min(\underline{f-h},\underline{g-f}+\underline{f-h}-t)$ for all $t\in[0,\delta]$.  Actually this is implied by  $\min (\underline{g-f}+\underline{f-h}-t)=\min(\underline{g-h}+\delta-t)\geq 0$ for $t\in[0,\delta]$. 

Therefore $\lowereta_\infty(t)=\Vert [\min(0,\underline{g-f}-t)+\underline{f-h}]\Vert=\sup(\min(0,\underline{g-f}-t)+\underline{f-h})$ is clearly non-increasing for $t\in[0,\delta]$. Now consider  $t_1,t_2\in[0,\delta]$ such that $t_1<t_2$. We claim $\lowereta_p(t_1)=\llfloor [\min(0,\underline{g-f}-t_1)+\underline{f-h}]\rrfloor_p<\lowereta_p(t_2)=\llfloor [\min(0,\underline{g-f}-t_2)+\underline{f-h}]\rrfloor_p$. Note that $\min(0,\underline{g-f}-t_1)+\underline{f-h}\gneqq\min(0,\underline{g-f}-t_2)+\underline{f-h}$ and again  the claim follows from Proposition~\ref {P:BnormProperty}(15). 
    
For (3) and (4), we let $\alpha=-\alpha'$, $\beta=-\beta'$ and $\gamma=-\gamma'$. 
Then $\Xmax(\beta-\alpha)=\Xmin(\beta'-\alpha')$, $\Xmax(\alpha-\gamma)=\Xmin(\alpha'-\gamma')$, $\Xmax(\beta-\gamma)=\Xmin(\beta'-\gamma')$,  $P^{(\alpha,\beta)}(t)=-P_{(\alpha',\beta')}(t)$, $\uppereta_p(t)=\llceil P^{(\alpha,\beta)}(t)-\gamma\rrceil_p=\llceil -P_{(\alpha',\beta')}(t)+\gamma'\rrceil_p=\llfloor P_{(\alpha',\beta')}(t)-\gamma'\rrceil_p$.  By replacing $\alpha$, $\beta$ and $\gamma$ with $\alpha'$, $\beta'$ and $\gamma'$ respectively in (1) and (2), we can derive (3) and (4) respectively. 
    
\end{proof}

We have the following corollary of Proposition~\ref{P:working} which summarize some concrete criteria of lower and upper tropical projections stated in Theorem~\ref{T:main}.

\begin{corollary}
[\textbf{Criteria for Tropical Projections}] \label{C:CritTropProj}
Let $T$ be a subset (not necessarily compact) of $\mbbTP(X)$. Let $\gamma\in\mbbTP(X)$ and $\alpha\in T$. 
\begin{enumerate}[(a)]
\item 
If $T$ is lower tropically convex, then the following are equivalent:
\begin{enumerate}[(1)]
\item $\alpha=\lowerpi_T(\gamma)$.
\item For every $p\in[1,\infty)$ and every $\beta\in T$, the function $\lowereta_p(t)=\llfloor P_{(\alpha,\beta)}(t)-\gamma\rrfloor_p$ is strictly increasing for $t\in[0,\rho(\alpha,\beta)]$.
\item For every $p\in[1,\infty)$ and every $\beta\in T$, the function $\lowereta_p(t)=\llfloor P_{(\alpha,\beta)}(t)-\gamma\rrfloor_p$ is strictly increasing at $t=0$.
\item For every $\beta\in T$, $\llfloor \beta-\gamma\rrfloor_1=\llfloor \beta-\alpha\rrfloor_1+\llfloor \alpha-\gamma\rrfloor_1$.
\item For every $\beta\in T$, $\Xmin(\beta-\alpha)\Cap\Xmin(\alpha-\gamma)\neq\emptyset$.
\item For every $\beta\in T$ such that  $\beta\neq \alpha$, $\Xmin(\alpha-\beta)\Cap\Xmin(\beta-\gamma)=\emptyset$.
\end{enumerate}

\item 
If $T$ is upper tropically convex, then the following are equivalent:
\begin{enumerate}[(1)]
\item $\alpha=\upperpi_T(\gamma)$.
\item For every $p\in[1,\infty)$ and every $\beta\in T$, the function $\uppereta_p(t)=\llfloor P_{(\alpha,\beta)}(t)-\gamma\rrfloor_p$ is strictly increasing for $t\in[0,\rho(\alpha,\beta)]$.
\item For every $p\in[1,\infty)$ and every $\beta\in T$, the function $\uppereta_p(t)=\llfloor P_{(\alpha,\beta)}(t)-\gamma\rrfloor_p$ is strictly increasing at $t=0$.
\item For every $\beta\in T$, $\llceil \beta-\gamma\rrceil_1=\llceil \beta-\alpha\rrceil_1+\llceil \alpha-\gamma\rrceil_1$.
\item For every $\beta\in T$, $\Xmax(\beta-\alpha)\Cap\Xmax(\alpha-\gamma)\neq\emptyset$.
\item For every $\beta\in T$ such that  $\beta\neq \alpha$, $\Xmax(\alpha-\beta)\Cap\Xmax(\beta-\gamma)=\emptyset$.
\end{enumerate}
\item The lower and upper tropical projections are independent to the Borel measure $\mu$  on $X$ as long as $\mu(X)\in(0,\infty)$ and $\mcalN(X,\mu)$ is trivial. 
\end{enumerate}
\end{corollary}

\begin{proof}
All the criteria in (a) and (b) easily follow from Proposition~\ref{P:working}. For (c), we note that the set-theoretical criteria (a5), (a6), (b5) and (b6) actually do not depend on the underlying measure $\mu$ itself.
\end{proof}

\begin{remark} \label{R:TropProjMeas}
In Theorem~\ref{T:main}, we see that minimizers of all the functions defined with respect to the $\underline{p}$-pseudonorms and $\overline{p}$-pseudonorms for all $p\in[1,\infty)$ coincide as the upper or lower  tropical projections respectively. However, one may note that the $\underline{p}$-pseudonorms and $\overline{p}$-pseudonorms with $p\in[1,\infty)$ depend on the underlying measure $\mu$ on $X$. On the other hand, Corollary~\ref{C:CritTropProj}(c) says that actually we can expect more: the tropical projections are so intrinsic that they are even independent of the measure $\mu$ on $X$. 
\end{remark}

\begin{example} \label{E:bpseudonorm}
\begin{figure}
\centering
\includegraphics[scale=1]{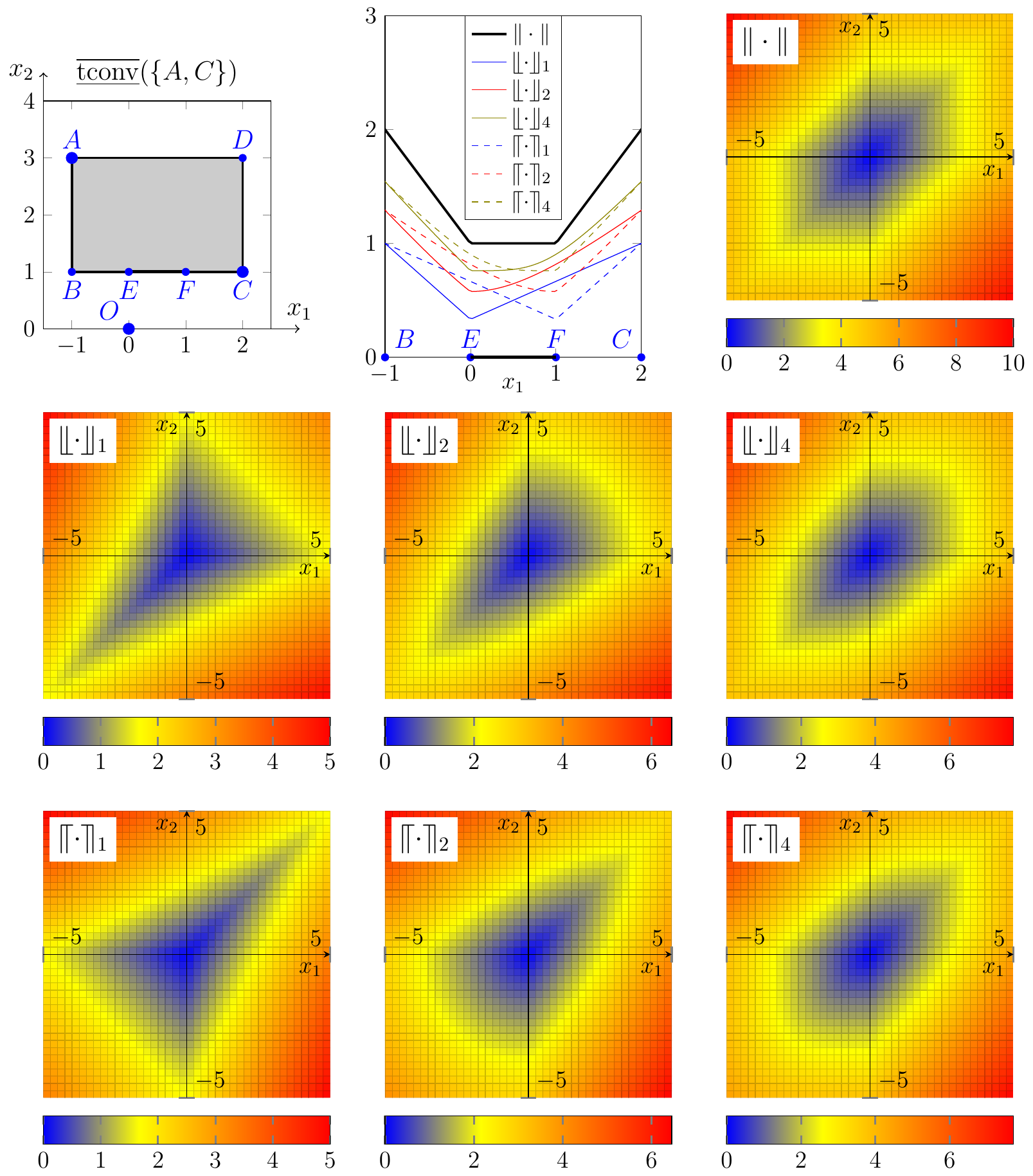}
\caption{Tropical norm and some $B^p$-pseudonorms on $\mbbTP(\{\ve_1,\ve_2,\ve_3\})$ represented by the $x_1x_2$-plane (as in Example~\ref{E:tconv}).} \label{F:bpseudonorm}
\end{figure}

Consider $\mbbTP(X)$ with $X=\{\ve_1,\ve_2,\ve_3\}$ which is represented by the $x_1x_2$-plane as in Example~\ref{E:tconv}. We associate $X$ with a measure $\mu$ such that $\mu(\ve_i)\in(0,\infty)$ for $i=1,2,3$. 
For $\alpha=(x_1,x_2)$ and $p\in[1,\infty)$, we have 
\begin{align*}
\Vert\alpha\Vert &= 
\begin{cases}
\max(x_1,x_2) & \mbox{if } \min(x_1,x_2)\geq 0 \\
x_2-x_1 & \mbox{if } x_1\leq 0 \leq x_2 \\
x_1-x_2 & \mbox{if } x_2\leq 0 \leq x_1 \\
-x_1 & \mbox{if } x_1\leq x_2 \leq 0 \\
-x_2 & \mbox{if } x_2\leq x_1 \leq 0 
\end{cases} \\
\llfloor\alpha\rrfloor_p &= 
\begin{cases}
(\mu(\ve_1)x_1^p+\mu(\ve_2)x_2^p)^{1/p} & \mbox{if } \min(x_1,x_2)\geq 0 \\
(\mu(\ve_2)(x_2-x_1)^p+\mu(\ve_3)(-x_1)^p)^{1/p} & \mbox{if } \min(x_2,0)\geq x_1 \\
(\mu(\ve_1)(x_1-x_2)^p+\mu(\ve_3)(-x_2)^p)^{1/p} & \mbox{if } \min(x_1,0)\geq x_2 
\end{cases} \\
\llceil\alpha\rrceil_p &= 
\begin{cases}
(\mu(\ve_1)(-x_1)^p+\mu(\ve_2)(-x_2)^p)^{1/p} & \mbox{if } \max(x_1,x_2)\leq 0 \\
(\mu(\ve_2)(x_1-x_2)^p+\mu(\ve_3)(x_1)^p)^{1/p} & \mbox{if } \max(x_2,0)\leq x_1 \\
(\mu(\ve_1)(x_2-x_1)^p+\mu(\ve_3)(x_2)^p)^{1/p} & \mbox{if } \max(x_1,0)\leq x_2 
\end{cases} \\
\end{align*}
by definition of the tropical norm and $B^p$-pseudonorms (Definition~\ref{D:pseudonorm}). In particular, Figure~\ref{F:bpseudonorm} shows the 2D plots of the functions $\Vert (x_1,x_2) \Vert$, $\llfloor (x_1,x_2) \rrfloor_1$, $\llfloor (x_1,x_2) \rrfloor_2$, $\llfloor (x_1,x_2) \rrfloor_4$, $\llceil (x_1,x_2) \rrceil_1$, $\llceil (x_1,x_2) \rrceil_2$ and $\llceil (x_1,x_2) \rrceil_4$ for $x_1\in[-5,5]$ and $x_2\in[-5,5]$ where $\mu(\ve_1)=\mu(\ve_2)=\mu(\ve_3)=1/3$. Moreover, in the first subfigure, we let $O=(0,0)$, $A=(-1,3)$, $B=(-1,1)$, $C=(2,1)$, $D=(2,3)$, $E=(0,1)$ and $F=(1,1)$. Let $T$ be the rectangle with vertices $A$, $B$, $C$ and $D$. Then it can be easily verified that $T=\loweruppertconv(\{A,C\})$ which is compact and both lower and upper tropically convex. Then $\Theta^{(T,O)}_\infty(\alpha) = \Vert\alpha\Vert$, $\lowerTheta^{(T,O)}_p(\alpha) =\llfloor \alpha \rrfloor_p$ and  $\upperTheta^{(T,O)}_p(\alpha) =\llceil \alpha \rrceil_p$. In the second subfigure, we plot the curves of $\Vert \alpha\Vert$, $\llfloor  \alpha \rrfloor_1$, $\llfloor  \alpha\rrfloor_2$, $\llfloor  \alpha \rrfloor_4$, $\llceil  \alpha \rrceil_1$, $\llceil  \alpha \rrceil_2$ and $\llceil  \alpha \rrceil_4$ for all $\alpha$ in the segment $BC$. We note that the minimizer of $\Theta^{(T,O)}_\infty$ is the segment $EF$ which is also compact and both lower and upper tropically convex, the minimizers of $\lowerTheta^{(T,O)}_1$, $\lowerTheta^{(T,O)}_2$ and $\lowerTheta^{(T,O)}_4$ are all identical to the singleton $\{E\}$, and  the minimizers of $\upperTheta^{(T,O)}_1$, $\upperTheta^{(T,O)}_2$ and $\upperTheta^{(T,O)}_4$ are all identical to the singleton $\{f\}$. Actually, Theorem~\ref{T:main} tells us that for all $p\in[1,\infty)$, the minimizers of $\lowerTheta^{(T,O)}_p$ are identical to a singleton which must be $\{E\}$, and the minimizers of $\upperTheta^{(T,O)}_p$ are identical to a singleton which must be $\{F\}$. Therefore, the lower tropical projection $\lowerpi_T(O)$ and upper tropical projection $\upperpi_T(O)$ of $O$ to $T$ are $E$ and $F$ respectively. 
\end{example}

\subsection{Basic Properties of Tropical Projections}

The following proposition shows  some basic properties of tropical projection maps.

\begin{proposition} \label{P:TropProj}
 Let $T$ be compact subsets of $\mbbTP(X)$. 
\begin{enumerate}
\item If $T$  is  lower tropically convex(respectively upper tropically convex), then $\alpha_0+cT$ is lower tropically convex(respectively upper tropically convex) for each $\alpha_0\in\mbbTP(X)$ and $c>0$, and $\lowerpi_{\alpha_0+cT}(\alpha_0+c\alpha)=\alpha_0+c\lowerpi_T(\alpha)$ (respectively $\upperpi_{\alpha_0+cT}(\alpha_0+c\alpha)=\alpha_0+c\upperpi_T(\alpha)$) for all $\alpha\in\mbbTP(X)$. 
\item If $T$  is  lower tropically convex (which is equivalent to say $-T$ upper tropically convex), then $\upperpi_{-T}(-\alpha)=-\lowerpi_T(\alpha)$ for all $\alpha\in\mbbTP(X)$. 
\item If  $T$  is lower tropically convex (respectively upper tropically convex), then $\lowerpi_T(\alpha)=\alpha$ (respectively $\upperpi_T(\alpha)=\alpha$) for all $\gamma\in T$. 
\item If $T$  is lower tropically convex (respectively upper tropically convex), then for each $\alpha\in\mbbTP(X)$ and each $\beta\in \underline{[\alpha,\lowerpi_T(\alpha)]}$ (respectively $\beta\in \overline{[\alpha,\upperpi_T(\alpha)]}$), we have $\lowerpi_T(\beta)=\lowerpi_T(\alpha)$ (respectively $\upperpi_T(\beta)=\upperpi_T(\alpha)$). 
\item For each $\alpha,\beta\in \mbbTP(X)$ and $T$ being lower tropically convex, $\rho(\lowerpi_T(\alpha), \lowerpi_T(\beta))\leq\rho(\alpha,\beta)$ and the equality holds if and only if $\alpha=\upperpi_L(\lowerpi_T(\alpha))$ and $\beta=\upperpi_L(\lowerpi_T(\beta))$ where $L=\overline{[\alpha,\beta]}$. Accordingly, for each $\alpha,\beta\in \mbbTP(X)$ and $T$ being upper tropically convex, $\rho(\upperpi_T(\alpha), \upperpi_T(\beta))\leq\rho(\alpha,\beta)$ and the equality holds if and only if $\alpha=\lowerpi_L(\upperpi_T(\alpha))$ and $\beta=\lowerpi_L(\upperpi_T(\beta))$ where $L=\underline{[\alpha,\beta]}$. 
\item For $\alpha,\beta,\gamma\in\mbbTP(X)$, if $T$ is lower tropically convex and $\rho(\lowerpi_T(\alpha), \lowerpi_T(\beta))=\rho(\alpha,\beta)=\rho(\alpha,\gamma)+\rho(\gamma,\beta)$, then $\rho(\alpha,\gamma)= \rho(\lowerpi_T(\alpha), \lowerpi_T(\gamma))$  and $\rho(\gamma,\beta)= \rho(\lowerpi_T(\gamma), \lowerpi_T(\beta))$;  if $T$ is upper tropically convex and $\rho(\upperpi_T(\alpha), \upperpi_T(\beta))=\rho(\alpha,\beta)=\rho(\alpha,\gamma)+\rho(\gamma,\beta)$, then $\rho(\alpha,\gamma)= \rho(\upperpi_T(\alpha), \upperpi_T(\gamma))$  and $\rho(\gamma,\beta)= \rho(\upperpi_T(\gamma), \upperpi_T(\beta))$. 

\end{enumerate}

\end{proposition}

\begin{remark}
The existence of $\lowerpi_{T}(\alpha)$, and $\upperpi_{T}(\alpha)$ in the above proposition  is not guaranteed for all $\alpha
\in \mbbTP(X)$ if we don't assume the compactness of $T$. As in Remark~\ref{R:NoCompact} for Theorem~\ref{T:main}, if the compactness assumption in the above proposition is withdrawn, the statements still hold if whenever $\lowerpi_{T}(\cdot)$ or $\upperpi_{T}(\cdot)$ is mentioned, its existence is assumed. 

\end{remark}

\begin{proof}
(1) follows from Lemma~\ref{L:OpTropHull}(1) and (2), Proposition~\ref{P:BnormProperty}(7), Corollary~\ref{C:CritTropProj}~(a4) and (b4)  easily. 
(2) follows from Lemma~\ref{L:OpTropHull}(3), Proposition~\ref{P:BnormProperty}(1) and Corollary~\ref{C:CritTropProj}~(a4) and (b4) easily. 
(3) is also clear by Corollary~\ref{C:CritTropProj}~(a4) and (b4). 

For (4), we note that if $T$  is lower tropically convex (respectively upper tropically convex), then for $\beta\in \underline{[\alpha,\lowerpi_T(\alpha)]}$ (respectively $\beta\in \overline{[\alpha,\upperpi_T(\alpha)]}$), we have $\Xmin(\lowerpi_T(\alpha)-\alpha)\subseteq \Xmin(\lowerpi_T(\alpha)-\beta)$ (respectively $\Xmax(\upperpi_T(\alpha)-\alpha)\subseteq \Xmax(\upperpi_T(\alpha)-\beta)$).  Therefore,  $\lowerpi_T(\beta)=\lowerpi_T(\alpha)$ (respectively $\upperpi_T(\beta)=\upperpi_T(\alpha)$) by Corollary~\ref{C:CritTropProj}~(a5) (respectively (b5)). 

For (5), we use Proposition~\ref{P:BnormProperty}(1)(5)(9), Corollary~\ref{C:CritTropProj}~(a4) and (b4), and see that when $T$ is lower tropically convex, 

\begin{align*}
&\rho(\lowerpi_T(\alpha), \lowerpi_T(\beta)) \mu(X) \\
&=\llfloor \lowerpi_T(\alpha)-\lowerpi_T(\beta)\rrfloor_1 + \llfloor \lowerpi_T(\beta)-\lowerpi_T(\alpha)\rrfloor_1\\
&= (\llfloor \lowerpi_T(\alpha)-\beta\rrfloor_1- \llfloor \lowerpi_T(\beta)-\beta\rrfloor_1)+(\llfloor \lowerpi_T(\beta)-\alpha\rrfloor_1- \llfloor \lowerpi_T(\alpha)-\alpha\rrfloor_1) \\
&= (\llfloor \lowerpi_T(\alpha)-\beta\rrfloor_1-\llfloor \lowerpi_T(\alpha)-\alpha\rrfloor_1)+(\llfloor \lowerpi_T(\beta)-\alpha\rrfloor_1-\llfloor \lowerpi_T(\beta)-\beta\rrfloor_1) \\
&\leq  \llfloor \alpha-\beta\rrfloor_1+\llfloor \beta-\alpha\rrfloor_1\\
&=\rho(\alpha, \beta) \mu(X).
\end{align*}
with equality holds if and only if $\llfloor \lowerpi_T(\alpha)-\beta\rrfloor_1=\llfloor \lowerpi_T(\alpha)-\alpha\rrfloor_1+\llfloor \alpha-\beta\rrfloor_1$ and $\llfloor \lowerpi_T(\beta)-\alpha\rrfloor_1=\llfloor \lowerpi_T(\beta)-\beta\rrfloor_1+\llfloor \beta-\alpha\rrfloor_1$ if and only if $\llceil \beta- \lowerpi_T(\alpha)\rrceil_1=\llceil \alpha-\lowerpi_T(\alpha)\rrceil_1+\llceil \beta-\alpha\rrceil_1$ and $\llceil \alpha-\lowerpi_T(\beta)\rrceil_1=\llceil \beta-\lowerpi_T(\beta)\rrceil_1+\llceil \alpha-\beta\rrceil_1$  if and only if $\alpha$ is the upper tropical projection of $\lowerpi_T(\alpha)$ to $\overline{[\alpha,\beta]}$ and $\beta$ is the upper tropical projection of $\lowerpi_T(\beta)$ to $\overline{[\alpha,\beta]}$. Accordingly, we can prove the case when $T$ is upper tropically convex. 

For (6), we have 
$\rho(\alpha,\beta)=\rho(\lowerpi_T(\alpha), \lowerpi_T(\beta))\leq\rho(\lowerpi_T(\alpha), \lowerpi_T(\gamma))+\rho(\lowerpi_T(\gamma), \lowerpi_T(\beta))\leq \rho(\alpha,\gamma)+\rho(\gamma,\beta)=\rho(\alpha,\beta)$ which implies $\rho(\alpha,\gamma)= \rho(\lowerpi_T(\alpha), \lowerpi_T(\gamma))$  and $\rho(\gamma,\beta)= \rho(\lowerpi_T(\gamma), \lowerpi_T(\beta))$. The case when $T$ is upper tropically convex can be proved analogously.
\end{proof}

\begin{proposition} \label{P:SeqTropProj}
Let $T$ and $T'$ be compact lower tropical convex (respectively upper tropical convex) subsets of $\mbbTP(X)$ such that  $T'\subseteq T$. Then for each $\alpha\in \mbbTP(X)$, $\lowerpi_{T'}(\alpha)=\lowerpi_{T'}(\lowerpi_T(\alpha))$ (respectively $\upperpi_{T'}(\alpha)=\upperpi_{T'}(\upperpi_T(\alpha))$). 
\end{proposition}

\begin{proof}
We will only prove the case of lower tropical convexity while the case of upper tropical convexity can be proved analogously. 

To prove $\lowerpi_{T'}(\alpha)=\lowerpi_{T'}(\lowerpi_T(\alpha))$, it suffices to show that $\llfloor \beta-\alpha \rrfloor_1 = \llfloor \beta-\lowerpi_{T'}(\lowerpi_T(\alpha)) \rrfloor_1 +\llfloor \lowerpi_{T'}(\lowerpi_T(\alpha))-\alpha \rrfloor_1 $ for every $\beta\in T'$ by Corollary~\ref{C:CritTropProj}(a4).

Actually, by applying Corollary~\ref{C:CritTropProj}(a4) to $T$ with respect to $\alpha$, we get
$\llfloor \gamma-\alpha \rrfloor_1 = \llfloor \gamma-\lowerpi_T(\alpha) \rrfloor_1 +\llfloor\lowerpi_T(\alpha)-\alpha \rrfloor_1 $ for every $\gamma\in T$, and in particular $\llfloor \lowerpi_{T'}(\lowerpi_T(\alpha))-\alpha \rrfloor_1 = \llfloor \lowerpi_{T'}(\lowerpi_T(\alpha))-\lowerpi_T(\alpha) \rrfloor_1 +\llfloor\lowerpi_T(\alpha)-\alpha \rrfloor_1 $

Now applying Corollary~\ref{C:CritTropProj}(a4) to $T'$  with respect to $\lowerpi_T(\alpha)$, we get 
$\llfloor \beta-\lowerpi_T(\alpha)\rrfloor_1 = \llfloor \beta-\lowerpi_{T'}(\lowerpi_T(\alpha)) \rrfloor_1 +\llfloor\lowerpi_{T'}(\lowerpi_T(\alpha))-\lowerpi_T(\alpha)\rrfloor_1 $ for every $\beta\in T'$.

Therefore,
\begin{align*}
& \llfloor \beta-\alpha \rrfloor_1 =\llfloor \beta-\lowerpi_T(\alpha)\rrfloor_1+\llfloor \lowerpi_T(\alpha)-\alpha\rrfloor_1 \\
&= \llfloor \beta-\lowerpi_{T'}(\lowerpi_T(\alpha)) \rrfloor_1 +\llfloor\lowerpi_{T'}(\lowerpi_T(\alpha))-\lowerpi_T(\alpha)\rrfloor_1+\llfloor\lowerpi_T(\alpha)-\alpha \rrfloor_1 \\
&= \llfloor \beta-\lowerpi_{T'}(\lowerpi_T(\alpha)) \rrfloor_1 +\llfloor \lowerpi_{T'}(\lowerpi_T(\alpha))-\alpha \rrfloor_1 
\end{align*}
for every $\beta\in T'$, which means that $\lowerpi_{T'}(\lowerpi_T(\alpha))$ is exactly the  lower tropical projection of $\alpha$ in $T'$ as claimed.

\end{proof}

For some initial results of tropical projections, we consider the tropical projections to tropical segments and have the following lemmas.

\begin{lemma} \label{L:TropProjSegment}
Let $\beta_1,\beta_2,\gamma$ be elements in $\mbbTP(X)$ and let $\alpha=\lowerpi_{\underline{[\beta_1,\beta_2]}}(\gamma)$ (respectively $\alpha=\upperpi_{\overline{[\beta_1,\beta_2]}}(\gamma)$). Then 
for all $\beta\in\underline{[\beta_1,\beta_2]}$,  $\Xmin(\beta-\gamma)\subseteq \Xmin(\alpha-\gamma)=\Xmin(\beta_1-\gamma)\bigcup\Xmin(\beta_2-\gamma)$ (respectively for all $\beta\in\overline{[\beta_1,\beta_2]}$,  $\Xmax(\beta-\gamma)\subseteq \Xmax(\alpha-\gamma)=\Xmax(\beta_1-\gamma)\bigcup\Xmax(\beta_2-\gamma)$). 
\end{lemma}
\begin{proof}
We will prove the case of lower tropical convexity while the case of upper tropical convexity can be proved analogously. 
Since $\alpha=\lowerpi_{\underline{[\beta_1,\beta_2]}}(\gamma)$, we have 
$$\Xmin(\beta_1-\gamma)=\Xmin(\beta_1-\alpha)\bigcap\Xmin(\alpha-\gamma),$$
$$\Xmin(\beta_2-\gamma)=\Xmin(\beta_2-\alpha)\bigcap\Xmin(\alpha-\gamma),$$
$$\Xmin(\beta-\gamma)=\Xmin(\beta-\alpha)\bigcap\Xmin(\alpha-\gamma)$$
by Corollary~\ref{C:CritTropProj} and Lemma~\ref{L:XminXmax}. 

In addition, $\Xmin(\beta_1-\alpha)\bigcup \Xmin(\beta_2-\alpha)=X$ by Proposition~\ref{P:TropSeg}(5) and therefore 
\begin{align*}
\Xmin(\beta-\gamma)&=\Xmin(\beta-\alpha)\bigcap\Xmin(\alpha-\gamma)\subseteq\Xmin(\alpha-\gamma) \\
&=(\Xmin(\beta_1-\alpha)\bigcap\Xmin(\alpha-\gamma))\bigcup(\Xmin(\beta_2-\alpha)\bigcap\Xmin(\alpha-\gamma)) \\
&=\Xmin(\beta_1-\gamma)\bigcup \Xmin(\beta_2-\gamma).
\end{align*}

\end{proof}

\begin{lemma} \label{L:SegSum}
Let $\alpha,\beta,\gamma$ be elements in $\mbbTP(X)$. Then $\rho(\alpha,\gamma)=\rho(\alpha,\beta)+\rho(\beta,\gamma)$ if and only if $\beta=\lowerpi_{\underline{[\alpha,\beta]}}(\gamma)=\lowerpi_{\underline{[\beta,\gamma]}}(\alpha)$ if and only if $\beta=\upperpi_{\overline{[\alpha,\beta]}}(\gamma)=\upperpi_{\overline{[\beta,\gamma]}}(\alpha)$. 
\end{lemma}

\begin{proof}
By Proposition~\ref{P:BnormProperty}, $\rho(\alpha,\gamma)=\rho(\alpha,\beta)+\rho(\beta,\gamma)$ if and only if 
$\llfloor \alpha-\gamma\rrfloor_1 =\llfloor \alpha-\beta\rrfloor_1+\llfloor \beta-\gamma \rrfloor_1 $ and 
$\llfloor \gamma-\alpha\rrfloor_1 =\llfloor \gamma-\beta\rrfloor_1+\llfloor \beta-\alpha \rrfloor_1 $
if and only if
$\llceil \alpha-\gamma\rrceil_1 =\llceil \alpha-\beta\rrceil_1+\llceil \beta-\gamma \rrceil_1 $ and 
$\llceil \gamma-\alpha\rrceil_1 =\llceil \gamma-\beta\rrceil_1+\llceil \beta-\alpha \rrceil_1 $. 

We then note that 
\begin{enumerate}
\item $\llfloor \alpha-\gamma\rrfloor_1 =\llfloor \alpha-\beta\rrfloor_1+\llfloor \beta-\gamma \rrfloor_1 $ if and only if $\beta=\lowerpi_{\underline{[\alpha,\beta]}}(\gamma)$;
\item $\llfloor \gamma-\alpha\rrfloor_1 =\llfloor \gamma-\beta\rrfloor_1+\llfloor \beta-\alpha \rrfloor_1 $ if and only if $\beta=\lowerpi_{\underline{[\beta,\gamma]}}(\alpha)$;
\item $\llceil \alpha-\gamma\rrceil_1 =\llceil \alpha-\beta\rrceil_1+\llceil \beta-\gamma \rrceil_1 $ if and only if $\beta = \upperpi_{\overline{[\alpha,\beta]}}(\gamma)$;
\item $\llceil \gamma-\alpha\rrceil_1 =\llceil \gamma-\beta\rrceil_1+\llceil \beta-\alpha \rrceil_1 $ if and only if $\beta = \upperpi_{\overline{[\beta,\gamma]}}(\alpha)$.
\end{enumerate}
\end{proof}

\subsection{Balls and Tropical Convex Functions}
For any $\alpha\in\mbbTP(X)$, we say that  $\mcalB(\alpha,r):=\{\beta\in \mbbTP(X)\mid \rho(\alpha,\beta)\leq r\}$ and $\mcalB^0(\alpha,r):=\{\beta\in \mbbTP(X)\mid \rho(\alpha,\beta)< r\}$ are respectively the \emph{closed ball} and \emph{open ball}  of radius $r$ centered at $\alpha$. Moreover, for a function $f$ on $\mbbTP(X)$ and a function $g$ on a subset $S$ of $\mbbTP(X)$, we write $\mcalL_{*r}(f):=\{\alpha\in\mbbTP(X)\mid f(\alpha)*r\}$ and $\mcalL^S_{*r}(g):=\{\alpha\in S\mid g(\alpha)*r\}$ where $*$ can be $=$, $<$, $\leq$, $>$ or $\geq$. Note that in this sense, $\mcalB(\alpha,r)=\mcalL_{\leq r}(\rho(\alpha,\cdot))=\mcalL_{\leq r}(\Vert \cdot -\alpha\Vert)$ and $\mcalB^0(\alpha,r)=\mcalL_{< r}(\rho(\alpha,\cdot))=\mcalL_{< r}(\Vert \cdot-\alpha\Vert)$. 

\begin{lemma}
Let $\alpha\in\mbbTP(X)$, $r\geq 0$ and $p\in[1,\infty]$. 
\begin{enumerate}
\item $\mcalL_{\leq r}(\llfloor \cdot-\alpha\rrfloor_p)$ and $\mcalL_{< r}(\llfloor \cdot-\alpha\rrfloor_p)$ are lower tropically convex.
\item $\mcalL_{\leq r}(\llceil \cdot-\alpha\rrceil_p)$ and $\mcalL_{< r}(\llceil \cdot-\alpha\rrceil_p)$ are upper tropically convex.
\item $\mcalB(\alpha,r)$ and $\mcalB^0(\alpha,r)$ are both lower and upper tropically convex. 
\end{enumerate}
 \end{lemma}
 
 \begin{proof}
 Let $\beta_1$ and $\beta_2$ be elements in $\mbbTP(X)$. Let $\beta_0=\lowerpi_{\underline{[\beta_1,\beta_2]}}(\alpha)$ Then by Theorem~\ref{T:main} and Proposition~\ref{P:working}, for all $\beta\in\underline{[\beta_0,\beta_1]}$, $\llfloor \beta_0-\alpha\rrfloor_p\leq \llfloor \beta-\alpha\rrfloor_p\ \leq \llfloor \beta_1-\alpha\rrfloor_p$, and for all $\beta\in\underline{[\beta_0,\beta_2]}$, $\llfloor \beta_0-\alpha\rrfloor_p\leq \llfloor \beta-\alpha\rrfloor_p\ \leq \llfloor \beta_2-\alpha\rrfloor_p$. In sum, 
 $\llfloor \beta-\alpha\rrfloor_p\leq \max(\llfloor \beta_1-\alpha\rrfloor_p,\llfloor \beta_2-\alpha\rrfloor_p)$. Then (1) follows from Definition~\ref{D:TropConv}. 
 
 Using an analogous argument with respect to the case of upper tropical convexity, we can derive (2). 
 
 For (3), we note that $\mcalB(\alpha,r)=\mcalL_{\leq r}(\Vert \cdot -\alpha\Vert)=\mcalL_{\leq r}(\llfloor \cdot -\alpha\rrfloor_\infty)=\mcalL_{\leq r}(\llceil \cdot -\alpha\rrceil_\infty)$ (Proposition~\ref{P:BnormProperty}(6)). Then (3) follows from (1) and (2). 
  \end{proof}

\begin{example}

\begin{figure}
\centering
\includegraphics[scale=1]{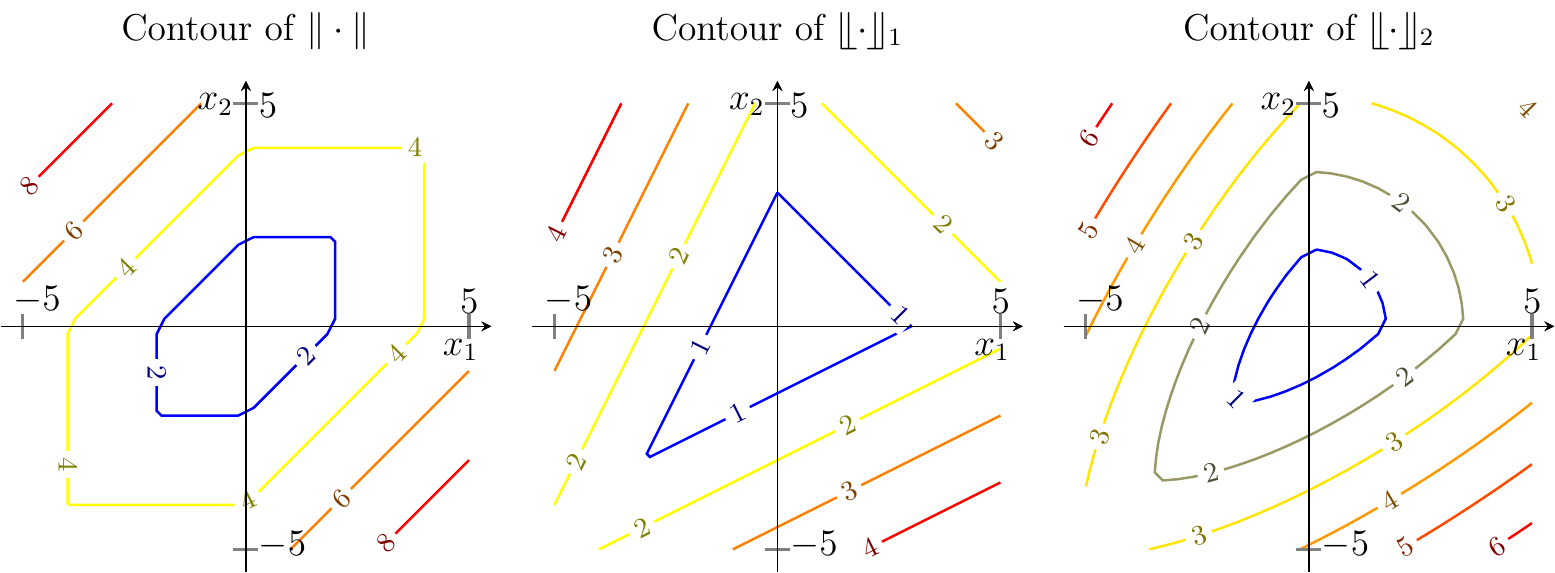}
 \caption{Contour lines of the tropical norm $\Vert\cdot\Vert$ and the  $B$-pseudonorms $\llfloor\cdot \rrfloor_1$ and $\llfloor\cdot \rrfloor_2$ on $\mbbTP(\{\ve_1,\ve_2,\ve_3\})$ represented by the $x_1x_2$-plane (as in Example~\ref{E:tconv} and Example~\ref{E:bpseudonorm}).} \label{F:contour}
 \end{figure}
 
 Figure~\ref{F:contour} shows some contour lines of $\Vert\cdot\Vert$, $\llfloor\cdot \rrfloor_1$ and $\llfloor\cdot \rrfloor_2$  under the same assumption of Example~\ref{E:bpseudonorm}. We observe that the contour lines of $\Vert\cdot\Vert$ are hexagons which enclose the balls  centered at the origin, the contour lines of $\llfloor\cdot \rrfloor_1$ are triangles which enclose the sub-level sets $\mcalL_{\leq r}(\llfloor \cdot\rrfloor_1)$ for different $r$'s, and the contour lines of $\llfloor\cdot \rrfloor_2$ are arrowhead-shaped curves which enclose the sub-level sets $\mcalL_{\leq r}(\llfloor \cdot\rrfloor_2)$ for different $r$'s. 
 
\end{example}

\begin{definition} \label{D:TropConvFunc}
Let $T$ be a lower tropically convex subset (respectively upper tropically convex subset) of $\mbbTP(X)$. Then we say a function $f$ on $T$ is \emph{lower tropically convex} (respectively \emph{upper tropically convex}) on $T$ if for each $\alpha,\beta\in T$, we have $f(P_{(\alpha,\beta)}(td))\leq tf(\alpha)+(1-t)f(\beta)$ (respectively $f(P^{(\alpha,\beta)}(td))\leq tf(\alpha)+(1-t)f(\beta)$) where $d=\rho(\alpha,\beta)$. 
\end{definition}

\begin{lemma}
If $T$ is a lower tropically convex subset   (respectively upper tropically convex subset) of $\mbbTP(X)$ and $f$ is a lower tropically convex  (respectively upper tropically convex) function on $T$, then $\mcalL^T_{\leq r}(f)$ and $\mcalL^T_{< r}(f)$ are lower tropically convex (respectively upper tropically convex). 
\end{lemma}
\begin{proof}
This can be easily verified from Definition~\ref{D:TropConv} and Definition~\ref{D:TropConvFunc}. 
\end{proof}

\section{Tropical Retracts and the Contractibiity of Tropical Convex Sets}  \label{S:TropRetract}
\begin{definition}
Let $W\subseteq\mbbTP(X)$ be lower (respectively upper)   tropically convex. Let $T$ be a compact lower (respectively upper) tropical convex subset of $W$. Then we say a strong deformation retraction $h:[0,1]\times W \rightarrow W$ of $W$ onto $T$ is a \emph{lower (respectively upper) tropical retraction} if at each $t\in[0,1]$, the set $h(t,W)$ is lower (respectively upper) tropically convex. In this sense, we say $T$ is a \emph{lower (respectively upper) tropical retract} of $W$.
\end{definition}

For $\gamma\in\mbbTP(X)$ and $S\subseteq\mbbTP(X)$, we write $\rho(\gamma,S):=\inf_{\beta\in S}\rho(\gamma,\beta)$. 

\begin{theorem} \label{T:TropRetract}
For each lower   (respectively upper)   tropically convex set $W\subseteq\mbbTP(X)$ and each compact lower (respectively upper) tropical convex subset  $T$ of $W$, there exists a lower   (respectively upper)  tropical retraction of $W$ onto $T$.
\end{theorem}
\begin{proof}
Here we give a proof for the case of lower tropical convexity, while a proof for the case of upper tropical convexity can be derived analogously. 

We will explicitly construct such a tropical retraction $h:[0,1]\times W \rightarrow W$. In particular, for each $\gamma\in W$, we want $h(0,\gamma)=\gamma$ and $h(1,\gamma)=\lowerpi_T(\gamma)$. 

Note that since $T$ is compact and lower tropically convex, we must have  $\rho(\gamma,T)=\min_{\beta\in T}\rho(\gamma,\beta)=\rho(\gamma,\lowerpi_T(\gamma))$ by Theorem~\ref{T:main}. Let $\kappa=\sup_{\gamma\in W}\rho(\gamma,T)$. Note that  $T$ must be bounded since it is compact, and $\kappa=\infty$ when $W$ is not bounded. 

Let $\phi:[0,\kappa]\to [0,1]$ be any continuous monotonically decreasing function such that $\phi(0)=1$ and $\phi(\kappa)=0$ (we allow $\kappa$ to be $\infty$). Then we define $h$ in the following way:

$$ h(t, \gamma) = 
  \begin{cases}
    \gamma, & \text{for } 0 \leq t < \phi(\rho(\gamma,T))\\
    P_{(\gamma,\lowerpi_T(\gamma))}(\rho(\gamma,T)-\phi^{-1}(t)), & \text{for } \phi(\rho(\gamma,T)) \leq t \leq 1.
  \end{cases}$$
  
 In other words, at $t=\phi(\rho(\gamma,T))$, the point at $\gamma$ starts to retract towards  $\lowerpi_T(\gamma)$ along $\underline{[\gamma,\lowerpi_T(\gamma)]}$ and at $t=1$, it hits $\lowerpi_T(\gamma)$. It is clear that $h(t,\gamma)=\gamma$ for all $t\in[0,1]$ and $\gamma\in T$. Now, to show $h$ is actually a lower tropical retraction of $W$ onto $T$, it remains to show that $h(t,W)$ is lower tropically convex for all $t\in[0,1]$ and $h$ is continuous. 
 
 To show that $h(t,W)$ is lower tropically convex, we note that the distance $\rho(h(t, \gamma),T)$ from $h(t, \gamma)$ to $T$
 is at most $\phi^{-1}(t)$. More precisely, $\rho(h(t, \gamma),T)=\phi^{-1}(t)$ when  $t\geq \phi(\rho(\gamma,T)) $ and $\rho(h(t, \gamma),T)=\rho(\gamma,T)$ when $t\leq \phi(\rho(\gamma,T)) $. Therefore, $h(t,W)$ is identical to the sublevel set $\{\gamma\in W\mid \rho(\gamma,T) \leq \phi^{-1}(t)\}$ of $\rho(\cdot,T)$. Then for each $\gamma_1,\gamma_2\in h(t,W)$, we have $\rho(\gamma_1,T)\leq \phi^{-1}(t)$ and $\rho(\gamma_2,T)\leq \phi^{-1}(t)$. To prove that $h(t,W)$ is lower tropically convex, it remains to show that for each $\gamma\in \underline{[\gamma_1,\gamma_2]}$, $\rho(\gamma,T)\leq \phi^{-1}(t)$. Let $\alpha_1=\lowerpi_T(\gamma_1)$ and $\alpha_2=\lowerpi_T(\gamma_2)$. Let $\alpha=\lowerpi_{\underline{[\alpha_1,\alpha_2]}}(\gamma)$. Then by Proposition~\ref{P:DistIneq}(2), 
 $$\rho(\gamma,T)\leq\rho(\gamma,\alpha)\leq\max(\rho(\gamma_1,\alpha_1),\rho(\gamma_2,\alpha_2))=\max(\rho(\gamma_1,T),\rho(\gamma_2,T))\leq \phi^{-1}(t).$$
 
 Now let us show that $h$ is continuous, i.e., $h(t_n,\gamma_n)\to h(t,\gamma)$ as $t_n\to t$ and $\gamma_n\to \gamma$. 
 By the triangle  inequality, 
 $$\rho(h(t_n,\gamma_n),h(t,\gamma))\leq\rho(h(t_n,\gamma_n),h(t_n,\gamma))+\rho(h(t_n,\gamma),h(t,\gamma)).$$
 
 We first note that $\rho(h(t_n,\gamma),h(t,\gamma))\leq \rho(\gamma,T)|\phi^{-1}(t)-\phi^{-1}(t_n)|$. Since $\phi$ and $\phi^{-1}$ are continuous, $\rho(h(t_n,\gamma),h(t,\gamma))\to 0$ as $t_n\to t$. Write $\gamma'_n=h(t_n,\gamma_n)$ and $\gamma'=h(t_n,\gamma)$. We then claim that $\rho(\gamma'_n,\gamma')\leq 2\rho(\gamma_n,\gamma)$ which is sufficient to guarantee  the continuity of $h$. 
 
Let $\alpha_n = \lowerpi_T(\gamma_n)=\lowerpi_T(\gamma'_n)$ and  $\alpha = \lowerpi_T(\gamma)=\lowerpi_T(\gamma')$. Then $\rho(\gamma_n,\alpha_n)=\rho(\gamma_n,T)$ and $\rho(\gamma,\alpha)=\rho(\gamma_n, T)$. Note that $\rho(\alpha_n,\alpha)\leq \rho(\gamma_n,\gamma)$ by Proposition~\ref{P:TropProj}(5).

Case (a): $\phi^{-1}(t) \geq \max(\rho(\gamma_n,T),\rho(\gamma,T))$. In this case, $\gamma'_n=\gamma_n$, $\gamma'=\gamma$, and thus $\rho(\gamma'_n,\gamma')=\rho(\gamma_n,\gamma)$.

Case (b): $ \min(\rho(\gamma_n,T),\rho(\gamma,T))\leq \phi^{-1}(t) \leq \max(\rho(\gamma_n,T),\rho(\gamma,T))$. Without loss of generality, we may assume 
$\rho(\gamma_n,T)\leq\rho(\gamma,T)$ which means $\rho(\gamma_n,\alpha_n)=\rho(\gamma_n,T)\leq\phi^{-1}(t) \leq\rho(\gamma,T)=\rho(\gamma,\alpha)$. Then $\gamma'_n=\gamma_n$ and $\rho(\gamma',\alpha)=\rho(\gamma',T)=\phi^{-1}(t)$. 

Let $\gamma'' =\lowerpi_{\underline{[\gamma,\alpha]}}(\gamma_n)$ and $\alpha'' =\lowerpi_{\underline{[\gamma,\alpha]}}(\alpha_n)$.  Depending on the relative positions of $\gamma'$ and $\gamma''$ in $\underline{[\gamma,\alpha]}$, there are two subcases:

Subcase (b1): $\gamma''\in\underline{[\alpha,\gamma']}$. Then $\rho(\gamma'_n,\gamma')=\rho(\gamma_n,\gamma')\leq \rho(\gamma_n,\gamma)$. 

Subcase (b2): $\gamma'\in\underline{[\alpha,\gamma'']}$. Note that $\rho(\alpha'',\alpha)\leq \rho(\alpha_n,\alpha)$, $\rho(\gamma'',\alpha'')\leq \rho(\gamma'_n,\alpha_n)$ (Proposition~\ref{P:TropProj}(5)) and $\rho(\gamma_n,\alpha_n)\leq\phi^{-1}(t) =\rho(\gamma',\alpha)\leq\rho(\gamma'',\alpha)$. 
Therefore,
\begin{align*}
\rho(\gamma'',\gamma') &=\rho(\gamma'',\alpha)-\rho(\gamma',\alpha)\leq(\rho(\gamma'',\alpha'')+\rho(\alpha'',\alpha))-\rho(\gamma',\alpha) \\
&\leq(\rho(\gamma'_n,\alpha_n)+\rho(\alpha_n,\alpha))-\rho(\gamma',\alpha)\leq \rho(\alpha_n,\alpha)\leq\rho(\gamma_n,\gamma).
\end{align*}
In addition, we have $\rho(\gamma_n,\gamma'')\leq\rho(\gamma_n,\gamma).$
It follows $$\rho(\gamma'_n,\gamma')=\rho(\gamma_n,\gamma')\leq\rho(\gamma_n,\gamma'')+\rho(\gamma'',\gamma')\leq 2\cdot\rho(\gamma_n,\gamma).$$

Case (c): $\phi^{-1}(t) \leq \min(\rho(\gamma_n,T),\rho(\gamma,T))$. Then $\rho(\gamma'_n,\alpha_n)=\rho(\gamma'_n,T) =\rho(\gamma',T)=\rho(\gamma',\alpha)=\phi^{-1}(t)$.

Let $\gamma''=\lowerpi_{\underline{[\gamma,\alpha]}}(\gamma'_n)$ and $\gamma''_n=\lowerpi_{\underline{[\gamma_n,\alpha_n]}}(\gamma')$. Depending on the relative positions of $\gamma'$ and $\gamma''$ in $\underline{[\gamma,\alpha]}$ and of the relative positions of $\gamma'_n$ and $\gamma''_n$ in $\underline{[\gamma_n,\alpha_n]}$, there are four subcases:

Subcase (c1): $\gamma''\in \underline{[\gamma',\alpha]}$ and $\gamma''_n\in \underline{[\gamma'_n,\alpha_n]}$. Then by Proposition~\ref{P:DistIneq}(1), 
$\rho(\gamma'_n,\gamma')\leq \rho(\gamma_n,\gamma)$.

Subcase (c2):  $\gamma''\in \underline{[\gamma',\gamma]}$ and $\gamma''_n\in \underline{[\gamma'_n,\gamma_n]}$. Then again by Proposition~\ref{P:DistIneq}(1), $\rho(\gamma'_n,\gamma')\leq \rho(\alpha_n,\alpha)\leq \rho(\gamma_n,\gamma)$.

Subcase (c3): $\gamma''\in \underline{[\gamma',\gamma]}$ and $\gamma''_n\in \underline{[\gamma'_n,\alpha_n]}$. Analyzing as in Subcase (b2) by introducing $\alpha'' =\lowerpi_{\underline{[\gamma,\alpha]}}(\alpha_n)$, we can derive $\rho(\gamma'',\gamma')\leq  \rho(\alpha_n,\alpha)\leq\rho(\gamma_n,\gamma)$. Moreover, $\rho(\gamma'_n,\gamma'')\leq \max(\rho(\alpha_n,\alpha),\rho(\gamma_n,\gamma))=\rho(\gamma_n,\gamma)$ by Proposition~\ref{P:DistIneq}(2). Therefore,
$\rho(\gamma'_n,\gamma')\leq\rho(\gamma'_n,\gamma'')+\rho(\gamma'',\gamma')\leq 2\cdot\rho(\gamma_n,\gamma)$.

Subcase (c4): $\gamma''\in \underline{[\gamma',\alpha]}$ and $\gamma''_n\in \underline{[\gamma'_n,\gamma_n]}$. Analogous to the analysis in Subcase (c3), we get $\rho(\gamma'_n,\gamma')\leq\rho(\gamma'_n,\gamma''_n)+\rho(\gamma''_n,\gamma')\leq 2\cdot\rho(\gamma_n,\gamma)$.

\end{proof}

\begin{corollary} \label{C:TropContr}
All lower or upper tropical convex sets are contractible. 
\end{corollary}
\begin{proof}
Apply Theorem~\ref{T:TropRetract} while letting $W$ be a singleton. 
\end{proof}

\begin{proposition} \label{P:DistIneq}

For $\alpha_1,\alpha_2,\gamma_1,\gamma_2\in \mbbTP(X)$, consider $\beta_1\in \underline{[\alpha_1,\gamma_1]}$ and $\beta_2\in \underline{[\alpha_2,\gamma_2]}$ (respectively consider $\beta_1\in \overline{[\alpha_1,\gamma_1]}$ and $\beta_2\in \overline{[\alpha_2,\gamma_2]}$).  

\begin{enumerate}
\item Let $\alpha'_1=\lowerpi_{\underline{[\alpha_1,\gamma_1]}}(\beta_2)$ and $\alpha'_2=\lowerpi_{ \underline{[\alpha_2,\gamma_2]}}(\beta_1)$.  If $\alpha'_1\in \underline{[\alpha_1,\beta_1]}$ and $\alpha'_2\in \underline{[\alpha_2,\beta_2]}$, then $\rho(\beta_1,\beta_2)\leq \rho(\gamma'_1,\gamma'_2)$ for all $\gamma'_1\in \underline{[\beta_1,\gamma_1]}$ and $\gamma'_2\in \underline{[\beta_2,\gamma_2]}$. 

(Respectively, let $\alpha'_1=\upperpi_{\overline{[\alpha_1,\gamma_1]}}(\beta_2)$ and $\alpha'_2=\upperpi_{ \overline{[\alpha_2,\gamma_2]}}(\beta_1)$.  If $\alpha'_1\in \overline{[\alpha_1,\beta_1]}$ and $\alpha'_2\in \overline{[\alpha_2,\beta_2]}$, then $\rho(\beta_1,\beta_2)\leq \rho(\gamma'_1,\gamma'_2)$ for all $\gamma'_1\in \overline{[\beta_1,\gamma_1]}$ and $\gamma'_2\in \overline{[\beta_2,\gamma_2]}$.)

\item If $\beta_1=\lowerpi_{\underline{[\alpha_1,\gamma_1]}}(\beta_2)$, then $\rho(\beta_1,\beta_2)\leq \max(\rho(\alpha_1,\alpha_2),\rho(\gamma_1,\gamma_2))$. If in addition $\alpha_1=\alpha_2$, then $\rho(\beta_1,\beta_2)\leq \rho(\gamma'_1,\gamma'_2)$ for all $\gamma'_1\in \underline{[\beta_1,\gamma_1]}$ and $\gamma'_2\in \underline{[\beta_2,\gamma_2]}$. 

(Respectively, if $\beta_1=\upperpi_{\overline{[\alpha_1,\gamma_1]}}(\beta_2)$, then $\rho(\beta_1,\beta_2)\leq \max(\rho(\alpha_1,\alpha_2),\rho(\gamma_1,\gamma_2))$. If in addition $\alpha_1=\alpha_2$, then $\rho(\beta_1,\beta_2)\leq \rho(\gamma'_1,\gamma'_2)$ for all $\gamma'_1\in \overline{[\beta_1,\gamma_1]}$ and $\gamma'_2\in \overline{[\beta_2,\gamma_2]}$.)

\item If  $\alpha_1=\alpha_2=\alpha$ and $\rho(\alpha,\beta_1)=\rho(\alpha,\beta_2)$, then $\rho(\beta_1,\beta_2)\leq \rho(\gamma'_1,\gamma'_2)$ for all $\gamma'_1\in \underline{[\beta_1,\gamma_1]}$ and $\gamma'_2\in \underline{[\beta_2,\gamma_2]}$. 

(Respectively, if  $\alpha_1=\alpha_2=\alpha$ and $\rho(\alpha,\beta_1)=\rho(\alpha,\beta_2)$, then $\rho(\beta_1,\beta_2)\leq \rho(\gamma'_1,\gamma'_2)$ for all $\gamma'_1\in \overline{[\beta_1,\gamma_1]}$ and $\gamma'_2\in \overline{[\beta_2,\gamma_2]}$.)

\end{enumerate}

\end{proposition}

\begin{proof}
Here we give proofs for the cases of lower tropical convexity. The proofs for upper tropical convexity cases can be derived analogously. 

For (1), based on the relative locations of the points, we have the following equalities by  Corollary~\ref{C:CritTropProj} 

\begin{align*}
&\llfloor \gamma'_1-\beta_2\rrfloor_1 \\
&=\llfloor \gamma'_1-\alpha'_1\rrfloor_1+\llfloor \alpha'_1-\beta_2\rrfloor_1 \\
&=\llfloor \gamma'_1-\beta_1\rrfloor_1+\llfloor \beta_1-\alpha'_1\rrfloor_1+\llfloor \alpha'_1-\beta_2\rrfloor_1 \\
&=\llfloor \gamma'_1-\beta_1\rrfloor_1+\llfloor \beta_1-\beta_2\rrfloor_1,
\end{align*}

and analogously $ \llfloor \gamma'_2-\beta_1\rrfloor_1=\llfloor \gamma'_2-\beta_2\rrfloor_1+\llfloor \beta_2-\beta_1\rrfloor_1$.

Therefore,
\begin{align*}
\rho(\beta_1,\beta_2)\mu(X) &= \llfloor \beta_1-\beta_2\rrfloor_1+\llfloor \beta_2-\beta_1\rrfloor_1 \\
&(\llfloor \gamma'_1-\beta_2\rrfloor_1-\llfloor \gamma'_1-\beta_1\rrfloor_1)+(\llfloor \gamma'_2-\beta_1\rrfloor_1-\llfloor \gamma'_2-\beta_2\rrfloor_1) \\
&=(\llfloor \gamma'_1-\beta_2\rrfloor_1-\llfloor \gamma'_2-\beta_2\rrfloor_1)+(\llfloor \gamma'_2-\beta_1\rrfloor_1-\llfloor \gamma'_1-\beta_1\rrfloor_1) \\
&\leq \llfloor \gamma'_1-\gamma'_2\rrfloor_1 +\llfloor \gamma'_2-\gamma'_1\rrfloor_1 \\
&=\rho(\gamma'_1,\gamma'_2)\mu(X).
\end{align*}

For (2), we want to apply (1). Note that $\alpha'_1=\beta_1=\lowerpi_{\underline{[\alpha_1,\gamma_1]}}(\beta_2)$. Thus by (1), if $\alpha'_2=\lowerpi_{ \underline{[\alpha_2,\gamma_2]}}(\beta_1)$ lies in $\underline{[\alpha_2,\beta_2]}$, then $\rho(\beta_1,\beta_2)\leq \rho(\gamma_1,\gamma_2)$, and if 
$\alpha'_2$ lies in $\underline{[\gamma_2,\beta_2]}$, then $\rho(\beta_1,\beta_2)\leq \rho(\alpha_1,\alpha_2)$.  This means $\rho(\beta_1,\beta_2)\leq \max(\rho(\alpha_1,\alpha_2),\rho(\gamma_1,\gamma_2))$. 

Now if in addition $\alpha=\alpha_1=\alpha_2$, then we claim that $\alpha'_2=\lowerpi_{ \underline{[\alpha_2,\gamma_2]}}(\beta_1)$ lies in $\underline{[\alpha_2,\beta_2]}$ which means that $\rho(\beta_1,\beta_2)\leq \rho(\gamma'_1,\gamma'_2)$ for all $\gamma'_1\in \underline{[\beta_1,\gamma_1]}$ and $\gamma'_2\in \underline{[\beta_2,\gamma_2]}$ by (1). Actually using  Proposition~\ref{P:TropProj}(5) twice, we can derive $\rho(\alpha,\alpha'_2)\leq \rho(\alpha,\beta_1) \leq \rho(\alpha,\beta_2)$. 

For (3), we can derive $\rho(\alpha,\alpha'_1)\leq \rho(\alpha,\beta_2)=\rho(\alpha,\beta_1)$ and $\rho(\alpha,\alpha'_2)\leq \rho(\alpha,\beta_1)=\rho(\alpha,\beta_2)$ by using Proposition~\ref{P:TropProj}(5).   Again,  we can apply (1).
\end{proof}

\section{Compact Tropical Convex Sets} \label{S:Compact}

\subsection{A Construction of Compact Tropical Convex Sets}
\begin{theorem}\label{T:construction}
Let $T,T'\subseteq\mbbTP(X)$ be lower (respectively upper) tropical convex set. Then $\lowertconv(T\bigcup T')=\bigcup_{\alpha\in T,\alpha'\in T'}\lowertconv(\alpha,\alpha')$ (respectively $\uppertconv(T\bigcup T')=\bigcup_{\alpha\in T,\alpha'\in T'}\uppertconv(\alpha,\alpha')$). If $T$ and $T'$ are compact in addition, then $\lowertconv(T\bigcup T')$ (respectively $\uppertconv(T\bigcup T')$) is compact and each $\beta\in\lowertconv(T\bigcup T')$ (respectively $\beta\in\uppertconv(T\bigcup T')$) lies on the tropical segment $\underline{[\lowerpi_T(\beta),\lowerpi_{T'}(\beta)]}$ (respectively $\overline{[\upperpi_T(\beta),\upperpi_{T'}(\beta)]}$).
\end{theorem}

\begin{proof}
We give a proof for the lower tropical convexity case and the proof for the upper tropical convexity case can be derived analogously. 

Denote $\bigcup_{\alpha\in T,\alpha'\in T'}\lowertconv(\alpha,\alpha')$ by $\tilde{T}$. Then clearly $\tilde{T}\subseteq\lowertconv(T\bigcup T')$. We claim that  $\lowertconv(T\bigcup T')\subseteq\tilde{T}$.

Note that by Theorem~\ref{T:TconvTlinear}, $\lowertconv(T\bigcup T')=\hminplus(T\bigcup T')$ which means that for  each $[g]\in\lowertconv(T\bigcup T')$, we can write  $g=(c_1\odot f_1)\minplus\cdots\minplus (c_m\odot f_m)\minplus (c'_1\odot f'_1)\minplus\cdots\minplus (c'_n\odot f'_n)$ for some $m,n\in \mbbZ^+$,  $c_i,c'_j\in\mbbR$, $[f_i]\in T$ and $[f'_j]\in T'$. Let $f=(c_1\odot f_1)\minplus\cdots\minplus (c_m\odot f_m)$ and $f'=(c'_1\odot f'_1)\minplus\cdots\minplus (c'_n\odot f'_n)$. Then $[f]\in T$, $[f']\in T'$ and $g=f\minplus f'$. Thus $[g]\in\underline{[[f],[f']]}$ which implies $\lowertconv(T\bigcup T')\subseteq\tilde{T}$ as claimed.

Recall that a metric space is compact if and only if it is complete and totally bounded. Now let us show that if in addition $T$ and $T'$ are complete and totally bounded, then $\tilde{T}$ is also complete and totally bounded.

First, we claim that for each $\beta\in\tilde{T}$, we must have $\beta\in\underline{[\alpha,\alpha']}$ where $\alpha=\lowerpi_T(\beta)$ and $\alpha'=\lowerpi_{T'}(\beta)$. We must assume that $\beta\in\underline{[\alpha+0,\alpha'_0]}$  for some $\alpha_0\in T$ and $\alpha'_0\in T'$. Then 
By Corollary~\ref{C:CritTropProj}(a5) and Lemma~\ref{L:XminXmax}, we have $\Xmin(\alpha_0-\beta)\subseteq \Xmin(\alpha-\beta)$ and $\Xmin(\alpha'_0-\beta)\subseteq \Xmin(\alpha'-\beta)$. Therefore, $\Xmin(\alpha-\beta)\bigcup \Xmin(\alpha'-\beta)\supseteq \Xmin(\alpha_0-\beta)\bigcup \Xmin(\alpha'_0-\beta)=X$ which means that  $\beta\in\underline{[\alpha,\alpha']}$ by Proposition~\ref{P:TropSeg}(5).

Now we want to show that $\tilde{T}$ is complete. Let $\beta_1,\beta_2,\cdots$ be a Cauchy sequence in $\tilde{T}$, i.e., $\rho(\beta_m,\beta_n)\to 0$ as $m,n\to\infty$. We claim that there exists $\beta\in\tilde{T}$ such that $\rho(\beta_n,\beta)\to 0$ as $n\to\infty$, which implies the completeness of $\tilde{T}$. Since $T$ is compact, we may let $\alpha_i=\lowerpi_T(\beta_i)$ and $\alpha'_i=\lowerpi_{T'}(\beta_i)$.  Note that $\beta_i\in\underline{[\alpha_i,\alpha'_i]}$. Moreover, $\alpha_1,\alpha_2,\cdots$ is a Cauchy sequence in $T$ and $\alpha'_1,\alpha'_2,\cdots$ is a Cauchy sequence in $T'$ since $\rho(\alpha_m,\alpha_n)\leq \rho(\beta_m,\beta_n)$ and $\rho(\alpha'_m,\alpha'_n)\leq \rho(\beta_m,\beta_n)$ by Proposition~\ref{P:TropProj}(5). Now let $\alpha$ be the limit of $\alpha_1,\alpha_2,\cdots$ and $\alpha'$ be the limit of $\alpha'_1,\alpha'_2,\cdots$. Consider the tropical segment $\underline{[\alpha,\alpha']}$ and let $\gamma_i = \lowerpi_{\underline{[\alpha,\alpha']}}(\beta_i)$. Again $\gamma_1,\gamma_2,\cdots$ is a Cauchy sequence in $\underline{[\alpha,\alpha']}$ by Proposition~\ref{P:TropProj}(5). We let $\beta$ be the limit of $\gamma_1,\gamma_2,\cdots$ and claim $\beta$ is also the limit of $\beta_1,\beta_2,\cdots$. 
Note that  $\rho(\beta_n,\beta) \leq \rho(\beta_n,\gamma_n)+\rho(\gamma_n,\beta)$. By Proposition~\ref{P:DistIneq}(2), we have $\rho(\beta_n,\gamma_n)\leq \max(\rho(\alpha_n,\alpha),\rho(\alpha'_n,\alpha'))$. Now $\rho(\alpha_n,\alpha)\to 0$, $\rho(\alpha'_n,\alpha')\to 0$ and $\rho(\gamma_n,\beta)\to 0$ as $n\to\infty$, which implies $\rho(\beta_n,\beta)\to 0$ as $n\to\infty$ as claimed. 

Next, we want to show that $\tilde{T}$ is totally bounded, i.e., for every real $\epsilon>0$, there exists a finite cover of $\tilde{T}$ by open balls of radius $\epsilon$. 
We start with a finite cover of $T$ by open balls $\mcalB^0(\alpha_i,\epsilon/2)$ of radius $\epsilon/2$ centered at $\alpha_i\in T$ for $i=1,\cdots,n$, and a finite cover of $T'$ by open balls $\mcalB^0(\alpha'_j,\epsilon/2)$ of radius $\epsilon/2$ centered at $\alpha'_j\in T'$ for $j=1,\cdots,m$.
Then for each $\underline{[\alpha_i,\alpha'_j]}$, we have a finite cover by open balls $\mcalB^0(\beta^{(i,j)}_{k^{(i,j)}},\epsilon/2)$ of radius $\epsilon/2$ centered at $\beta^{(i,j)}_{k^{(i,j)}}\in \underline{[\alpha_i,\alpha'_j]}$ for $k^{(i,j)}=1,\ldots,m^{(i,j)}$. We claim that there is a finite cover of $\tilde{T}$ by open balls $\mcalB^0(\beta^{(i,j)}_{k^{(i,j)}},\epsilon)$ of radius $\epsilon$ centered at  $\beta^{(i,j)}_{k^{(i,j)}}\in\tilde{T}$ for $i=1,\ldots,n$, $j=1,\ldots,m$ and $k^{(i,j)}=1,\ldots,m^{(i,j)}$.
For any $\beta\in\tilde{T}$, there exist $\alpha\in T$ and $\alpha'\in T'$ such that $\beta\in\underline{[\alpha,\alpha']}$. Suppose $\alpha\in \mcalB^0(\alpha_i,\epsilon/2)$ for some $1\leq i\leq n$ and $\alpha'\in \mcalB^0(\alpha'_j,\epsilon/2)$ for some $1\leq j\leq m$. Let $\gamma=\lowerpi_{\underline{[\alpha_i,\alpha'_j]}}(\beta)$ and suppose $\gamma\in \mcalB^0(\beta^{(i,j)}_{k^{(i,j)}},\epsilon/2)$ for some $\beta^{(i,j)}_{k^{(i,j)}}$. Then by Proposition~\ref{P:DistIneq}(2),  we have 
$$\rho(\beta,\beta^{(i,j)}_{k^{(i,j)}})\leq \rho(\beta,\gamma)+\rho(\gamma,\beta^{(i,j)}_{k^{(i,j)}})\leq \max(\rho(\alpha,\alpha_i),\rho(\alpha',\alpha'_j))+\rho(\gamma,\beta^{(i,j)}_{k^{(i,j)}})<\epsilon/2+\epsilon/2=\epsilon.$$
Therefore $\beta$ lies in $\mcalB^0(\beta^{(i,j)}_{k^{(i,j)}},\epsilon)$ which means $\tilde{T}$ is covered by this finite collection of open balls as claimed.
\end{proof}

\begin{corollary} \label{C:Polytope}
Let $T_1,\cdots,T_n$ be compact subsets of $\mbbTP(X)$ which are also lower (respectively upper) tropically convex. Then $\lowertconv(T_1\bigcup \cdots\bigcup T_n)$ (respectively $\uppertconv(T_1\bigcup \cdots\bigcup T_n)$) is compact. In particular, all lower and upper tropical polytopes are compact. 
\end{corollary}

By the corollary, since tropical polytopes are compact, we can always apply tropical projections from $\mbbTP(X)$ to any tropical polytope. 

\subsection{Closed Tropical Convex Hulls and The Tropical Version of Mazur's Theorem}
Recall that the closed (conventional) convex hull of a compact subset $S$ of a Banach space is also compact (Mazur's theorem). Note that $\mbbTP(X)$ is a Banach space and here we prove an analogue of Mazur's theorem for tropical convexity.

\begin{proposition} \label{P:closure}
The topological closure $\cl(T)$ of any lower  (respectively upper) tropical convex set $T$ is lower  (respectively upper) tropically convex. 
\end{proposition}
\begin{proof}
We show here the case for lower tropical convexity and the proof for the case of upper tropical convexity can be derived analogously. 
Let $\alpha$ and $\beta$ be elements in $\cl(T)$. Then there is a sequence $\alpha_1,\alpha_2,\cdots$ in $T$ converging to $\alpha$ and a sequence $\beta_1,\beta_2,\cdots$ in $T$ converging to $\beta$. Then for each $\gamma\in\underline{[\alpha,\beta]}$, let $\gamma_n=\lowerpi_{\underline{[\alpha_n,\beta_n]}}(\gamma)$. By Proposition~\ref{P:DistIneq}(2), we see that $\rho(\gamma,\gamma_n)\leq\max(\rho(\alpha,\alpha_n),\rho(\beta,\beta_n))$. Therefore, $\rho(\gamma,\gamma_n)\to 0$ as $n\to\infty$ since $\rho(\alpha,\alpha_n)\to 0$ and $\rho(\beta,\beta_n)\to 0$ as $n\to\infty$. Now since $\gamma_1,\gamma_2,\cdots$ is a sequence in $T$, we conclude that $\gamma\in\cl(T)$ which means $\cl(T)$ is also lower tropically convex. 
\end{proof}

\begin{definition} \label{D:ClTropConv}
Let $S$ be a subset of $\mbbTP(X)$. The \emph{closed lower tropical convex hull $\lowerTCONV(S)$} (respectively \emph{closed upper tropical convex hull $\upperTCONV(S)$}) generated by $S$ is the intersection of all closed lower (respectively upper) tropically convex subsets of $\mbbTP(X)$ containing $S$.
\end{definition}

\begin{lemma}
$\lowerTCONV(S)=\cl(\lowertconv(S))$, $\upperTCONV(S)=\cl(\uppertconv(S))$ and $\lowerTCONV(\upperTCONV(S))=\upperTCONV(\lowerTCONV(S))=\lowerTCONV(\uppertconv(S))=\upperTCONV(\lowertconv(S))=\cl(\loweruppertconv(S))$. 
\end{lemma}

\begin{proof}
The statements follow from Definition~\ref{D:ClTropConv}  and Proposition~\ref{P:closure} directly. 
\end{proof} 

\begin{remark}
We will write $\lowerupperTCONV(S):=\cl(\loweruppertconv(S))$. 

\end{remark}

Before discussing the tropical version of Mazur's theorem, we will need the following proposition which is a generalization of Proposition~\ref{P:DistIneq}(2). 

\begin{proposition} \label{P:GeneralIneq}
Consider $\alpha_1,\cdots,\alpha_n,\beta_1,\cdots,\beta_n\in\mbbTP(X)$ and let $d_i=\rho(\alpha_i,\beta_i)$ for $i=1,\cdots,n$. Let $T=\lowertconv(\{\beta_1,\cdots,\beta_n\})$ (respectively $T=\uppertconv(\{\beta_1,\cdots,\beta_n\})$). Then for each $\alpha\in \lowertconv(\{\alpha_1,\cdots, \alpha_n\})$, $\rho(\alpha,T)=\rho(\alpha,\lowerpi_T(\alpha))\leq \max(d_1,\cdots,d_n)$ (respectively for each $\alpha\in \uppertconv(\{\alpha_1,\cdots, \alpha_n\})$, $\rho(\alpha,T)=\rho(\alpha,\upperpi_T(\alpha))\leq \max(d_1,\cdots,d_n)$). 
\end{proposition}

\begin{proof}
First we note that $\rho(\alpha,T)=\rho(\alpha,\lowerpi_T(\alpha))$ is always true by Theorem~\ref{T:main}. 

We will proceed  by using induction on $n$. 
Suppose the statement is true for $n$.

Now let us consider  $\alpha_1,\cdots,\alpha_{n+1},\beta_1,\cdots,\beta_{n+1}\in\mbbTP(X)$ and let $d_i=\rho(\alpha_i,\beta_i)$ for $i=1,\cdots,n+1$.
Let $S_n=\lowertconv(\{\alpha_1,\cdots,\alpha_n\})$, $S_{n+1}=\lowertconv(\{\alpha_1,\cdots,\alpha_{n+1}\})$, $T_n=\lowertconv(\{\beta_1,\cdots,\beta_n\})$ and $T_{n+1}=\lowertconv(\{\beta_1,\cdots,\beta_{n+1}\})$. 

For each $\alpha\in S_{n+1}$, let $\alpha_0=\lowerpi_{S_n}(\alpha)$ and $\beta_0=\lowerpi_{T_n}(\alpha_0)$. Then by assumption, since $\alpha_0$ is an element in $S_n$, we have $\rho(\alpha_0,T_n)=\rho(\alpha_0,\beta_0)\leq \max(d_1,\cdots,d_n)$. Moreover, 
we have $\alpha\in\underline{[\alpha_0,\alpha_{n+1}]}$ by Theorem~\ref{T:construction}. Let $\gamma = \lowerpi_{\underline{[\beta_0,\beta_{n+1}]}}(\alpha)$. Then $\rho(\alpha,T_{n+1})\leq \rho(\alpha,\gamma)\leq \max(\rho(\alpha_0,\beta_0),\rho(\alpha_{n+1,},\beta_{n+1}))\leq \max(d_1,\cdots,d_{n+1})$ where the second inequality follows from Proposition~\ref{P:DistIneq}(2).

The case for upper tropical convexity can be proved analogously. 

\end{proof}

\begin{theorem}[The Tropical Version of Mazur's Theorem] \label{T:TropMazur}
If $S$ is a compact subset of $\mbbTP(X)$, then the closed tropical convex hulls $\lowerTCONV(S)$, $\upperTCONV(S)$ and $\lowerupperTCONV(S)$ are all compact. 
\end{theorem}

\begin{proof}

Since $\lowerTCONV(S)$ is closed in the Banach space $\mbbTP(X)$ (Proposition~\ref{P:TPBanach}), we know that $\lowerTCONV(S)$ is complete. So to show that $\lowerTCONV(S)$ is compact, we need to show that it is totally bounded. Actually first we will show that $\lowertconv(S)$ is totally bounded.

Since $S$ is compact and thus totally bounded, for each $\epsilon>0$, $S$ can be finitely covered by open balls $\mcalB^0(\beta_i,\epsilon/2)$ of radius $\epsilon/2$ centered at $\beta_i\in S$ for $i=1,\cdots,n$. Now consider the lower tropical polytope $\lowertconv(\{\beta_1,\cdots,\beta_n\})$ which is also compact by Corollary~\ref{C:Polytope}. Then we may also assume that $\lowertconv(\{\beta_1,\cdots,\beta_n\})$can be finitely covered by open balls $\mcalB^0(\beta_i,\epsilon/2)$ of radius $\epsilon/2$ centered at $\beta_i\in\lowertconv(\{\beta_1,\cdots,\beta_n\})$ for $i=1,\cdots,N$ where $N\geq n$.

Now let $\alpha$ be any element in $\lowertconv(S)$. We may assume that $\alpha$ is contained in a lower tropical polytope $\lowertconv(\{\alpha_1,\cdots,\alpha_m\})$ with $\alpha_i\in S$ for $i=1,\cdots,m$ by Corollary~\ref{C:LocalFinite}. Then there is a function $\phi:\{1,\cdots,m\}\to\{1,\cdots, n\}$ such that $\rho(\alpha_i,\beta_{\phi(i)})<\epsilon/2$ for $i=1,\cdots,m$. Therefore, if $\gamma$ be the lower tropical projection of $\alpha$ to $\lowertconv(\{\beta_1,\cdots,\beta_N\})$, then $\rho(\alpha,\gamma)=\rho(\alpha,\lowertconv(\{\beta_1,\cdots,\beta_N\}))\leq \rho(\alpha,\lowertconv(\{\beta_\phi(1),\cdots,\beta_\phi(m)\}))\leq \max(\rho(\alpha_i,\beta_{\phi(i)})\mid i=1,\cdots,m)<\epsilon/2$ by Proposition~\ref{P:GeneralIneq}. Again, since $\gamma$ is an element of $\lowertconv(\{\beta_1,\cdots,\beta_N\})$, there must be some $1\leq i(\alpha)\leq N$ such that $\rho(\gamma,\beta_{i(\alpha)})< \epsilon/2$. Therefore, $\rho(\alpha,\beta_{i(\alpha)})\leq \rho(\alpha,\gamma)+\rho(\gamma,\beta_{i(\alpha)})<\epsilon$. As a conclusion, $\lowertconv(S)$ can be finitely covered by open balls $\mcalB^0(\beta_i,\epsilon)$ of radius $\epsilon$ centered at $\beta_i$ for $i=1,\cdots,N$ which implies that $\lowertconv(S)$ is totally bounded. 

As $\lowerTCONV(S)=\cl(\lowertconv(S))$, each element $\alpha\in \lowerTCONV(S)$ is approachable by a sequence $\alpha_1,\alpha_2,\cdots$ in $\lowertconv(S)$. Note that by the above argument, $\rho(\alpha_i,\{\beta_1,\cdots,\beta_N\})< \epsilon$. 
 Therefore,  $\rho(\alpha,\{\beta_1\cdots,\beta_N\})=\lim\limits_{i\to\infty} \rho(\alpha_i,\{\beta_1\cdots,\beta_N\})\leq \epsilon$. This means that $\lowerTCONV(S)$ can be finitely covered by open balls $\mcalB^0(\beta_i,2\epsilon)$ of radius $2\epsilon$ centered at $\beta_i$ for $i=1,\cdots,N$. Thus $\lowerTCONV(S)$ is totally bounded and thus compact. 
 
 The compactness of $\upperTCONV(S)$ and $\lowerupperTCONV(S)$ can be derived analogously. 
\end{proof}

\begin{remark}
In case $\mbbTP^0(X)\neq \mbbTP(X)$, the above theorem is not generally true restricted to $\mbbTP^0(X)$, i.e., the compactness of $S\subseteq\mbbTP^0(X)$ is not able to guarantee the compactness of the tropical convex sets $\lowerTCONV(S)\bigcap \mbbTP^0(X)$, $\upperTCONV(S)\bigcap \mbbTP^0(X)$ and 
$\lowerupperTCONV(S)\bigcap \mbbTP^(X)$. The reason is that $\mbbTP^0(X)$ is not necessarily a Banach space as $\mbbTP(X)$. 
\end{remark}

\section{A Criterion for Tropical Weak Independence} \label{S:TropIndep}

The following theorem provides a set-theoretical criterion for tropical weak independence, which is a generalization of Proposition~\ref{P:TropSeg} and Lemma~\ref{L:TropProjSegment}.

\begin{theorem} \label{T:CriterionTropIndep}
Let $T\subseteq\mbbTP(X)$ be a lower (respectively upper) tropical convex hull generated by $S\subseteq \mbbTP(X)$. Then for any $\gamma\in\mbbTP(X)$, $\gamma$ lies in $T$ if and only if there is a finite subset $\{\beta_1,\cdots, \beta_n\}$ of $S$ such that   $\bigcup_{i=1}^n\Xmin(\beta_i-\gamma)=X$ (respectively $\bigcup_{i=1}^n\Xmax(\beta_i-\gamma)=X$). Furthermore, for each element $\beta$ in $T$, 
$\Xmin(\beta-\gamma)\subseteq\Xmin(\lowerpi_T(\gamma)-\gamma)=\bigcup_{i=1}^n\Xmin(\beta_i-\gamma)$ (respectively $\Xmax(\beta-\gamma)\subseteq\Xmax(\upperpi_T(\gamma)-\gamma)=\bigcup_{i=1}^n\Xmax(\beta_i-\gamma)$). 
\end{theorem}
\begin{proof}
We prove here the lower tropical convexity case and the proof of upper tropical convexity case follows analogously. 

By Corollary~\ref{C:LocalFinite}, any element $\gamma$ in $T$ must be contained in lower tropical polytope with generators in a finite subset of $S$. So we only need to show the case when $S$ is itself finite. 

Let us prove by an induction on the number of generators. Suppose the statements are true for all lower tropical convex hulls generated by $n$ elements. Now consider a lower tropical polytope $T=\lowertconv(\{\beta_1,\cdots, \beta_n, \beta_{n+1}\})$ generated by $n+1$ elements and let $T'=\lowertconv(\{\beta_1,\cdots, \beta_n\})$. For $\gamma\in \mbbTP(X)$, let $\alpha=\lowerpi_T(\gamma)$. 

Then by Theorem~\ref{T:construction}, there exists $\alpha'\in T'$ such that $\alpha\in \lowerpi{[\alpha',\beta_{n+1}]}$. This implies that $\Xmin(\alpha'-\alpha)\bigcup \Xmin(\beta_{n+1}-\alpha)=X$. Then by assumption, $\Xmin(\alpha'-\alpha)\subseteq \bigcup_{i=1}^n\Xmax(\beta_i-\alpha)$. 
Thus $\bigcup_{i=1}^{n+1}\Xmax(\beta_i-\gamma)=X$. In addition, $\Xmin(\beta_i-\gamma)=\Xmin(\alpha-\gamma)\bigcap\Xmin(\beta_i-\gamma)$ for $i=1,\cdots,n+1$. Therefore
\begin{align*}
\Xmin(\alpha-\gamma)&=\Xmin(\alpha-\gamma)\bigcap(\bigcup_{i=1}^{n+1}\Xmin(\beta_i-\alpha)) \\
&= \bigcup_{i=1}^{n+1}(\Xmin(\alpha-\gamma)\bigcap\Xmin(\beta-\alpha)) \\
&= \bigcup_{i=1}^{n+1}\Xmin(\beta-\gamma).
\end{align*}
And this also implies $\gamma\in T$ if and only if $\bigcup_{i=1}^{n+1}\Xmin(\beta_i-\gamma)=X$.

\end{proof}

\begin{corollary}
Let $S,S_1,\cdots,S_n$ be finite subsets of $\mbbTP(X)$ such that $S=S_1\bigcup \cdots \bigcup S_n$. Let $T=\lowertconv(S)$ (respectively $T=\uppertconv(S)$) and for $i=1,\cdots,n$, let $T_i = \lowertconv(S_i)$ (respectively $T_i = \uppertconv(S_i)$). For $\gamma\in \mbbTP(X)$,  $\gamma\in T$ if and only if $\gamma\in\lowertconv(\{\lowerpi_{T_1}(\gamma),\cdots,\lowerpi_{T_n}(\gamma)\})$ (respectively $\gamma\in\uppertconv(\{\lowerpi_{T_1}(\gamma),\cdots,\upperpi_{T_n}(\gamma)\})$). 
\end{corollary}
\begin{proof}
Note that by Theorem~\ref{T:CriterionTropIndep}, $\Xmin(\lowerpi_{T_i}(\gamma)-\gamma)=\bigcup_{\beta\in S_i}\Xmin(\beta-\gamma)$ for $i=1,\cdots,n$. 
Then $\bigcup_{i=1}^n\Xmin(\lowerpi_{T_i}(\gamma)-\gamma)=\bigcup_{\beta\in S_1\bigcup \cdots \bigcup S_n}\Xmin(\beta-\gamma)=\bigcup_{\beta\in S}\Xmin(\beta-\gamma)$. Again, by Theorem~\ref{T:CriterionTropIndep}, this means that $\gamma\in T$ if and only if $\gamma\in\lowertconv(\{\lowerpi_{T_1}(\gamma),\cdots,\lowerpi_{T_n}(\gamma)\})$. The case for upper tropical convexity can be shown analogously. 
\end{proof}

Let $T$ be a lower (respectively upper) tropical convex set. For $\alpha\in T$, if $\alpha\notin \lowertconv(T\setminus\{\alpha\})$ (respectively $\alpha\notin \uppertconv(T\setminus\{\alpha\})$),  then we say $\alpha$ is an \emph{extremal} of $T$. It is clear from definition that any generating set of $T$ must contain all the extremals of $T$ and the set of extemals of $T$ is (lower or  upper) tropically independent.

\begin{theorem} \label{T:extremal}
Every lower (respectively upper) tropical polytope $T$ contains finitely many extremals. The set of all extremals of $T$ generates $T$ and is minimal among all generating sets of $T$.
\end{theorem}
\begin{proof}
We will prove the case for lower tropical convexity and the case for upper tropical convexity can be proved analogously. 

Since $T$ is a lower tropical polytope, we may choose a finite generating set  $S$ of $T$, i.e., $\lowertconv(S)=T$. Choose a subset $V$ of $S$ such that $\lowertconv(V)=T$ and $V$ is lower tropically independent,  i.e., $\lowertconv(V\setminus\{\alpha\})\neq T$ for all $\alpha\in V$. Note that this is doable since $S$ is finite. We claim that $V$ must be the set of all extremals of $T$, i.e.,  $V=\{\alpha\in T\mid \lowertconv(T\setminus\{\alpha\})= T\setminus\{\alpha\}\}$. 

First we note that all extremals must be contained in $V$ by definition. Now let us show that all elements of $V$ are extremals of $T$.

Let $V=\{\alpha_1,\cdots,\alpha_n\}$. Without loss of generality, we will show that  $\alpha_1$ is an extremal of $T$. 
We know that $\lowertconv(V)=T$ and $\lowertconv(V\setminus\{\alpha_1\})\neq T$. To show that $\alpha_1$ is an extremal of $T$, we will need to show that $T\setminus\{\alpha_1\}$ is lower tropically convex or equivalently $\alpha_1\notin \lowertconv(T\setminus\{\alpha_1\})$. Actually it suffices to show that for arbitrarily $\beta_1$ and $\beta_2$ in $T\setminus\{\alpha_1\}$,  $\alpha\notin \underline{[\beta_1,\beta_2]}$.

Let $T'=\lowertconv(\{\alpha_2,\cdots,\alpha_n\})$. By Theorem~\ref{T:construction}, there exist $\gamma_1$ and $\gamma_2$ in $T'$ such that $\beta_1\in\underline{[\alpha_1,\gamma_1]}$ and $\beta_2\in\underline{[\alpha_1,\gamma_2]}$. Then $\Xmin(\beta_1-\alpha_1)=\Xmin(\gamma_1-\alpha_1)$ and $\Xmin(\beta_2-\alpha_1)=\Xmin(\gamma_2-\alpha_1)$. 

Note that  $\alpha_1\notin T'$ since $V$ is lower tropically independent. Hence $\bigcup_{i=2}^n\Xmin(\alpha_i-\alpha_1)\neq X$ by Theorem~\ref{T:CriterionTropIndep}.

Moreover, $\Xmin(\gamma_1-\alpha_1)\subseteq \bigcup_{i=2}^n\Xmin(\alpha_i-\alpha_1)$ and $\Xmin(\gamma_2-\alpha_1)\subseteq \bigcup_{i=2}^n\Xmin(\alpha_i-\alpha_1)$ by Theorem~\ref{T:CriterionTropIndep}. Then 
$\Xmin(\beta_1-\alpha_1) \bigcup \Xmin(\beta_2-\alpha_1) =\Xmin(\gamma_1-\alpha_1) \bigcup  \Xmin(\gamma_2-\alpha_1) \subseteq \bigcup_{i=2}^n\Xmin(\alpha_i-\alpha_1) \neq X$ which means that $\alpha\notin \underline{[\beta_1,\beta_2]}$. 

\end{proof}

\begin{remark}
The statements in Theorem~\ref{T:extremal} is not generally true to all tropical convex sets. For example, if $X=\{1,\cdots,n\}$, then the whole space $\mbbTP(X)=\mbbR^{n-1}$  does not contain any extremals. 
\end{remark}

\section{A Fixed Point Theorem for Tropical Projections} \label{S:FixedPoint}

\begin{theorem} \label{T:FixedPoint}
Let $S$ and $T$ be compact subsets of $\mbbTP(X)$ and suppose $S$ is lower tropically convex and $T$ is upper tropically convex. Let $S_0=\{\lowerpi_S(\gamma)\mid \gamma\in T\}$ and $T_0=\{\upperpi_T(\gamma)\mid \gamma\in S\}$. Then $S_0$ and $T_0$ are isometric under $\upperpi_T\mid _{S_0}:S_0\to T_0$ and $\lowerpi_S\mid _{T_0}:T_0\to S_0$ which are inverse maps.
\end{theorem}

\begin{proof}
For each $\alpha\in S_0$, there exists $\gamma\in T$ such that $\alpha=\lowerpi_S(\gamma)$. Let $\beta=\upperpi_T(\alpha)$ and $\alpha'=\lowerpi_S(\beta)$. We want to show that $\alpha=\alpha'$. 

Using Theorem~\ref{T:main} and Corollary~\ref{C:CritTropProj}, we have the following relations:

\begin{align}
&\llfloor \alpha-\alpha'\rrfloor_1 + \llfloor \alpha' - \beta\rrfloor_1  = \llfloor \alpha - \beta\rrfloor_1; \\
&\llfloor \alpha'-\alpha\rrfloor_1 + \llfloor \alpha - \gamma\rrfloor_1  = \llfloor \alpha' - \gamma\rrfloor_1; \\
&\llceil \gamma-\beta\rrceil_1 + \llceil \beta - \alpha\rrceil_1  = \llceil \gamma - \alpha\rrceil_1. 
\end{align}

Then we have 

\begin{align*}
0\leq &\rho(\alpha,\alpha')\mu(X)=\llfloor \alpha-\alpha'\rrfloor_1+\llfloor \alpha'-\alpha\rrfloor_1 \quad\text{(Proposition~\ref{P:BnormProperty}(5))}\\
&=(\llfloor \alpha - \beta\rrfloor_1-\llfloor \alpha' - \beta\rrfloor_1)+(\llfloor \alpha' - \gamma\rrfloor_1- \llfloor \alpha - \gamma\rrfloor_1)\quad \text{(by (1) and (2))}\\
&=(\llceil \beta - \alpha\rrceil_1-\llceil \beta-\alpha' \rrceil_1)+(\llceil \gamma-\alpha' \rrceil_1- \llceil \gamma-\alpha\rrceil_1) \quad\text{(Proposition~\ref{P:BnormProperty}(1))}\\
&= ((\llceil \gamma - \alpha\rrceil_1-\llceil \gamma - \beta\rrceil_1)-\llceil \beta-\alpha' \rrceil_1)+(\llceil \gamma-\alpha' \rrceil_1- \llceil \gamma-\alpha\rrceil_1) \quad \text{(by (3))}\\
&= \llceil \gamma-\alpha' \rrceil_1-(\llceil \gamma - \beta\rrceil_1+\llceil \beta-\alpha' \rrceil_1) \\
& = \llceil(\gamma - \beta) +(\beta-\alpha' ) \rrceil_1-(\llceil \gamma - \beta\rrceil_1+\llceil \beta-\alpha' \rrceil_1) \leq 0 \quad\text{(Proposition~\ref{P:BnormProperty}(9))}. 
\end{align*}

Therefore $\rho(\alpha,\alpha')$ must be $0$ which means that $\alpha=\alpha'$ as claimed. 

Using an analogous argument, we can show that $\beta=\upperpi_T(\lowerpi_S(\beta))$ for each $\beta\in T_0$. 
\end{proof}

\section{An Application to the Divisor Theory on Metric Graphs} \label{S:AppMetGra}

\subsection{Tropical Convexity for Divisors and $\mbbR$-Divisors}

Let $\Gamma$ be a compact metric graph with finite edge lengths. We also denote the set of points of $\Gamma$ by $\Gamma$ for simplicity. Let $\Div(\Gamma)$ be the free abelian group on $\Gamma$ and $\RDiv(\Gamma) = \Div(\Gamma)\otimes\mR$. As in convention, we call the elements of $\Div(\Gamma)$ \emph{divisors} (or \emph{$\mbbZ$-divisors} when we want to emphasize the integer coefficients) and elements of $\RDiv(\Gamma)$ \emph{$\mR$-divisors}. An $\mR$-divisor on $\Gamma$ can be written as $D=\sum_{p\in\Gamma}m_p\cdot(p)$ where $m_p\in\mbbR$ is called the \emph{value} of $D$ at $p\in\Gamma$ which is zero for all but finitely many points $p\in \Gamma$. We also write the value of $D$ at $p\in\Gamma$ as $D(p)$.  Moreover, for an $\mbbR$-divisor $D=\sum_{p\in\Gamma}m_p\cdot(p)$, 

\begin{enumerate}
\item the \emph{degree} of $D$ is $\sum_{p\in\Gamma}m_p$;
\item the \emph{support} of $D$, denoted by $\supp(D)$, is the set of points $p\in\Gamma$ such that $m_p\neq 0$;  
\item $D$ is the \emph{zero divisor} if and only  if $\supp(D)=\emptyset$;
\item $D$ is a $\mbbZ$-divisor if and only  if $m_p$ is an integer for all $p\in\Gamma$;
\item we say $D$  is \emph{effective} if $m_p\geq 0$ for all $p\in\Gamma$;
\item $D$ can be uniquely written as $D^+-D^-$ where $D^+$ and $D^-$ are both effective divisors which have disjoint supports and are called the \emph{effective part} and \emph{noneffective part} of $D$ respectively.
\end{enumerate}

Note that there is a natural partial ordering on $\RDiv(\Gamma)$. We say $D'\leq D$ if $D-D'$ is effective. 

Let $\DivPlus(\Gamma)$ and $\RDivPlus(\Gamma)$ be the semigroups of effective $\mZ$-divisors and effective $\mR$-divisors respectively. If $d$ is a nonnegative integer, denote the set of divisors of degree $d$ by $\DivD(\Gamma)$ and the set of effective divisors of degree $d$ by $\DivPlusD(\Gamma)$. If $d$ is a nonnegative real, denote the set of $\mbbR$-divisors of degree $d$ by $\RDivD(\Gamma)$ and the set of effective $\mbbR$-divisors of degree $d$ by $\RDivPlusD(\Gamma)$. 

In this section, we will explore a tropical convexity theory on  $\RDivPlusD(\Gamma)$. To make a connection to the whole theory developed in the previous sections, we will need to relate $\mbbR$-divisors to elements in $\mbbTP(\Gamma)$ (note that $\mbbTP^0(\Gamma)=\mbbTP(\Gamma)$ since $\Gamma$ is compact). This can be realized using a potential theory on metric graphs \cite{BF06,BR07,BS13} with some results briefly summarized in Appendix~\ref{S:potential}. 

Let $CPA(\Gamma)\subset C(\Gamma)$ be the vector space consisting of all continuous piecewise-affine  functions on $\Gamma$.  If $D=\sum_{p\in\Gamma}m_p\cdot(p)\in\RDiv$, we let $\delta_D:=\sum_{p\in\Gamma}m_p\cdot\delta_p$ with $\delta_p$ the Dirac measure at $p$.
Consider $D_1,D_2\in\RDivPlusD(\Gamma)$. Then based on the potential theory on $\Gamma$, there exists a piecewise-linear function $f_{D_2-D_1}\in CPA(\Gamma)$  such that $\Delta f_{D_2-D_1} = \delta_{D_2}-\delta_{D_1}$ which is unique up to  a constant translation. We say $f_{D_2-D_1}$ is an \emph{associated function} of $D_2-D_1$, and in this sense, we may associate the unique element $[f_{D_2-D_1}]$ in $\mbbTP(\Gamma)$ to $D_2-D_1$. On the other hand, for each $f\in[f_{D_2-D_1}]$, we say $\divf(f)=\divf([f]):=D_2-D_1$ is the \emph{associated divisor} of $f$ or $[f]$. In particular, the value of $\divf(f)$ at $p\in\Gamma$ is the sum of slopes of $f$ along all incoming tangent directions at $p$.  Note that $[f_{D_2-D_1}]+[f_{D_1-D_0}]=[f_{D_2-D_0}]$ for all $D_0,D_1,D_2\in \RDivPlusD(\Gamma)$. 

More precisely, if $D_1=(q)$ and $D_2=(p)$ for some $p,q\in\Gamma$, then $\underline{f_{D_2-D_1}}(x)=j_q(x,p)$ (see the definition of $j_q(x,p)$ in Appendix~\ref{S:potential}).
Now let $D_1=\sum_{i=1}^{d_1} m_{1,i}\cdot(p_{1,i})$ and $D_2=\sum_{i=1}^{d_2} m_{2,i}\cdot(p_{2,i})$ such that $D_1,D_2\in\RDivPlusD(\Gamma)$ (this means $d=\sum_{i=1}^{d_1} m_{1,i}=\sum_{i=1}^{d_2} m_{2,i}$). Then by the linearity of the Laplacian, for an arbitrary $q\in\Gamma$, $\sum_{i=1}^{d_1}m_{1,i}\cdot j_q(x,p_{1,i})$, $\sum_{i=1}^{d_2}m_{2,i}\cdot j_q(x,p_{2,i})$, $\sum_{i=1}^{d_2}m_{2,i}\cdot j_q(x,p_{2,i})-\sum_{i=1}^{d_1}m_{1,i}\cdot j_q(x,p_{1,i})$ are  associated functions of $D_1-d\cdot(q)$, $D_2-d\cdot(q)$ and $D_2-D_1$ respectively. To have a quick understanding, one may think of $\Gamma$ as an electrical network with resistances given by the edge lengths. Then an associated function of $D_2-D_1$ is the electrical potential function on $\Gamma$ (with an arbitrary point of $\Gamma$ grounded)  when $m_{2,i}$ units of current enter the network at $p_{2,i}$ for all $i=1,\cdots,d_2$ and $m_{1,i}$ units of current exit the network at $p_{1,i}$ for all $i=1,\cdots,d_1$.

We may define $\iota:\RDivPlusD(\Gamma)\times \RDivPlusD(\Gamma)\to \mbbTP(\Gamma)$ by $\iota(D_1,D_2)=[f_{D_2-D_1}]$. Moreover, fixing an arbitrary $\mbbR$-divisor $D_0$ in  $\RDivPlusD(\Gamma)$, the map $\iota_{D_0}:\RDivPlusD(\Gamma)\to\mbbTP(\Gamma)$ with $D\mapsto [f_{D-D_0}]$ is an embedding of $\RDivPlusD(\Gamma)$ into $\mbbTP(\Gamma)$. Note that $\divf(\iota_{D_0}(D))=D-D_0$. 

Now we can translate the tropical convexity theory from $\mbbTP(\Gamma)$ to  $\RDivPlusD(\Gamma)$ based on the maps $\iota$ and $\iota_{D_0}$. 
\begin{enumerate}
\item The \emph{tropical metric} on $\RDivPlusD(\Gamma)$ is defined by the distance function $\rho(D_1,D_2):=\rho(\iota_{D_0}(D_1), \iota_{D_0}(D_2))=\Vert [f_{D_2-D_0}]-[f_{D_1-D_0}]\Vert=\Vert\iota(D_1,D_2)\Vert=\Vert [f_{D_2-D_1}]\Vert = \max(f_{D_2-D_1})-\min(f_{D_2-D_1})$. 
\item Let $l=\rho(D_1,D_2)$. Let $\alpha=\iota_{D_0}(D_1)$ and $\beta=\iota_{D_0}(D_2)$. The \emph{tropical path} from $D_1$ to $D_2$ in $\RDivPlusD(\Gamma)$ is defined as a map  $P_{D_2-D_1}:[0,l]\to \RDivPlusD(\Gamma)$ given by 
\begin{align*}
&P_{D_2-D_1}(t):=\divf(P_{(\alpha,\beta)})+D_0 \\
&=\divf([\min(t,\underline{f_{D_2-D_0}-f_{D_1-D_0}})+f_{D_1-D_0}])+D_0 \\
&=\divf([\min(t,\underline{f_{D_2-D_1}})])+\divf([f_{D_1-D_0}])+D_0 \\
&=\divf([\min(t,\underline{f_{D_2-D_1}})])+D_1.
\end{align*}
Note that the value of $\divf([\min(t,\underline{f_{D_2-D_1}})])$ at each point $p\in\Gamma$ is at least $-D_1(p)$, and thus $P_{D_2-D_1}(t)$ is effective of degree $d$ which means $P_{D_2-D_1}$ is well-defined.
In particular, $P_{D_2-D_1}(0)=D_1$ and $P_{D_2-D_1}(l) = D_2$. The image of $P_{D_2-D_1}$, denoted by $\tconv(D_1,D_2)$ or $[D_1,D_2]$, is called the \emph{tropical segment} connecting $D_1$ and $D_2$. 

\item It is easy to see that the tropical metric and the tropical paths on $\RDivPlusD(\Gamma)$ are independent of $D_0$ which we choose for the embedding $\iota_{D_0}$. 
\item A set $T\subseteq\RDivPlusD(\Gamma)$ is \emph{tropically convex} if for every $D_1,D_2\in T$, the whole tropical segment $[D_1,D_2]$ is contained in $T$. For a subset $S$ of $\RDivPlusD(\Gamma)$, the \emph{tropical convex hull} generated by $S$, denoted by $\tconv(S)$, is the intersection of all tropical convex sets in $\RDivPlusD(\Gamma)$ containing $S$. If $S$ is finite, we also call $\tconv(S)$ a \emph{tropical polytope}. 
\item All theorems about lower tropical convex sets in $\mbbTP(X)$ in the previous sections can be applied to tropical convex sets in $\RDivPlusD(\Gamma)$. 
\end{enumerate}

\begin{remark}
Recall that we have defined two types of tropical paths, the lower ones and upper ones, for $\mbbTP(X)$ in Definition~\ref{D:tpath}. But here we have essentially only one way to define tropical paths for $\RDivPlusD(\Gamma)$. As an analogy to Definition~\ref{D:tpath}., one may want to call $P_{D_2-D_1}$ the lower tropical path from $D_1$ to $D_2$, and define the upper tropical path from $D_1$ to $D_2$ as $P^{D_2-D_1}(t):=\divf(P^{(\alpha,\beta)})+D_0$. But an issue here is that in general $P^{D_2-D_1}(t)$ is not contained in $\RDivPlusD(\Gamma)$. Actually it can be verified that $P^{D_2-D_1}(t)=D_1+D_2-P_{D_2-D_1}(l-t)$ which is in general not effective. There are a few more points worth mentioning here:
\begin{enumerate}
\item In the above, we use lower tropical paths in $\mbbTP(\Gamma)$ to define tropical paths in $\RDivPlusD(\Gamma)$. But we can also use upper  tropical paths in $\mbbTP(\Gamma)$ to give an equivalent definition of tropical paths in $\RDivPlusD(\Gamma)$. Let $\alpha'=[f_{D_0-D_1}]$ and $\beta'=[f_{D_0-D_1}]$. Then $P_{D_2-D_1}(t)=D_0-\divf(P^{(\alpha',\beta')})$. 
\item As for $\RDivPlusD(\Gamma)$, we may also embed $\RDivD(\Gamma)$ into $\mbbTP(\Gamma)$ by fixing an $\mbbR$-divisor $D_0$ in  $\RDivD(\Gamma)$ and sending each $\mbbR$-divisor $D$ in  $\RDivD(\Gamma)$ to $[f_{D-D_0}]$ in $\mbbTP(\Gamma)$. Therefore a tropical convexity theory for $\RDivD(\Gamma)$ can be derived from the tropical convexity theory for  $\mbbTP(\Gamma)$, just as what we have done for $\RDivPlusD(\Gamma)$.  But in this way, there is some difference: both lower tropical paths $P_{D_2-D_1}(t):=\divf(P_{(\alpha,\beta)})+D_0$ and upper tropical paths $P^{D_2-D_1}(t):=\divf(P^{(\alpha,\beta)})+D_0$ can be defined for $\RDivD(\Gamma)$ since the $\mbbR$-divisors in the paths are not required to be effective and are fully contained in $\RDivD(\Gamma)$. 
\item Throughout the rest of this section, we will make our discussion focused on the tropical convexity theory on $\RDivPlusD(\Gamma)$ instead of that on $\RDivD(\Gamma)$, and  stick to the translation of theorems for lower tropical convexity on $\mbbTP(\Gamma)$ to those for the  tropical convexity on $\RDivPlusD(\Gamma)$. 
\end{enumerate}
\end{remark}

Now suppose $d$ is an integer and  let us consider $\DivPlusD(\Gamma)$ which is a subset of $\RDivPlusD(\Gamma)$. As in convention, a function $f\in CPA(\Gamma)$ is said to be \emph{rational} if $f$ is piecewise-linear with integral slopes. Clearly, $\divf(f)$ is a divisor of degree $0$. For each $D_1,D_2\in\DivD(\Gamma)$, we say $D_1$ is \emph{linearly equivalent} to $D_2$ (denoted $D_1\sim D_2$) if $f_{D_2-D_1}$ is rational. Note that the linear equivalence is an equivalence relation on $\Div(\Gamma)$ and we denote the linear equivalence class of a divisor $D$ by $[D]$.  The \emph{complete linear system} $|D|$ associated to $D\in\DivD(\Gamma)$ is the set of all effective divisors linearly equivalent to $D$ and we say the degree of $|D|$ is $d$. The following are some facts about $\DivPlusD(\Gamma)$:
\begin{enumerate}
 \item $\DivPlusD(\Gamma)$ with the metric topology induced from the tropical metric on $\RDivPlusD(\Gamma)$ is homeomorphic to the $d$-th symmetric product $\Gamma^d/S_d$ of $\Gamma$ with the symmetric product topology. We prove this statement in Appendix~\ref{S:EquiTop}. 
 \item For each $D\in\DivPlusD(\Gamma)$, $|D|$ is contained in $\DivPlusD(\Gamma)$. 
 \item For $D_1,D_2\in \DivPlusD(\Gamma)$, $D_1\sim D_2$ if and only if $[D_1,D_2]\subseteq \DivPlusD(\Gamma)$. 
 \item By definition, an equivalent way to say that a subset $T$ of  $\DivPlusD(\Gamma)$ is tropically convex is that $T$ is tropical-path-connected.
  \item $\DivPlusD(\Gamma)$ is not tropical-path-connected in general, and the nonempty complete linear systems of degree $d$ are exactly the tropical-path-connected components in $\DivPlusD(\Gamma)$.
  \item All complete linear systems $|D|$ are tropical polytopes which can be generated by the finite set of extremals of $|D|$. (See \cite{HMY12} for the definition and finiteness of  extremals of complete linear systems. This notion of extremals agrees with our notion of extremals in Section~\ref{S:TropIndep} in the case of $|D|$.) Note that by Corollary~\ref{C:Polytope}, complete linear systems are compact. 
\end{enumerate}

\begin{remark}
The divisor theory on finite graphs can be considered as a discretization of the divisor theory on metric graphs. Actually the divisor theory for $\Div(G)$ where $G$ is a finite graph is closely related to the divisor theory for $\Div(\Gamma)$ where $\Gamma$ is the metric graph geometrically realized as assigning the unit interval to all edges of $G$ \cite{Luo11,HKN13}. 
\end{remark}

\begin{remark} \label{R:ChipFiring}
Another way to think of linear equivalence on $\Div(\Gamma)$ is to use the so-called \emph{chip-firing moves}. We say a non-constant rational function $f$   on $\Gamma$ is $\emph{primitive}$ if $\Gamma\setminus(\Gmin(f)\bigcup \Gmax(f))$ is a disjoint union of open segments where $\Gmin(f)$ and $\Gmax(f)$ are the minimizer and maximizer of $f$ respectively. Note that these open segments form a cut of the metric graph. Then $\divf(f)=D^+-D^-$ where $D^+$ and $D^-$ are the effective part and noneffective part of $\divf(f)$ respectively. Now consider a divisor $D$ of degree $d$.  Then $D'=D+\divf(f)=D^++(D-D^-)$ is also a divisor of degree $d$ which is linearly equivalent to $D$. This also means $f_{D'-D}=f$. Note that $\supp(D^+)=\partial \Gmax(f)$ and $\supp(D^-)=\partial \Gmin(f)$. Let $l=\max(f)-\min(f)=\rho(D,D')$.  We may visualize the evolution from $D$ to $D'$ as follows:
\begin{enumerate}
\item For each point $x\in\supp(D)$ such that $D(x)\geq 0$, we put $D(x)$ chips at the point $x$. If $D(x)<0$, we simply say that the point $x\in\Gamma$ is in debt of $-D(x)$ chips. In this sense, the divisor $D$ is represented by the configuration of chips on $\Gamma$. 
\item Now we take a chip-firing move  from $\Gmin(f)$ to $\Gmax(f)$ along the cut formed by the open segments in $\Gamma\setminus(\Gmin(f)\bigcup \Gmax(f))$. More precisely, for each point $x\in\partial \Gmin(f)$ and the outgoing tangent  direction $\vt$ at $x$ along an open segment in $\Gamma\setminus(\Gmin(f)\bigcup \Gmax(f))$ adjacent to $x$, take out $m_\vt$ chips at $x$ where $m_\vt$ is the slope of $f$ along $\vt$ which is an integer (this might make the point $x$ in debt) and move these $m_\vt$ along $\vt$ at the speed of $1/m_\vt$. After a period of time $l$, we get a new configuration of chips which represents the divisor $D'$. 
\end{enumerate}
We call the above process the \emph{chip-firing move} associated to the primitive rational function $f$ or the chip-firing move from $D$ to $D'$ or the chip-firing move from $D$ of \emph{distance} $l$. By the \emph{direction} of this chip firing move, we mean the process of moving  $m_\vt$ chips from $x$ into $\Gmin(f)^c$ at the speed of of $1/m_\vt$  for each $x\in\partial \Gmin(f)$ and each tangent direction $\vt$ at $x$ shooting into $\Gmin(f)^c$ (ignoring the distance information). We also call $\Gmin(f)$ the \emph{base} of this chip-firing move or of this chip-firing move direction. If in addition, $D$ is effective and $D\geq D^-$, then we always have enough chips on the boundary of $\Gmin(f)$ to fire  and $D'=D^++(D-D^-)$ is effective, i.e., this chip-firing move makes no point in debt.  We call it an \emph{effective chip-firing move} from $D$. Moreover, it can be easily verified that any rational function is a finite sum of primitive rational functions. Therefore, two divisors $D,D'\in \Div(\Gamma)$ are linearly equivalent if and only if $D'$ can be reached by $D$ via finitely many steps of chip-firing moves. 
If in addition $D$ and $D'$ are effective, then it can also be shown that all the intermediate  chip-firing moves can be chosen to be effective, i.e., the firings are within $|D|$. Actually, one way to choose the chip-firing moves is along the tropical path from $D$ to $D'$. Furthermore, The above notions of chip-firing moves, chip-firing move directions, effective chip-firing moves, and bases of chip-firing moves can be straightforwardly generalized to  cases for $\mbbR$-divisors by allowing the slopes $m_\vt$ of the corresponding functions to be real numbers. 
\end{remark}

The divisor theory on finite graphs or metric graphs is an analogue of the divisor theory on algebraic curves, while in the latter not only complete linear systems are studied, linear systems in general are also studied. Here we give the following definition of linear systems in the context of metric graphs. 

\begin{definition} \label{D:LinSys}
For a divisor $D\in\DivPlusD(\Gamma)$, a tropical polytope $T\subseteq |D|$ containing $D$ is called a \emph{linear system} associated to $D$. We say the \emph{degree} of $T$ is $d$. For convenience, we also allow a linear system to be an empty set which is called the \emph{empty linear system}. 
\end{definition}

\begin{example} \label{E:LinearSys}
In Figure~\ref{F:LinearSys}, we consider a metric circle $\Gamma$ and effective divisors of degree $3$ on $\Gamma$. Suppose the total length of $\Gamma$ is $2\pi$. A point $x$ on $\Gamma$ can be represented by its polar angle $\theta(x)\in \mbbR/2\pi$. Let $v_1$, $v_2$, $v_3$, $w_{12}$, $w_{23}$ and $w_{13}$ be points on $\Gamma$ such that $\theta(v_1)=7\pi/6$, $\theta(v_2)=\pi/2$, $\theta(v_3)=-\pi/6$, $\theta(w_{12})=5\pi/6$, $\theta(w_{23})=\pi/6$ and $\theta(w_{13})=-\pi/2$ respectively.  Let $D_0=(v_1)+(v_2)+(v_3)$. Then one can verify that a divisor $D=(x_1)+(x_2)+(x_3)\in \DivPlusThree(\Gamma)$ is linearly equivalent to $D_0$ if and only if $\theta(x_1)+\theta(x_2)+\theta(x_3)=3\pi/2\ \mod 2\pi$. Then $|D_0|$ can be represented by the locus of $\theta(x_1)+\theta(x_2)+\theta(x_3)=3\pi/2\ \mod 2\pi$ in the quotient of $(\mbbR/2\pi)^3$ by the natural action of the symmetric group $S_3$, which can actually be viewed as an equilateral triangle centered at $D_0$ with vertices $D_1=3(v_1)$, $D_2=3(v_2)$ and $D_3=3(v_3)$ as shown in Figure~\ref{F:LinearSys}. Moreover, the midpoint  of the side $D_1D_2$ corresponds to the divisor $D_{12}=2(w_{12})+(v_3)$, the midpoint  of the side $D_2D_3$ corresponds to the divisor $D_{23}=2(w_{23})+(v_1)$, and the midpoint  of the side $D_1D_3$ corresponds to the divisor $D_{13}=2(w_{13})+(v_2)$. In Figure~\ref{F:LinearSys}, we also show some sub-linear systems of $|D|$. For example, one can verify that the tropical segment $[D_1,D_3]$ is exactly the side $D_1D_3$ of the triangle, the tropical segment $[D_1,D_{23}]$ is the exactly the median $D_1D_{23}$ of the triangle, the tropical segment $[D_{12},D_{23}]$ is the union of $[D_0,D_{12}]$ and  $[D_0,D_{23}]$ which are straight segments in the medians $D_3D_{12}$ and $D_1D_{23}$ respectively. In addition,
\begin{enumerate}
\item  the tropical path from $D_1$ to $D_3$ corresponds to exactly one step of chip-firing move where one chip moves from $v_1$ to $v_2$ along the segment through $w_{12}$ at the speed of two units and two chips move from $v_1$ to $v_3$ along the segment through $w_{13}$ at the speed of one unit,
\item the tropical path from $D_1$ to $D_{23}$ corresponds to exactly one step of chip-firing move where one chip moves from $v_1$ to $w_{23}$ along the segment through $v_2$ at the speed of one unit and one chip moves from $v_1$ to $w_{23}$ along the segment through $v_3$ at the speed of one unit, and
\item the tropical path from $D_{12}$ to $D_{23}$ corresponds to two steps of chip-firing moves: the first step is the chip-firing move from $D_{12}$ to $D_0$ where one chip moves from $w_{12}$ to $v_1$ at the speed of one unit and one chip moves from $w_{12}$ to $v_2$ at the speed of one unit, and the second step is the chip-firing move from $D_0$ to $D_{23}$ where one chip moves from $v_2$ to $w_{23}$ at the speed of one unit and one chip moves from $v_3$ to $w_{23}$ at the speed of one unit.
\end{enumerate}
Moreover, the linear systems $\tconv(\{D_{12},D_{23},D_{13}\})$, $\tconv(\{D_{12},D_2,D_{13}\})$ and $\tconv(\{D_{12},D_2,D_{23}\})$ are also illustrated, which is purely $1$-dimensional, not of pure dimension and purely $2$-dimensional respectively. 

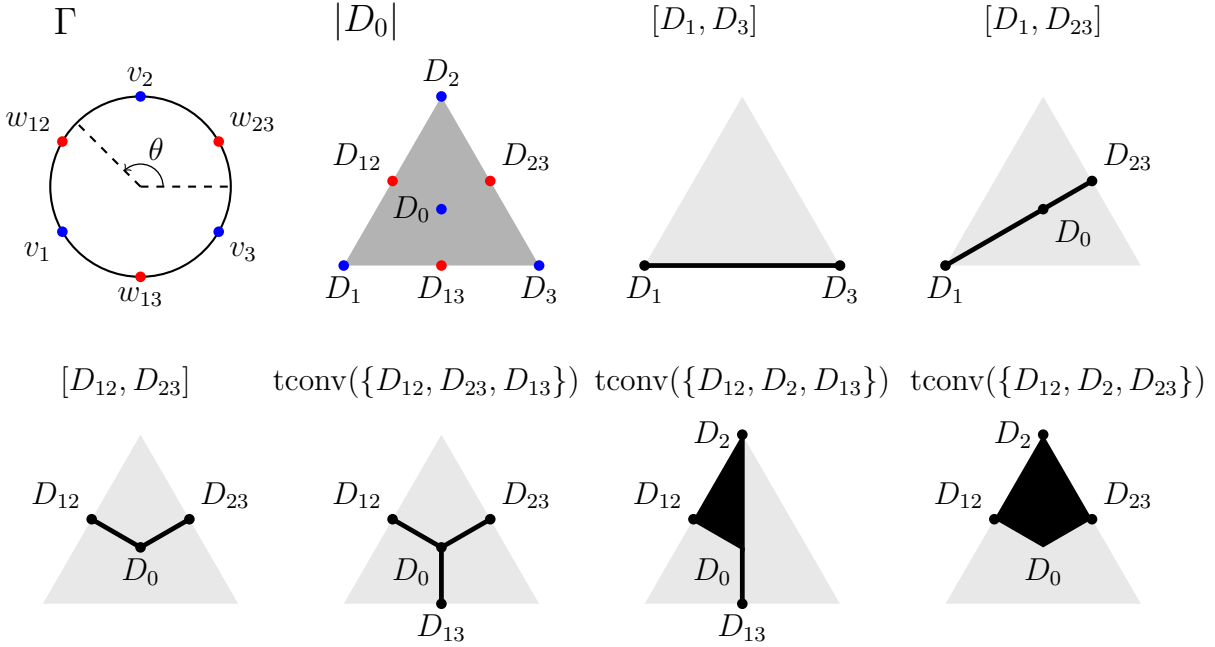
\begin{figure} 
\centering
\begin{tikzpicture}
\begin{scope}
 \draw (-1,2.5) node {\large $\Gamma$};
  \coordinate (center) at (0,0.3);
  \def\radius{1.2};
  \coordinate (v1) at ($(center)+(210:\radius)$);
  \coordinate (v2) at ($(center)+(90:\radius)$);
  \coordinate (v3) at ($(center)+(-30:\radius)$);
  \coordinate (w12) at ($(center)+(150:\radius)$);
  \coordinate (w23) at ($(center)+(30:\radius)$);
  \coordinate (w13) at ($(center)+(-90:\radius)$);
  \coordinate (th1) at ($(center)+(0:\radius)$);
  \coordinate (th2) at ($(center)+(135:\radius)$);

  \draw[thick] (center) circle[radius=\radius];
  \draw [dashed,line width = .8pt] (center) -- (th1);
  \draw [dashed,line width = .8pt] (center) -- (th2);
  \pic [draw, ->, "$\theta$", line width = .5pt, angle radius=0.3cm, angle eccentricity=1.8] {angle = th1--center--th2};

  \fill[blue] (v1) circle[radius=2pt];
  \fill[blue] (v2) circle[radius=2pt];
  \fill[blue] (v3) circle[radius=2pt];
  \fill[red] (w12) circle[radius=2pt];
  \fill[red] (w23) circle[radius=2pt];
  \fill[red] (w13) circle[radius=2pt];
  
  \draw [anchor=north east] (v1) node {$v_1$};
    \draw [anchor=south] (v2) node {$v_2$};
    \draw [anchor=north west] (v3) node {$v_3$};
    \draw [anchor=south east] (w12) node {$w_{12}$};
    \draw [anchor=south west] (w23) node {$w_{23}$};
    \draw [anchor=north] (w13) node {$w_{13}$};
    
\end{scope}

\begin{scope}[shift={(4,0)}]
\draw (-1,2.5) node {\large $|D_0|$};
  \coordinate (center) at (0,0);
  \def\radius{1.5};
  \coordinate (D1) at ($(center)+(210:\radius)$);
  \coordinate (D2) at ($(center)+(90:\radius)$);
  \coordinate (D3) at ($(center)+(-30:\radius)$);
  \coordinate (D12) at ($(center)+(150:\radius/2)$);
  \coordinate (D23) at ($(center)+(30:\radius/2)$);
  \coordinate (D13) at  ($(center)+(-90:\radius/2)$);
   
  \fill [black!30,opacity=1]  (D1) --(D2) -- (D3)--(D1);
  
 \fill[blue] (center) circle[radius=2pt];  
  \fill[blue] (D1) circle[radius=2pt];
  \fill[blue] (D2) circle[radius=2pt];
  \fill[blue] (D3) circle[radius=2pt];
  \fill[red] (D12) circle[radius=2pt];
  \fill[red] (D23) circle[radius=2pt];
  \fill[red] (D13) circle[radius=2pt];
  
\draw [anchor=east] (center) node {$D_0$};
  \draw [anchor=north] (D1) node {$D_1$};
    \draw [anchor=south] (D2) node {$D_2$};
    \draw [anchor=north] (D3) node {$D_3$};
    \draw [anchor=south east] (D12) node {$D_{12}$};
    \draw [anchor=south west] (D23) node {$D_{23}$};
    \draw [anchor=north] (D13) node {$D_{13}$};
    
\end{scope}

\begin{scope}[shift={(8,0)}]
\draw (-0.5,2.5) node {$[D_1,D_3]$};
  \coordinate (center) at (0,0);
  \def\radius{1.5};
  \coordinate (D1) at ($(center)+(210:\radius)$);
  \coordinate (D2) at ($(center)+(90:\radius)$);
  \coordinate (D3) at ($(center)+(-30:\radius)$);
  \coordinate (D12) at ($(center)+(150:\radius/2)$);
  \coordinate (D23) at ($(center)+(30:\radius/2)$);
  \coordinate (D13) at  ($(center)+(-90:\radius/2)$);
   
  \fill [black!30,opacity=0.3]  (D1) --(D2) -- (D3)--(D1);
  \draw [line width = 1.8pt] (D1) -- (D3);
  
  \fill[black] (D1) circle[radius=2pt];
  \fill[black] (D3) circle[radius=2pt];

  \draw [anchor=north] (D1) node {$D_1$};
    \draw [anchor=north] (D3) node {$D_3$};
\end{scope}

\begin{scope}[shift={(12,0)}]
\draw (0,2.5) node {$[D_1,D_{23}]$};
  \coordinate (center) at (0,0);
  \def\radius{1.5};
  \coordinate (D1) at ($(center)+(210:\radius)$);
  \coordinate (D2) at ($(center)+(90:\radius)$);
  \coordinate (D3) at ($(center)+(-30:\radius)$);
  \coordinate (D12) at ($(center)+(150:\radius/2)$);
  \coordinate (D23) at ($(center)+(30:\radius/2)$);
  \coordinate (D13) at  ($(center)+(-90:\radius/2)$);
   
  \fill [black!30,opacity=0.3]  (D1) --(D2) -- (D3)--(D1);
    \draw [line width = 1.8pt] (D1) -- (D23);
  
 \fill[black] (center) circle[radius=2pt];  
  \fill[black] (D1) circle[radius=2pt];
  \fill[black] (D23) circle[radius=2pt];

\draw [anchor=north west] (center) node {$D_0$};
  \draw [anchor=north] (D1) node {$D_1$};
    \draw [anchor=south west] (D23) node {$D_{23}$};
    
\end{scope}

\begin{scope}[shift={(0,-4.5)}]
\draw (-0.2,2.2) node {$[D_{12},D_{23}]$};
  \coordinate (center) at (0,0);
  \def\radius{1.5};
  \coordinate (D1) at ($(center)+(210:\radius)$);
  \coordinate (D2) at ($(center)+(90:\radius)$);
  \coordinate (D3) at ($(center)+(-30:\radius)$);
  \coordinate (D12) at ($(center)+(150:\radius/2)$);
  \coordinate (D23) at ($(center)+(30:\radius/2)$);
  \coordinate (D13) at  ($(center)+(-90:\radius/2)$);
   
  \fill [black!30,opacity=0.3]  (D1) --(D2) -- (D3)--(D1);
    \draw [line width = 1.8pt] (D12) --(center)-- (D23);
  
 \fill[black] (center) circle[radius=2pt];  
  \fill[black] (D12) circle[radius=2pt];
  \fill[black] (D23) circle[radius=2pt];

\draw [anchor=north] (center) node {$D_0$};
  \draw [anchor=south east] (D12) node {$D_{12}$};
    \draw [anchor=south west] (D23) node {$D_{23}$};
    
\end{scope}

\begin{scope}[shift={(4,-4.5)}]
\draw (-0.2,2.2) node {$\tconv(\{D_{12},D_{23},D_{13}\})$};
  \coordinate (center) at (0,0);
  \def\radius{1.5};
  \coordinate (D1) at ($(center)+(210:\radius)$);
  \coordinate (D2) at ($(center)+(90:\radius)$);
  \coordinate (D3) at ($(center)+(-30:\radius)$);
  \coordinate (D12) at ($(center)+(150:\radius/2)$);
  \coordinate (D23) at ($(center)+(30:\radius/2)$);
  \coordinate (D13) at  ($(center)+(-90:\radius/2)$);
   
  \fill [black!30,opacity=0.3]  (D1) --(D2) -- (D3)--(D1);
    \draw [line width = 1.8pt] (D12) --(center)-- (D23);
    \draw [line width = 1.8pt] (center)-- (D13);
  
 \fill[black] (center) circle[radius=2pt];  
  \fill[black] (D12) circle[radius=2pt];
  \fill[black] (D23) circle[radius=2pt];
  \fill[black] (D13) circle[radius=2pt];

\draw [anchor=north east] (center) node {$D_0$};
  \draw [anchor=south east] (D12) node {$D_{12}$};
    \draw [anchor=south west] (D23) node {$D_{23}$};
    \draw [anchor=north] (D13) node {$D_{13}$};
    
\end{scope}

\begin{scope}[shift={(8,-4.5)}]
\draw (0,2.2) node {$\tconv(\{D_{12},D_2,D_{13}\})$};
  \coordinate (center) at (0,0);
  \def\radius{1.5};
  \coordinate (D1) at ($(center)+(210:\radius)$);
  \coordinate (D2) at ($(center)+(90:\radius)$);
  \coordinate (D3) at ($(center)+(-30:\radius)$);
  \coordinate (D12) at ($(center)+(150:\radius/2)$);
  \coordinate (D23) at ($(center)+(30:\radius/2)$);
  \coordinate (D13) at  ($(center)+(-90:\radius/2)$);
   
  \fill [black!30,opacity=0.3]  (D1) --(D2) -- (D3)--(D1);
  \fill [black!100,opacity=1]  (D12) --(center)-- (D2) -- cycle;
    \draw [line width = 1.8pt] (center)-- (D13);
    \draw [line width = 1.8pt] (D12)--(center)-- (D2);

  \fill[black] (D12) circle[radius=2pt];
  \fill[black] (D2) circle[radius=2pt];
  \fill[black] (D13) circle[radius=2pt];

\draw [anchor=north east] (center) node {$D_0$};
  \draw [anchor=south east] (D12) node {$D_{12}$};
    \draw [anchor=east] (D2) node {$D_2$};
    \draw [anchor=north] (D13) node {$D_{13}$};
    
\end{scope}

\begin{scope}[shift={(12,-4.5)}]
\draw (0.2,2.2) node {$\tconv(\{D_{12},D_2,D_{23}\})$};
  \coordinate (center) at (0,0);
  \def\radius{1.5};
  \coordinate (D1) at ($(center)+(210:\radius)$);
  \coordinate (D2) at ($(center)+(90:\radius)$);
  \coordinate (D3) at ($(center)+(-30:\radius)$);
  \coordinate (D12) at ($(center)+(150:\radius/2)$);
  \coordinate (D23) at ($(center)+(30:\radius/2)$);
  \coordinate (D13) at  ($(center)+(-90:\radius/2)$);
   
  \fill [black!30,opacity=0.3]  (D1) --(D2) -- (D3)--(D1);
  \fill [black!100,opacity=1]  (D12) --(center)--(D23)-- (D2) -- cycle;

  \fill[black] (D12) circle[radius=2pt];
  \fill[black] (D2) circle[radius=2pt];
  \fill[black] (D23) circle[radius=2pt];

\draw [anchor=north] (center) node {$D_0$};
  \draw [anchor=south east] (D12) node {$D_{12}$};
    \draw [anchor=east] (D2) node {$D_2$};
    \draw [anchor=south west] (D23) node {$D_{23}$};
    
\end{scope}
 
\end{tikzpicture}
\caption{A metric circle $\Gamma$, a complete linear system $|D_0|$ on $\Gamma$ and several linear systems in $|D_0|$.} \label{F:LinearSys}
\end{figure}

\end{example}

\subsection{Tropical Projections and the Tropical Weak Independence Criterion for Divisors and $\mbbR$-Divisors}

As stated in the previous subsection, using the embedding map $\iota_{D_0}:\RDivPlusD(\Gamma)\to\mbbTP(\Gamma)$ with $D\mapsto [f_{D-D_0}]$ where $D_0\in \RDivPlusD$, we may translate the theorems for lower tropical convexity on $\mbbTP(\Gamma)$ to theorems for tropical convexity on $\RDivPlusD$ and this translation is essentially independent from the $\mbbR$-divisor $D_0$ used for the embedding $\iota_{D_0}$. 

First we may define the $B$-pseudonorms to all degree zero $\mbbR$-divisors $D$  as $\llfloor D \rrfloor_p := \llfloor f_D \rrfloor_p$ for all $p\in[1,\infty]$. In particular, 
$\llfloor D_2-D_1 \rrfloor_p = \llfloor f_{D_2-D_1} \rrfloor_p=\llfloor  \iota_{D_0}(D_2)-\iota_{D_0}(D_1)\rrfloor_p$ for all $D_1,D_2\in \RDivPlusD$.
Then we can rewrite the tropical projection theorem (Theorem~\ref{T:main} and Corollary~\ref{C:CritTropProj}) as follows:

\begin{theorem}  \label{T:TropProjDiv} 
 For a compact  tropically convex subset $T$ of $\RDivPlusD(\Gamma)$ and an arbitrary element $E$ in $\RDivPlusD(\Gamma)$, consider  the following real-valued  functions $\Theta^{(T,E)}_p:T \to [0,\infty)$ with $p\in[1,\infty]$ given by  $D\mapsto\llfloor D-E\rrfloor_p$. In particular, $\Theta^{(T,E)}_\infty(D)= \rho(D,E)$. 
 \begin{enumerate}
 \item  There is a unique element $\pi_T(E)$ called the tropical projection of $E$ to $T$ which minimizes $\Theta^{(T,E)}_p$  for all $p\in[1,\infty)$.
 \item  The minimizer of $\Theta^{(T,E)}_\infty$ is compact and tropically convex which contains $\pi_T(E)$. 
 \item The following are equivalent:
 \begin{enumerate}
 \item $D_0=\pi_T(E)$;
 \item For each $D\in T$, $\llfloor D-E\rrfloor_1=\llfloor D-D_0\rrfloor_1+\llfloor D_0-E\rrfloor_1$;
 \item For each $D\in T$, $\Gmin(f_{D-D_0})\bigcap\Gmin(f_{D_0-E})\neq\emptyset$;
 \item For each $D\in T$, $\Gmin(f_{D-D_0})\bigcap\Gmin(f_{D_0-E})=\Gmin(f_{D-E})$;
 \item For each $D\in T$ such that $D\neq D_0$,  $\Gmin(f_{D_0-D})\bigcap\Gmin(f_{D-E})=\emptyset$.
 \end{enumerate}
 \end{enumerate}
\end{theorem}

\begin{corollary} \label{C:TropProjDiv}
For a compact  tropically convex subset $T$ of $\RDivPlusD(\Gamma)$, let $cT+F:=\{c\cdot D+F\mid D\in T\}$ for $F\in\RDivPlusM(\Gamma)$ and $c>0$. Then $cT+F$ is a compact tropically convex subset of $\RDivPlusDM(\Gamma)$ and $\pi_{cT+F}(c\cdot E+F)=c\cdot \pi_T(E)+F$ for all $E\in\RDivPlusD(\Gamma)$. 
\end{corollary}
\begin{proof}
This is an interpretation of Proposition~\ref{P:TropProj}(1) for $\RDivPlus(\Gamma)$. 

Let $D_0=\pi_T(E)$. Then for each $D\in T$, $\Gmin(f_{D-D_0})\bigcap\Gmin(f_{D_0-E})\neq\emptyset$ by Theorem~\ref{T:TropProjDiv}. Since  $\Gmin(f_{(c\cdot D+F)-(c\cdot D_0+F)})=\Gmin(c\cdot f_{D-D_0})=\Gmin(f_{D-D_0})$ and $\Gmin(f_{(c\cdot D_0+F)-(c\cdot E+F)})=\Gmin(c\cdot f_{D_0-E})=\Gmin(f_{D_0-E})$, we must have for each $c\cdot D+F \in cT+F$, 
$$\Gmin(f_{(c\cdot D+F)-(c\cdot D_0+F)})\bigcap \Gmin(f_{(c\cdot D_0+F)-(c\cdot E+F)})\neq \emptyset$$ which means $\pi_{cT+F}(c\cdot E+F)=c\cdot D_0+F$. 
\end{proof}

The following proposition is the version of Proposition~\ref{P:SeqTropProj} for $\RDivPlusD$. 
\begin{proposition} \label{P:SeqTropProjDiv}
Let $T$ and $T'$ be compact tropical convex subsets of $\RDivPlusD(\Gamma)$ such that $T'\subseteq T$. Then for each $E\in\RDivPlusD(\Gamma)$,  $\pi_{T'}(E)=\pi_{T'}(\pi_T(E))$.
\end{proposition}

The following theorem is the version of Theorem~\ref{T:CriterionTropIndep} and Theorem~\ref{T:extremal} for $\RDivPlusD$. 

\begin{theorem}\label{T:FiniteCriterion}
Let $T\subseteq\RDivPlusD$ be a tropical polytope generated by $D_1,\cdots,D_n$. Let $E$ be an $\mbbR$-divisor of degree $d$.
\begin{enumerate}
\item $E\in T$ if and only if $\bigcup_{i=1}^n\Gmin(f_{D_i-E})=\Gamma$.
\item If $D_0=\pi_T(E)$ and $D$ is an arbitrary $\mbbR$-divisor in $T$, then
$\Gmin(f_{D-E})\subseteq\Gmin(f_{D_0-E})=\bigcup_{i=1}^n\Gmin(f_{D_i-E})$.
\item Let $S$ be the set of extremals of $T$. Then $S\subseteq \{D_1,\cdots,D_n\}$ and $T=\tconv(S)$. 
\end{enumerate}
\end{theorem}

Note that in the above theorem, if we make a further restriction to $\DivPlusD(\Gamma)$ with $T$ being a linear system (Definition~\ref{D:LinSys}), then it will be a finitely verifiable criterion to tell whether an arbitrary divisor $D\in \DivPlusD(\Gamma)$ is contained in $T$. Moreover, we have the following immediate corollary. 

\begin{corollary}
Let $T$ be a linear system, i.e., $T=\tconv(\{D_1,\cdots,D_n\})$ where $D_1,\cdots,D_n$  are linearly equivalent divisors. Then for each divisor $E\in\RDivPlusD\setminus \DivPlusD$, we must have $\bigcup_{i=1}^n\Gmin(f_{D_i-E})\neq\Gamma$. 
\end{corollary}
\begin{proof}
This follows from Theorem~\ref{T:FiniteCriterion} knowing that  $E$  cannot be an element of $T$. 
\end{proof}

\begin{example}
Let $T$ be the linear system $\tconv(\{D_{12},D_2,D_{13}\})$ in Figure~\ref{F:LinearSys}. Here we will use Theorem~\ref{T:FiniteCriterion} to verify that $D_0\in T$ and $D_{23}\notin T$. We note that $\Gmin(f_{D_{12}-D_0})$ is the segment $v_1v_3v_2$, $\Gmin(f_{D_2-D_0})$ is the segment $v_1v_3$, $\Gmin(f_{D_{13}-D_0})$ is the segment $v_1v_2v_3$, $\Gmin(f_{D_{12}-D_{23}})$ is the point $w_{23}$, $\Gmin(f_{D_2-D_{23}})$ is the segment $v_1v_3w_{23}$ and $\Gmin(f_{D_{13}-D_{23}})$ is the point $w_{23}$. Thus $\Gmin(f_{D_{12}-D0})\bigcup \Gmin(f_{D_2-D_0}) \bigcup \Gmin(f_{D_{13}-D_0}) =\Gamma$ which means $D_0\in T$ and $\Gmin(f_{D_{12}-D_{23}})\bigcup \Gmin(f_{D_2-D_{23}}) \bigcup \Gmin(f_{D_{13}-D_{23}}) $ is the segment $v_1v_3w_{23}$ which means $D_{23}\notin T$ by Theorem~\ref{T:FiniteCriterion}. 
\end{example}

\subsection{Reduced Divisors: from  $b$-Functions to $B$-Pseudnorms} \label{SS:bFuncBPseudoNorm}
Let us first recall the definition of reduced divisors on a metric graph $\Gamma$. Let $X$ be a closed subset of $\Gamma$. For each $p\in X$, we denote the number of segments leaving $X$ at $p$ by  $\outdeg_X(p)$. Note that $\outdeg_X(p)=0$ for all $p\in X\setminus \partial X$. 

\begin{definition} \label{D:RedDiv}
Fix a point $q\in \Gamma$. We say a divisor $D\in\Div(\Gamma)$ is \emph{$q$-reduced} if 
\begin{enumerate}
\item $D(x)\geq 0$ for all $x\in \Gamma\setminus\{q\}$ and 
\item for every closed subset $X$ of $\Gamma\setminus\{q\}$,  there exists a point $p\in\partial X$ such that $D(p)<\outdeg_X(p)$. 
\end{enumerate}
\end{definition}

The most important property of reduced divisors is that  for each point $q\in\Gamma$ each divisor $D$, there exists a unique $q$-reduced divisor $D_q$ in the  linear equivalence class $[D]$. In particular,  if $D$ is effective, then $D_q\in|D|$. 

Adapted from an algorithm  in the context of sanpile models \cite{Dhar90},  there is a classical algorithm commonly called Dhar's algorithm which can efficiently determine whether a divisor is $q$-reduced for finite graphs and metric graphs \cite{Luo11}. 
Baker and Shokrieh \cite{BS13}  obtained  an elegant characterization of reduced divisors for finite graphs using energy pairing. In particular, they introduced the notion of $b_q$-function on the divisor group of a finite graph $G$ where $q$ is an arbitrary vertex of $G$ and proved that an effective divisor $D$ is $q$-reduced if and only if $D$ minimizes the $b_q$-function restricted to the complete linear system $|D|$ on $G$ (Theorem~4.14 in \cite{BS13}). 

An interesting observation is that the natural translation of $b_q$-function in the context of metric graphs is exactly the $\underline{B}^1$-pseudonorm of $f_{D-d\cdot(q)}$ where $D$ is a divisor of degree $d$ on $\Gamma$ (actually we name $B$-pseudonorms after $b$-functions) . On one hand,  the $q$-reduced divisors in the context of metric graphs should also be minimizers of $b_q$-function on complete linear systems as in the context of finite graphs proved in \cite{BS13}. 
On the other hand, we know that $|D|$ is a tropical polytope  and thus we can make a tropical projection of any effective divisor  $E$ of degree $d$ to $|D|$ which minimizes all $B$-pseudonorms of $\llfloor \cdot - E\rrfloor_p$ restricted to $|D|$. This actually means that the $q$-reduced divisor in $|D|$ should be exactly the tropical projection of the divisor $d\cdot(q)$ to $|D|$. We give a precise characterization in the following proposition.

\begin{proposition} \label{P:RedTroPro}
Consider a complete linear system $|D|\subseteq \DivPlusD(\Gamma)$ on a metric graph $\Gamma$. For an arbitrary point $q\in\Gamma$, the following are equivalent:
\begin{enumerate}
\item $D_0\in |D|$ is $q$-reduced. 
\item The base of any possible effective chip-firing move from $D_0$ contains $q$. 
\item For each $D'\in |D|$, $\Gmin(f_{D'-D_0})$ contains $q$. 
\item For each $D'\in |D|$, $\Gmin(f_{D'-D_0})\bigcap \Gmin(f_{D_0-d\cdot (q)})\neq\emptyset$.
\item $D_0=\pi_{|D|}(d\cdot (q))$.
\end{enumerate}
\end{proposition}

\begin{proof}
(1)$\Leftrightarrow$(2): By Remark~\ref{R:ChipFiring}, to make an effective chip-firing move from $D_0$ with base $X$, we must have enough chips at all boundary points of $X$, i.e., for each point $p\in\partial X$, $D_0(p)\geq D^-(p)\geq \outdeg_X(p)$ where $D^-$ is the noneffective part of the primitive rational function with respect to the chip-firing move. Therefore, by Definition~\ref{D:RedDiv}, the base of any effective chip-firing move from $D_0$ must contain $q$, since $D_0$ is $q$-reduced. 

(2)$\Leftrightarrow$(3): This is because on one hand chip-firing moves are associated to primitive rational functions and on the other hand $\Gmin(f_{D'-D_0})$ for each $D'\in |D|$ must coincide with the base of some effective chip-firing move from $D_0$. 

(3)$\Leftrightarrow$(4): This follows from Lemma~\ref{L:CritEqui}. 

(4)$\Leftrightarrow$(5): This follows from  Theorem~\ref{T:TropProjDiv}. 
\end{proof}

\begin{lemma} \label{L:CritEqui}
For each $q\in\Gamma$ and $D,D'\in\RDivPlusD(\Gamma)$ with $d>0$, $\Gmin(f_{D'-D})\bigcap \Gmin(f_{D-d\cdot (q)})\neq \emptyset$ if and only if $q\in\Gmin(f_{D'-D})$.
\end{lemma}
\begin{proof}
Clearly $\Gmin(f_{D-d\cdot (q)})$ must contain $q$. So if $q\in\Gmin(f_{D'-D})$, then $\Gmin(f_{D'-D})\bigcap \Gmin(f_{D-d\cdot (q)})\neq \emptyset$.

Suppose $\Gamma_1,\cdots,\Gamma_m$ are the connected components of $\Gamma\setminus\{q\}$. We note that for $i=1,\cdots,m$, 
\begin{enumerate}
\item $\Gmin(f_{D-d\cdot (q)})\bigcap \Gamma_i\neq \emptyset$ if and only if $\Gmin(f_{D-d\cdot (q)})\bigcap \Gamma_i=\Gamma_i$ if and only if $\supp(D)\bigcap \Gamma_i=\emptyset$,
and \item if $q\notin\Gmin(f_{D'-D})$, then $\Gmin(f_{D'-D})\bigcap \Gamma_i\neq \emptyset$ if and only if $\supp(D)\bigcap \Gamma_i\neq\emptyset$. 
\end{enumerate}

Now suppose $\Gmin(f_{D'-D})\bigcap \Gmin(f_{D-d\cdot (q)})\neq \emptyset$. If $q\notin\Gmin(f_{D'-D})$, then there must be some $\Gamma_i$ such that $\Gmin(f_{D'-D})\bigcap \Gamma_i\neq \emptyset$ and $\Gmin(f_{D-d\cdot (q)})\bigcap \Gamma_i\neq \emptyset$. Then by (1), $\Gmin(f_{D-d\cdot (q)})\bigcap \Gamma_i\neq \emptyset$ means that $\supp(D)\bigcap \Gamma_i=\emptyset$, and by (2), $\Gmin(f_{D-d\cdot (q)})\bigcap \Gamma_i\neq \emptyset$ means that $\supp(D)\bigcap \Gamma_i\neq\emptyset$, a contradiction. 
\end{proof}

\begin{remark}
Let $X$ be the vertex set of $G$ equipped with the counting measure and we can also use our theory to study the $b$-functions on a finite graph $G$ directly. In this way, the $b_q$-function on is exactly the $\underline{B}^1$-pseudonorm of $f_{D-d\cdot(q)}$ where $q$ is a vertex of $G$ and $D$ is a divisor of degree $d$ on $G$. Note that Remark~4.15 of \cite{BS13} provides variations of the $b_q$-function which says that the weights of different vertices can be different as long as all being non-negative. Such variations  do not affect the $q$-reduced divisors being the minimizers of the $b_q$-function restricted to complete linear systems. This fact is exactly reflected in Corollary~\ref{C:CritTropProj}(c) which says tropical projections are independent of the measure on $X$, since different measures on $X$ can be considered as different distributions of weights on vertices. 
\end{remark}

Proposition~\ref{P:RedTroPro} tells us that the $q$-reduced divisor in a complete linear system $|D|$ of degree $d$ is no more than the tropical projection of the divisor $d\cdot(q)$ to $|D|$. Therefore we have the following natural generalization of the notion of reduced divisors. 

\begin{definition} \label{D:GeneralRed}
Let $T$ be a compact  tropically convex subset of $\RDivPlusD$. Then for each $q\in\Gamma$, the $q$-reduced $\mbbR$-divisor in $T$ is the tropical projection of $d\cdot(q)$ to $T$. 
\end{definition} 

We can derive the following proposition as an analogue of Proposition~\ref{P:RedTroPro}. 
\begin{proposition} \label{P:GeneralRed}
Consider  a compact  tropically convex subset $T$ of $\RDivPlusD(\Gamma)$. For an arbitrary point $q\in\Gamma$, the following are equivalent:
\begin{enumerate}
\item $D_0$ is the $q$-reduced in $T$. 
\item The base of any possible effective chip-firing move inside $T$ from $D_0$ contains $q$. 
\item For each $D\in T$, $\Gmin(f_{D-D_0})$ contains $q$. 
\item For each $D\in T$, $\Gmin(f_{D-D_0})\bigcap \Gmin(f_{D_0-d\cdot (q)})\neq\emptyset$.
\item $D_0=\pi_T(d\cdot (q))$.
\end{enumerate}
\end{proposition}
\begin{proof}
The equivalence of (1) and (5) follows from Definition. The equivalence of (2)--(5) follows from arguments analogous to those in the proof of Proposition~\ref{P:RedTroPro} with the notion of chip-firing moves generalized to cases for $\mbbR$-divisors (Remark~\ref{R:ChipFiring}). 
\end{proof}

\begin{corollary}\label{C:RedVal}
Let $q$ be an arbitrary point in $\Gamma$ and $T$ be a compact tropical convex subset of $\RDivPlusD(\Gamma)$. For each $D\in T$, the value of $D$ at $q$ is at most the value of $q$-reduced divisor in $T$ at $q$. 
\end{corollary}
\begin{proof}
Let $D_0$ be the $q$-reduced divisor in $T$. Then by Proposition~\ref{P:GeneralRed}, we see that for each $D\in T$, $\Gmin(f_{D-D_0})$ contains $q$. This actually implies the value of $D-D_0$ at $q$ is non-positive. 
\end{proof}

\begin{remark} \label{R:RedVal}
We can actually be more precise about Corollary~\ref{C:RedVal}. As in the above proof, let $B=\Gmin(f_{D-D_0})$. If $q$ is a boundary point of $B$, then $D(q)<D_0(q)$, and if $q$ is not a boundary point of $B$, then $D(q)=D_0(q)$. 
\end{remark}

The following proposition says that reduced $\mbbR$-divisors are invariant under ``translation'' and scaling of the corresponding compact tropically convex set.

\begin{proposition}
Let $T$ be a compact  tropically convex subset of $\RDivPlusD$ and $F$ be an effective $\mbbR$-divisor of degree $m$. Let $cT+F:=\{c\cdot D+F\mid D\in T\}\subseteq \RDivPlusDM$. Then $\pi_{cT+F}((cd+m)\cdot (q))=\pi_{cT+F}(cd\cdot (q)+F)=c\cdot \pi_T(d\cdot (q))+F$ for each $q\in\Gamma$. 
\end{proposition}

\begin{proof}
First we note that $\pi_{cT+F}(cd\cdot (q)+F)=c\cdot \pi_T(d\cdot (q))+F$ follows from Corollary~\ref{C:TropProjDiv} directly. 

Now let $D_0$ be the $q$-reduced $\mbbR$-divisor $\pi_T(d\cdot (q))$ in $T$. Then by Proposition~\ref{P:GeneralRed}, for each $D\in T$, $\Gmin(f_{D-D_0})$ contains $q$. Therefore, for each $c\cdot D+F\in cT+F$, we have $q\in \Gmin(f_{(c\cdot D+F)-(c\cdot D_0+F)})=\Gmin(f_{D-D_0})$. Using Proposition~\ref{P:GeneralRed} again, this means that $c\cdot D_0+E$ is exactly the $q$-reduced $\mbbR$-divisor $\pi_{cT+F}((cd+e)\cdot (q))$ in $cT+F$.
\end{proof}

\begin{example} \label{E:Red}
Consider the metric circle $\Gamma$ in Figure~\ref{F:LinearSys}. Let us use Proposition~\ref{P:GeneralRed}(2) to verify some divisors as being reduced with respect to certain points. 
\begin{enumerate}
\item Consider the complete linear system $|D_0|$. There are three possible effective chip-firing move directions from $D_1$: 
\begin{enumerate}
\item along the direction of $D_1D_2$, i.e., two chips move from $v_1$ towards $w_{12}$ at the speed of one unit and one chip moves from $v_1$ towards $w_{13}$ at the speed of two unit;
\item along the direction of $D_1D_3$, i.e., two chips move from $v_1$ towards $w_{13}$ at the speed of one unit and one chip moves from $v_1$ towards $w_{12}$ at the speed of two unit; and
\item along the direction of $D_1D_0$, i.e.,  one chip moves from $v_1$ towards $w_{12}$ and one chip moves from $v_1$ towards $w_{13}$, both at the speed of one unit. 
\end{enumerate}
The bases of all the above chip-firing move directions are identical to $\{v_1\}$. Therefore, $v_1$-reduced divisor in $|D_0|$ is $D_1$ by Proposition~\ref{P:GeneralRed}(2). It can be verified analogously that the $v_2$-reduced divisor in $|D_0|$ is $D_2=3(v_2)$, and the $v_3$-reduced divisor in $|D_0|$ is $D_3=3(v_3)$. 

\item Consider the tropical segment $[D_1,D_3]$. We observe that  an effective chip-firing move inside $[D_1,D_3]$ from $D_1$ can only be along the direction of $D_1D_3$ whose base is the point $v_1$, and an effective chip-firing move inside $[D_1,D_3]$ from $D_3$ can only be along the direction of $D_3D_1$ whose base is the point $v_3$. Therefore, $D_1$ and $D_3$ are the $v_1$-reduced and $v_3$-reduced divisors in $T$ respectively. Moreover, an effective chip-firing move inside $[D_1,D_3]$ from $D_{13}$ can be either along the direction of $D_{13}D_1$ whose base is the segment $w_{13}v_3v_2$ or along the direction of $D_{13}D_3$ whose base is the segment $w_{13}v_1v_2$. Both bases contain $w_{13}$ and $v_2$ which means that $D_{13}$ is the reduced divisor in $[D_1,D_3]$ with respect to both $w_{13}$ and $v_2$. 
\item Let $T=\tconv(\{D_{12},D_{23},D_{13}\})$. We observe that an effective chip-firing move inside $T$ from $D_0$ should be along one of the directions of  $D_0D_{12}$, $D_0D_{23}$ and $D_0D_{13}$ whose bases are the segment $v_1v_3v_2$, the segment $v_2v_1v_3$ and the segment $v_1v_2v_3$ respectively. This means that the reduced divisors in $T$ with respect to $v_1$, $v_2$ and $v_3$ are all identical to $D_0$. From $D_{12}$, there is only one direction of chip-firing move which is along $D_{12}D_0$ with base being the point $w_{12}$. Therefore, $D_{12}$ is the $w_{12}$-reduced divisor in $T$. Analogously, $D_{23}$ and $D_{13}$ are the $w_{23}$-reduced and $w_{13}$-reduced divisors in $T$ respectively. 

\end{enumerate}
\end{example}

\subsection{Reduced Divisor Maps and Tropical Trees}
By sending a point $q$ in a metric graph $\Gamma$ to a complete linear system $|D|$ on $\Gamma$, we can naturally define a map called the reduced divisor map from $\Gamma$ to $|D|$, which was originally studied in \cite{Amini13}. Using the generalized notion of reduced divisors introduced in the previous subsection, we can broaden the notion of reduced divisor maps. 

\begin{definition} \label{D:RedDivMap}
Let $T$ be a compact  tropically convex subset of $\RDivPlusD$. The map $\Red_T:\Gamma\to T$ given by $q\mapsto \pi_T(d\cdot (q))$ is called the \emph{reduced divisor map} from $\Gamma$ to $T$. 
\end{definition}

The following proposition is a direct corollary of Proposition~\ref{P:SeqTropProjDiv}. 
\begin{proposition} \label{P:SeqRedjDiv}
Let $T$ and $T'$ be compact tropical convex subsets of $\RDivPlusD(\Gamma)$ such that $T'\subseteq T$. Then for each $q\in\Gamma$,  $\Red_{T'}(q)=\pi_{T'}(\Red_T(q))$.
\end{proposition}

Clearly reduced divisor maps are continuous since tropical projections are continuous. In the following discussions, we will focus on reduced divisor maps to linear systems. 

\begin{figure} 
\centering
\begin{tikzpicture}[scale=0.9]
\begin{scope}
\draw (-1.5,1.5) node[anchor=east] {(a)};
\begin{scope} [shift = {(1,0)}]
  \coordinate (center) at (0,0);
  \def\radius{1.5};
  \coordinate (v1) at ($(center)+(210:\radius)$);
  \coordinate (v2) at ($(center)+(90:\radius)$);
  \coordinate (v3) at ($(center)+(-30:\radius)$);
  \coordinate (w12) at ($(center)+(150:\radius)$);
  \coordinate (w23) at ($(center)+(30:\radius)$);
  \coordinate (w13) at ($(center)+(-90:\radius)$);
   
  \draw[thick] (center) circle[radius=\radius];

  \fill[blue] (v1) circle[radius=3pt];
  \fill[blue] (v2) circle[radius=3pt];
  \fill[blue] (v3) circle[radius=3pt];
  
  \draw [anchor=north east] (v1) node {$v_1$};
    \draw [anchor=south] (v2) node {$v_2$};
    \draw [anchor=north west] (v3) node {$v_3$};
    \draw [anchor=north east] (w13) node {$w_{13}$};
   
    \coordinate (center) at (0,-3.5);
      \def\radius{1.5};
  \coordinate (D1) at ($(center)+(210:\radius)$);
  \coordinate (D2) at ($(center)+(90:\radius)$);
  \coordinate (D3) at ($(center)+(-30:\radius)$);
  \coordinate (D12) at ($(center)+(150:\radius/2)$);
  \coordinate (D23) at ($(center)+(30:\radius/2)$);
  \coordinate (D13) at  ($(center)+(-90:\radius/2)$);
   
  \fill [black!30,opacity=0.3]  (D1) --(D2) -- (D3)--(D1);
    \draw [line width = 1.8pt] (D1)-- (D3);

  \fill[black] (D1) circle[radius=2pt];
  \fill[black] (D3) circle[radius=2pt];
  \fill[black] (D13) circle[radius=2pt];

\draw [anchor=north] (D1) node {$D_1$};
  \draw [anchor=north] (D3) node {$D_3$};
    \draw [anchor=north] (D13) node {$D_{13}$};

   \def\ht{0.2}

 \draw [->, dashed, line width = 1pt] (v1)--  ($(D1)+(0,\ht)$);
 \draw [->, dashed, line width = 1pt] (v3)-- ($(D3)+(0,\ht)$);
 \draw [->, dashed, line width = 1pt] (v2)-- ($(D13)+(0,\ht)$);
 
  \fill[red] (w13) circle[radius=3pt];
    
\end{scope}

\begin{scope} [shift= {(6,0)}]
  \coordinate (center) at (0,0);
  \def\radius{1.5};
  \coordinate (v1) at ($(center)+(210:\radius)$);
  \coordinate (v2) at ($(center)+(90:\radius)$);
  \coordinate (v3) at ($(center)+(-30:\radius)$);

  \fill[black] (v1) circle[radius=3pt];
  \fill[black] (v2) circle[radius=3pt];
  \fill[black] (v3) circle[radius=3pt];
  
  \draw [anchor=north east] (v1) node {$v_1$};
    \draw [anchor=south] (v2) node {$v_2$};
    \draw [anchor=north west] (v3) node {$v_3$};
    
   \coordinate (D1) at ($(v1)+(0,-3)$);
  \coordinate (D3) at ($(v3)+(0,-3)$);   
    \draw [line width = 1.8pt] (D1)-- (D3);

  \fill[black] (D1) circle[radius=2pt];
  \fill[black] (D3) circle[radius=2pt];
  
  \draw [anchor=north] (D1) node {$D_1$};
  \draw [anchor=north] (D3) node {$D_3$};

 \path [line width = 1pt]
(v1) edge node[pos=0.5,left]{$1$} (v2)
(v2) edge node[pos=0.5,right]{$1$} (v3)
(v1) edge node[pos=0.5,below]{$2$} (v3);

 \draw [->, dashed, line width = 1pt] ($(center)+(0,-1.5)$)--  ($(center)+(0,-3.5)$);
 
\end{scope}
\end{scope}

\begin{scope} [shift={(0,-7)}]
\draw (-1.5,1.5) node[anchor=east] {(b)};
\begin{scope}[shift={(0,-0.5)}]
  \coordinate (center) at (0,0);
  \def\radius{1.5};
  \coordinate (v1) at ($(center)+(210:\radius)$);
  \coordinate (v2) at ($(center)+(90:\radius)$);
  \coordinate (v3) at ($(center)+(-30:\radius)$);
  \coordinate (w12) at ($(center)+(150:\radius)$);
  \coordinate (w23) at ($(center)+(30:\radius)$);
  \coordinate (w13) at ($(center)+(-90:\radius)$);
   
  \draw [name path=  circle, thick] (center) circle (\radius);

  \fill[blue] (v1) circle[radius=3pt];
  \fill[blue] (v2) circle[radius=3pt];
  \fill[blue] (v3) circle[radius=3pt];
    \fill[red] (w12) circle[radius=3pt];
    \fill[red] (w23) circle[radius=3pt];
    \fill[red] (w13) circle[radius=3pt];
  
  \draw [anchor=north east] (v1) node {$v_1$};
    \draw [anchor=south] (v2) node {$v_2$};
    \draw [anchor=north west] (v3) node {$v_3$};
    \draw [anchor=south east] (w12) node {$w_{12}$};
    \draw [anchor=south west] (w23) node {$w_{23}$};
    \draw [anchor=north] (w13) node {$w_{13}$};
    
        \draw[blue, dashed] (center) -- (v1);
    \draw[blue, dashed] (center) -- (v2);
    \draw[blue, dashed] (center) -- (v3);

 \foreach \i in {1,2,3}
       {\coordinate (P) at ($(v1)+(150:\i*\radius/4)$) ;
       \coordinate (Q) at ($(v2)+(150:\i*\radius/4)$) ;
       \coordinate (O) at ($(center)+(150:\i*\radius/4)$) ;
       
          \path [name path=OP] (O) -- (P);
             \path [name path=OQ] (O) -- (Q);

        \draw [name intersections={of=circle and OP}]
        (intersection-1) coordinate (P1);    
                \draw [name intersections={of=circle and OQ}]
        (intersection-1) coordinate (Q1); 
     \draw [dashed] (P1)--  (O) -- (Q1);
      }
      
 \foreach \i in {1,2,3}
       {\coordinate (P) at ($(v2)+(30:\i*\radius/4)$) ;
       \coordinate (Q) at ($(v3)+(30:\i*\radius/4)$) ;
       \coordinate (O) at ($(center)+(30:\i*\radius/4)$) ;
       
          \path [name path=OP] (O) -- (P);
             \path [name path=OQ] (O) -- (Q);

        \draw [name intersections={of=circle and OP}]
        (intersection-1) coordinate (P1);    
                \draw [name intersections={of=circle and OQ}]
        (intersection-1) coordinate (Q1); 
     \draw [dashed] (P1)--  (O) -- (Q1);
      }
      
\foreach \i in {1,2,3}
       {\coordinate (P) at ($(v1)+(-90:\i*\radius/4)$) ;
       \coordinate (Q) at ($(v3)+(-90:\i*\radius/4)$) ;
       \coordinate (O) at ($(center)+(-90:\i*\radius/4)$) ;
       
          \path [name path=OP] (O) -- (P);
             \path [name path=OQ] (O) -- (Q);

        \draw [name intersections={of=circle and OP}]
        (intersection-1) coordinate (P1);    
                \draw [name intersections={of=circle and OQ}]
        (intersection-1) coordinate (Q1); 
     \draw [dashed] (P1)--  (O) -- (Q1);
      }
   
    \coordinate (center) at (0,-3.5);
  \def\radius{1.5};
  \coordinate (D1) at ($(center)+(210:\radius)$);
  \coordinate (D2) at ($(center)+(90:\radius)$);
  \coordinate (D3) at ($(center)+(-30:\radius)$);
  \coordinate (D12) at ($(center)+(150:\radius/2)$);
  \coordinate (D23) at ($(center)+(30:\radius/2)$);
  \coordinate (D13) at  ($(center)+(-90:\radius/2)$);
   
  \fill [black!30,opacity=0.3]  (D1) --(D2) -- (D3)--(D1);
    \draw [line width = 1.8pt] (D12) --(center)-- (D23);
    \draw [line width = 1.8pt] (center)-- (D13);
  
 \fill[black] (center) circle[radius=2pt];  
  \fill[black] (D12) circle[radius=2pt];
  \fill[black] (D23) circle[radius=2pt];
  \fill[black] (D13) circle[radius=2pt];

\draw [anchor=north east] (center) node {$D_0$};
  \draw [anchor=south east] (D12) node {$D_{12}$};
    \draw [anchor=south west] (D23) node {$D_{23}$};
    \draw [anchor=north] (D13) node {$D_{13}$};
   
\end{scope}

\begin{scope}[shift={(5,0)}, x  = {(1cm,0cm)},
                    y  = {(0.5cm,0.8cm)},
                    z = {(0cm,1cm)}]

    \coordinate (center) at (0,0,-3.5);
      \def\radius{2};
  \coordinate (D12) at ($(center)+(150:\radius)$);
  \coordinate (D23) at ($(center)+(30:\radius)$);
  \coordinate (D13) at  ($(center)+(-90:\radius)$);
  
      \coordinate (v1) at (0,0,0);
      \coordinate (v2) at (0,0,0.4);
      \coordinate (v3) at (0,0,-0.4);
  \coordinate (w12) at ($(v2)+(150:\radius)$);
  \coordinate (w23) at ($(v1)+(30:\radius)$);
  \coordinate (w13) at  ($(v3)+(-90:\radius)$);
  
  \draw [line width = 1.8pt] (center)-- (D12);
  \draw [line width = 1.8pt] (center)-- (D23);
  \draw [line width = 1.8pt] (center)-- (D13);
  
 \draw [line width = 1pt] (w12)-- (v2) -- (w23);
 \draw [line width = 1pt] (w23)-- (v3) -- (w13);
 \draw [line width = 1pt] (w13)-- (v1) -- (w12);
 
   \fill[black] (v1) circle[radius=2pt];
  \fill[black] (v2) circle[radius=2pt];
  \fill[black] (v3) circle[radius=2pt];
  \fill[black] (w12) circle[radius=2pt];
  \fill[black] (w23) circle[radius=2pt];
  \fill[black] (w13) circle[radius=2pt];
  \fill[black] (D12) circle[radius=2pt];
  \fill[black] (D23) circle[radius=2pt];
  \fill[black] (D13) circle[radius=2pt];
 
 \draw [anchor=east] (v1) node {$v_1$};
  \draw [anchor=south] (v2) node {$v_2$};
    \draw [anchor=north west] (v3) node {$v_3$};
 \draw [anchor=south] (w12) node {$w_{12}$};
  \draw [anchor=south] (w23) node {$w_{23}$};
    \draw [anchor=west] (w13) node {$w_{13}$};
    \draw [anchor=north west] (center) node {$D_0$};
     \draw [anchor=east] (D12) node {$D_{12}$};
  \draw [anchor=north] (D23) node {$D_{23}$};
    \draw [anchor=north] (D13) node {$D_{13}$};
    
    \def\ht{0.25}
     \draw [->, dashed, line width = 1pt] (v2)-- ($(center)+(0,0,\ht)$);
     \draw [->, dashed, line width = 1pt] (w12)-- ($(D12)+(0,0,\ht)$);
     \draw [->, dashed, line width = 1pt] (w23)-- ($(D23)+(0,0,\ht)$);
     \draw [->, dashed, line width = 1pt] (w13)-- ($(D13)+(0,0,\ht)$);
\end{scope}

\begin{scope}[shift={(10,0)}, x  = {(1cm,0cm)},
                    y  = {(0.5cm,0.8cm)},
                    z = {(0cm,1cm)}]

    \coordinate (center) at (0,0,-3.5);
      \def\radius{2};
  \coordinate (D12) at ($(center)+(150:\radius)$);
  \coordinate (D23) at ($(center)+(30:\radius)$);
  \coordinate (D13) at  ($(center)+(-90:\radius)$);
  
      \coordinate (v1) at (0,0,0);
      \coordinate (v2) at (0,0,0.4);
      \coordinate (v3) at (0,0,-0.4);
  \coordinate (w12) at ($(v2)+(150:\radius)$);
  \coordinate (w23) at ($(v1)+(30:\radius)$);
  \coordinate (w13) at  ($(v3)+(-90:\radius)$);
    \coordinate (v12) at ($(v3)+(150:\radius)$);
  \coordinate (v23) at ($(v2)+(30:\radius)$);
  \coordinate (v13) at  ($(v2)+(-90:\radius)$);
  
  \draw [line width = 1.8pt] (center)-- (D12);
  \draw [line width = 1.8pt] (center)-- (D23);
  \draw [line width = 1.8pt] (center)-- (D13);
  
 \draw [line width = 1pt] (w12)-- (v2) -- (w23);
 \draw [line width = 1pt] (w23)-- (v3) -- (w13);
 \draw [line width = 1pt] (w13)-- (v1) -- (w12);
 
  \draw [red, line width = 1pt] (v2) -- (v13);
 \draw [red, line width = 1pt] (v1) -- (v23);
 \draw [red, line width = 1pt] (v3) -- (v12);
 
   \fill[black] (v1) circle[radius=2pt];
  \fill[black] (v2) circle[radius=2pt];
  \fill[black] (v3) circle[radius=2pt];
  \fill[black] (w12) circle[radius=2pt];
  \fill[black] (w23) circle[radius=2pt];
  \fill[black] (w13) circle[radius=2pt];
    \fill[red] (v12) circle[radius=2pt];
  \fill[red] (v23) circle[radius=2pt];
  \fill[red] (v13) circle[radius=2pt];
  \fill[black] (D12) circle[radius=2pt];
  \fill[black] (D23) circle[radius=2pt];
  \fill[black] (D13) circle[radius=2pt];

    \def\ht{0.25}
     \draw [->, dashed, line width = .8pt] (v2)-- ($(center)+(0,0,\ht)$);
     \draw [->, dashed, line width = .8pt] (w12)-- ($(D12)+(0,0,\ht)$);
     \draw [->, dashed, line width = .8pt] (v23)-- ($(D23)+(0,0,\ht)$);
     \draw [->, dashed, line width = .8pt] (v13)-- ($(D13)+(0,0,\ht)$);
\end{scope}
\end{scope}
\end{tikzpicture}
\caption{Examples of reduced divisor maps and harmonic morphisms.}  \label{F:RedHarm} 
\end{figure}
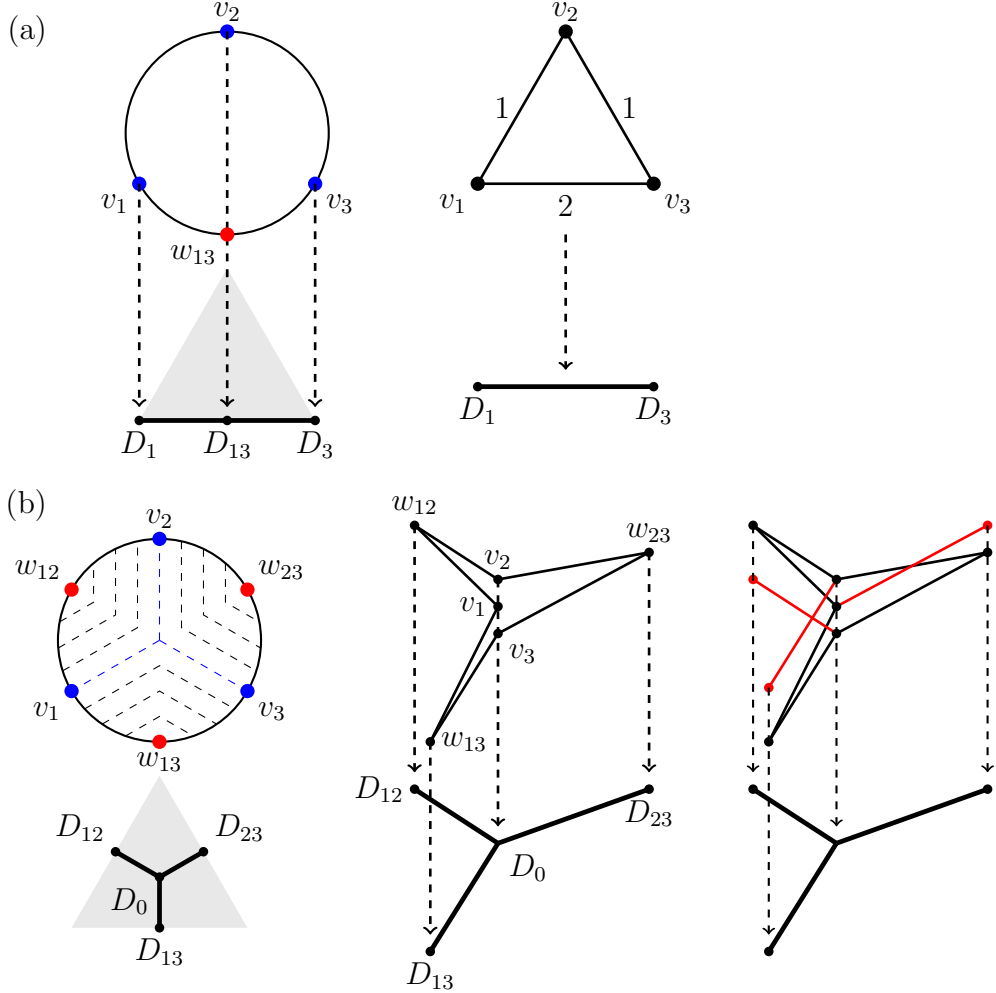

\begin{example} \label{E:RedMap}
As discussed in Example~\ref{E:Red}(1), we see that $\Red_{|D_0|}(v_1)=D_1$, $\Red_{|D_0|}(v_2)=D_2$ and $\Red_{|D_0|}(v_3)=D_3$. 
Actually $\Red_{|D_0|}$ is an embedding of $\Gamma$ into $|D_0|$ where the image of $\Red_{|D_0|}$ is the circumference of $|D_0|$. Now consider the reduced divisor map to the tropical segment $[D_1,D_3]$. In Example~\ref{E:Red}(2), we've shown that $\Red_{[D_1,D_3]}(v_1)=D_1$, $\Red_{[D_1,D_3]}(v_3)=D_3$ and $\Red_{[D_1,D_3]}(v_2)=\Red_{[D_1,D_3]}(w_{13})=D_{13}$. Note that from Example~\ref{E:LinearSys}, we know that in the chip-firing move from $D_1$ to $D_3$, two chips move from $v_1$ to $v_3$ along the segment $v_1w_{13}v_3$ at the speed of one unit and one chip moves from $v_1$ to $v_3$ along the segment $v_1v_2v_3$ at the speed of two units. 
Thus we may write the $P_{D_3-D_1}(t) = 2\cdot (x)+(y)$ where $x$ is the point on the segment  $v_1w_{13}v_3$ of distance  $t/2$ from $v_1$ and $x$ is the point on the segment  $v_1v_2v_3$ of distance  $t$ from $v_1$ for $t\in[0,\rho(D_1,D_3)]$. One can verify that $\Red_{[D_1,D_3]}(x)=\Red_{[D_1,D_3]}(y)=P_{D_3-D_1}(t)$. This reduced divisor map to $[D_1,D_3]$ is also illustrated in Figure~\ref{F:RedHarm}(a). Now consider $T=\tconv(\{D_{12},D_{23},D_{13}\})$. In Example~\ref{E:Red}(3), we've shown that $\Red_T(v_1)=\Red_T(v_2)=\Red_T(v_3)=D_0$, $\Red_T(w_{12})=D_{12}$, $\Red_T(w_{23})=D_{23}$ and $\Red_T(w_{13})=D_{13}$. Actually, $T$ can be realized by gluing  $v_1w_{12}$ with $v_2w_{12}$, $v_2w_{23}$ with $v_3w_{23}$, and $v_1w_{13}$ with $v_3w_{13}$, as illustrated in Figure~\ref{F:RedHarm}(b). More precisely, the divisors on the tropical path from $D_{12}$ to $D_0$ can be written as $P_{D_{12}-D_0}(t)=(x)+(y)+(v_3)$ where $x$ and $y$ are points on the segments $w_{12}v_1$ and $w_{12}v_2$ respectively, both of distance $t$ to $w_{12}$ where $t\in[0,\rho(D_{12},D_0)]$. Then $\Red_T(x)=\Red_T(y)=P_{D_{12}-D_0}(t)$. Similarly, we can derive the reduced divisor map restricted to the segment $v_2w_{23}v_3$ whose image is $[D_{23},D_0]$ and restricted to the segment $v_1w_{13}v_3$ whose image is $[D_{13},D_0]$. 
\end{example}

Now let us study one-dimensional linear systems and the corresponding reduced divisor maps. For a  linear system $T$, let $\supp(T)=\bigcup_{D\in T}\supp(D)$. We have the following notions. 

\begin{definition} \label{D:TropTree}
Let $T$ be a linear system on a metric graph $\Gamma$ generated by $D_1,\cdots,D_m$. We say that $T$ is a \emph{tropical tree} if $T=\bigcup_{i,j=1,\cdots,m}[D_i,D_j]$. A tropical tree $T$ is called \emph{dominant} if $\supp(T)=\Gamma$. For two tropical trees $T$ and $T'$, we say $T'$ is a tropical \emph{subtree} of $T$ if $T'\subseteq T$. We say a tropical tree $T$ is \emph{maximal} if the only tropical tree containing $T$ is $T$ itself.
\end{definition}

Clearly, when not being a singleton, since the intersection of two tropical segments is a tropical segment (Proposition~\ref{P:TropSeg}(6)),  a tropical tree $T=\tconv(\{D_1,\cdots,D_m\})$ actually has a tree structure and the extremals of $T$ are the leaves of $T$. Moreover, if $T'$ is a tropical subtree of $T$, denote the set of connected components of  $T\setminus T'$ by $\mcalU(T\setminus T')$. Note that the closure $\cl(U)$ of each $U\in \mcalU(T\setminus T')$  is a also a tropical subtree of $T$ and  the intersection of $T'$ and $\cl(U)$ is a single point  which we call the attaching point  between $T'$ and $\cl(U)$. Then following lemma says that the tropical projection from $T$ to $T'$ respects the natural retraction from $T$ to $T'$. 

\begin{lemma} \label{L:ConnCompRed}
Let $T'$ be a tropical subtree of a tropical tree $T$. For $U\in \mcalU(T\setminus T')$, let $D_0$ be the attaching point between $T'$ and $\cl(U)$. Then $\pi_{T'}(D)=D_0$ for all $D\in \cl(U)$. 
\end{lemma}
\begin{proof}
This is a corollary of Proposition~\ref{P:TropProj} (3) and (4). Actually for $D\in \cl(U)$, $D_0$ must be contained in the tropical segment $[D,\pi_{T'}(D)]$ since $T$ is a tropical tree. Now by Proposition~\ref{P:TropProj} (3) and (4), $\pi_{T'}(D)=\pi_{T'}(D_0)=D_0$. 
\end{proof}

The following proposition says that tropical trees  can be tested locally. 

\begin{proposition} \label{P:TropTreeCrit}
A linear system $T$ is a tropical tree if and only if for each $D\in T$ and the bases $B_1,\cdots,B_s$ of all possible effective chip-firing moves inside $T$ from $D$, we have $B_i\bigcup B_j=\Gamma$ for all distinct $B_i,B_j\in\{B_1,\cdots,B_s\}$. 
\end{proposition}
\begin{proof}
Suppose $T$ is a tropical tree generated by linear equivalent divisors $D_1,\cdots,D_m$ of degree $d$. We may assume that $D_1,\cdots,D_m$ are all extremals (or leaves) of $T$. Consider an arbitrary divisor $D\in T$.

 If $D$ is an extremal of $T$, then we may assume $D=D_1$ without loss of generality. In this way, $\bigcap_{i=2}^m [D,D_i]$ is a tropical segment $[D,D']$ for some $D'\in T$ which is distinct from $D$ (Proposition~\ref{P:TropSeg}(6)). This actually means that from $D$, we can only have exactly one effective chip-firing base $B_1=\Gmin(f_{D'-D})$. 

Now suppose that $D$ is not an extremal of $T$. Let $D'_1,\cdots,D'_s$ be all the distinct divisors close enough to $D$ such that $f_{D'_i-D}$ is a primitive rational function for $i=1,\cdots,s$. Note that $s$ is at least two since $D$ is not a leaf of $T$. Let $B_i=\Gmin(f_{D'_i-D})$ for $i=1,\cdots,s$. Then $B_1,\cdots,B_s$ are the bases of all possible effective chip-firing moves inside $T$ from $D$. Note that $B_i\neq \Gamma$. For each distinct $i,j=1,\cdots,s$, we know that $D\in[D'_i,D'_j]$ which means $B_i\bigcup B_j=\Gmin(f_{D'_i-D})\bigcup \Gmin(f_{D'_j-D})=\Gamma$ by Theorem~\ref{T:FiniteCriterion}. 

Conversely, suppose $D$ is a divisor in a linear system $T$ generated by the extremals $D_1,\cdots,D_m$ such that for the bases $B_1,\cdots,B_s$ of all possible effective chip-firing moves inside $T$ from $D$, we have $B_i\bigcup B_j=\Gamma$ for all distinct $B_i,B_j\in\{B_1,\cdots,B_s\}$. Claim that $D\in [D_i,D_j]$ for some $D_i,D_j\in \{D_1,\cdots,D_m\}$. This will imply that $T$ is a tropical tree. 

Let $B'_i=\Gmin(f_{D_i-D})$ for $i=1,\cdots,m$. If there is some $B'_i$ identical to $\Gamma$, then $D$ must be the extremal $D_i$ of $T$. 
Now suppose $D$ is not an extremal of $T$ which means $B'_i\neq \Gamma$ for all $i=1,\cdots,m$. Then it is clear that $\{B'_1,\cdots,B'_m\}\subseteq \{B_1,\cdots,B_s\}$. Since $D\in T=\tconv(\{D_1,\cdots,D_m\})$, we must have $\bigcup_{i=1}^m B'_i = \Gamma$ by Theorem~\ref{T:FiniteCriterion}. This actually means that $B'_1,\cdots,B'_m$ can not be all identical since they are proper subsets of $\Gamma$. Without loss of generality, we may assume $B'_1=B_1$ is distinct from $B'_2=B_2$. Then by assumption, $B'_1\bigcup B'_2 = B_1\bigcup B_2 =\Gamma$.  Therefore, by applying Theorem~\ref{T:FiniteCriterion} again, we conclude that $D\in[D_1,D_2]$. 
\end{proof}

\begin{remark} \label{R:TropTree}
For a linear system $T$ and a divisor $D\in T$, let $C_1,\cdots, C_n$ be the connected components of $\Gamma\setminus \supp(D)$ and $B_1,\cdots,B_s$ be the bases of all possible effective chip-firing moves inside $T$ from $D$. Then a chip firing of base $B_i$ will move some chips on the boundary of $B_i$ into  the complement $B_i^c$ of $B_i$. Note that  $B_i^c$ must contain  some $\mcalC_i=C_{i_1}\bigcup \cdots \bigcup C_{i_{n_i}}$ as a dense subset.  The local criterion for $D\in T$ with $T$ being a  tropical tree $T$ (Proposition~\ref{P:TropTreeCrit})  is equivalent to saying that when $D$ is not an extremal of $T$, $s$ must be at least $2$ and for each distinct $\mcalC_i=C_{i_1}\bigcup \cdots \bigcup C_{i_{n_i}}$ and $\mcalC_j=C_{j_1}\bigcup \cdots \bigcup C_{j_{n_j}}$, we must have $\{C_{i_1}, \cdots , C_{i_{n_i}}\}\bigcap\{C_{j_1},\cdots,C_{j_{n_j}}\}=\emptyset$. This also implies that if $T$ is a tropical tree and $D\in T$, then there is a one-to-one correspondence among the following objects:
\begin{enumerate}
\item  the bases of all possible effective chip-firing moves inside $T$ from $D$, 
\item the directions of all possible effective chip-firing moves inside $T$ from $D$,  
\item  all the outgoing tangent directions from $D$ in $T$, and 
\item the connected components of $T\setminus\{D\}$.  
\end{enumerate}
For future discussions, we denote the set of bases of all possible effective chip-firing moves inside $T$ from $D$ by $\mcalB_T(D)$, the set of outgoing tangent directions from $D$ in $T$ by $\Tan_T(D)$, and the set of components of $T\setminus \{D\}$ by $\mcalU_T(D)$. Then we have one-to-one correspondence among elements in $\mcalB_T(D)$, $\Tan_T(D)$ and $\mcalU_T(D)$. Moreover, when talking about chip-firing directions,  given an outgoing tangent direction $\vt$ from $D$ in $T$, we may also say that $D$ takes an effective chip-firing move along $\vt$. 
\end{remark}

\begin{example} \label{E:TropTree}
Again, we consider the linear systems on the metric circle $\Gamma$ in Figure~\ref{F:LinearSys}. Let $T=\tconv(\{D_{12},D_{23},D_{13}\})$ and $T'=\tconv(\{D_{12},D_2,D_{13}\})$. One can easily verify that $T$ is a tropical tree while $T'$ is not by Definition~\ref{D:TropTree}.  Note $D_0=(v_1)+(v_2)+(v_3)$ is a divisor in both $T$ and $T'$. There are three connected components $C_1$, $C_2$ and $C_3$ of $\Gamma\setminus \supp(D_0)$ where $C_1$ is the open segment between $v_2$ and $v_3$ through $w_{23}$, $C_2$ is the open segment between $v_1$ and $v_3$ through $w_{13}$, and $C_3$ is the open segment between $v_1$ and $v_2$ through $w_{12}$.  Let us test $D_0$ locally based on Proposition~\ref{P:TropTreeCrit} and Remark~\ref{R:TropTree}.
\begin{enumerate}
\item There are three directions of effective chip-firing moves from $D_0$ allowed in $T$: (1) the chips at $v_2$ and $v_3$ move towards $w_{23}$ at the same speed whose base is $B_1=C_2\bigcup C_3\bigcup \{v_1,v_2,v_3\}$; (2) the chips at $v_1$ and $v_3$ move towards $w_{13}$ at the same speed whose base is $B_2=C_1\bigcup C_3\bigcup \{v_1,v_2,v_3\}$; and (3) the chips at $v_1$ and $v_2$ move towards $w_{12}$ at the same speed whose base is $B_3=C_1\bigcup C_2\bigcup \{v_1,v_2,v_3\}$. Clearly $B_i^c=C_i$ for $i=1,2,3$.  Then $B_1\bigcup B_2=B_2\bigcup B_3=B_1\bigcup B_3=\Gamma$ or equivalently $C_1\bigcap C_2=C_2\bigcap C_3=C_1\bigcap C_3=\emptyset$. Therefore $D_0$ satisfies the local tropical tree condition for $T_1$. 
\item There are three directions of effective chip-firing moves from $D_0$ allowed in $T'$: (1) the chips at $v_1$ and $v_3$ move towards $w_{13}$ at the same speed whose base is $B'_1=C_1\bigcup C_3\bigcup \{v_1,v_2,v_3\}$; (2) the chips at $v_1$ and $v_3$ move towards $v_2$ at the same speed whose base is $B'_2=C_2\bigcup \{v_1,v_3\}$;  (3) the chips at $v_1$ and $v_2$ move towards $w_{12}$ at the same speed whose base is $B'_3=C_1\bigcup C_2\bigcup \{v_1,v_2,v_3\}$. Note that ${B'_1}^c=C_2$, ${B'_2}^c=C_1\bigcup C_3\bigcup \{v_2\}$, and ${B'_3}^c=C_3$. Now $B'_2\bigcup B'_3=B'_3\neq \Gamma$ or equivalently $\{C_1,C_3\}\bigcap \{C_3\}=\{C_3\}\neq\emptyset$. Therefore, $D_0$ does not satisfy the local tropical tree condition for $T'$ which means $T'$ cannot be a tropical tree. 
\end{enumerate}
\end{example}

\begin{lemma} \label{L:TropTreeSurj}
For any tropical tree $T$, the corresponding reduced divisor map $\Red_T$ is surjective and the preimage of $D\in T$ under $\Red_T$ is $\bigcap_{B\in \mcalB_T(D)}B$. 
\end{lemma}
\begin{proof}
First we note that for each $p$ and $q$ in $\Gamma$ and any path $P$ connecting $p$ and $q$, the tropical segment $[\Red_T(p),\Red_T(q)]$ must be a subset of the image of $P$ under $\Red_T$ since $\Red_T$ is continuous. So it remains to show that the extremals $D_1,\cdots,D_m$ of $T$ are contained in the image of $\Red_T$. Consider an extremal $D_i$ of $T$. Note that there is exactly one base $B$ for all possible effective chip-firing moves inside $T$ from $D_i$. By Proposition~\ref{P:GeneralRed}, this means that $D_i$ is reduced with respect to all the points in $B$. Also by Proposition~\ref{P:GeneralRed}, in general we have $\Red_T^{-1}(D)=\bigcap_{B\in \mcalB_T(D)}B$ for all $D\in T$. 
\end{proof}

\begin{corollary} \label{C:RedFunc}
Let $T$ be a tropical tree. For each $D_1,D_2\in T$, let $f$ be an associated function of $D_2-D_1$, i.e., $\divf(f)=D_2-D_1$. Consider a divisor $D=P_{D_2-D_1}(t)$ for some $t\in[0,\rho(D_1,D_2)]$. 
\begin{enumerate}
\item $\Red_{[D_1,D_2]}^{-1}(D)=\underline{f}^{-1}(t)$, i.e.,  $\Red_{[D_1,D_2]}$ is exactly $P_{D_2-D_1} \circ \underline{f}$. 
\item  If  $\underline{f}^{-1}(t)$ is a finite set, then $\Red_T^{-1}(D)=\underline{f}^{-1}(t)$. 
\end{enumerate}
\end{corollary}
\begin{proof}
Note that for $D$ in the interior of the tropical segment  $[D_1,D_2]$, there are  two directions of effective chip-firing moves from $D$ inside $[D_1,D_2]$, one towards $D_1$ with base $\{p\in\Gamma\mid \underline{f}(p)\geq t\}$ and the other towards $D_2$ with base $\{p\in\Gamma\mid \underline{f}(p)\leq t\}$. Then $\Red_{[D_1,D_2]}^{-1}=\bigcap_{B\in \mcalB_T(D)}B=\{p\in\Gamma\mid \underline{f}(p)\geq t\} \bigcap \{p\in\Gamma\mid \underline{f}(p)\leq t\}=\underline{f}^{-1}(t)$. The same result for $D$ being $D_1$ or $D_2$ follows from analogous arguments while instead only one direction of effective chip-firing moves from $D$ is allowed. 

Now if $\underline{f}^{-1}(t)$ is a finite set, we can extend the result even to $\Red_T^{-1}(D)$. The reason is that the finiteness of $\underline{f}^{-1}(t)$ can guarantee that there are no more directions of effective chip-firing moves from $D$ inside $T$ than the directions only inside $[D_1,D_2]$. 
\end{proof}

\begin{remark}
For Corollary~\ref{C:RedFunc}(2), we note that the level sets $\underline{f}^{-1}(t)$ of $\underline{f}$  is generically finite, i.e., the set of  $t$ such that  $\underline{f}^{-1}(t)$ is infinite  is a finite set. 
\end{remark}

\begin{lemma} \label{L:RedComp}
For each $q\in \Gamma$ and  $U\in \mcalU_T(\Red_T(q))$, all divisors in $U$ take the same value at $q$, which is at most the value of $\Red_T(q)$ at $q$. 
\end{lemma}
\begin{proof}
Let $D_0=\Red_T(q)$. By Corollary~\ref{C:RedVal}, we know that the value of all divisors in $T$ is at most the value of $D_0$ at $q$. Now consider any two divisors $D_1$ and $D_2$ in a component $U$ of  $T\setminus \{D\}$. Let $B$ be the base of effective chip-firing moves inside $T$ from $D$ corresponding to $U$ (Remark~\ref{R:TropTree}). Then $B=\Gmin(f_{D_2-D_0})=\Gmin(f_{D_1-D_0})$ and $q$ belongs to $B$ by Proposition~\ref{P:GeneralRed}. Since $T$ is a tropical tree, we may assume that $D_1$ is in the interior of the tropical segment $[D_0,D_2]$.  Let $B' = \Gmin(f_{D_2-D_1}) $. Then $B$ is a proper subset of $\Gmax(f_{D_1-D_0})^c$, $B'$ is the closure of $\Gmax(f_{D_1-D_0})^c$, and we conclude that $q$ belongs to $B'\setminus \partial B'$. Thus the value of $D_2-D_1$ is $0$ at $q$. 
\end{proof}

\begin{remark} \label{R:tauq}
By Lemma~\ref{L:RedComp}, we can define a function $\tau_q$ on $\mcalU_T(\Red_T(q))$ such that for each $U\in\mcalU_T(\Red_T(q))$,  $\tau_q(U)$ is the value of a divisor on $U$ at $q$.  We say $\tau_q$ is trivial if $\tau_q(U)=0$ for all $U\in \mcalU_T(\Red_T(q))$.

Moreover, let $T_q:=\{D\in T\mid q\in\supp(D)\}$. We have the following cases for $T_q$:
\begin{enumerate}
\item If $q\notin \supp(\Red_T(q))$, then $T_q=\emptyset$. 
\item If $q\in \supp(\Red_T(q))$, then $T_q=\{q\}\bigcup(\bigcup_{U\in \mcalU_T(\Red_T(q)),\tau_q(U)>0} U)$. In particular, when in addition $\tau_q$ is trivial, $T_q$ is exactly the singleton $\{q\}$. Moreover, there are only finitely many points $q\in\Gamma$ with nontrivial $\tau_q$. This can be proved by the following arguments. Let $Q$ be the set of points $q$ such that $\tau_q$ is nontrivial. Consider the extremals $D_1,\cdots,D_m$ of $T$. Then for each $q\in Q$, $T_q\bigcap \{D_1,\cdots,D_m\}$ must be nonempty. Then if $Q$ is an infinite set, then there must be some $D_i\in \{D_1,\cdots,D_m\}$ such that there are infinitely many points $q\in Q$ with $D_i\in T_q$. But this means that $\supp(D_i)$ is an infinite set, which is impossible. 
\end{enumerate}
\end{remark}

The following proposition provides several criteria for a tropical tree being dominant. 

\begin{proposition} \label{P:DomTreeCrit}
Let $T$ be a tropical tree. Then the following are equivalent:
\begin{enumerate}
\item $T$ is dominant. 
\item For each $p\in\Gamma$, $p\in\supp(\Red_T(p))$. 
\item The corresponding reduced divisor map $\Red_T$ is finite., i.e.,  the preimage of any divisor in $T$ is a finite set.
\item (Local criterion) For each $D\in T$, we have $\supp(D)\bigcup(\bigcup_{B\in\mcalB_T(D)} B^c)=\Gamma$ where $B^c$ is the complement of $B$ in $\Gamma$. 
\end{enumerate}
\end{proposition}

\begin{proof}
(1)$\Leftrightarrow$(2): Suppose $T$ is dominant, i.e., for each $p\in \Gamma$, there exists a divisor $D\in T$ such that $p\in\supp(D)$. By Corollary~\ref{C:RedVal}, we know that the value of $\Red_T(p)$ at $p$ is at least the value of $D$ at $p$, which means $p\in\supp(\Red_T(p))$. 

(2)$\Leftrightarrow$(3): If $p\in\supp(\Red_T(p))$ for each $p\in\Gamma$, then clearly the preimage of each $D\in T$ under $\Red_T$ must be finite since it is a subset of $\supp(D)$. Conversely,  if there exists  a point $q\in\Gamma$ such that $q\notin \supp(\Red_T(p))$, then we claim that the preimage of $\Red_T(q)$ under the reduced divisor map $\Red_T$ is an infinite set. Let $D_0=\Red_T(q)$. By Lemma~\ref{L:TropTreeSurj}, we know that $\Red_T^{-1}(D_0)=\bigcap_{B\in \mcalB_T(D_0)}B\neq \emptyset$. Note that the boundary points of each $B\in \mcalB_T(D_0)$ must be contained in $\supp(D_0)$. Now $q$ is contained in $\Red_T^{-1}(D_0)$ but $q\notin \supp(D_0)$. Let $C$ be the connected component of $\Gamma\setminus\supp(D_0)$ which contains $q$. Then we must have $C\subseteq \Red_T^{-1}(D_0)$. 

(2)$\Leftrightarrow$(4): Recall that  $\Red_T^{-1}(D)=\bigcap_{B\in\mcalB_T(D)}B$ by Lemma~\ref{L:TropTreeSurj}. Therefore $\Red_T^{-1}(D)\subseteq \supp(D)$ is equivalent to $\supp(D)\bigcup(\bigcup_{B\in\mcalB_T(D)} B^c)=\supp(D)\bigcup(\bigcap_{B\in\mcalB_T(D)}B)^c=\Gamma$.

\end{proof}

Lemma~\ref{L:TropTreeSurj} says that the reduced divisor map to a tropical tree $T$ is surjective and provides a characterization of the inverse image of the reduced divisor map. By Proposition~\ref{P:DomTreeCrit}, we know that if in addition $T$ is dominant, then the preimage of each $D\in T$ under the reduced divisor map must be a subset of $\supp(D)$. The following proposition  provides a criterion to characterize the preimage of $D\in T$ under $\Red_T$  more easily. 

\begin{proposition} \label{P:PreImagRed}
For a dominant tropical tree $T$ and a divisor $D\in T$, the following are equivalent:
\begin{enumerate}
\item $D=\Red_T(q)$.
\item $q\in\supp(D)$ and $q$ is a boundary point of some $B\in\mcalB_T(D)$. 
\item  $q\in\supp(D)$ and there exists at least one outgoing tangent direction $\vt\in\Tan_T(D)$ such that as $D$ takes an effective chip-firing move along  $\vt$, at least one chip at $q$ moves.
\end{enumerate}
\end{proposition}
\begin{proof}
The equivalence of (2) and (3) follows from the one-to-one correspondence of $\mcalB_T(D)$ and $\Tan_T(D)$ (Remark~\ref{R:TropTree}) and the fact that an effective chip-firing move along  $\vt$ moves chips on the boundary of the base of that chip-firing move. It remains to prove that $\Red^{-1}(D)=\bigcup_{B\in\mcalB_T(D)}\partial B\subseteq \supp(D)$ where $\partial B$ is the set of boundary points of $B$ in $\Gamma$.

By Lemma~\ref{L:TropTreeSurj}, we have $\Red_T^{-1}(D)=\bigcap_{B\in\mcalB_T(D)}B$. Then $\Red_T^{-1}(D) \subseteq \supp(D)$ follows from Proposition~\ref{P:DomTreeCrit} directly. Note that this also means that $\Red_T^{-1}(D)$ is finite. Now by Proposition~\ref{P:TropTreeCrit}, for each $B\in \mcalB_T(D)$, we must have $B^c\subseteq B'$ for all $B'\in \mcalB_T(D)$ such that $B'\neq B$. This actually means that $\partial B\subseteq B'$ since $B'$ is closed in $\Gamma$. Therefore, $\bigcup_{B\in\mcalB_T(D)}\partial B\subseteq \bigcap_{B\in\mcalB_T(D)}B=\Red_T^{-1}(D)$. Now let $A=\bigcup_{B\in\mcalB_T(D)}\partial B$ which is a finite set and suppose that there exists $q\in \Red_T^{-1}(D)\setminus A$.  Let $C$ be the connected component of $\Gamma\setminus A$ which contains $q$. Then we must have $C\subseteq \Red_T^{-1}(D)$ since $q\in C\subseteq B$ for all $B\in\mcalB_T(D)$. This contradicts the fact that $\Red_T^{-1}(D)$ is finite.
\end{proof}

\begin{corollary}
Dominant tropical trees are maximal. 
\end{corollary}
\begin{proof}
Consider a dominant tropical tree $T$. Then by Lemma~\ref{L:TropTreeSurj}, we know that the reduced divisor map $\Red_T:\Gamma\to T$ is surjective. If $T$ is not maximal, then there exists a larger tropical tree $T'\supseteq T$ such that $T'\setminus T$ is nonempty. Consider a connected component $U$ of $T'\setminus T$. First we note that $\Red_{T'}^{-1}(U)$ must be an infinite subset of $\Gamma$ (this can be shown in various ways, for example as a simple consequence of Corollary~\ref{C:RedFunc} by considering a tropical segment contained in $U$). By Lemma~\ref{L:ConnCompRed} and Proposition~\ref{P:SeqRedjDiv}, this means that for all $p\in \Red_{T'}^{-1}(U)$,  $\Red_T(p)=\pi_T(\Red_{T'}(p))=D_0$ with $D_0$ being the attaching point between $T'$ and $\cl(U)$. However, by Proposition~\ref{P:DomTreeCrit}, it follows that $\Red_{T'}^{-1}(U)\subseteq \supp(D_0)$ which is impossible since $\Red_{T'}^{-1}(U)$ is an infinite set. Therefore, $T$ must be maximal. 
\end{proof}

We have the following characterization of the reduced divisor map to a dominant tropical tree as a whole. 

\begin{proposition} \label{P:CritRedDivMap}
Let $\omega$ be a continuous map from $\Gamma$ to a complete linear system. Let $T$ be the image of $\omega$. Let $T_p:=\{D\in T\mid p\in \supp(D)\}$. The following are equivalent.
\begin{enumerate}
\item $T$ is a dominant tropical tree and $\omega$ is the reduced divisor map to $T$.
\item $T$ is a linear system and for each $p\in\Gamma$, we have $\omega(p)\in T_p$. 
\item $T$ is a linear system and there exists a dense subset $\Gamma_0$ of $\Gamma$ such that for each $p\in\Gamma_0$, $T_p$ is exactly the singleton $\{\omega(p)\}$. 
\end{enumerate}
\end{proposition}
\begin{proof}
(1)$\Rightarrow$(2): This follows from Proposition~\ref{P:DomTreeCrit} directly. 

(2)$\Rightarrow$(3): Since the image $T$ of the continuous map $\omega$  is a linear system, it can be  verified easily that $T$ must be a tropical tree (using  Definition~\ref{D:TropTree} directly). Let $\Gamma_0:=\{p\in\Gamma\mid T_p=\{\omega(p)\}\}$. Then we want to show that $\Gamma_0$ is dense in $\Gamma$. Suppose for contradiction that $\Gamma_0$ is not dense in $\Gamma$, which means that there exists $q\in \Gamma\setminus \Gamma_0$ and a small enough closed neighborhood $N_q=\{p\in\Gamma\mid \dist(p,q)\leq\epsilon\}\subseteq \Gamma\setminus \Gamma_0$ with $\epsilon>0$. Here we suppose $\epsilon$ is small enough that $N_q$ is star-shaped with center $q$. Then $\omega(N_q)=\bigcup_{i=1}^k[\omega(q),E_i]$ where $E_i$'s are some divisors close enough to $\omega(q)$ in $T$. Note that $\omega(N_p)$ is tropically convex by this construction and cannot be the singleton $\{\omega(q)\}$. Otherwise the support of $\omega(q)$ is an infinite set which is impossible. 

First we claim that for each $p\in \Gamma$,  $T_p$ is tropically convex. For each $D_1,D_2\in T_p$, we have $D_1(p),D_2(p)\geq 1$. Note that the divisors on the tropical path from $D_1$ to $D_2$ are $P_{D_2-D_1}(t)=\divf([\min(t,\underline{f_{D_2-D_1}})])+D_1$ for $t\in[0,\rho(D_1,D_2)]$. Then the value of $P_{D_2-D_1}(t)$ at $p$ is either $D_1(p)$ or $D_2(p)$ for $t\in[0,\rho(D_1,D_2)]$, which means that $P_{D_2-D_1}(t) \in T_p$. 

Now let us construct a sequence  of divisors in $T$. We choose a divisor $D_1$ from the interior of the tropical segment $[\omega(q),E_1]$ (with $\omega(q)\neq E_1$). Then there must be a point $q_1\in N_q$ such that $D_1=\omega(q_1)$. This means that $T_{q_1}$ is not a singleton. Since $T_{q_1}$ is tropically convex, the set $T_{q_1}\bigcap [\omega(q),E_1]$ is a tropical segment  which contains $D_1$ but is not the singleton $\{D_1\}$. Let $[D_1,F_1]$ with $D_1\neq F_1$ be a tropical segment contained in $T_{q_1}\bigcap [\omega(q),E_1]$. Choose a divisor $D_2$ from the interior of $[D_1,F_1]$. Again, there is a point $q_2\in N_q$ such that $D_2=\omega(q_2)$. Clearly $q_1\neq q_2$ since $D_1\neq D_2$. But we note that $\{q_1,q_2\}\subseteq \supp(D_2)$ since $D_2\in T_{q_1}$ and $D_2=\omega(q_2)$. Moreover, we can keep doing this process and construct a sequence of divisors $D_1,D_2,\cdots$:
\begin{enumerate}
\item  Suppose we have already derived the tropical segment $[D_i,F_i]\in T_{q_i}\bigcap [D_{i-1},F_{i-1}]$. 
\item Choose a divisor $D_{i+1}$ from the interior of $[D_i,F_i]$ and find a point $q_{i+1}\in N_q$ such that $D_{i+1}=\omega(q_{i+1})$. 
\item The set  $T_{q_{i+1}}\bigcap [D_i,F_i]$ is a tropical segment  which contains $D_{i+1}$ but is not the singleton $\{D_{i+1}\}$. Let $[D_{i+1},F_{i+1}]$ with $D_{i+1}\neq F_{i+1}$ be a tropical segment contained in $T_{q_{i+1}}\bigcap [D_i,F_i]$. 
\item Let $i\leftarrow i+1$ and go to (1). 
\end{enumerate}

In this way, we derive a nested sequence of tropical segments $[D_1,F_1]\supseteq[D_2,F_2]\supseteq \cdots$ where $[D_i,F_i]\subseteq T_{q_i}$. Note that all $q_i$'s are distinct since $D_i$'s are distinct. Then we conclude that $\{q_1,\cdots,q_i\}\subseteq \supp(D_i)$. But this is impossible since $\supp(D_i)$ must be a finite set. Therefore, $\Gamma_0$ must be a dense subset of $\Gamma$. 

(3)$\Rightarrow$(1): For $p\in\Gamma_0$, we must have $\omega(p)=\Red_T(p)$ since $\omega(p)$ is the only divisor in $T$ with $p$ belonging to $\supp(\omega(p))$ and the value of $\Red_T(p)$ at the point $p$ is the largest among the values of all divisors in $T$ at the point $p$ (Corollary~\ref{C:RedVal}). Now since $\omega$ and $\Red_T$ are both continuous maps which coincide restricted to the dense subset $\Gamma_0$ of $\Gamma$ and , we conclude that $\omega=\Red_T$. By the continuity of $\Red_T$ and the equivalence of metric topology and symmetric product topology of $T$ (Appendix~\ref{S:EquiTop}), $p\in\supp(\Red_T)$ is true for all $p\in \Gamma$, which means that $T$ is a dominant tropical tree by  Proposition~\ref{P:DomTreeCrit}. 
\end{proof}

\begin{example}
Let us reconsider Example~\ref{E:TropTree} about the tropical tree $T=\tconv(\{D_{12},D_{23},D_{13}\})$ shown in Figure~\ref{F:LinearSys}. We note that $T$ is a dominant tropical tree (which can be verified in several ways, e.g., using Definition~\ref{D:TropTree} directly, using Proposition~\ref{P:DomTreeCrit} or Proposition~\ref{P:CritRedDivMap}). Now specifically let us consider the divisor $D_0\in T$. In Example~\ref{E:TropTree},  we have verified that $D_0$ satisfies the condition in the local criterion for tropical trees (Proposition~\ref{P:TropTreeCrit}). Here we will verify that $D_0$ satisfies the condition in the  local criterion for dominant tropical trees  (Proposition~\ref{P:DomTreeCrit}(4)). As shown in Example~\ref{E:TropTree}, there are three possible effective chip-firing directions from $D_0$ allowed in $T$: along $D_0D_{23}$ with base $B_1$ being the segment $v_2v_1v_3$, along $D_0D_{13}$ with the base $B_2$ being the segment $v_1v_2v_3$, and along $D_0D_{12}$ with the base $B_3$ being the segment $v_1v_3v_2$. Therefore, the local condition $\supp(D)\bigcup(\bigcup_{B\in\mcalB_T(D)} B^c)=\Gamma$ in Proposition~\ref{P:DomTreeCrit}(4) is satisfied. Furthermore, by Proposition~\ref{P:PreImagRed}, we get $\Red_T^{-1}(D_0)=\{v_1,v_2,v_3\}$. 
\end{example}

\subsection{Harmonic Morphisms to Trees}
For a metric graph $\Gamma$ and a point $p\in\Gamma$, we denote the set of all outgoing tangent directions from $p$ in $\Gamma$ by $\Tan_\Gamma(p)$. 

\begin{definition} \label{D:Harmonic}
Let $\Gamma$ and $\Gamma'$ be two metric graphs. 
\begin{enumerate}
\item A \emph{pseudo-harmonic morphism} $\phi$ from  $\Gamma$ to $\Gamma'$  is a continuous finite surjective piecewise-linear map with nonzero integral slopes. (Here we say $\phi$ is finite if $\phi^{-1}(y)$ is finite for all $y\in\Gamma'$.) In particular, for all points $p\in\Gamma$ and tangent directions $\vt\in \Tan_\Gamma(p)$, the \emph{expansion factor} $d_\vt(\phi)$ along $\vt$ is the absolute value of the integral slope of $\phi$ along $\vt$, i.e., the ratio of the distance between $\phi(x)$ and $\phi(y)$ in $\Gamma'$ over the distance between $x$ and $y$ in $\Gamma$ where $x$ and $y$ are points close enough to $p$ in the direction $\vt$.
\item  We say a pseudo-harmonic morphism $\phi:\Gamma\to\Gamma'$ is \emph{harmonic} at a point $p\in\Gamma$ if $\phi$  satisfies the following \emph{balancing} condition:
 for any tangent direction $\vt' \in \Tan_{\Gamma'}(\phi(p))$, the  sum of the expansion factors $d_\vt(\phi)$ over all tangent directions $\vt$ in $\Tan_\Gamma(p)$ that map to $\vt'$,  i.e., the integer $$\sum_{\vt \in \Tan_\Gamma(p),~\vt \mapsto \vt'}d_\vt(\phi),$$ is independent of $\vt'$ and is called the \emph{degree} of $\phi$ at $p$, denoted by $\deg_p(\phi)$.
\item A pseudo-harmonic morphism $\phi:\Gamma\to\Gamma'$ is a \emph{harmonic morphism} if $\phi$ is harmonic at all $p\in\Gamma$.
\item For a harmonic morphism $\phi:\Gamma\to\Gamma'$, we define the \emph{degree} of $\phi$ to be $$\deg(\phi):=\sum_{p \in \phi^{-1} (q)}\deg_p(\phi),$$ which can be shown to be independent of $q\in \Gamma'$. 
\end{enumerate}
\end{definition}

\begin{remark}
\begin{enumerate}
\item Harmonic morphisms and pseudo-harminic morphisms can also be defined to metrized complexes (metric graphs with certain points assoicated with algebraic curves) \cite{ABBR15,ABBR15_2,LM18}.
\item Recall that the genus of a metric graph $\Gamma$ is the first Betti number of $\Gamma$. For example, a metric tree has genus $0$ and a metric circle has genus $1$. We say that a metric graph $\Gamma^{\Mod}$ is a  modification of $\Gamma$ if $\Gamma$ is isometric to a subgraph of $\Gamma^{\Mod}$ and the genus of $\Gamma^{\Mod}$ is the same as the genus of $\Gamma$. This actually means that $\Gamma^{\Mod}$ can be realized by attaching metric trees to $\Gamma$ as extra ``branches''. For simplicity, we will treat $\Gamma$ just as a subgraph of $\Gamma^{\Mod}$.
\item For a modification $\Gamma^{\Mod}$ of $\Gamma$, there is a natural retraction map $\gamma:\Gamma^{\Mod}\to\Gamma$. By abuse of notation, for each divisor $D=\sum_{p\in\Gamma^{\Mod}}m_p\cdot(p)$ on $\Gamma^{\Mod}$, we also write $\gamma(D)=\sum_{p\in\Gamma^{\Mod}}m_p\cdot(\gamma(p))$ which is a divisor on $\Gamma$ called the retraction $D$. Note that for each $D_1$  and $D_2$ in $\Div(\Gamma^{\Mod})$, $D_1$ and $D_2$ are linearly equivalent as divisors on $\Gamma^{\Mod}$  if and only if $\gamma(D_1)$ and $\gamma(D_2)$ are linearly equivalent as divisors on $\Gamma$. 
\item Consider a harmonic morphism $\phi:\Gamma\to\Gamma'$. Then any divisor $D$ on $\Gamma'$ can be naturally pulled back to a divisor $\phi^*(D)=\sum_{p\in\Gamma}\deg_p(\phi) D(\phi(p))\cdot(p)$ on $\Gamma$. It can be easily verified that $\deg(\phi^*(D))=\deg(\phi)\deg(D)$. 

\end{enumerate}

\end{remark}

Now let us focus on pseudo-harmonic morphisms and harmonic morphisms to metric trees. 
The following lemma says that there is not much difference between pseudo-harmonic morphisms and harmonic morphisms to metric trees. 

\begin{lemma} \label{L:Pseudo}
For a pseudo-harmonic morphism $\phi:\Gamma\to\ T$ where $T$ is a metric tree, there exists a harmonic morphism $\phi^{\Mod}:\Gamma^{\Mod}\to\ T$ where $\Gamma^{\Mod}$ is a modification of $\Gamma$ such that the restriction of $\phi^{\Mod}$ to $\Gamma$ is exactly $\phi$. 
\end{lemma}
\begin{proof}
For a point $p\in \Gamma$, if the balancing condition in Definition~\ref{D:Harmonic}(2) is not satisfied, i.e., the integer $d_{p,\vt'}=\sum_{\vt \in \Tan_\Gamma(p),~\vt \mapsto \vt'}d_\vt(\phi)$ is not independent of  $\vt' \in \Tan_{\Gamma'}(\phi(p))$, then we ``modify'' $\Gamma$ as follows: 
\begin{enumerate}
\item Let $d_p:=\max_{\vt' \in \Tan_{\Gamma'}(\phi(p))}(d_{p,\vt'})$.
\item Let $\mcalU$ be the set of connected components of $T\setminus\{\phi(p)\}$. Note that elements in $\mcalU$ is in one-to-one correspondence to elements in $\Tan_\Gamma(p)$. 
\item For each $\vt' \in \Tan_{\Gamma'}(\phi(p))$, if $d_{p,\vt'}<d_p$, then attach $d_p-d_{p,\vt'}$ copies of the connected component $U\in \mcalU$ corresponding to $\vt'$ to $\Gamma$ at point $p$ as extra branches. 
\item The natural identification of each extra branch with the connected component below is incorporated into the morphism from the modification of $\Gamma$ to $T$. 
\end{enumerate}
After these operations at $p$, we derive a new morphism which is balanced at $p$ of degree $d_p$. Note that there can only be finitely many unbalanced points. Applying the above procedure to all of them, we can derive a new metric graph $\Gamma^{\Mod}$  which is a modification of $\Gamma$ and a harmonic morphism $\phi^{\Mod}:\Gamma^{\Mod}\to\ T$ such that the restriction of $\phi^{\Mod}$ to $\Gamma$ is exactly $\phi$. 
\end{proof}

\begin{remark}
The modification in the above proof is not necessarily unique and can be quite flexible. For example, in Step (3), we may instead just attach a single branch to $\Gamma$ at $p$ which is a scaling of $U$ by a factor of $1/(d_p-d_{p,\vt'})$. Actually in Proposition~6.2 of \cite{LM18}, a pseudo-harmonic morphism of a metrized complex $\mfrakC(\Gamma)$ to a genus-$0$ metrized complex $\mfrakC(T)$ can be extended to a harmonic morphism from a modification of $\mfrakC(\Gamma)$ to $\mfrakC(T)$ where the corresponding modification of the underlying metric graph $\Gamma$ can be subtly adjusted to respect the the finite morphisms between the associated algebraic curves. 
\end{remark}

\begin{remark}
We will call the harmonic morphism $\phi^{\Mod}:\Gamma^{\Mod}\to\ T$ an \emph{extension} of the pseudo-harmonic morphism in Lemma~\ref{L:Pseudo}. Note that depending on different allowable modifications of $\Gamma$, a pseudo-harmonic morphism can have various extensions to harmonic morphisms. 
\end{remark}

Lemma~\ref{L:Pseudo} tells us that pseudo-harmonic morphisms and harmonic morphisms to metric trees are strongly correlated. The following theorem says that pseudo-harmonic morphisms to metric trees are exactly the reduced divisor maps to dominant tropical trees discussed in the previous subsection.

\begin{theorem} \label{T:RedHarm}
\begin{enumerate}
\item For each dominant tropical tree $T\subseteq\DivPlusD(\Gamma)$, if $T$ is treated as a metric tree,  the reduced divisor map $\Red_T:\Gamma\to T$ is a pseudo-harmonic morphism which can be extended to a degree-$d$ harmonic morphism. 

\item For each pseudo-harmonic morphism $\phi:\Gamma\to\ T$ where $T$ is a metric tree which which can be extended to a degree-$d$ harmonic morphism, $T$ can be isometrically embedded into a degree-$d$ complete linear system as a  dominant tropical tree and $\phi$ coincides with the reduced divisor map $\Red_T$.
\end{enumerate}
\end{theorem}

\begin{proof}
For (1), let $D_1,\cdots, D_m$ be the extremals of $T$. Note that the rational function $f_{D_i-D_j}$ is a piecewise linear function with integral slopes (possibly zero at certain points) for each $i,j=1,\cdots,m$. Let $\Gamma_{ij}=\Red_T^{-1}([D_i,D_j])$. Note that  $\Red_{[D_i,D_j]}=\pi_{[D_i,D_j]}\circ \Red_T$ by Proposition~\ref{P:SeqRedjDiv} and $\pi_{[D_i,D_j]}$ restricted to $T$ is exactly the natural retraction from $T$ to $[D_i,D_j]$ by Lemma~\ref{L:ConnCompRed}. Hence $\Red_T\mid_{\Gamma_{ij}}=\Red_{[D_i,D_j]}\mid_{\Gamma_{ij}}$. Recall that by   Corollary~\ref{C:RedFunc}, $\Red_{[D_i,D_j]}=P_{D_j-D_i} \circ \underline{f_{D_j-D_i}}$ where the tropical path $P_{D_j-D_i}$ is an isometry from the segment $[0,\rho(D_i,D_j)]$ to the tropical segment $[D_i,D_j]$ and $f_{D_j-D_i}$ is a rational function associated to $D_j-D_i$. 

Consider a point $p\in\Gamma$ and a tangent direction $\vt\in \Tan_\Gamma(p)$. Since $T$  is a dominant tropical tree which means that $\Red_T$ is a continuous finite surjection with $p\in\supp(\Red_T(p))$ for all $p\in\Gamma$ (Proposition~\ref{P:DomTreeCrit}),  there is a tangent direction $\vt' \in \Tan_T(\Red_T(p))$ which is the pushforward of $\vt$ by $\Red_T$. Therefore we can choose a tropical path from $D_i$ to $D_j$ which goes through $\Red_T(p)$ along $\vt'$. Note that $p\in\Gamma_{ij}$ and $\Red_T$ coincides with $\Red_{[D_i,D_j]}=P_{D_j-D_i} \circ \underline{f_{D_j-D_i}}$ over $\Gamma_{ij}$. Letting the expansion factor $d_\vt(\Red_T)$ of $\Red_T$ at $p$ along $\vt$  be the slope of $f_{D_j-D_i}$ at $p$ along $\vt$ (which is always a positive integer), we conclude that $\Red_T$ is a pseudo-harmonic morphism from $\Gamma$ to $T$.  

Now let us  extend $\Red_T$ to a degree-$d$ harmonic morphism. Recall that by Lemma~\ref{L:RedComp} and Remark~\ref{R:tauq}, for each $p\in \Gamma$ and  $U\in \mcalU_T(\Red_T(p))$ where  $\mcalU_T(\Red_T(p))$ is the set of connected components of $T\setminus \Red_T(p)$, all the divisors in $U$ take the same value at $p$ and therefore we can define a function $\tau_p$ on $\mcalU_T(\Red_T(p))$ such that for each $U\in\mcalU_T(\Red_T(p))$,  $\tau_p(U)$ is the value of a divisor on $U$ at $q$. Moreover, we say $\tau_p$ is trivial if $\tau_p(U)=0$ for all $U\in \mcalU_T(\Red_T(p))$, and there are only finitely many $p\in\Gamma$ with nontrivial $\tau_p$. Since $T$ is a dominant tropical tree, the divisor $\Red_T(p)$ takes value at least $1$ at $q$. 

We will get a modification $\Gamma^{\Mod}$ of $\Gamma$ as follows. 
\begin{enumerate}
\item First we note that for each point $p\in\Gamma$ and $\vt' \in \Tan_{\Gamma'}(\phi(p))$,  the integer $d_{p,\vt'}=\sum_{\vt \in \Tan_\Gamma(p),~\vt \mapsto \vt'}d_\vt(\Red_T(p))$ is exactly $\Red_T(p)(p)-\tau_p(U_{\vt'})$ where $\Red_T(p)(p)$ is the value of the $p$-reduced divisor $\Red_T(p)$ at $p$ and $U_{\vt'}$ is the connected component of $T\setminus \Red_T(p)$ corresponding to $\vt'$. An interpretation of this is that the chip-firing move along $\vt'$ consumes $d_{p,\vt'}=\Red_T(p)(p)-\tau_p(U_{\vt'})$ chips at $p$. 
\item The above argument implies that $\Red_T$ is harmonic at $p$ if and only if $\tau_p$ is trivial or a constant function. That is, for each possible effective chip-firing from $\Red_T(p)$ inside $T$, all chips at $\Red_T(p)$ are consumed. Therefore, the degree $\deg_p(\Red_T)$ of $\Red_T$ at $p$ is exactly $\Red_T(p)(p)$. 
\item We will modify $\Gamma$ at all those points $p$ with nontrivial $\tau_p$. Note that this is a little different from what we have done in the proof of Lemma~\ref{L:Pseudo} where modifications are taken only to non-harmonic points (recall that $\tau_p$ being a constant function also implies that $\Red_T$ is harmonic at $p$). 
\item Let $p$ be a point with nontrivial $\tau_p$. For each $\vt' \in \Tan_{\Gamma'}(\phi(p))$, $U_{\vt'}$ is the connected component of $T\setminus \Red_T(p)$ corresponding to $\vt'$. We attach $\tau_p(U_{\vt'})$ copies of $U_{\vt'}$ to $\Gamma$ at $p$ as extra branches. 
\item After such attachments for all $p$ with nontrivial $\tau_p$, we obtain a modification $\Gamma^{\Mod}$ of $\Gamma$. Incorporating the natural identification of each extra branch with the connected component below, we extend the reduced divisor map $\Red_T$ to a new pseudo-harmonic morphism $\Red_T^{\Mod}:\Gamma^{\Mod}\to T$. 
\item $\Red_T^{\Mod}$ is then harmonic at all points $p\in\Gamma$ of degree $\Red_T(p)(p)$ at $p$. Since $\Red_T^{\Mod}$ is automatically harmonic at the points in the newly attached branches of degree $1$, $\Red_T^{\Mod}$ is overall a harmonic morhpism and the degree of $\Red_T^{\Mod}$ is identical to $d$ which is the degree of $T$. 
\end{enumerate}

For (2), extend the pseudo-harmonic morphism  $\phi:\Gamma\to\ T$  to a harmonic morphism $\phi^{\Mod}:\Gamma^{\Mod}\to\ T$ of degree $d$.  Let $\gamma$ be the natural retraction of divisors on $\Gamma^{\Mod}$ to divisors on $\Gamma$. Choose an arbitrary point $x\in T$, let $D_x^{\Mod}$ be the pullback divisor of the divisor $(x)$ on $T$ by $\phi^{\Mod}$ and let $D_x=\gamma(D_x^{\Mod})$. Then $D_x$ is an effective divisor on $\Gamma$ of degree $d$ and clearly $p\in\supp(D_{\phi(p)})$ for all $p\in\Gamma$. For every two points $x_1,x_2\in T$, denote the segment connecting $x_1$ and $x_2$ in $T$ by $[x_1,x_2]_T$ and let $\rho$ be the distance between $x_1$ and $x_2$ in $T$. Claim that the tropical segment $[D_{x_1},D_{x_2}]$ is exactly $\{D_x\mid x\in[x_1,x_2]_T\}$ which will imply that $T$ is a tropical tree with each $x\in T$ identified with $D_x$. By sending each $p\in\Gamma$ to the distance between $x$ and $x_1$ where $x$ is the retraction of $\phi(p)$ to the segment $[x_1,x_2]_T$, we can derive a rational function $f$ on $\Gamma$. It can be easily verified that $\divf(f)=D_{x_2}-D_{x_1}$ and $P_{D_{x_2}-D_{x_1}}(t)=D_x$ where $t\in[0,\rho]$ and $x$ is the unique point of distance $t$ to $x_1$ in the segment $[x_1,x_2]_T$. Now $T$ is a tropical tree and $p\in\supp(D_{\phi(p)})$ for all $p\in\Gamma$. Therefore, by Proposition~\ref{P:CritRedDivMap}, we conclude that $T$ is a dominant tropical tree and $\phi$ is exactly the reduced divisor map to $T$. 
\end{proof}

\begin{example}
In Figure~\ref{F:RedHarm}, we give two examples of reduced divisor maps with extensions to harmonic morphisms. 
In Example~\ref{E:RedMap}, we have discussed the reduced divisor map $\Red_{[D_1,D_3]}$ to the tropical segment $[D_1,D_3]$ and the reduced divisor map $\Red_T$ to the tropical tree $T=\tconv(\{D_{12},D_{23},D_{13}\})$. Note that $[D_1,D_3]$ and $T$ are both dominant tropical trees. Thus $\Red_{[D_1,D_3]}$  and $\Red_T$ are both pseudo-harmonic. 
\begin{enumerate}
\item As shown in  Figure~\ref{F:RedHarm}(a), the expansion factor of $\Red_{[D_1,D_3]}$ along the segment $v_1w_{13}v_3$ is $2$ and the expansion factor of $\Red_{[D_1,D_3]}$ along the segment $v_1v_2v_3$ is $1$, since the chip-firing from $D_1$ to $D_3$ moves two chips along $v_1w_{13}v_3$ and one chip along $v_1v_2v_3$. $\Red_{[D_1,D_3]}$ is harmonic at all points and thus is a harmonic morphism.
\item The reduced divisor map $\Red_T$ is illustrated in Figure~\ref{F:RedHarm}(b) (middle panel). By analyzing the chip-firing moves from $D_0$ to $D_{12}$, from $D_0$ to $D_{23}$ and from $D_0$ to $D_{13}$, we see that the expansion factors along all segments $v_1w_{12}$, $v_1w_{13}$, $v_2w_{12}$, $v_2w_{23}$, $v_3w_{13}$ and $v_3w_{23}$ are identical to $1$. Note that $\Red_T^{-1}(D_0)=\{v_1,v_2,v_3\}$.  Therefore, the only non-harmonic points are $v_1$, $v_2$ and $v_3$. The right panel of Figure~\ref{F:RedHarm}(b) shows an extension of $\Red_T$ to a degree-$3$ harmonic morphism. The extra attachments in the modification of $\Gamma$ include an attachment of  a copy of $D_0D_{23}$ at $v_1$, an attachment of  a copy of $D_0D_{13}$ at $v_2$ and an attachment of  a copy of $D_0D_{12}$ at $v_3$. 
\end{enumerate}
\end{example}

\subsection{Stable Gonality and  Geometric Rank Functions}
The gonality of an algebraic curve $C$ has several interpretations such as the minimum degree of a rank-$1$ linear system on $C$ or the minimum degree of a morphism from $C$ to the projective line $\mbbP^1$. Accordingly, there are several types of gonality defined on finite graphs or metric graphs, e.g., the divisorial gonality \cite{Baker08} and the stable gonality \cite{CKK15}. However, unlike curve gonality, these definitions of gonality in the case of finite graphs or metric graphs are not equivalent \cite{Caporaso14}. 

In their foundational paper on graph-theoretical Riemann-Roch theory\cite{BN07}, Baker and Norine introduced a rank function on the set of divisors for a finite graph. In the context of metric graphs, the Baker-Norine rank function or combinatorial rank function $r_c:\Div(\Gamma) \to \mbbZ$ on a metric graph $\Gamma$ is defined as follows:
for a divisor $D\in \Div(\Gamma)$, if the complete linear system $|D|$ is empty, then $r_c(D)=-1$, and  otherwise, $r_c(D)$ is the greatest non-negative integer $r$ such that for every $E\in\DivPlusR(\Gamma)$, there exists a divisor $D'\in|D|$ such that $E\leq D'$. Essentially, this combinatorial rank function can be consider to be a function on the set of complete linear systems since all divisors in a complete linear system have the same rank. The \emph{divisorial gonality} of $\Gamma$ is defined as the minimum degree $d$ such that there exists a complete linear system on $\Gamma$ of degree $d$ and rank at least $1$. 

Recall that we have defined linear systems in general as tropical polytopes contained in some complete linear systems in Definition~\ref{D:LinSys}. Therefore, the combinatorial rank function can be generalized as a function on the set of linear systems.
\begin{definition}
For  a linear system $T$ on a metric graph $\Gamma$, the combinatorial rank $r_c(T)$ of $T$ is defined as follows:
\begin{enumerate}
\item If $T$ is the empty linear system, then $r_c(T)=-1$;
\item Otherwise, $r_c(T)=\max\{r\mid\forall E\in \DivPlusR(\Gamma), \exists D'\in T\ \mbox{such that}\ E\leq D'\}$.
\end{enumerate}
\end{definition}

The following notion of stable gonality of metric graphs comes from \cite{CKK15}. 
\begin{definition} \label{D:StaGonality}
A metric graph $\Gamma$ is \emph{stably $d$-gonal} if it admits a degree-$d$ harmonic morphism from a modification $\Gamma^{\Mod}$ of $\Gamma$ to a metric tree. The \emph{stable gonality} of $\Gamma$ is the minimum degree $d$ such that there exists a harmonic morphism of degree $d$ from a modification of $\Gamma$ to a metric tree. 
\end{definition}

Now we introduce a new rank function called geometric rank function here. 
\begin{definition} \label{D:GeoRank}
For  a linear system $T$ on a metric graph $\Gamma$, the \emph{geometric rank} $r_g(T)$ of $T$ is defined as follows:
\begin{enumerate}
\item If $T$ is the empty linear system, then $r_g(T)=-1$;
\item Otherwise, $r_g(T)$ is the maximum of the integers $r$ such that there exists a map $\pi:\DivPlusR\rightarrow T$ with the following conditions satisfied:
\begin{enumerate}
  \item $\pi$ is continuous.
  \item For every $E\in\DivPlusR(\Gamma)$, $E\leqslant\pi(E)$.
  \item The image $\pi(\DivPlusR(\Gamma))$ is tropically convex.
\end{enumerate}
In particular, we say $\pi$ is a \emph{rank-$r$} map to $T$.
\end{enumerate}
Moreover, the geometric rank of a divisor $D\in\Div(\Gamma)$ is defined to be the geometric rank of the complete linear system $|D|$. 
\end{definition}

\begin{proposition} \label{P:SGon}
$\Gamma$ is stable $d$-gonal if and only if $\Gamma$ admits a degree-$d$ dominant tropical tree if and only if there exists a divisor  $D\in\Div(\Gamma)$ of degree $d$ and geometric rank $1$. 
\end{proposition}

\begin{proof}
By  Theorem~\ref{T:RedHarm}, we see that $\Gamma$ is stable $d$-gonal if and only if $\Gamma$ admits a degree-$d$ dominant tropical tree $T$. By Proposition~\ref{P:CritRedDivMap}, the reduced divisor map to $T$ is exactly a rank-$1$ map in Definition~\ref{D:GeoRank}. Therefore, the existence of a degree-$d$ dominant tropical tree is equivalent to the existence of a divisor  $D\in\Div(\Gamma)$ of degree $d$ and geometric rank $1$. 
\end{proof}

An immediate consequence of Proposition~\ref{P:SGon} is that the stable gonality is larger than or equal to the divisorial gonality, while the inequality can be strict as shown in the following example.

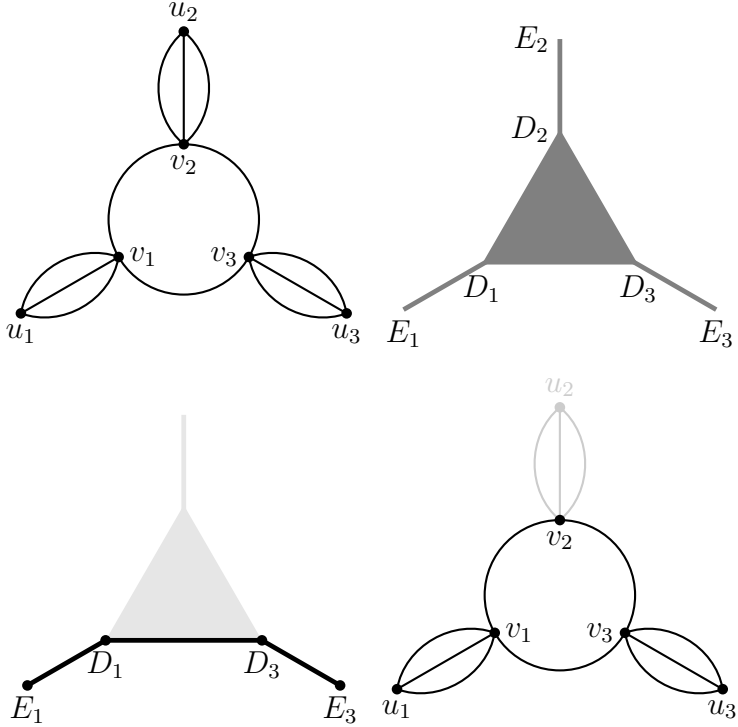
\begin{figure}
\centering
\begin{tikzpicture}
\begin{scope}
  \coordinate (center) at (0,0);
  \def\radius{1};
  \def\sca{2.5};
  \coordinate (v1) at ($(center)+(210:\radius)$);
  \coordinate (v2) at ($(center)+(90:\radius)$);
  \coordinate (v3) at ($(center)+(-30:\radius)$);
    \coordinate (u1) at ($(center)+(210:\sca*\radius)$);
  \coordinate (u2) at ($(center)+(90:\sca*\radius)$);
  \coordinate (u3) at ($(center)+(-30:\sca*\radius)$);

  \draw[thick] (center) circle[radius=\radius];

  \fill[black] (v1) circle[radius=2pt];
  \fill[black] (v2) circle[radius=2pt];
  \fill[black] (v3) circle[radius=2pt];
   \fill[black] (u1) circle[radius=2pt];
  \fill[black] (u2) circle[radius=2pt];
  \fill[black] (u3) circle[radius=2pt];

  \foreach \i in {1,2,3}
  \path [thick]
		(u\i) edge (v\i)
		(u\i) edge [bend left=50] (v\i)
		(u\i) edge [bend right=50] (v\i);
  
  \draw [anchor=west] (v1) node {$v_1$};
    \draw [anchor=north] (v2) node {$v_2$};
    \draw [anchor=east] (v3) node {$v_3$};
      \draw [anchor=north] (u1) node {$u_1$};
    \draw [anchor=south] (u2) node {$u_2$};
    \draw [anchor=north] (u3) node {$u_3$};

\end{scope}

\begin{scope}[shift={(5,0)}]
  \coordinate (center) at (0,0);
  \def\radius{1.2};
  \def\sca{2};
  \coordinate (D1) at ($(center)+(210:\radius)$);
  \coordinate (D2) at ($(center)+(90:\radius)$);
  \coordinate (D3) at ($(center)+(-30:\radius)$);
    \coordinate (E1) at ($(center)+(210:\sca*\radius)$);
  \coordinate (E2) at ($(center)+(90:\sca*\radius)$);
  \coordinate (E3) at ($(center)+(-30:\sca*\radius)$);
  \coordinate (D12) at ($(center)+(150:\radius/2)$);
  \coordinate (D23) at ($(center)+(30:\radius/2)$);
  \coordinate (D13) at  ($(center)+(-90:\radius/2)$);
   
  \fill [black!50,opacity=1]   (D1) --(D2) -- (D3)-- (D1);
  \draw [black!50, line width = 1.8pt]  (center) -- (E1);
  \draw [black!50, line width = 1.8pt]  (center) -- (E2);
  \draw [black!50, line width = 1.8pt]  (center) -- (E3);

  \draw [anchor=north] (D1) node {$D_1$};
    \draw [anchor=east] (D2) node {$D_2$};
    \draw [anchor=north] (D3) node {$D_3$};
      \draw [anchor=north] (E1) node {$E_1$};
    \draw [anchor=east] (E2) node {$E_2$};
    \draw [anchor=north] (E3) node {$E_3$};
\end{scope}

\begin{scope}[shift={(0,-5)}]
  \coordinate (center) at (0,0);
  \def\radius{1.2};
  \def\sca{2};
  \coordinate (D1) at ($(center)+(210:\radius)$);
  \coordinate (D2) at ($(center)+(90:\radius)$);
  \coordinate (D3) at ($(center)+(-30:\radius)$);
    \coordinate (E1) at ($(center)+(210:\sca*\radius)$);
  \coordinate (E2) at ($(center)+(90:\sca*\radius)$);
  \coordinate (E3) at ($(center)+(-30:\sca*\radius)$);
  \coordinate (D12) at ($(center)+(150:\radius/2)$);
  \coordinate (D23) at ($(center)+(30:\radius/2)$);
  \coordinate (D13) at  ($(center)+(-90:\radius/2)$);
   
  \fill [black!10,opacity=1]   (D1) --(D2) -- (D3)-- (D1);

  \draw [black!10, line width = 1.8pt]  (center) -- (E2);
  \draw [black, line width = 1.8pt]  (E1) -- (D1) -- (D3) -- (E3);

  \fill (D1) circle[radius=2pt];
  \fill (D3) circle[radius=2pt];
    \fill (E1) circle[radius=2pt];
  \fill (E3) circle[radius=2pt];

  \draw [anchor=north] (D1) node {$D_1$};
    \draw [anchor=north] (D3) node {$D_3$};
      \draw [anchor=north] (E1) node {$E_1$};

    \draw [anchor=north] (E3) node {$E_3$};
\end{scope}

\begin{scope}[shift={(5,-5)}]
  \coordinate (center) at (0,0);
  \def\radius{1};
  \def\sca{2.5};
  \coordinate (v1) at ($(center)+(210:\radius)$);
  \coordinate (v2) at ($(center)+(90:\radius)$);
  \coordinate (v3) at ($(center)+(-30:\radius)$);
    \coordinate (u1) at ($(center)+(210:\sca*\radius)$);
  \coordinate (u2) at ($(center)+(90:\sca*\radius)$);
  \coordinate (u3) at ($(center)+(-30:\sca*\radius)$);
  
\path [black!20,thick]
		(u2) edge (v2)
		(u2) edge [bend left=50] (v2)
		(u2) edge [bend right=50] (v2);
		
\fill[black!20] (u2) circle[radius=2pt];
 \draw [black!20,anchor=south] (u2) node {$u_2$};
   
  \draw[thick] (center) circle[radius=\radius];

  \fill[black] (v1) circle[radius=2pt];
  \fill[black] (v2) circle[radius=2pt];
  \fill[black] (v3) circle[radius=2pt];
   \fill[black] (u1) circle[radius=2pt];
  \fill[black] (u3) circle[radius=2pt];
  
  \foreach \i in {1,3}
  \path [thick]
		(u\i) edge (v\i)
		(u\i) edge [bend left=50] (v\i)
		(u\i) edge [bend right=50] (v\i);
  
  \draw [anchor=west] (v1) node {$v_1$};
    \draw [anchor=north] (v2) node {$v_2$};
    \draw [anchor=east] (v3) node {$v_3$};
      \draw [anchor=north] (u1) node {$u_1$};
    \draw [anchor=north] (u3) node {$u_3$};

\end{scope}
\end{tikzpicture} 
\caption{An example of a metric graph of divisorial gonality of $3$ and stable gonality $4$. A complete linear system $|D|$ whose combinatorial rank is $1$ and geometric rank is $0$ is shown. Divisors $D_1=3(v_1)$, $D_2=3(v_2)$, $D_3=3(v_3)$, $E_1=3(u_1)$, $E_2=3(u_2)$ and $E_3=3(u_3)$ are all elements of $|D|$.} \label{F:gonality}
\end{figure}

\begin{example}
Figure~\ref{F:gonality} shows an example of a metric graph $\Gamma$ of divisorial gonality $3$ and stable gonality $4$. This example also appears in  \cite{ABBR15_2} (Example~5.13). $\Gamma$ is genus-$7$ metric graph made of a circle attached with three banana graphs of genus $2$. We let $u_1$, $u_2$, $u_3$, $v_1$, $v_2$ and $v_3$ be the vertices of $\Gamma$ with all edge lengths being identical. We consider the linearly equivalent divisors $D_1=3(v_1)$, $D_2=3(v_2)$, $D_3=3(v_3)$, $E_1=3(u_1)$, $E_2=3(u_2)$ and $E_3=3(u_3)$.  In Figure~\ref{F:gonality}, the complete linear system $|D|$ which contains $D_1$, $D_2$, $D_3$, $E_1$, $E_2$ and $E_3$ is shown in the upper-right panel, the tropical segment $[E_1,E_3]$ is shown in the lower-left panel with its support shown in the lower-right panel. As a tropical tree, $[E_1,E_3]$ is maximal but not dominant. On the other hand, $\supp(|D|)=\Gamma$ and it can be easily verified the $|D|$ is the only complete linear system of degree at most $3$ and combinatorial rank at least $1$. Therefore, the divisorial gonality of $\Gamma$ is $3$. On the other hand, $|D|$ contains no degree-$3$ dominant tropical tree, which means that the geometric rank of $|D|$ is $0$ and the stable gonality is strictly larger than $3$. Actually, it can be verified  that $\tconv(\{E_1+(u_2),E_3+(u_2), 2(w_{13})+2(x_1), 2(w_{13})+2(x_2), 2(w_{13})+2(x_3)\})$ is a tropical dominant tree of degree $4$ where $w_{13}$ is the middle point of $v_1v_3$, and $x_i$'s are the middle points of the three edges between $u_2$ and $v_2$ respectively. Therefore the stable gonality is $4$ by Proposition~\ref{P:SGon}. 
\end{example}

Lastly, we note that a dominant tropical tree or the image of a rank-$1$ map in Definition~\ref{D:GeoRank} is the degeneration of a linear series of rank $1$ on an algebraic curve $C$ which degenerates to $\Gamma$. A conjecture is that in general, the image of a rank-$r$ map is the degeneration of a linear series of rank $r$ on $C$.

\section*{Acknowledgments}
We are very thankful to Matthew Baker for his encouragement and support of broadening our original results of tropical convexity in the context of metric graphs \cite{Luo13} to the more general context of analysis. 

\bibliographystyle{alpha}
\bibliography{citation}

\newcommand{\noopsort}[1]{} \newcommand{\printfirst}[2]{#1}
  \newcommand{\singleletter}[1]{#1} \newcommand{\switchargs}[2]{#2#1}
\begin{thebibliography}{ABBR15b}

\bibitem[ABBR15a]{ABBR15}
Omid Amini, Matthew Baker, Erwan Brugall{\'e}, and Joseph Rabinoff.
\newblock Lifting harmonic morphisms {I}: metrized complexes and {B}erkovich
  skeleta.
\newblock {\em Res. Math. Sci.}, 2:Art. 7, 2015.

\bibitem[ABBR15b]{ABBR15_2}
Omid Amini, Matthew Baker, Erwan Brugall{\'e}, and Joseph Rabinoff.
\newblock Lifting harmonic morphisms {II}: tropical curves and metrized
  complexes.
\newblock {\em Algebra \& Number Theory}, 9(2):267--315, 2015.

\bibitem[AGG09]{AGG09}
Marianne Akian, St{\'e}phane Gaubert, and Alexander Guterman.
\newblock Linear independence over tropical semirings and beyond.
\newblock {\em Contemporary Mathematics}, 495:1, 2009.

\bibitem[Ami13]{Amini13}
Omid Amini.
\newblock Reduced divisors and embeddings of tropical curves.
\newblock {\em Transactions of the American Mathematical Society},
  365(9):4851--4880, 2013.

\bibitem[Bak08]{Baker08}
Matthew Baker.
\newblock Specialization of linear systems from curves to graphs.
\newblock {\em Algebra \& Number Theory}, 2(6):613--653, 2008.

\bibitem[Ber12]{Berkovich12}
Vladimir~G Berkovich.
\newblock {\em Spectral theory and analytic geometry over non-{A}rchimedean
  fields}.
\newblock 33. American Mathematical Soc., 2012.

\bibitem[BF06]{BF06}
Matthew Baker and Xander Faber.
\newblock Metrized graphs, {L}aplacian operators, and electrical networks.
\newblock {\em Contemporary Mathematics}, 415:15--34, 2006.

\bibitem[BN07]{BN07}
Matthew Baker and Serguei Norine.
\newblock Riemann--{R}och and {A}bel--{J}acobi theory on a finite graph.
\newblock {\em Advances in Mathematics}, 215(2):766--788, 2007.

\bibitem[BN09]{BN09}
Matthew Baker and Serguei Norine.
\newblock Harmonic morphisms and hyperelliptic graphs.
\newblock {\em International Mathematics Research Notices},
  2009(15):2914--2955, 2009.

\bibitem[BPR13]{BPR13}
Matthew Baker, Sam Payne, and Joseph Rabinoff.
\newblock On the structure of nonarchimedean analytic curves.
\newblock In {\em Proceedings of the 2011 Bellairs Workshop in Number Theory},
  2013.

\bibitem[BPR16]{BPR16}
Matthew Baker, Sam Payne, and Joseph Rabinoff.
\newblock Nonarchimedean geometry, tropicalization, and metrics on curves.
\newblock {\em Algebr. Geom.}, 3:63--105, 2016.

\bibitem[BR07]{BR07}
Matthew Baker and Robert~S Rumely.
\newblock Harmonic analysis on metrized graphs.
\newblock {\em Cand. J. Math.}, 59:225--275, 2007.

\bibitem[BS13]{BS13}
Matthew Baker and Farbod Shokrieh.
\newblock Chip-firing games, potential theory on graphs, and spanning trees.
\newblock {\em Journal of Combinatorial Theory, Series A}, 120(1):164--182,
  2013.

\bibitem[But10]{Butkovivc10}
Peter Butkovi{\v{c}}.
\newblock {\em Max-linear systems: theory and algorithms}.
\newblock Springer Science \& Business Media, 2010.

\bibitem[Cap14]{Caporaso14}
Lucia Caporaso.
\newblock Gonality of algebraic curves and graphs.
\newblock In {\em Algebraic and complex geometry}, pages 77--108. Springer,
  2014.

\bibitem[CDPR12]{CDPR12}
Filip Cools, Jan Draisma, Sam Payne, and Elina Robeva.
\newblock A tropical proof of the {B}rill--{N}oether theorem.
\newblock {\em Adv. Math.}, 230(2):759--776, 2012.

\bibitem[CG79]{CR79}
Raymond~A Cuninghame-Green.
\newblock Minimax algebra, vol. 166 of lecture notes in economics and
  mathematical systems, 1979.

\bibitem[CKK15]{CKK15}
Gunther Cornelissen, Fumiharu Kato, and Janne Kool.
\newblock A combinatorial {L}i--{Y}au inequality and rational points on curves.
\newblock {\em Mathematische Annalen}, 361(1-2):211--258, 2015.

\bibitem[CLM15]{Caporaso15}
Lucia Caporaso, Yoav Len, and Margarida Melo.
\newblock Algebraic and combinatorial rank of divisors on finite graphs.
\newblock {\em Journal de Math{\'e}matiques Pures et Appliqu{\'e}es},
  104(2):227--257, 2015.

\bibitem[CP18]{CP18}
Scott Corry and David Perkinson.
\newblock {\em Divisors and sandpiles: An Introduction to chip-firing}, volume
  114.
\newblock American Mathematical Soc., 2018.

\bibitem[Dha90]{Dhar90}
Deepak Dhar.
\newblock Self-organized critical state of sandpile automaton models.
\newblock {\em Physical Review Letters}, 64(14):1613, 1990.

\bibitem[DS04]{DS04}
Mike Develin and Bernd Sturmfels.
\newblock Tropical convexity.
\newblock {\em Doc. Math}, 9(1-27):7--8, 2004.

\bibitem[GK08]{GK08}
Andreas Gathmann and Michael Kerber.
\newblock A {R}iemann--{R}och theorem in tropical geometry.
\newblock {\em Math. Z.}, 259(1):217--230, 2008.

\bibitem[GM84]{GM84}
Michel Gondran and Michel Minoux.
\newblock Linear algebra in dioids: a survey of recent results.
\newblock In {\em North-Holland mathematics studies}, volume~95, pages
  147--163. Elsevier, 1984.

\bibitem[HKN13]{HKN13}
Jan Hladk{\`y}, Daniel Kr{\'a}l, and Serguei Norine.
\newblock Rank of divisors on tropical curves.
\newblock {\em Journal of Combinatorial Theory, Series A}, 120(7):1521--1538,
  2013.

\bibitem[HMY12]{HMY12}
Christian Haase, Gregg Musiker, and Josephine Yu.
\newblock Linear systems on tropical curves.
\newblock {\em Mathematische Zeitschrift}, 270(3-4):1111--1140, 2012.

\bibitem[JP14]{JP14}
David Jensen and Sam Payne.
\newblock Tropical independence {I}: shapes of divisors and a proof of the
  {G}ieseker--{P}etri theorem.
\newblock {\em Algebra Number Theory}, 8(9):2043--2066, 2014.

\bibitem[JP16]{JP16}
David Jensen and Sam Payne.
\newblock Tropical independence {II}: The maximal rank conjecture for quadrics.
\newblock {\em Algebra \& Number Theory}, 10(8):1601--1640, 2016.

\bibitem[Kag18]{Kageyama18}
Yuki Kageyama.
\newblock Divisorial condition for the stable gonality of tropical curves.
\newblock {\em arXiv preprint arXiv:1801.07405}, 2018.

\bibitem[KM97]{KM97}
Vasily Kolokoltsov and Victor~P Maslov.
\newblock {\em Idempotent analysis and its applications}, volume 401.
\newblock Springer Science \& Business Media, 1997.

\bibitem[Kol94]{K94}
Vasily Kolokoltsov.
\newblock Idempotent structures in optimisation.
\newblock {\em Applied Mathematics}, 48:45--68, 1994.

\bibitem[KS01]{KS01}
Maxim Kontsevich and Yan Soibelman.
\newblock Homological mirror symmetry and torus fibrations.
\newblock In {\em Symplectic geometry and mirror symmetry}, pages 203--263.
  World Scientific, 2001.

\bibitem[LM98]{LM98}
Grigori{\u\i}~L Litvinov and Viktor~P Maslov.
\newblock The correspondence principle for idempotent calculus and some
  computer applications.
\newblock {\em Idempotency}, 11:420--443, 1998.

\bibitem[LM05]{LM05}
Grigori{\u\i}~L Litvinov and Viktor~P Maslov.
\newblock {\em Idempotent mathematics and mathematical physics: international
  workshop, {F}ebruary 3-10, 2003, {E}rwin {S}chr{\"o}dinger {I}nternational
  {I}nstitute for {M}athematical {P}hysics, {V}ienna, {A}ustria}, volume 377.
\newblock American Mathematical Soc., 2005.

\bibitem[LM18]{LM18}
Ye~Luo and Madhu Manjunath.
\newblock Smoothing of limit linear series of rank one on saturated metrized
  complexes of algebraic curves.
\newblock {\em Cand. J. Math.}, 70:628--682, 2018.

\bibitem[Luo11]{Luo11}
Ye~Luo.
\newblock Rank-determining sets of metric graphs.
\newblock {\em Journal of Combinatorial Theory, Series A}, 118(6):1775--1793,
  2011.

\bibitem[Luo13]{Luo13}
Ye~Luo.
\newblock Tropical convexity and canonical projections.
\newblock {\em arXiv preprint arXiv:1304.7963}, 2013.

\bibitem[Mik05]{Mikhalkin05}
Grigory Mikhalkin.
\newblock Enumerative tropical algebraic geometry in $\mathbb{R}^2$.
\newblock {\em Journal of the American Mathematical Society}, 18(2):313--377,
  2005.

\bibitem[MS15]{MS15}
Diane Maclagan and Bernd Sturmfels.
\newblock {\em Introduction to tropical geometry}, volume 161.
\newblock American Mathematical Soc., 2015.

\bibitem[PS04]{PS04}
Alexander Postnikov and Boris Shapiro.
\newblock Trees, parking functions, syzygies, and deformations of monomial
  ideals.
\newblock {\em Transactions of the American Mathematical Society},
  356(8):3109--3142, 2004.

\bibitem[RGST05]{RST05}
Jurgen Richter-Gebert, Bernd Sturmfels, and Thorsten Theobald.
\newblock First steps in tropical geometry.
\newblock {\em Contemporary Mathematics}, 377:289--318, 2005.

\bibitem[Vir08]{Viro08}
Oleg Viro.
\newblock From the sixteenth {H}ilbert problem to tropical geometry.
\newblock {\em Japanese Journal of Mathematics}, 3(2):185--214, 2008.

\end{thebibliography}

\appendix
\section{Potential Theory on Metric Graphs} \label{S:potential}
We list here some standard terminologies and basic facts concerning potential theory on metric graphs. The reader may refer \cite{BF06,BR07,BS13} for details.

For a metric graph $\Gamma$, we let $C(\Gamma)$ be the $\mbbR$-algebra of continuous real-valued functions on $\Gamma$, and let $CPA(\Gamma)\subset C(\Gamma)$ be the vector space consisting of all continuous piecewise-affine  functions on $\Gamma$. Note that $CPA(\Gamma)$ is dense in $C(\Gamma)$. Let $\measure(\Gamma)$ be the vector space of finite signed Borel measures of total mass zero on $\Gamma$. Denote by $\mbbR\in C(\Gamma)$ the space of constant functions on $\Gamma$.

In terms of electric network theory, we may think of $\Gamma$ as an electrical network with resistances given by the edge lengths. For $p,q,x\in \Gamma$, we define a $j$-function $j_q(x,p)$ as the potential at $x$ when one unit of current enters the network at $p$ and exits at $q$ with $q$ grounded (potential $0$).

We summarize some properties of  $j$-functions as follows.
\begin{enumerate}
\item $j_q(x,p)$ is jointly continuous in $p$, $q$ and $x$.
\item $j_q(x,p)\in\CPA$.
\item $j_q(q,p)=0$.
\item $0\leqslant j_q(x,p) \leqslant j_q(p,p)$.
\item $j_q(x,p)=j_q(p,x)$.
\item $j_q(x,p)+j_p(x,q)$ is constant for all $x\in\Gamma$. Denoted by $r(p,q)$, this constant is the effective resistance between $p$ and $q$.
\item $r(p,q)=j_q(p,p)=j_p(q,q)$.
\item $r(p,q)\leqslant\dist(p,q)$ where $\dist(p,q)$ is the distance between $p$ and $q$ on $\Gamma$.
\item $\frac{r(p,q)}{\dist(p,q)}\to 1$ as $\dist(p,q)\to 0$.
\end{enumerate}

Let $BDV(\Gamma)$ be the vector space of functions of bounded differential variation \cite{BR07}. Then we have $CPA(\Gamma)\subset BDV(\Gamma)\subset C(\Gamma)$.

The Laplacian $\Delta: BDV(\Gamma)\rightarrow\measure(\Gamma)$ is defined as an operator in the following sense.
\begin{enumerate}
\item $\Delta$ induces an isomorphism between $BDV(\Gamma)/\mR$ and $\measure(\Gamma)$ as vector spaces.
\item For $f\in CPA(\Gamma)$,  we have $$\Delta f=\sum_{p\in\Gamma}\sigma_p(f)\delta_p$$ where $-\sigma_p(f)$ is the sum of the slopes of $f$ in all tangent directions emanating from $p$ and $\delta_p$ is the Dirac measure (unit point mass) at $p$. In particular, $\Delta j_q(x,p)=\delta_p(x)-\delta_q(x)$.
\item An inverse to $\Delta$ is given by $$\nu\mapsto\int_{\Gamma}j_q(x,y)d\nu(y)\in\{f\in BDV(\Gamma):f(q)=0\}.$$
\end{enumerate}

\section{Equivalence of Topologies on Metric Graphs} \label{S:EquiTop}
\begin{proposition}
On $\DivPlusD(\Gamma)$, the metric topology is the same as the induced topology as a $d$-fold symmetric product of $\Gamma$.
\end{proposition}
\begin{proof}
Denote the first topology by $\mathcal{T}_1$ and the second by $\mathcal{T}_2$. To show $\mathcal{T}_1=\mathcal{T}_2$, it suffices to show that for a divisor $D=\sum_{i=1}^d (q_i)$ with $q_i\in\Gamma$, a sequence $\{D^{(n)}\}_n$ converges to $D$ in $\mathcal{T}_2$ if and only if $\rho(D^{(n)},D)\to 0$. In addition, we note that to say $D^{(n)}\to D$ in $\mathcal{T}_2$ is equivalent to say that there exists $d$ sequences of points on $\Gamma$, $\{p_i^{(n)}\}_n$ for $i=1,\dots,d$, such that $D^{(n)}=\sum_{i=1}^d(p_i^{(n)})$ and $p_i^{(n)}\to q_i$ on $\Gamma$.

Suppose $D^{(n)}\to D$ in $\mathcal{T}_2$. Since $$\rho(D^{(n)},D)\leqslant\sum_{i=1}^d\rho((p_i^{(n)}),(q_i))=\sum_{i=1}^d r(p_i^{(n)},q_i) \leqslant\sum_{i=1}^d\dist(p_i^{(n)},q_i)$$ where $r(p_i^{(n)},q_i)$ is the effective resistance between $p_i^{(n)}$ and $q_i$ (see Section~\ref{S:potential}),
we conclude that $D^{(n)}\to D$ in $\mathcal{T}_1$.

Now suppose $D^{(n)}\to D$ in $\mathcal{T}_1$ which means $\rho(D^{(n)},D)=\max(\underline{f_{D^{(n)}-D}})\to 0$.
Considering the divisors $D$ and $D^{(n)}$, for each point $q_i\in\supp D$, we will associate a point $p_i^{(n)}\in\supp D^{(n)}$ with an procedure as follows.

Let $M$ be the maximum of degrees among all the points in $\Gamma$. This means each point $p\in\Gamma$ has at most $M$ adjacent edges. Denote the sum of slopes of $f_{D^{(n)}-D}$ for all outgoing directions from $p\in\Gamma$ by $\chi(p)$. Then $\chi(p)=D(p)-D^{(n)}(p)$. Let $V(\Gamma)$ be a vertex set of $\Gamma$.

First, we will determine $p_1^{(n)}$ for $q_1$.

If $q_1\in\supp D^{(n)}$, we let $p_1^{(n)}=q_1$.

Otherwise, we must have $\chi(q_1)\geqslant 1$ and there must be an outgoing direction $\vt_{q_1}$ from $q_1$ with a slope at least $1/M$. Let $w(q_1)\in V(\Gamma)$ be the adjacent vertex of $q_1$ in direction $\vt_{q_1}$. If there exists a point in $\supp D^{(n)}$ that lies in the half-open-half-closed segment $(q_1,w(q_1)]$, then we let $p_1^{(n)}$ be this point. Clearly, $f_{D^{(n)}-D}(q_1)<f_{D^{(n)}-D}(p_1^{(n)})$ and $\dist(p_1^{(n)},q_1)\leqslant M\cdot\rho(D^{(n)},D)$ in this case.

Otherwise, we must have $\chi(w(q_1))\geqslant 0$. Since the slope of the outgoing direction from $w(q_1)$ to $q_1$ is at most $-1/M$, the sum of slopes in the remaining outgoing directions from $w(q_1)$ is at least $1/M$ and there must be an outgoing direction $\vt_{w(q_1)}$ from $w(q_1)$ with a slope at least $1/(M(M-1))$. Let $w^2(q_1)\in V(\Gamma)$ be the adjacent vertex of $w(q_1)$ in direction $\vt_{w(q_1)}$. Following the same procedure, we let $p_1^{(n)}$ be a point contained in both $\supp D^{(n)}$ and $(w(q_1),w^2(q_1)]$ if their intersection is nonempty, and otherwise keep seeking $p_1^{(n)}$ in the next outgoing direction from $w^2(q_1)$ with slope at least $1/(M(M-1)^2)$.

The procedure must terminate in finitely many steps since we only have finitely many elements in $V(\Gamma)$. Let $N=|V(\Gamma)|$. We conclude that
we can find $p_1^{(n)}$ within $N$ steps and $\dist(p_1^{(n)},q_1)\leqslant C_1\cdot\rho(D^{(n)},D)$ where $C_1=M(M-1)^N$.

Next we will determine $p_i^{(n)}$ one by one inductively. Suppose for $i=2,\ldots,d'$ ($d'<d$), we have determined $p_i^{(n)}$ and known that $\dist(p_i^{(n)},q_i)\leqslant C_i\cdot\rho(D^{(n)},D)$ where $C_i$'s are constants. We let $D^{(n)}_{d'}=D^{(n)}-\sum_{i=1}^{d'}(p_i^{(n)})$ and
$D_{d'}=D-\sum_{i=1}^{d'}(q_i)$. Then
\begin{align*}
\rho(D^{(n)}_{d'},D_{d'})&\leqslant\rho(D^{(n)},D)+\sum_{i=1}^{d'}r(p_i^{(n)},q_i) \\
&\leqslant\rho(D^{(n)},D)+\sum_{i=1}^{d'}\dist(p_i^{(n)},q_i) \\
&=(1+\sum_{i=1}^{d'}C_i)\rho(D^{(n)},D).
\end{align*}
Following exactly the same procedure we used to seek $p_1^{(n)}$, we can find $p_{d'+1}^{(n)}\in\supp D^{(n)}_{d'}$ such that
$$\dist(p_{d'+1}^{(n)},q_{d'+1})\leqslant C_1\cdot\rho(D^{(n)}_{d'},D_{d'})=C_{d'+1}\cdot\rho(D^{(n)},D)$$ where $C_{d'+1}=C_1(1+\sum_{i=1}^{d'}C_i)$.

In this way, for each $D^{(n)}$, we can find $p_i^{(n)}$ such that $D^{(n)}=\sum_{i=1}^d(p_i^{(n)})$ and $\dist(p_i^{(n)},q_i)$ is bounded by $C_i\cdot\rho(D^{(n)},D)$. This means $D^{(n)}\rightarrow D$ in $\mathcal{T}_1$ implies $D^{(n)}\rightarrow D$ in $\mathcal{T}_2$.
\end{proof}

\end{document}